\documentclass[a4paper,11pt]{article}
\usepackage[utf8]{inputenc}
\usepackage[T1]{fontenc}
\usepackage{amsmath,amsthm,amssymb}
\usepackage{mathtools}
\usepackage[usenames]{xcolor}
\usepackage{stmaryrd}\SetSymbolFont{stmry}{bold}{U}{stmry}{m}{n}
\usepackage{enumitem}
\usepackage{tikz}\usetikzlibrary{positioning,decorations.markings, patterns}
\usepackage{pgfplots}\pgfplotsset{compat=1.15}
\usepackage{caption, subcaption}
\usepackage[normalem]{ulem}
\usepackage{bm}

\newtheorem*{thm*}{Theorem}
\newtheorem{thm}{Theorem}[section]
\newtheorem{theorem}[thm]{Theorem}
\newtheorem{cor}[thm]{Corollary}
\newtheorem{corollary}[thm]{Corollary}
\newtheorem{prop}[thm]{Proposition}
\newtheorem{lemma}[thm]{Lemma}
\newtheorem{proposition}[thm]{Proposition}
\theoremstyle{definition}

\newtheorem{definition}[thm]{Definition}
\newtheorem{assumptions}[thm]{Assumptions}
\newtheorem{example}[thm]{Example}

\newtheorem{remark}[thm]{Remark}

\newcommand{\com}[1]{}

\newcommand{\Com}[1]{}

\newcommand{\GRN}{}
\newcommand{\FGRN}{}
\definecolor{purple}{rgb}{.5, 0, .3}

\newcommand{\Grn}[1]{{#1}}

\definecolor{burntorange}{rgb}{0.8, 0.33, 0.0}

\newcommand{\mfrac}[2]{\tfrac{#1\strut}{#2\strut}}

\newcommand{\C}{\mathbb C}
\newcommand{\CP}{\mathbb{CP}}
\newcommand{\R}{\mathbb R}

\newcommand{\Z}{\mathbb Z}
\newcommand{\Q}{\mathbb Q}
\renewcommand{\H}{\mathbb H}
\newcommand{\SL}{\operatorname{SL}}

\newcommand{\bX}{\bm X}
\newcommand{\bY}{\bm Y}
\newcommand{\bZ}{\bm Z}

\newcommand{\bt}{\bm t}
\newcommand{\bz}{\bm z}
\newcommand{\bT}{\bm T}

\newcommand{\bmu}{\bm\mu}
\newcommand{\bexp}{\operatorname{\mathbf{exp}}}

\newcommand{\bE}{\bm E}
\newcommand{\hatbXnf}{\hat{\bm X}_{\mathrm{nf}}}

\newcommand{\bXnf}{{\bm X}_{\mathrm{nf}}}

\newcommand{\bXnfeps}{{\bm X}_{\mathrm{nf},\epsilon}}
\newcommand{\bXnfS}{{\bm X}_{\mathrm{nf,S}}}

\newcommand{\hatphinf}{\hat\phi_{\mathrm{nf}}}

\newcommand{\phinf}{\phi_{\mathrm{nf}}}

\newcommand{\phinfeps}{\phi_{\mathrm{nf},\epsilon}}
\newcommand{\phinfS}{\phi_{\mathrm{nf},S}}
\newcommand{\hattaunf}{\hat\tau_{\mathrm{nf}}}
\newcommand{\taunf}{\tau_{\mathrm{nf}}}
\newcommand{\hattaunff}[1]{\hat\tau_{#1,\mathrm{nf}}}
\newcommand{\taunff}[1]{\tau_{#1,\mathrm{nf}}}
\newcommand{\rhonf}{\rho_{\mathrm{nf}}}
\newcommand{\hatrhonf}{\hat\rho_{\mathrm{nf}}}
\newcommand{\hatchinf}{\hat\chi_{\mathrm{nf}}}
\newcommand{\chinf}{\chi_{\mathrm{nf}}}
\newcommand{\hatfnf}{f_{\mathrm{nf}}}
\newcommand{\bXmod}{{\bm X}_{\mathrm{mod}}}
\newcommand{\bXmodh}{{\bm X}_{\mathrm{mod},h}}
\newcommand{\rhomod}{\rho_{\mathrm{mod}}}
\newcommand{\phimod}{\phi_{\mathrm{mod}}}
\newcommand{\phimodh}{\phi_{\mathrm{mod},h}}
\newcommand{\chimod}{\chi_{\mathrm{mod}}}
\newcommand{\taumod}{\tau_{\mathrm{mod}}}
\newcommand{\taumodf}[1]{\tau_{#1,\mathrm{mod}}}
\newcommand{\fmod}{f_{\mathrm{mod}}}
\newcommand{\Mnf}{M_{\mathrm{nf}}}

\newcommand{\Mmod}{M_{\mathrm{mod}}}
\newcommand{\CMmod}{\Cal M_{\mathrm{mod}}}

\newcommand{\GL}{\operatorname{GL}}
\newcommand{\mon}{\operatorname{mon}}
\newcommand{\const}{\operatorname{const}}
\newcommand{\ord}{\operatorname{ord}}
\newcommand{\Ends}{\mathit{Ends}}

\newcommand{\dd}[1]{\frac{\partial}{\partial #1}}
\newcommand{\tdd}[1]{\tfrac{\partial}{\partial #1 \vphantom{\hat X}}}
\newcommand{\ov}{\overline}

\newcommand{\id}{\operatorname{id}}
\newcommand{\cotg}{\operatorname{cotan}}

\newcommand{\transp}[1]{#1^\mathsf{T}}
\newcommand{\tr}{\operatorname{tr}}
\newcommand{\Fix}{\operatorname{Fix}}

\newcommand{\Cal}{\mathcal}

\newcommand{\sminus}{\smallsetminus}
\newcommand{\hot}{\mathrm{h.o.t.}}
\newcommand{\res}{\operatorname{res}}

\newcommand{\RE}{\operatorname{Re}}
\newcommand{\IM}{\operatorname{Im}}

\newcommand{\Diff}{\mathrm{Diff}}
\newcommand{\hatDiff}{\widehat{\Diff}}

\newcommand{\re}[1]{(\ref{#1})}

\newcommand{\rd}[1]{Definition~\ref{#1}}

\def\beq{\begin{equation}}
\def\eeq{\end{equation}}
\begin{document}

\title{\Grn{Reversible parabolic diffeomorphisms of $(\mathbb{C}^2,0)$\\ and exceptional hyperbolic CR-singularities}}
\author{Martin Klime\v{s}\,\thanks{FER, University of Zagreb, email: {\tt martin.klimes@fer.hr}. Supported by the grant PZS-2019-02-3055 ``Fractal properties of bifurcations of dynamical systems'' of Croatian Science Foundation.}
\ and 
Laurent Stolovitch\,\thanks{CNRS and Laboratoire J.-A. Dieudonn\'e
U.M.R. 7351, Universit\'e de Nice - Sophia Antipolis, Parc Valrose
06108 Nice Cedex 02, France, email: {\tt stolo@unice.fr}. This research was supported by ANR grant ``ANR-10-BLAN 0102'' for the project DynPDE.}
}

\date{\today}
\maketitle
\def\abstractname{Abstract} 

\begin{abstract}

The aim of this article is twofold: 
First we study holomorphic germs of parabolic diffeomorphisms of $(\mathbb{C}^2,0)$ that are reversed by a holomorphic reflection and posses an analytic first integral
with non-degenerate critical point at the origin. 
We find a canonical formal normal form and provide a complete analytic classification (in formal generic cases) in terms of a collection of functional invariants. Their restriction to an irreductible component of the zero locus of the first integral reduces to the Birkhoff--\'Ecalle--Voronin modulus of the 1-dimensional restricted parabolic germ.

\Grn{We then generalize this classification also to germs of anti-holomorphic diffeomorphisms of $(\mathbb{C}^2,0)$ whose square iterate is of the above form.}

Related to it, we solve the problem of both formal and analytic classification of germs of real analytic surfaces in $\C^2$ with non-degenerate CR singularities of  exceptional hyperbolic type, under the assumption that the surface is holomorphically flat, i.e. that it can be locally holomorphically embedded in a real hypersurface of $\C^2$.

\end{abstract}


\section{Introduction}\label{intro}

Early works on iterations of germs of holomorphic maps of $(\mathbb{C},0)$ of the form $\phi: z\mapsto e^{2i\pi\alpha}z +\hot(z)$ in a neighborhood of the origin, the fixed point, can be traced back to Leau \cite{Leau} in the 19th century. 
The structure of orbits of points near the origin under iteration exhibits quite different features depending on whether $\alpha$ is an irrational number or a rational one. 
In the first case, one either encounters Siegel discs on which the dynamics is holomorphically linearizable: conjugate to $z'\mapsto e^{2i\pi\alpha} z'$ by a germ of holomorphic change of coordinate $z'=\psi(z)$ at the origin \cite{siegel,Brjuno,yoccoz}, 
or otherwise, if the dynamics is non-linearizable, one encounters complicated invariant sets known as ``hedgehogs'' \cite{Perez-Marco}. 
On the other hand, \emph{parabolic dynamic} concerns the case of a rational $\alpha=\frac{q}{p}$, meaning that $\phi^{\circ p}(z)=z+\hot(z)$ is tangent to identity. 
Its main feature is the organization of orbits of $\phi^{\circ p}$ into invariant \emph{petals} attached to the origin. 
Furthermore, such germ is formally equivalent to a polynomial \emph{normal form} of the form 
$\phinf:z'\mapsto e^{2i\pi\frac{q}{p}} z' + a z'^{kp+1}+b z'^{2kp+1}$, for some $k\geq 1$ and $a\neq 0,\ b\in\C$. 
It is well known that \emph{normalizing transformations} conjugating $\phi$ to such a normal form $\phinf$ are usually divergent power series of Gevrey type.
Nevertheless, G.D.~Birkhoff \cite{Birkhoff} and T.~Kimura \cite{kimura} proved the existence of \emph{sectorial normalizations}, that is of a finite ``cochain'' of local biholomorphisms $\{\Psi_i\}_{i\in\Z_{2kp}}$, defined on some covering of a neighborhood of the origin by $2kp$ \emph{sectors} (petals) $\{\Omega_i\}_{i\in\Z_{2kp}}$, and conjugating $\phi$ to its normal form, $\phinf\circ\Psi_i=\Psi_{i+2kq}\circ\phi$. 
This is the starting point of the \emph{holomorphic classification problem}, solved first partially by G.D.~Birkhoff \cite{Birkhoff}, and later independently by J.~\'Ecalle \cite{Ecalle,Ecalle2} and S.M.~Voronin \cite{Voronin} (see also \cite{Malg-diffeo,Ilyashenko,IlYa}). Its aim is to describe the equivalence classes of biholomorphisms which are holomorphically conjugate with each other on a neighborhood of the origin. 
\Grn{In the one-dimensional parabolic case, the classifying space, called \emph{Birkhoff--\'Ecalle--Voronin moduli space}, is an infinite-dimensional space
consisting of \emph{cocycles}: $2kp$-tuples of equivalence classes of the transition maps $\{\Psi_{i-1}\circ \Psi_i^{\circ-1}\}$ over the intersection sectors $\Omega_i\cap \Omega_{i-1}$. }

The vector field counterpart of this theory was devised by J. Martinet and J.-P. Ramis \cite{MR,MR2} for 2-dimensional vector fields (corresponding to a \emph{saddle--node} and to a \emph{resonant saddle} respectively) and generalized by the second author to any dimension to \emph{$1$-resonant} vector fields \cite{Stolo-classif}. 
Similar types of functional moduli spaces have since then been discovered in several other contexts (e.g. \cite{Ilyashenko, Ahern-Gong2, lohrmann1, Bittman2}...). 
The common thread through most of these works is that the divergent behavior is concentrated to a single variable or a single resonant monomial, and that there is a finite  covering of a full neighborhood of the singularity by domains projecting to onto sectors in the divergent variable.

The primary goal of this article is to obtain an analytic classification of germs of \emph{parabolic reversible diffeomorphisms} of $(\mathbb{C}^2,0)$, that is
of pairs $(\phi,\tau)$, where $\phi$ is a holomorphic diffeomorphism fixing $0$, such that $\phi^{\circ p}=\id+\hot$ is tangent to identity
for some power $p\geq 1$, and $\tau$ is a holomorphic \Grn{reflection} reversing $\phi$:
\[\tau^{\circ 2}=\id,\qquad  \tau\circ\phi\circ\tau=\phi^{\circ(-1)}.\]
We restrict our attention only to those germs $\phi$ that posses a \emph{holomorphic first integral} $H=H\circ\phi$ of Morse type \Grn{(i.e. with nondegenerate critical point)} at $0$. 

Afterwards we extend the classification also to germs of \emph{parabolic reversible antiholomorphic diffeomorphisms} of $(\C^2,0)$, \Grn{that is to pairs $(\chi,\tau)$ where $\chi$ is an antiholomorphic germ, $\tau\circ\chi\circ\tau=\chi^{\circ(-1)}$, and $\phi=\chi^{\circ 2}$, $\tau$ are as above.}

Following the same general approach as Birkhoff--\'Ecalle--Voronin, we first obtain a formal classification by finding canonical formal normal forms $(\hatphinf,\hattaunf)$, resp. $(\hatchinf,\hattaunf)$,
and then construct a normalizing cochain of transformations on a certain covering of a neighborhood of the origin, which conjugate $(\phi,\tau)$, resp. $(\chi,\tau)$, to an analytic model $(\phimod,\taumod)$, resp. $(\chimod,\taumod)$, representing an equivalence class slightly broader than the formal class.  
The peculiarity of the normalizing cochain is due to its domains no longer being sector-like, but 
having more complicated two-dimensional shapes, attached to the fixed-points divisor. 
This is similar to the domains encountered in the theory of parametric unfolding of 1-dimensional parabolic germs developed by C.~Christopher, P.~Marde\v{s}i\'c, R.~Roussarie \& C.~Rousseau \cite{MRR, Rousseau-Christopher, Rousseau10, Christopher-Rousseau, Rousseau} and by J.~Ribon \cite{Ribon-f, Ribon-a, Ribon-c} building on the works A.~Douady, P.~Lavaurs \cite{Lavaurs}, R.~Oudekerk \cite{Oudekerk} and M.~Shishikura \cite{Shi} on the parabolic bifurcation.
In a striking difference to these works, the covering in our case consists of an infinite number of domains in general.

We emphasize that our result is one of the very first classification results in  parabolic dynamics in a higher dimension. 
In fact, most previous studies focus solely on the existence of \emph{parabolic curves}, notion generalizing that of ``petals'' (see for instance \cite{hakim,weickert,Abate, Abate2}).
\Grn{Under our assumption on existence of Morse first integral $H$, this follows trivially from the 1-dimensional theory by restriction to each irreducible component of the zero level set of $H$.}

Besides, the dynamical system interest, this work is largely motivated by the seemingly unrelated problem of understanding the geometry and holomorphic classification of exceptional hyperbolic Cauchy-Riemann singularities of real analytic surfaces in $(\mathbb{C}^2,0)$. These are real surfaces 
of the form
\begin{equation*}
	M:\ z_2=\begin{cases}\gamma^{-1}z_1\bar z_1+z_1^2+\bar z_1^2+\hot(z,\bar z),& \gamma\in\,\,]0,\infty],\\
		z_1\bar z_1+\hot(z,\bar z),&\gamma=0,
	\end{cases}
\end{equation*}
whose the tangent plane 
\Grn{at the origin is a complex subspace of $\C^2$, but those at neighboring points are not (except if $\gamma=\frac12$ when the set of points with a complex tangent can form a real curve).}
As shown by J.~Moser and S.~Webster \cite{Moser-Webster}, for $\gamma\neq 0$, the moduli space of such surfaces with respect to biholomorphic changes of the ambient space $(\C^2,0)$ 
is in fact isomorphic to the moduli space of 
\Grn{holomorphic conjugacy classes of triples $(\tau_1,\tau_2,\rho)$ where $\tau_1,\tau_2$ are holomorphic reflections and $\rho$ is an anti-holomorphic one such that $\tau_1\circ\rho=\rho\circ\tau_2$. This is the same as the space of conjugacy classes of reversible antiholomorphic diffeomorphisms  
$(\chi,\tau)=(\tau_1\circ\rho,\tau_1)$.}
To the best of our knowledge, this article presents the very first systematic investigation of the \emph{exceptional hyperbolic} case, that is the case when the multipliers $\lambda,\lambda^{-1}$ defined by 
$\lambda+\lambda^{-1}=\gamma^{-2}-2$ 
are non-trivial roots of unity. 
\Grn{The assumption on existence of Morse first integral translates to a condition on the surface $M$ to be \emph{holomorphically flat}: contained in the real hypersurface $\{\RE z_2=0\}$ of $\C^2$.}

\goodbreak
\newpage 

\subsection{Notations}\label{notation:bar}~

\begin{itemize}[leftmargin=12pt]
	    \item[\textasteriskcentered] \Grn{$\hot(\xi)$ stands for ``higher order terms'' in the variable $\xi$.}
	    
		\item[\textasteriskcentered] $\mathbb{Z}_l:=\mathbb{Z}/l\mathbb{Z}\simeq \{0,\ldots, l-1\}$.
		
		\item[\textasteriskcentered] $(\C^n,0)$ stands for a \emph{germ of a neighborhood of $0$ in $\C^n$}.
		
		\item[\textasteriskcentered] \Grn{$\Diff(\C^n,0)\supset\Diff_{\id}(\C^n,0)$ denote the \emph{group of germs of holomorphic diffeomorphisms} fixing the origin in $\C^n$ and its subgroup of \emph{diffeomorphisms tangent to the identity.} }
		
		\item[\textasteriskcentered] \Grn{$\hatDiff(\C^n,0)\supset\hatDiff_{\id}(\C^n,0)$ denote the \emph{group of formal diffeomorphisms of $\C^n$} and its subgroup of elements \emph{tangent to the identity}.}
		
		\item[\textasteriskcentered] If $f(\xi)=\sum_{\bm m\in \mathbb{N}^2}f_{\bm m} \xi^{\bm m}$ is a  germ, then its \emph{complex conjugate} $\bar f(\xi)$ is defined by $\ov{f(\xi)}=\bar f(\bar\xi)$, i.e. $\bar f(\xi)=\sum_{\bm m\in \mathbb{N}^2}\bar f_{\bm m} \xi^{\bm m}$.\\
		Likewise, if $\bX(\xi)=X_1(\xi)\dd{\xi_1}+ X_2(\xi)\dd{\xi_2}$ is a  vector field, then we denote 
		$\ov\bX(\xi)=\ov X_1(\xi)\dd{\xi_1}+ \ov X_2(\xi)\dd{\xi_2}$.
		
		\item[\textasteriskcentered] For a vector field $\bX(\xi)=X_1(\xi)\dd{\xi_1}+X_2(\xi)\dd{\xi_2}$ and a germ $f(\xi)$, we denote 
		\[\bX.f(\xi)=X_1(\xi)\tdd{\xi_1} f (\xi)+X_2(\xi)\tdd{\xi_2} f (\xi)\] 
		the \emph{Lie derivative} of $f$ along $\bX$.  If $f=\transp{(f_1, f_2)}$ is a vector valued function, then $\bX.f=\transp{(\bX.f_1,\,\bX.f_2)}$.
		In particular, $\bX.\xi=\transp{(X_1,X_2)}$
		
		\item[\textasteriskcentered] If $\bX$ is a vector field, then $\exp(t\bX)(\xi)$ denotes the \emph{flow map of $\bX$ at time $t$} (see Section~\ref{sec:3infgen}), 	
		and $\exp(t\bX)\big|_{t=f(\xi)}$ is the map obtained by substituting $f(\xi)$ for $t$ in the map $(\xi,t)\mapsto\exp(t\bX)(\xi)$.
		
		\item[\textasteriskcentered] Let $\Psi:\xi\mapsto\xi'=\Psi(\xi)$ be a diffeomorphism, conjugating two vector fields $\bX(\xi)$ and $\bX'(\xi')$,
		that is such that
		$\bX'.\xi'\big|_{\xi=\Psi^{\circ(-1)}}=D\Psi(\bX.\xi)=\bX.\Psi$,
		then $\bX$ is the \emph{pullback} of $\bX'$
		\[ \bX=\Psi^*\bX', \] 
		and \[\exp(\bX)=\Psi^{\circ(-1)}\circ\exp(\bX')\circ\Psi.\]
		
		\item[\textasteriskcentered] A function $f:(\C^2,0)\to (\C,0)$ is \emph{$\tau$-invariant} if $f\circ\tau=f$.\\
		A map $F:(\C^2,0)\to (\C^2,0)$ is \emph{$\tau$-equivariant} if $F\circ\tau=\tau\circ F$.\\
		A vector field $\bX$ is \emph{$\tau$-equivariant} if $\tau^*\bX=\bX$. 
	\end{itemize}

\goodbreak

\subsection{Recall: Birkhoff--\'Ecalle--Voronin theory of parabolic diffeomorphisms of $(\C,0)$}

To motivateour results, let us shortly recall some of the basics of analytic theory of  parabolic diffeomorphisms of $(\C,0)$.
For more details see \cite{Ecalle,Ecalle2, Malg-diffeo, Voronin, Voronin2} or \cite{Bracci, Loray, IlYa}.

Let $\phi(z)=\lambda z+\hot(z)\in\Diff(\C,0)$ be a germ of analytic diffeomorphism fixing the origin in $\C$, 
where $\lambda$ is a root of unity of some order $p\geq 1$,  $\lambda^p=1$.
Its $p$-th iteration $\phi^{\circ p}(z)=z+\hot(z)$ is a diffeomorphism tangent to the identity, and as such it possesses a unique formal infinitesimal generator: a formal vector field $\hat{\bX}(z)$ at $0$ with vanishing linear part, such that the Taylor series of $\phi^{\circ p}(z)$ agrees with the formal time-1-flow $\exp(\hat{\bX})(z)$ of $\hat{\bX}$.
The formal vector field $\hat{\bX}$ can be conjugated by some formal tangent-to-identity map $\hat{\Psi}(z)=z+\hot(z)$ to its \emph{normal form}
\begin{equation*}
\bXnf(z)=\mfrac{c\,z^{kp}}{1+c\,\mu\, z^{kp}}z\tdd{z},\quad k\geq 1,\quad c\neq 0,	
\end{equation*}
which is invariant by the rotation $z\mapsto\lambda z$. Consequently also the germ $\phi(z)$ is \Grn{formally} conjugated to the \emph{normal form}
\begin{equation*}
\phinf(z)=\lambda\exp(\tfrac1p\bXnf)(z).
\end{equation*}	
As it turns out, while the formal conjugacy $\hat\Psi(z)$ is generically divergent, it is Borel summable (of order $kp$) on sectors.

\begin{theorem}[\Grn{Birkhoff, Kimura, \'Ecalle, Voronin,...}]\label{thm:kimura}
The germ $\phi(z)$ is conjugated to its  normal form $\phinf(z)$ by a \emph{cochain of bounded analytic transformations} 
$\big\{\Psi_{\Omega_j}(z)=z+\hot(z)\big\}_{j\in\Z_{2kp}}$
on a covering by $2kp$ sectors (Leau--Fatou petals) $\Omega_j$, $j\in\Z_{2kp}$,\footnote{\Grn{The sectorial covering is $\lambda$-invariant: writing $\lambda=e^{2\pi i\frac{q}{p}}$ then for every sector $\Omega_j$ the rotated sector $\lambda\Omega_j=\Omega_{l}$, $l=j+2kq\mod 2kp$, belongs again to the covering.}}
\begin{equation*}
\Psi_{\lambda\Omega_j}\circ\phi(z)=\phinf\circ\Psi_{\Omega_j}(z),\quad z\in\Omega_j. 
\end{equation*}	
Such normalizing cochain $\big\{\Psi_{\Omega_j}(z)\big\}_{j\in\Z_{2kp}}$ is \emph{unique up to} left composition with cochains $\big\{\exp(C_{\Omega_j}\bXnf)(z)\big\}_{j\in\Z_{2kp}}$, $C_{\Omega_j}\in\C$, of flow maps of $\bXnf$.
\end{theorem}
	
\Grn{The form of these sectorial domains $\Omega_j$ (Figure~\ref{figure:parabolicpetals}) is related to the dynamics of $\bXnf$: they} are spanned by the real-time trajectories of the family of rotated vector fields $e^{i\theta}\bXnf(z)$, 
i.e. by the real curves 
\[\mfrac{dz}{dt}=e^{i\theta}\mfrac{c\,z^{kp}}{1+c\,\mu\, z^{kp}}z,\qquad t\in\R,\]
that stay inside some disc $\{|z|<\delta_1\}$, where $\theta$ is allowed to vary in some interval $]\delta_3,\pi-\delta_3[$,
 for some $\delta_1,\delta_3>0$. See Figure~\ref{figure:parabolictrajectories}.

\begin{figure}[t]
	\centering	
	\begin{subfigure}{.31\textwidth} \includegraphics[angle=90, width=\textwidth]{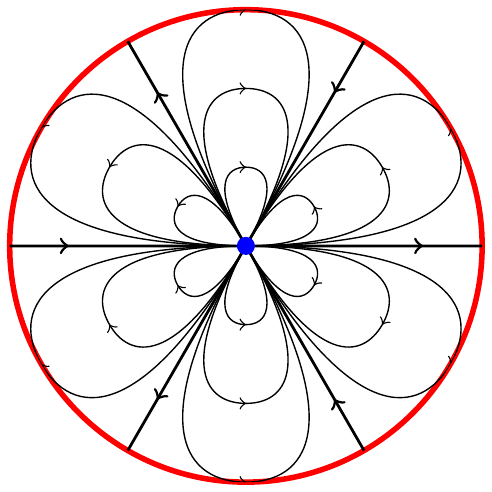} \caption{} \label{figure:parabolictrajectories} \end{subfigure}	
	\qquad
	\begin{subfigure}{.4\textwidth} \includegraphics[width=\textwidth]{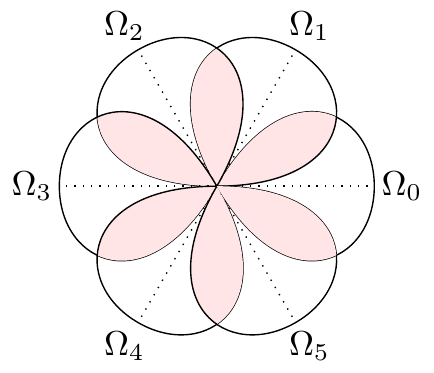} \caption{} \label{figure:parabolicpetals}  \end{subfigure}	
	\caption{(a) Real-time trajectories of the vector field $e^{i\theta}\bXnf$ inside a small disc. 
		(b) The Leau--Fatou petals $\Omega_j$, $j\in\Z_{2kp}$.}
	\label{figure:parabolic}
\end{figure}

\medskip
The equivalence class of the set of the $2kp$ \emph{transition maps} 
\[\psi_j=\Psi_{\Omega_{j-1}}\circ\Psi_{\Omega_j}^{\circ(-1)}\quad \text{on the \emph{intersections} }\ \Omega_{j-1}\cap\Omega_j,\quad j\in\Z_{2kp},\]
modulo conjugation by cochains of flow maps $\big\{\exp(C_{\Omega_j}\bXnf)\big\}_{j\in\Z_{2kp}}$, $C_{\Omega_j}\in\C$, 
\[\psi_{j}\simeq \exp(C_{\Omega_{j-1}}\bXnf)\circ\psi_{j}\circ\exp(-C_{\Omega_{j}}\bXnf)\]
is then called a \emph{cocycle}. It is an analytic invariant of $\phi$  which expresses the obstruction to convergence of the formal normalizing transformation. 
It was initially described by G.D.~Birkhoff \cite{Birkhoff} and later independently rediscovered by J.~\'Ecalle \cite{Ecalle} and S.M.~Voronin \cite{Voronin}.

\begin{theorem}~
\begin{enumerate}[wide=0pt, leftmargin=\parindent]
\item \textnormal{(Birkhoff, \'Ecalle, Voronin).}
Two germs $\phi$, $\phi'$ that are formally tangent-to-identity equivalent are analytically tangent-to-identity equivalent if and only if their cocycles $\big\{\psi_j\big\}_{j\in\Z_{2kp}}$, $\big\{\psi_j'\big\}_{j\in\Z_{2kp}}$ agree.

\item \textnormal{(\'Ecalle, Malgrange, Voronin).}
For each formal normal form $\phinf$ and each collection of maps $\big\{\psi_j\big\}_{j\in\Z_{2kp}}$ on the intersections sectors, that are asymptotic to the identity and commute with $\phinf$: 
\[\phinf\circ\psi_j=\psi_{j+2kq}\circ\phinf,\]
there exists an analytic map $\phi$ whose cocycle is represented by $\big\{\psi_j\big\}_{j\in\Z_{2kp}}$.
\end{enumerate}	
\end{theorem}

If one wants to obtain the modulus of analytic equivalence with respect to conjugation by general transformations in $\Diff(\C,0)$, one has to consider the cocycles modulo an action of the group of rotations $z\mapsto e^{2\pi i\frac{r}{kp}}z$, $r\in \Z_{kp}$, which preserve $\bXnf$. 	

\medskip
The theory can be generalized also to analytic classification of germs of antiholomorphic diffeomorphisms of parabolic type, see \cite{GR1}.

\goodbreak

\subsection{\Grn{Classification of reversible parabolic diffeomorphisms}}

Let $(\phi,\tau)$ be a pair of a reversible map $\phi$ and its reversing involution $\tau$
\begin{equation}\label{eq:reversible}
	\tau^{\circ 2}=\id,\qquad  \tau\circ\phi\circ\tau=\phi^{\circ(-1)}.
\end{equation}
Denoting $\Cal G$ the group of diffeomorphisms generated by $\{\phi,\tau\}$, then
\[\Cal G=\big\{\phi^{\circ n}\mid n\in\Z\big\}\cup\big\{\tau\circ\phi^{\circ n}\mid n\in\Z\big\},\]
where each $\tau_{n+1}=\tau\circ\phi^{\circ n}$ is an involution reversing $\phi$.
For every $n\in\Z$ the pair of involutions $(\tau_n,\tau_{n+1})$ satisfies $\tau_n\circ\tau_{n+1}=\phi$ and therefore generates $\Cal G$. 
\emph{Thus the problem of classification of reversible maps $(\phi,\tau)$ with respect to conjugation is equivalent to that of pairs of involutions $(\tau_n,\tau_{n+1})$.} 
Since all unordered pairs $\{\tau_{n},\tau_{n+1}\}$ are conjugated to each other, one may consider just the pair  
\[(\tau_1,\tau_2)=(\tau,\tau\circ\phi).\]

\begin{assumptions}\label{assumptions}
We shall assume that $(\phi,\tau)$ are holomorphic diffeomorphisms of $(\C^2,0)$  such that \Grn{$\phi\neq\tau$} and:
\begin{enumerate}
	\item $\phi\in\Diff(\C^2,0)$ is \emph{parabolic}: $\phi^{\circ p}=\id+\hot$ for some positive integer $p\geq 1$ (the minimal with such property),
	\item $\tau\in\Diff(\C^2,0)$ is a \emph{holomorphic reflection} (an involution whose linear part has eigenvalues  $\{1,-1\}$) which reverses $\phi$,
	\[	\tau^{\circ 2}=\id,\qquad  \tau\circ\phi\circ\tau=\phi^{\circ(-1)},\]
   \item the pair $(\phi,\tau)$ possesses an \emph{analytic first integral of Morse type}, i.e. with non-degenerate critical point at $0$, $H(0)=0$, $DH(0)=0$, $\det D^2H(0)\neq 0$,
\[H=H\circ\phi=H\circ\tau.\]
\end{enumerate}	
\end{assumptions}

Up to a linear change of variables (Lemma~\ref{lemma:phitau}) they take the form
\begin{equation}\label{eq:phitau}
\phi(\xi)=\Lambda\xi+\hot(\xi),\qquad 	\tau(\xi)=\sigma\xi+\hot(\xi),	\qquad H(\xi)=\xi_1\xi_2+\hot(\xi),
\end{equation}	
\GRN
where
\begin{equation}\label{eq:sigmaLambda}
\sigma=\left(\begin{smallmatrix}0\, &1\\[4pt]1&\, 0\end{smallmatrix}\right),
\qquad 	 \Lambda=\begin{cases}
\left(\begin{smallmatrix}\lambda & 0 \\[4pt] 0 & \lambda^{-1}\end{smallmatrix}\right),& \Lambda^{p}=I,\quad p\geq 1,\\[6pt]
-\sigma,& \Lambda^{2}=I,\quad p=2.
\end{cases}
\end{equation}
The case of diagonal $ \Lambda=	\left(\begin{smallmatrix}\lambda & 0 \\[3pt] 0 & \lambda^{-1}\end{smallmatrix}\right)$ arises when the involution
$\tau_2=\tau\circ\phi$ is a holomorphic reflection as well, while
the case $\Lambda=-\sigma$ happens when the involution $\tau_2=\tau\circ\phi$ is tangent to $-\id$.
We shall note that a possibility of $\tau_2=\tau\circ\phi$ being tangent to $\id$ is excluded by the assumption that $\phi\neq\tau$, since any involution tangent to the identity is in fact the identity.
\FGRN

\com{We didn't consider the case $\Lambda=-\sigma$ before, and it doesn't arise in the Moser-Webster context, but since its addition doesn't cost much (all the proofs work exactly the same way except for some small accommodations), i think it's reasonable to include it as well. It gives rise to the case (c) in Theorem~\ref{thm:1}.}

The diffeomorphism $\phi^{\circ p}(\xi)=\xi+\hot(\xi)$ is tangent to the identity, and as such it possesses  a unique formal infinitesimal generator (see Section~\ref{sec:3infgen}): a formal vector field $\hat\bX(\xi)$ whose formal time-1-flow $\exp(\hat\bX)(\xi)$ is equal to the Taylor expansion of $\phi^{\circ p}(\xi)$.  This formal vector field $\hat\bX(\xi)$ has $H(\xi)$ as a first integral, and is reversed by $\tau$:  $\tau^*\hat\bX(\xi)=-\hat\bX(\xi)$.
This allows to reduce the problem of formal classification of $(\phi,\tau)$ to a formal classification of such integrable reversible formal vector fields $\hat\bX$ 
(Theorem~\ref{thm:FNFX1}).

 \begin{thm}[Formal classification]\label{thm:1}~\\
Let $(\phi,\tau)$ and $H$ be as above satisfying Assumptions~\ref{assumptions}.
\Grn{Let $s\geq 0$ be the multiplicity of the zero level set $\{H(\xi)=0\}$ in the fixed point divisor $\Fix(\phi^{\circ p})$ which is of the form
$\{H^s(\xi)\cdot g(\xi)=0\}$ for some analytic germ $g(\xi)$,
and denote $kp=\ord_0 g(\xi)$ its order of vanishing.}

There exists a formal transformation $\xi\mapsto\hat\Psi(\xi)\in\widehat{\Diff}_{\id}(\C^2,0)$ and a formal diffeomorphism $\hat G\in\widehat{\Diff}(\C,0)$,
such that
\[\hat\Psi\circ\phi=\hatphinf\circ\hat\Psi, \qquad  \hat\Psi\circ\tau=\sigma\,\hat\Psi, \qquad \hat G(H)=h\circ\hat\Psi, \]
where
\begin{equation*}
\hatphinf(\xi)=\Lambda\cdot\exp(\tfrac{1}{p}\hatbXnf)(\xi),\quad
\text{with $\sigma,\Lambda$ as in \eqref{eq:sigmaLambda}, \quad and \quad $h=\xi_1\xi_2.$}
\end{equation*}
Here  $\hatbXnf(\xi)=\Lambda^*\hatbXnf(\xi)=-\sigma^*\hatbXnf(\xi)$ is one of the following vector fields:
\begin{itemize}[wide=0pt, leftmargin=\parindent]
\item[(o)] $s=+\infty$: $\hatbXnf(\xi)=0$. 
		\Grn{This happens if and only if $\phi^{\circ p}=\id$, and there exists such normalizing transformation $\hat\Psi$ which is convergent.}
\end{itemize}
If $\Lambda$ is diagonal:	
\begin{itemize}[wide=0pt, leftmargin=\parindent]	
\item[(a)] $k=0$, $s\geq 1:$ \ $\hatbXnf(\xi)= c\,h^s\big(\xi_1\tdd{\xi_1}-\xi_2\tdd{\xi_2}\big)$, \ $c\neq 0$.
			
\item[(b)] $k\geq 1$, $s\geq 0:$ \ $\displaystyle{\hatbXnf(\xi)=\frac{c\, h^s P(u,h)}{1+c\,\hat\mu(h)P(u,h)}\big(\xi_1\tdd{\xi_1}-\xi_2\tdd{\xi_2}\big)}$, \ $c\neq0$,\\
	where $P(u,h)$ is polynomial in $u(\xi):=\xi_1^p+\xi_2^p$ of order $k$,
	\[P(u,h)=u^k+ P_{k-1}(h)u^{k-1}+\ldots+P_0(h),\quad P(u,0)=u^k,\]
    and $\hat\mu(h)=\sum_{n=0}^{+\infty}\mu_nh^n$ is a formal power series. 
\end{itemize} 
\GRN  
If $\Lambda=-\sigma$:    
\begin{itemize}[wide=0pt, leftmargin=\parindent]
	\item[(c)] $k=\tilde k+\frac12$, $s\geq 0:$ \ $\displaystyle{\hatbXnf(\xi)= c\,h^s\tilde P(\tilde u,h)\big(\xi_1\!+\!\xi_2\big)\big(\xi_1\tdd{\xi_1}-\xi_2\tdd{\xi_2}\big)}$, \ $c\neq0$,\\
    where $\tilde P(\tilde u,h)$ is polynomial in $\tilde u(\xi):=(\xi_1+\xi_2)^2$ of order $\tilde k\geq 0$,
    \[\tilde P(\tilde u,h)=\tilde u^{\tilde k}+ \tilde P_{\tilde k-1}(h)\tilde u^{\tilde k-1}+\ldots+\tilde P_0(h),\quad \tilde P(\tilde u,0)=\tilde u^{\tilde k}.\]
\end{itemize}
\FGRN

\begin{itemize}[wide=0pt, leftmargin=\parindent, label=$\triangleright$]
\item In the cases (a),\,(b),\,(c) the formal normalizing transformation $\hat\Psi\in\widehat{\Diff}_{\id}(\C^2,0)$ is unique. 
Furthermore, in the cases (b),\,(c) $h\circ\hat\Psi$ is convergent. 

\item The formal equivalence class of $(\phi,\tau)$ with respect to conjugation by the group $\widehat{\Diff}_{\id}(\C^2,0)$ contains a unique representative in the above formal normal form $(\hatphinf,\hattaunf)$.

\item In the formal equivalence class of $(\phi,\tau)$ with respect to conjugation in the full group $\widehat{\Diff}(\C^2,0)$ the above formal normal form $(\hatphinf,\hattaunf)$ and its infinitesimal generator $\hatbXnf$ are determined uniquely up to the action of scalar transformation
$\xi\mapsto\zeta\cdot\xi$, $\zeta\in\C^*$, and also of $\xi\mapsto\sigma\xi$ in case $p\in\{1,2\}$ when $\sigma\Lambda=\Lambda\sigma$, by which the constant $c\neq 0$ can be further normalized.

\item The group $\Cal Z(\hatphinf,\sigma)=\Cal Z(\hatbXnf,\sigma,\Lambda)$ of formal diffeomorphisms commuting with $\hatphinf,\sigma$, which is the same as 
the $(\sigma,\Lambda)$-equivariant diffeomorphisms preserving $\hatbXnf$, is in the cases (a),\,(b),\,(c) identified with some subgroup of $\Z_{2kp+4s}$ acting on $(\C^2,0)$ by $\xi\mapsto e^{\frac{\pi i r}{kp+2s}}\sigma^r\xi,\ r\in\Z_{2kp+4s}$.
If $p>2$ then only the action with $r$ commute with $\Lambda$. 
\end{itemize}
\end{thm}

\begin{remark}
\begin{enumerate}[wide=0pt, leftmargin=\parindent]
\item The variables $h=\xi_1\xi_2$ and $u=\xi_1^p+\xi_2^p$, \Grn{resp. $\tilde u=(\xi_1+\xi_2)^2$,} are basic $(\sigma,\Lambda)$-invariant functions: any formal/analytic $(\sigma,\Lambda)$-invariant function can be written as a formal/analytic function of $(h,u)$, \Grn{resp. $(h,\tilde u)$,} (see e.g. \cite[\S XII-4]{GSS}).
		
\item The cases (o), (a) and (b)+(c), of Theorem~\ref{thm:1} are distinguished by the position of their fixed point divisor $\Fix(\phi^{\circ p})=\{\xi\in(\C^2,0): \phi^{\circ p}(\xi)=\xi\}$ with respect to the foliation by level sets of $H(\xi)$:
\[\begin{cases}
	\text{(o)}\ \Fix(\phi^{\circ p})=(\C^2,0),\\
	\text{(a)}\ \Fix(\phi^{\circ p})=\{H^s(\xi)=0\},\\
	\text{(b),\,(c)}\ \Fix(\phi^{\circ p})=\{H^s(\xi)\cdot g(\xi)=0\},
\end{cases}\]	
where $\{g(\xi) =0\}$ is a divisor  transverse to the foliation \Grn{intersecting each level set $\{H(\xi)=\const\}$ at $2kp$ points (counted with multiplicity).}
\end{enumerate}
\end{remark}

In the case (a)	of Theorem~\ref{thm:1}, there exist analytic germs that are formally equivalent to the normal form but not analytically (Theorem~\ref{thm:divergent}). In fact, there are topological obstructions to convergence.
\Grn{However, somewhat surprisingly, there are also interesting examples where the conjugacy is analytic (Example~\ref{example:Painleve} below).}

The cases (b) \Grn{and (c)} of Theorem~\ref{thm:1} carry close analogy with the Birkhoff--\'Ecalle--Voronin theory of parabolic diffeomorphisms in dimension 1.  
While the formal normalizing transformation is generically divergent, the obstructions to convergence are of a purely analytic nature and can be expressed in terms of 
an infinite-dimensional functional modulus  (Theorem~\ref{thm:analytic} below).

The formal invariant $\hat{\mu}(h)$ in Theorem~\ref{thm:1}\,(b), or more precisely $2\pi i\, h^{-s}\hat\mu(h)$ is the formal period of 
any formal differential 1-form  dual to $\hatbXnf$ along  the ``vanishing cycles'' generating the fundamental group of the leaves $\{h=\const\neq 0\}$.
\Grn{Correspondingly,} the composition $\big(2\pi i\, h^{-s}\hat\mu(h)\big)\circ\hat{\Psi}$ is the formal period of 
any formal differential 1-form dual to the infinitesimal generator $\hat\bX$ of $\phi^{\circ p}$ along  the ``vanishing cycles'' generating the fundamental group of the leaves $\{H=\const\neq 0\}$. (Lemma~\ref{lemma:mu}).
At the present moment it is not known to us whether the formal series $\hat\mu(h)$ is convergent in general or under what condition.

\begin{remark}
The formal classification of Theorem~\ref{thm:1} is quite similar to the study of 1-parameter families of holomorphic germs $\phi_{\epsilon}(z)$
unfolding a parabolic germ $\phi_0(z)=\lambda z+\hot(z)$, $\lambda^{p}=1$.
The formal normal form for such germs in the case $p>1$ is
\[\phinfeps(z)=\lambda\exp\big(\tfrac1p\bXnfeps\big)(z),\qquad 
\bXnfeps=\begin{cases}
	0,&\\
	c\,\epsilon^sz\tdd{z},&s>0,\\[6pt]
	\frac{c\,\epsilon^sP(z^p,\epsilon)}{1+c\mu(\epsilon)z^{kp}}z\tdd{z},& s\geq 0,\ k>0,
\end{cases}\]
with $P(z^p,\epsilon)=z^{kp}+P_{k-1}(\epsilon)z^{(k-1)p}+\ldots+P_0(\epsilon)$, $P(z^p,0)=z^{kp}$.
The study of such families in the finite codimension case $s=0$ was carried independently by C.~Christopher, P.~Marde\v{s}i\'c, R.~Roussarie \& C.~Rousseau \cite{MRR, Rousseau-Christopher, Rousseau10, Christopher-Rousseau, Rousseau} and by J.~Ribon \cite{Ribon-f, Ribon-a, Ribon-c}.\footnote{Prior to that, this was investigated also by J.~Martinet \cite{Martinet}, P.~Lavaurs \cite{Lavaurs}, R.~Oudekerk \cite{Oudekerk}, M.~Shishikura \cite{Shi} and A.~Glutsyuk \cite{Glutsyuk}.}
It leads to a modulus of analytic classification that ``unfolds'' the Birkhoff--\'Ecalle--Voronin modulus.

A 1-parameter family of diffeomorphisms $\phi_{\epsilon}(z)$ unfolding $\phi_0(z)$ can be thought of as a parabolic diffeomorphism $(z,\epsilon)\mapsto\big(\phi_{\epsilon}(z),\epsilon\big)$ of $(\C^2,0)$ with a first integral $\epsilon$, and hence with a locally trivial leaf-wise invariant foliation by level curves $\{\epsilon=\const\}$ (Figure~\ref{figure:fol-a}).
\Grn{The essential difference to the situation considered here is that in our case the leaf-wise invariant foliation, given by level curves of the first integral $H(\xi)=\xi_1\xi_2+\hot(\xi)$,
is topologically non-trivial (Figure~\ref{figure:fol-b}).}
\end{remark}

\begin{figure}[t]
\centering	
\begin{subfigure}{.3\textwidth} \includegraphics{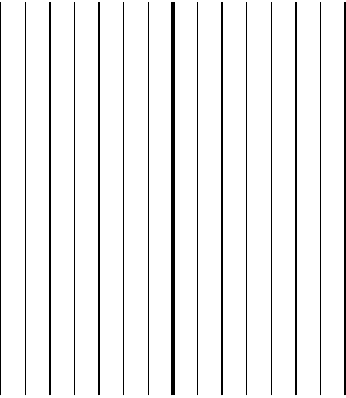} \caption{ } \label{figure:fol-a}	\end{subfigure}	
\qquad
\begin{subfigure}{.3\textwidth} \includegraphics{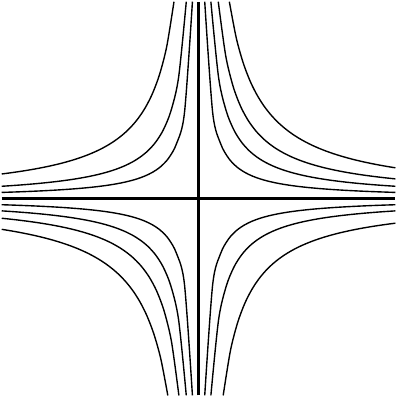} \caption{ } \label{figure:fol-b} \end{subfigure}	
\caption{(a) Trivial invariant foliation of a 1-parameter family of diffeomorphisms of $\C$.
	(b) Invariant foliation of a diffeomorphism of $(\C^2,0)$ with a first integral $H(\xi)=\xi_1\xi_2+\hot(\xi)$.}
\label{figure:foliation}
\end{figure}

\GRN
\begin{remark}
The formal classification of reversible diffeomorphisms 
\[\phi(\xi)=\left(\begin{smallmatrix}\lambda &0\\[3pt]0&\lambda^{-1}\end{smallmatrix}\right)\xi+\hot(\xi),\qquad \tau(\xi)=\sigma\xi+\hot(\xi),\quad\text{with }\ \lambda\notin e^{\pi i\Q},\]
has been achieved by Moser \& Webster \cite{Moser-Webster} with a formal normal form  
\begin{equation}\label{eq:phinf}
	\phinf(\xi)=\left(\begin{smallmatrix}\lambda\, e^{\tilde ch^s} &0\\[3pt]0&\lambda^{-1} e^{-\tilde ch^s}\end{smallmatrix}\right)\xi,\qquad \taunf(\xi)=\sigma\xi=\left(\begin{smallmatrix}0 &1\\[3pt]1&0\end{smallmatrix}\right)\xi,\qquad h=\xi_1\xi_2.
\end{equation}
This classification has been later generalized to all elements of $\Diff(\C^2,0)$ that are formally conjugated to their inverse by O'Farrell \& Zaitsev \cite{OFarrellZaitsev}. 

The classification is analytic if $|\lambda|\neq 1$, which in particular implies the existence of Morse first integral $H(\xi)$ of $(\phi,\tau)$.  
\end{remark}

The following example is of a non-trivial situation in which a reversible diffeomorphism is analytically conjugated to its formal normal form of type (a) of Theorem~\ref{thm:1}. 

\begin{example}[Monodromy of the Sixth Painlev\'e equation]\label{example:Painleve}
The operator of a local monodromy of Sixth Painlev\'e equation at either of its singular points is one that acts on solutions by their analytic continuation along a loop around the singularity. Considered as a map on the space of ``initial conditions'', it is a reversible holomorphic map $\phi$ with up to 4 fixed points (corresponding to locally non-ramified solutions near the singularity), and with a first integral which is of Morse at each of the fixed points. The local multipliers $\lambda,\lambda^{-1}$ of $\phi$ near a fixed point
depend on the parameters of the equations, and for some of the parameters they are indeed roots of unity, however, no matter what they are, the map is always locally analytically conjugated to the formal normal form \eqref{eq:phinf} with $s=1$. The normalizing map is essentially given by the Riemann--Hilbert correspondance. More details in Section~\ref{sec:Painleve}.  
\end{example}
\FGRN

Since the formal normal form $(\hatphinf,\sigma)$ of Theorem~\ref{thm:1} in the case (b) is a priori purely formal \Grn{(due to the formal invariant $\hat\mu(h)$)}, we introduce instead a larger \emph{model class} represented by an analytic model.

\begin{definition}[Model]\label{def:model}
Let $(\hatphinf,\sigma)$ be the formal normal form of Theorem~\ref{thm:1} for $(\phi,\tau)$, with infinitesimal generator 
\[\hatbXnf(\xi)=\begin{cases} 
	0,\\ 
	c\,h^s\big(\xi_1\tdd{\xi_1}-\xi_2\tdd{\xi_2}\big), \\[6pt] 
	c\,h^s\frac{P(u,h)}{1+c\,\hat\mu(h)P(u,h)}\big(\xi_1\tdd{\xi_1}-\xi_2\tdd{\xi_2}\big), \\[6pt] 
	\Grn{c\,h^s\tilde P(\tilde u,h)\big(\xi_1\!+\!\xi_2\big)\big(\xi_1\tdd{\xi_1}-\xi_2\tdd{\xi_2}\big). }
\end{cases}\]
Let us introduce a \emph{model} $(\phimod,\sigma)$ for $(\phi,\tau)$, as $\phimod=\Lambda\exp\big(\tfrac{1}{p}\bXmod \big)$ where
\begin{equation}\label{eq:Xmodel}
	\bXmod(\xi)=\begin{cases} 
		0,\\ 
		c\,h^s\big(\xi_1\tdd{\xi_1}-\xi_2\tdd{\xi_2}\big), \\[6pt] 
		c\,h^sP(u,h)\big(\xi_1\tdd{\xi_1}-\xi_2\tdd{\xi_2}\big), \\[6pt] 
	\Grn{c\,h^s\tilde P(\tilde u,h)\big(\xi_1\!+\!\xi_2\big)\big(\xi_1\tdd{\xi_1}-\xi_2\tdd{\xi_2}\big). }
\end{cases}
\end{equation}
with the same $c$ and $P(u,h)$, resp. $\tilde P(\tilde u,h)$, as in the formal normal form.
The \emph{model class} of $\phimod$ is the set of all analytic $\phi$ with the same model, i.e. it is the union of formal classes with over all invariants $\hat{\mu}(h)$.
\end{definition}


\begin{remark}\label{remark:power-log}
	The normal form vector field $\hatbXnf=h^s\frac{cP}{1+\hat\mu cP}\bE$ is equivalent to the model $\bXmod=h^scP\bE$ by means of a $(\sigma,\Lambda)$-equivariant formal power-log transformation \[\hat\Psi=\exp(tcP\bE)\big|_{t=\tfrac{\hat\mu(h)}{2p}(\log\xi_1^p-\log\xi_2^p)},\qquad \hat\Psi^*\bXmod=\hatbXnf.\]	
	This follows from Lemma~\ref{lemma:X} by writing $\hat\mu(h)=\bE.\big[\tfrac{\hat\mu(h)}{2p}(\log\xi_1^p-\log\xi_2^p)\big]$.
\end{remark}

\begin{proposition}[Prenormalization]\label{prop:prenormalphi}
Let $(\phi,\tau)$ and $H$ be as in Theorem~\ref{thm:1} of formal type (b) or (c), and let $(\hatphinf,\sigma)$ be its formal normal form and $(\phimod,\sigma)$ its model.
There exists an analytic tangent-to-identity change of coordinates, after which $\tau(\xi)=\sigma\xi$ and $\phi(\xi)$ is such that $h\circ\phi(\xi)=h(\xi)$ and
\[\begin{aligned}
\xi_j\circ\phi(\xi)&=\xi_j\circ\hatphinf(\xi)\mod h^{-s}f(\xi)^2\xi_j\\
&=\xi_j\circ\phimod(\xi)\mod h^{-s}f(\xi)^2\xi_j,\qquad j=1,2,\end{aligned}\]
where $f(\xi)=\frac{\xi_1\circ\phi^{\circ p}(\xi)-\xi_1}{\xi_1}$ generates the same ideal of $\C\llbracket\xi\rrbracket$ as 
\Grn{$\hatfnf(\xi)=\frac{\xi_1\circ\hatphinf^{\circ p}(\xi)-\xi_1}{\xi_1}$ and $\fmod(\xi)=\frac{\xi_1\circ\phimod^{\circ p}(\xi)-\xi_1}{\xi_1}$,
and $\Fix(\phi^{\circ p})=\{f(\xi)=0\}$.}

\com{I've removed the part about the infinitesimal generator to keep it shorter.}
\end{proposition}

\Grn{The following is our analogy of Theorem~\ref{thm:kimura}.}
	
\begin{thm}[``Sectorial'' equivalence]\label{thm:sectorial}
Let $(\phi,\sigma)$ of formal type \textit{(b)} or \textit{(c)} be in the prenormal form of Proposition~\ref{prop:prenormalphi},
and let $(\phimod,\sigma)$ be its model.
There exists a countable collection of cuspidal sectors\footnote{See Figure~\ref{figure:sectors} and Definition~\ref{def:sector}.} 
covering a disc $\{|h(\xi)|<\delta_2\}$ for some $\delta_2>0$, and for each given sector $S$ a $(\sigma,\Lambda)$-invariant family\footnote{If $\Omega_{S}$ is a domain in the family, then its images $\sigma(\Omega_S)$ and $\Lambda(\Omega_S)$ are also in the family.} 
of $4kp$ ``Lavaurs domains'' $\{\Omega_{S}^j\}_{j=1,\ldots, 4kp}$ covering together the set $B_S\sminus\Fix(\phimod^{\circ p})$
\begin{equation}\label{eq:BS}
	B_S=\{\xi\in\C^2: |\xi|<\delta_1,\ h(\xi)\in S\}
\end{equation} 
 (see Figure~\ref{figure:sectors}), a family of bounded analytic transformations $\{\Psi_{\Omega_{S}^j}\}_{j=1,\ldots, 4kp}$ defined on  the Lavaurs domains $\Omega_{S}^j$,
such that
\[\Psi_{\Lambda(\Omega_{S}^j)}\circ\phi=\phimod\circ\Psi_{\Omega_{S}^j},\qquad \Psi_{\sigma(\Omega_{S}^j)}\circ\sigma=\sigma\circ\Psi_{\Omega_{S}^j},\qquad h\circ\Psi_{\Omega_{S}^j}=h,\] 
 for all $j$ and $S$.
\Grn{We call the family $\{\Psi_{\Omega_{S}^j}\}_{j=1,\ldots, 4kp}$ a \emph{normalizing cochain}.
Such normalizing cochain is unique up to left composition with cochains of flow maps
\begin{equation}\label{eq:cochainofflowmaps}
	\Big\{\exp\big(h^{-s}C_{\Omega^{j}_{S}}(h)\bXmod\big)\Big\}_{j=1,\ldots 4kp},
\end{equation} 
where the $C_{\Omega^{j}_{S}}(h)$ are bounded analytic functions on $S$.}
\end{thm}

The form of the domains $\Omega_S^j$ in the covering of Theorem~\ref{thm:sectorial} is determined by the dynamics of the model vector field
$\bXmod$ \eqref{eq:Xmodel}.
\Grn{The set $B_S$ \eqref{eq:BS} has two ``essential'' boundary components: ``outer'' one at $\{|\xi_1|=\delta_1\}$ and ``inner'' one $\{|\xi_2|=\delta_1\}$, and the $4kp$ domains $\Omega_{S}^j$ are correspondingly grouped into two sets: $2kp$ cyclically ordered \emph{outer domains} (touching the ``outer'' boundary) and $2kp$ cyclically ordered \emph{inner domains} (touching the ``inner'' boundary), see Figure~\ref{figure:sectors}. }

\begin{figure}[t]
	\centering 
	\includegraphics[width=0.99\textwidth]{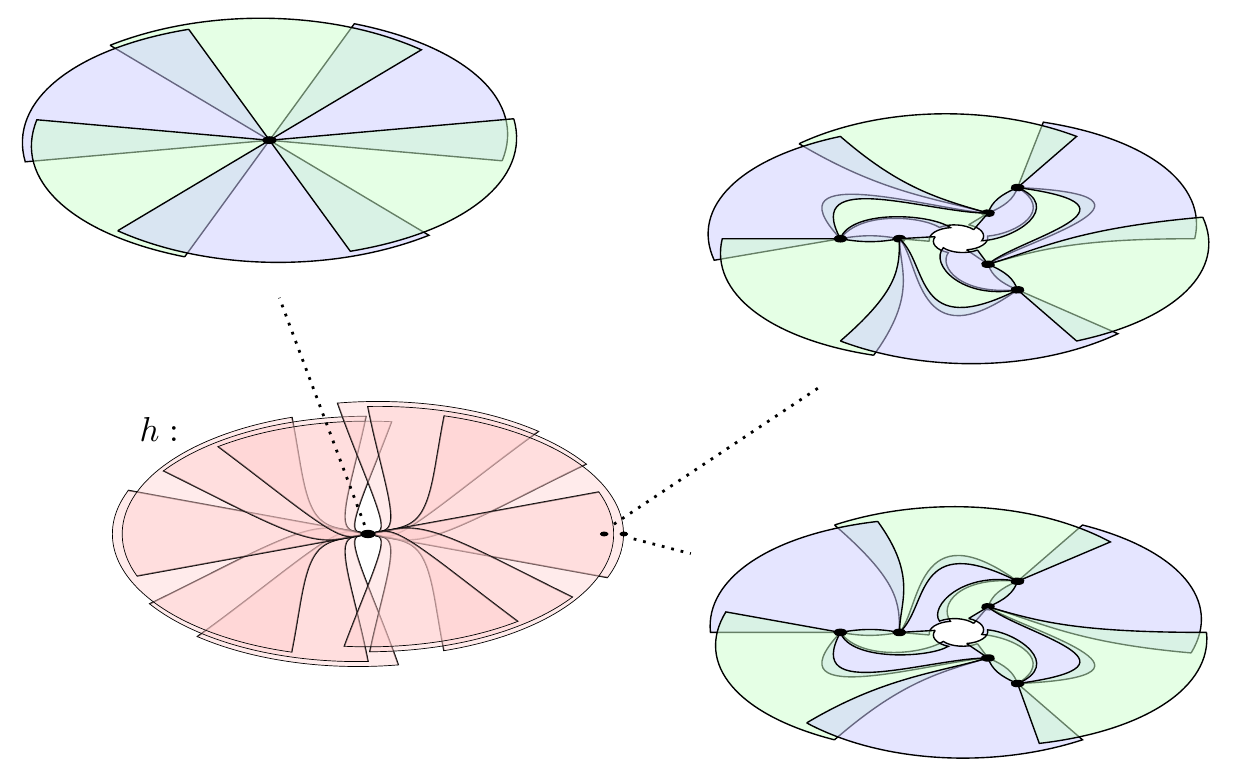}
	\caption{ Example of the domains of Theorem~\ref{thm:sectorial} in the case $p=3$, $k=1$, $s=0$, 
	for a model $\bXmod=i(u^3+h)\big(\xi_1\dd{\xi_1}-\xi_2\dd{\xi_2}\big)$.
	In the center of the figure: a covering of a small disc $\{|h|<\delta_2\}$ by a collection of cuspidal sectors $S$ (in pink). 
	For each sector $S$ and $h\in S$, the leaf $B_h=\{h=\const\}\cap\{|\xi|<\delta_1\}$, which in the coordinate $\xi_1$ has the form of an annulus,
	is covered by $4kp$ Lavaurs domains $\Omega_{S,h}^j=\Omega_S^j\cap B_h$ attached to the fixed points $\Fix(\phi^{\circ p})\cap B_h$. 
	When $h$ belongs to several sectors $S$ the associated coverings differ.
	The zero level set $\{h=0\}$ consists of two irreducible components, the figure shows the covering of only one of them.
	}
	\label{figure:sectors}
\end{figure}

We express the modulus of analytic classification as a countable collection of ``cocycles'' of transition maps on certain intersections of the covering. 
Namely, to each cuspidal sector $S$ in the $h$-plane and its associated normalizing cochain $\{\Psi_{\Omega^j_{S}}\}_{j=1,\ldots 4kp}$ on the $4kp$ Lavaurs domains $\{\Omega_S^j\}$, we  associate  a set of \emph{transition maps} on the intersections of \Grn{two subsequent outer/inner domains}:
\[ \left\{\psi_S^{j,i}:=\Psi_{\Omega^j_{S}}\circ\Psi_{\Omega^i_{S}}^{\circ(-1)}\right\}_{i,j},\quad \text{defined on } \Omega^i_S\cap\Omega^j_S, \]
preserving $\bXmod$ and possessing a $(\sigma,\Lambda)$-equivariance property.
The equivalence class of the set of transition maps 
modulo conjugation by cochains of flow maps \eqref{eq:cochainofflowmaps}
\[\psi_S^{j,i}\simeq \exp(C_{\Omega^j_{S}}(h)\bXmod)\circ\psi_S^{j,i}\circ\exp(-C_{\Omega^i_{S}}(h)\bXmod),\]
is then called a \emph{cocycle}. 

\GRN

\com{This below is an improved version of the classification theorem: the original was equivalence of (1) and (3), but here (4) is more interesting.}

\begin{thm}[Analytic classification]\label{thm:analytic}
Let $\phi,\phi'=\Lambda\xi+\hot$ be two analytic $\sigma$-reversible germs of formal type \textit{(b)} or  \textit{(c)} of Theorem~\ref{thm:1} in the prenormal form of Proposition~\ref{prop:prenormalphi}, both with the same model $\phimod$ \eqref{eq:Xmodel}. The following are equivalent:
\begin{enumerate}[wide=0pt, leftmargin=\parindent]
	\item $(\phi,\sigma)$ and $(\phi',\sigma)$ are analytically conjugated by an element of $\Diff_{\id}(\C^2,0)$. 
	\item $(\phi,\sigma)$ and $(\phi',\sigma)$ are analytically conjugated by an element of 
	\[\Diff_{\id}^{h}(\C^2,0)=\{\psi\in \Diff_{\id}(\C^2,0),\ h\circ\psi=h\}.\]
	\item For every cuspidal sector $S$ of Theorem~\ref{thm:sectorial} their associated cocycles $\{\psi_S^{j,i}\}$, $\{{\psi'}_S^{j,i}\}$ agree.
 	\item For some cuspidal sector $S$ of Theorem~\ref{thm:sectorial} their associated cocycles $\{\psi_S^{j,i}\}$, $\{{\psi'}_S^{j,i}\}$ agree.
\end{enumerate}
\end{thm}

\FGRN

In order to obtain the modulus of analytic equivalence with respect to conjugation by general transformations in $\Diff(\C^2,0)$, one has to further consider the action 
on the cocycles of the group 
\[\Cal Z^{\sigma,\Lambda}(\bXmod)=\{\psi\in\Diff(\C^2,0):\ \psi=\sigma\psi\circ\sigma=\Lambda^{-1}\psi\circ\Lambda,\ \psi^*\bXmod=\bXmod\},\]
which by Theorem~\ref{thm:1} is a subgroup of the group $\{\xi\mapsto e^{\frac{\pi i r}{kp+2s}}\sigma^r\xi,\ r\in\Z_{2kp+4s}\}$,
details are left to Section~\ref{sec:6.3}.

\begin{remark}
The restriction of $\phi$ to either irreducible component of the zero level set $\{h(\xi)=0\}$	 is a parabolic diffeomorphism of $(\C,0)$ whose   Birkhoff--\'Ecalle--Voronin modulus agrees with the corresponding restriction of the classifying  cocycle $\{\psi_S^{j,i}\}$  for each of the cuspidal sectors $S$ in the $h$-plane. In particular, this implies that the modulus is indeed infinite-dimensional (see Example~\ref{example:EV}).
\end{remark}

\com{I've removed a remark, as it is no longer needed.}

\subsection{Antiholomorphic parabolic reversible diffeomorphisms \\ and Moser--Webster tripples of involutions}\label{sec:1-anti}

Let $(\chi,\tau)$ be a pair of an antiholomorphic diffeomorphism $\chi$ (i.e. $\xi\mapsto\ov{\chi(\xi)}$ belongs to $\Diff(\C^2,0)$),
and a \Grn{holomorphic} involution $\tau\in\Diff(\C^2,0)$ such that 
	\begin{equation}\label{eq:reversible2}
		\tau^{\circ 2}=\id,\qquad  \tau\circ\chi\circ\tau=\chi^{\circ(-1)}.
	\end{equation}
Then $\rho=\tau\circ\chi$ is an antiholomorphic involution reversing $\chi$, 
	\[\rho^{\circ 2}=\id,\qquad  \rho\circ\chi\circ\rho=\chi^{\circ(-1)},\]
and the problem of classification of pairs $(\chi,\tau)$ with respect to holomorphic conjugation is equivalent to that of classification of pairs of a holomorphic and anti-holomorphic involution $(\tau,\rho)$, or of \emph{Moser--Webster tripples of involutions} 
\begin{equation}\label{eq:MWtriple}
	(\tau_1,\tau_2,\rho)=(\tau,\tau\circ\chi^{\circ 2},\tau\circ\chi),
\end{equation}
where two holomorphic involutions $\tau_1,\tau_2$ are intertwined by a third antiholomorphic involution $\rho$:
\begin{equation}\label{eq:intertwining1}
	\tau_1\circ\rho=\rho\circ\tau_2.
\end{equation}

\Grn{Assume that the reversible holomorphic diffeomorphism $(\phi,\tau)=(\chi^{\circ 2},\tau)$ satisfies Assumptions~\ref{assumptions}:
it is parabolic, $\chi^{\circ 2p}\in\Diff_{\id}(\C^2,0)$ for some $p\geq 1$, and has a first integral $H=H\circ\chi^{\circ 2}=H\circ\tau$ of Morse type.}
Up to a linear change of coordinate (Lemma~\ref{lemma:chitau}), $(\chi,\tau)$ and $H$ take the form:
\begin{equation}\label{eq:taurho}
	\chi(\xi)=\Lambda^{\frac12}\sigma\ov\xi+\hot(\ov\xi),\qquad \tau(\xi)=\sigma\xi+\hot(\xi), \qquad H(\xi)=\xi_1\xi_2+\hot(\xi),
\end{equation}
where
\[\sigma=\left(\begin{smallmatrix}0& 1\\[3pt] 1&0\end{smallmatrix}\right),\qquad
\Lambda=\left(\begin{smallmatrix}\lambda& 0\\[3pt] 0&\lambda^{-1}\end{smallmatrix}\right)=\ov\Lambda^{-1}.\]

\GRN
Since $\phi=\chi^{\circ 2}$ is a holomorphic diffeomorphism of $(\C^2,0)$ reversed by both $\tau$ and $\rho$,
the classification of pairs $(\chi,\tau)$ is a priori a refinement of that of holomorphic pairs $(\phi,\tau)$ with an additional antiholomorphic symmetry.

\begin{theorem}\label{thm:antiholomorphicclassification}
	Two pairs $(\chi,\tau),\ (\chi',\tau')$ \eqref{eq:reversible2} with $(\chi^{\circ2},\tau)$,  $(\chi'^{\circ2},\tau')$ satisfying Assumptions~\ref{assumptions} 
	are: 
	\begin{enumerate}
		\item Analytically (resp. formally) conjugated by a tangent-to-identity transformation if and only if $(\chi^{\circ 2},\tau)$, $({\chi'}^{\circ 2},\tau')$ are.
		\item Analytically (resp. formally) conjugated by a general transformation if and only if $(\chi^{\circ 2},\tau)$, $({\chi'}^{\circ 2},\tau')$ are analytically conjugated by a transformation with real linear part.
	\end{enumerate}
\end{theorem} 
\FGRN

\GRN

So by virtue of Theorem~\ref{thm:antiholomorphicclassification}, the formal normal form $(\hatphinf,\hattaunf)$ of Theorem~\ref{thm:1} for $(\chi^{\circ 2},\tau)$, provides in fact also a formal normal form $(\hatchinf,\hattaunf)$, namely
\begin{equation}\label{eq:chinormalform1}
	\hatchinf(\xi)=\sigma\hatrhonf,\qquad
	\hatrhonf=\exp(-\tfrac{1}{2p}\hatbXnf)(\Lambda^{-\frac12}\bar\xi), \qquad 
	\hattaunf(\xi)=\sigma\xi,
\end{equation}
where $\hatbXnf$ is as in Theorem~\ref{thm:1}\,(o)--(b), and satisfies $\hatbXnf=-\ov{(\Lambda^{\frac12})^*\hatbXnf}$, see Theorem~\ref{thm:normalform}. 
\FGRN

However we find it more convenient to linearize the antiholomorphic involution $\rho=\tau\circ\chi$ instead of $\tau$, 
\Grn{which leads to a more symmetric formal normal form for the Moser--Webster triple \eqref{eq:MWtriple}. }

\begin{thm}[Formal classification]\label{thm:2}
	Let $(\chi,\tau)$ and $H$ be as above, with $(\chi^{\circ 2},\tau)$ satisfying Assumptions~\ref{assumptions}.
	\Grn{Let $s\geq 0$ be the multiplicity of the zero level set $\{H(\xi)=0\}$ in the fixed point divisor $\Fix(\chi^{\circ 2p})$ which is of the form
		$\{H^s(\xi)\cdot g(\xi)=0\}$ for some analytic germ $g(\xi)$,
		and denote $kp=\ord_0 g(\xi)$ its order of vanishing.}
	
Then $(\chi,\tau)$ is formally conjugated to the following {\bf normal form} 
	$(\hatchinf', \hattaunf')$: 
	\begin{equation}\label{eq:SFNF}
		\hatchinf'=\hattaunf'\circ\rhonf',
		\qquad \hattaunf'(\xi)=\exp(\tfrac{1}{2p}\hatbXnf')(\Lambda^{\frac12}\sigma\xi),
		\qquad \rhonf'(\xi)=\bar\xi, 
	\end{equation}
	where  $\Lambda^{\frac12}=\left(\begin{smallmatrix} \lambda^{\frac12} & 0\\[3pt] 0 & \lambda^{-\frac12} \end{smallmatrix}\right)$,  
	and where 	
	\[\hatbXnf'(\xi)=-\ov{\hatbXnf'}(\xi)=-(\sigma\Lambda^{\frac12})^*\hatbXnf'(\xi)=\Lambda^*\hatbXnf'(\xi)\]
	is one of the following vector fields
	\begin{itemize}[wide=0pt, leftmargin=\parindent]
		\item[(o)] $\hatbXnf'(\xi)=0$, i.e.
		\[\hattaunf'(\xi)=\Lambda^{\frac12}\sigma\xi \quad\text{is a linear map.}\] 
		\Grn{The group $\Cal Z(\hattaunf',\rhonf')=\Cal Z(\sigma\Lambda^{\frac12},\rhonf')$ of formal $\sigma\Lambda^{\frac12},\rhonf'$-equivariant diffeomorphisms consists of maps $\xi\mapsto \zeta(h,u')\cdot\xi$, where $\zeta(h,u')=\bar{\zeta}(h,u'), \zeta(0,0)\neq 0$.}
		
		This case happens if and only if $\chi^{\circ2p}=\id$, and the conjugation is convergent.
		
		\item[(a)]  $\hatbXnf'(\xi)=\pm 2i p\, h^s\big(\xi_1\tdd{\xi_1}-\xi_2\tdd{\xi_2}\big)$, $s\geq 1$,
		i.e.
		\[\hattaunf'(\xi)=\Lambda^{\frac12} e^{\pm ih^sJ}\sigma\xi,
		\qquad\text{where\ }\ J=\left(\begin{smallmatrix} 1& 0\\[3pt]0 & -1 \end{smallmatrix}\right).\]
		
		\Grn{The group  $\Cal Z(\hattaunf',\rhonf')=\Cal Z(\hatbXnf',\sigma\Lambda^{\frac12},\rhonf')$ of formal diffeomorphisms commuting with $\hattaunf',\rhonf'$, which are the same as $\sigma\Lambda^{\frac12},\rhonf'$-equivariant diffeomorphisms preserving $\hatbXnf'$, is generated by the involution $\xi\mapsto -\xi$.}
				
		\item[(b)] $\displaystyle{\hatbXnf'(\xi)= h^s\frac{c\,P'(u',h)}{1+c\,\hat\mu(h)P'(u',h)}\big(\xi_1\tdd{\xi_1}-\xi_2\tdd{\xi_2}\big)}$,  $s\geq 0$,\\
		where $P'(u',h)$ is an analytic polynomial in $u'(\xi)=\xi_1^p+\lambda^{\frac{p}2}\xi_2^p$ (note that $\lambda^{\frac{p}2}=\pm 1$),  
		\[P'(u',h)=u'^k+ P_{k-1}(h) u'^{k-1}+\ldots+ P'_0(h),\quad P'(u',0)=u'^k,\quad k>0,\] 
		$\hat\mu(h)=\sum_{n=0}^{+\infty}\mu_nh^n$ is a formal power series, and
		\[c=\pm 2ip,\quad \ov{P'}(u',h)= P'(u',h), \quad \ov{\hat\mu}(h)=-\hat\mu(h),\]
		are unique up to a change 
		\[\big(c,\,\hat\mu(h),\, P'(u',h)\big)\mapsto \big((-1)^{kp}c,\,\hat\mu(h),\,(-1)^{kp} P'(-u',h)\big).\]
		 \Grn{The group $\Cal Z(\hattaunf',\rhonf')=\Cal Z(\hatbXnf',\sigma\Lambda^{\frac12},\rhonf')$   of formal diffeomorphisms commuting with $\hattaunf',\rhonf'$, which are the same as $\sigma\Lambda^{\frac12},\rhonf'$-equivariant diffeomorphisms preserving $\hatbXnf'$, is either trivial or generated by the involution $\xi\mapsto -\xi$; in particular if $kp$ is odd then it is trivial.}
	\end{itemize}	
\end{thm}

The associated Moser--Webster triple of involutions $(\hattaunff1',\hattaunff2',\rhonf')=(\hattaunf',\rhonf'\circ\hattaunf'\circ\rhonf',\rhonf')$ takes the form:
\begin{equation}\label{eq:taunf12}
	\hattaunff1'(\xi)=\exp(\tfrac{1}{2p}\hatbXnf')(\Lambda^{\frac12}\sigma\xi),
\qquad \hattaunff2'(\xi)=\sigma\Lambda^{\frac12}\exp(\tfrac{1}{2p}\hatbXnf')(\xi),
\qquad \rhonf'(\xi)=\bar\xi. 
\end{equation}

\begin{remark}
	The variables $h=\xi_1\xi_2$ and $u'=\xi_1^p+\lambda^{\frac{p}{2}}\xi_2^p$ are basic $\Lambda^{\frac12}\sigma$-invariant functions satisfying
	$h\circ\rhonf'(\xi)=\ov{h(\xi)}$, $u'\circ\rhonf'(\xi)=\ov{u'(\xi)}$.
\end{remark}

In the case (a)	of Theorem~\ref{thm:2}, there exist analytic germs that are formally equivalent to the normal form but in general not analytically (Theorem~\ref{thm:divergent}), and there are topological obstructions to convergence.

\GRN
In the case (b) of Theorem~\ref{thm:2}, the model (Definition~\ref{def:model}) associated to the normal form \eqref{eq:chinormalform1}
\[\chimod(\xi)=\sigma\rhomod(\xi),\qquad
\rhomod(\xi)=\exp(-\tfrac{1}{2p}\bXmod)(\ov{\Lambda^{\frac12}\xi}),\qquad \taumod(\xi)=\sigma\xi,\]
is such that
\[ \bXmod=-\rho_0^*\bXmod=-\ov{(\Lambda^{\frac12})\vphantom{\big|}^*\bXmod},\quad\text{where }\ \rho_0(\xi)=\ov{\Lambda^{\frac12}\xi} .\] 
Now in the Theorem~\ref{thm:sectorial} on sectorial normalization of $\phi=\chi^{\circ 2}$ to the model $\phimod=\Lambda\exp(\tfrac1p\bXmod)=(\sigma\rhomod)^{\circ 2}$
by means of a cochain $\{\Phi_{\Omega_{S}^j}\}_{j=1,\ldots,4kp}$,
such cochain also exists that furthermore satisfies
\[\Psi_{\rho_0(\Omega^j_S)}\circ\rho=\rhomod\circ\Psi_{\Omega^j_S},\]
which is equivalent to 
\[\Psi_{\sigma\rho_0(\Omega^j_S)}\circ\chi=\chimod\circ\Psi_{\Omega^j_S},\]
(note that if the Lavaurs domain $\Omega^j_S$ is defined over a sector $S$, then $\rho_0(\Omega^j_S)$ is defined over $\ov S$.)
By Theorem~\ref{thm:antiholomorphicclassification}, the analytic classification is achieved in terms of the same functional modulus as in Theorem~\ref{thm:analytic}. 
\FGRN

Example~\ref{example:EV} below shows that the moduli space  \Grn{in the formal case (b)} is indeed infinite-dimensional.

\begin{example}\label{example:EV}
	Let $z\mapsto f(z)=\lambda^{\frac12}z+\hot(z)$ be any \Grn{holomorphic diffeomorphism of $(\C,0)$ that is reversed by the complex conjugation $z\mapsto\bar z$,}
	i.e. $\bar f(z)=f^{\circ(-1)}(z)$.
	Let 
	\[\psi(\xi)=\left(\lambda^{\frac14}\frac{\bar f(\xi_1)}{f(\xi_2)}\xi_2,\ \lambda^{-\frac14}\frac{f(\xi_2)}{\bar f(\xi_1)}\xi_1 \right).\] 
	We have $(\xi_1\xi_2)\circ\psi=\xi_1\xi_2$ and $\ov\psi(\xi)=\sigma\psi(\sigma\xi)$.
	Therefore, 
	\[\tau_1(\xi)=\sigma\psi\circ\sigma\circ\psi^{\circ(-1)}\circ\sigma(\xi),\quad  
	\tau_2(\xi)=\psi\circ\sigma\circ\psi^{\circ(-1)}(\xi),\quad \rho(\xi)=\ov\xi,\]
	is a Moser--Webster triple, for which the restriction of $\phi=\tau_1\circ\tau_2$ to $\{\xi_2=0\}$ is the diffeomorphism $\phi_0:\xi_1\mapsto f^{\circ 4}(\xi_1)$.
	The space of analytic moduli of such diffeomorphisms $\phi_0$ inside the Birkhoff--\'Ecalle--Voronin moduli space is defined by some symmetry conditions (see \cite{Ahern-Gong2,GR1,Nakai,Trepreau}), nevertheless it is infinite-dimensional.
	Therefore also the analytic moduli space of Moser--Webster triples 
	of formal type (b) of Theorem~\ref{thm:2} (with $s=0$ and $k\in 4\Z_{>0}$) is infinite-dimensional.
\end{example}

\subsection{Non-degenerate CR-singularities of surfaces in $\C^2$}

Let us consider a germ of real analytic surface $M$ in $(\C^2,0)$ 
\begin{equation}\label{eq:M}
	M:\ {z_2=F(z_1,\bar z_1)},\qquad z\in (\C^2,0),
\end{equation}
that is a higher order perturbation of the quadric 
\begin{equation}\label{eq:quadric}
	Q_\gamma:\ \begin{cases} z_2=\gamma^{-1}z_1\bar z_1+z_1^2+\bar z_1^2,\quad&\text{for }\ \gamma\in\,\,]0,\infty],\\
		z_2=z_1\bar z_1,&\text{for }\ \gamma=0.\end{cases} 
\end{equation}
When $\gamma\neq\frac12$ then such surface $M$ is \emph{totally real} outside of the origin: its real tangent space $T_zM\subset\C^2$ has no non-trivial complex subspace, meaning that $T_zM\oplus iT_zM=\C^2$ for $z\neq 0$,
but not at $z=0$, where the tangent $T_0M=iT_0M$ is a complex subspace of $\C^2$.
In another words, $M$ exhibits a \emph{CR-singularity} at the origin.
The problem of interest is that of \emph{formal and analytic classification} of such germs $M\subset(\C^2,0)$ with respect to holomorphic changes of the complex coordinate 
\begin{equation}\label{eq:transformationrule}
	z\mapsto f(z),
\end{equation} 
that preserve the CR singularity.
This problem has a long history going back to the works of E.~Bishop \cite{Bishop} and J.~Moser, S.~Webster \cite{Moser-Webster}. 
\Grn{The type of the quadric \eqref{eq:quadric}} depends on the value of the \emph{Bishop invariant} $\gamma$ -- one commonly distinguishes:
\[\begin{cases}
	\gamma\in[0,\frac{1}{2}[\,: &\textit{elliptic case},\\
	\gamma=\frac{1}{2}: &\textit{parabolic case},\\
	\gamma\in\,\,]\frac{1}{2},\infty]: &\textit{hyperbolic case}.
\end{cases}\]
A hyperbolic case is called \emph{exceptional} if the roots $\lambda,\lambda^{-1}$ of 
\begin{equation}\label{eq:lambda}
	\lambda+\lambda^{-1}=\gamma^{-2}-2 
\end{equation}	
are complex roots of unity.

The basic premise of the seminal work of J.~Moser \& S.~Webster \cite{Moser-Webster} is that for all $\gamma\in\,]0,\infty]$ the moduli space of analytic equivalence classes of surfaces \eqref{eq:M} is isomorphic to the moduli space of \Grn{certain Moser--Webster triples of involutions $(\tau_1,\tau_2,\rho)$ \eqref{eq:intertwining1}.}
To understand this correspondence, we need to complexify the surface.

Let 
\begin{equation}\label{eq:Mcomplex}
	\mathcal M:\ z_2=F(z_1,w_1),\quad w_2=\bar F(w_1,z_1),\qquad (z,w)\in (\C^4,0),
\end{equation}
be the \emph{complexification} of $M$, and let 
\[\rho^{\Cal M}:(z,w)\mapsto (\bar w,\bar z)\]
be the induced antiholomorphic involution acting on $\Cal M$. Then 
\[M=\Cal M\cap\Fix(\rho^{\Cal M}).\]
The transformation rule \eqref{eq:transformationrule} becomes
\[(z,w)\mapsto\big(f(z),\ov f(w)\big),\]
which splits between the two variables $z$ and $w$ and commutes with $\rho^{\Cal M}$. 

\Grn{If $\gamma\neq0$}, then the
\[ \pi_1,\pi_2:\mathcal M\to(\C^2,0),\quad \pi_1:(z,w)\mapsto z,\quad \pi_2:(z,w)\mapsto w, \]
are two-sheeted branched covering maps. 
Associated to them is a pair of holomorphic involutions $\tau_1^{\Cal M},\tau_2^{\Cal M}$ of $\Cal M$, that change the sheet of the projections\footnote{Our naming here of $\tau_1^{\Cal M},\tau_2^{\Cal M}$ is the opposite than in \cite{Moser-Webster}.} 
\[\pi_j\circ\tau_j^{\Cal M}=\pi_j,\quad j=1,2.\] 
They are the \emph{deck transformations} of covering maps, and are intertwined by $\rho^{\Cal M}$ 
\begin{equation*}
	\tau_2^{\Cal M}=\rho^{\Cal M}\circ\tau_1^{\Cal M}\circ\rho^{\Cal M}. 
\end{equation*} 
In the coordinates $(z_1,w_1)$ on $\mathcal M$ the triple of involutions $(\tau_1^{\Cal M},\tau_2^{\Cal M},\rho^{\Cal M})$ is identified 
with a \emph{Moser--Webster triple of involutions} $(\tau_1,\tau_2,\rho)$  of $(\C^2,0)$, of the form
\begin{equation}\label{eq:tau1tau2}
	\tau_1:\left(\begin{smallmatrix} z_1 \\[3pt] w_1\end{smallmatrix}\right)\mapsto
	\left(\begin{smallmatrix} z_1 \\[3pt] -\gamma^{-1}z_1-w_1+ \hot \end{smallmatrix}\right),\qquad 
	\tau_2:\left(\begin{smallmatrix} z_1 \\[3pt] w_1\end{smallmatrix}\right)\mapsto
	\left(\begin{smallmatrix} -z_1-\gamma^{-1}w_1+\hot\\[3pt] w_1 \end{smallmatrix}\right),
\end{equation}		
and $\rho: \left(\begin{smallmatrix} z_1 \\[3pt] w_1\end{smallmatrix}\right)\mapsto 
\left(\begin{smallmatrix} \ov w_1\\[3pt] \ov z_1 \end{smallmatrix}\right)$. \label{page:MW}
The composition
\[ \phi:=\tau_1\circ\tau_2=\chi^{\circ 2},\quad\text{where}\quad \chi=\tau_1\circ\rho, \] 
is a germ of analytic diffeomorphism of $(\C^2,0)$, the linear part of which has eigenvalues $\lambda$ and $\lambda^{-1}$ related to $\gamma$ by \eqref{eq:lambda}.

\GRN
\begin{definition}
A Moser--Webster triple of involutions 	$(\tau_1,\tau_2,\rho)$ consist of a pair of holomorphic reflections $\tau_1,\tau_2\in\Diff(\C^2,0)$ and of an antiholomorphic involution $\rho$ such that
\[\tau_1^{\circ 2}=\tau_2^{\circ 2}=\rho^{\circ 2}=\id,\qquad \tau_2=\rho\circ\tau_1\circ\rho.\]
\end{definition}
\FGRN

\begin{theorem}[Moser, Webster \cite{Moser-Webster}]\label{prop:MW0}
	Two germs of surfaces \eqref{eq:M} with the same Bishop invariant \Grn{$\gamma\neq 0$} are analytically equivalent if and only if their associated Moser--Webster triples of involutions are analytically conjugated.
	Furthermore \Grn{if $\gamma\neq\frac12$}, then there is a bijective correspondence:
	\emph{\small
		\begin{equation*}
			\left\{\!\! \begin{array}{c}\text{Germs of surfaces}\\ (M,0)\ \text{ with}\\ \gamma\in\ ]0,+\infty]\sminus\{\tfrac12\} \end{array}\!\!\right\}_{\!\!\Big/_{\!\!
					\left[\!\!\begin{array}{c}\text{Biholomorphic}\\ \text{equivalence}\end{array}\!\!\right]}}
			\longleftrightarrow\
			\left\{\!\! \begin{array}{c}\text{Moser--Webster}\\ \text{triples}\ (\tau_1,\tau_2,\rho) \\ \text{with }\lambda,\lambda^{-1} \end{array}\!\!\right\}_{\!\!\Big/_{\!\!
					\left[\!\!\begin{array}{c}\text{Conjugation}\\ \text{in } \Diff(\C^2,0)\end{array}\!\!\right]}}
		\end{equation*}	
	}
\end{theorem}

Analytic classification in the elliptic case with $\gamma\in\,\,]0,\frac{1}{2}[\,$, corresponding to $\lambda\in\R_{>0}\sminus\{1\}$, 
was achieved in the original study by J.~Moser \& S.~Webster \cite{Moser-Webster}. Using the correspondence of Theorem~\ref{prop:MW0}, they showed that in this case the formal classification agrees with the analytic one, and that each such surface is analytically equivalent to  one of the following normal forms
\GRN
\begin{equation}\label{eq:MWnf}
\Mnf:\  z_2=\big(\lambda^{\frac12} e^{\epsilon (-\Re z_2)^s}+\lambda^{-\frac12} e^{-\epsilon (-\Re z_2)^s}\big)z_1\bar z_1+z_1^2+\bar z_1^2,
\end{equation}
with
\[\epsilon=\begin{cases} 0,&s=+\infty,\\ \pm1,& s\geq1,\end{cases}\]
associated to the Moser--Webster tripple
\[\taunff{1}(\xi)=\phinf^{\circ(\frac12)}(\sigma\xi),\qquad \taunff{2}=\sigma\phinf^{\circ(\frac12)}(\xi),\qquad \rhonf(\xi)=\sigma\ov\xi,\]
where
\begin{equation}\label{eq:MWtriplenf}
	\phinf^{\circ\frac12}(\xi)=\left(\begin{smallmatrix} \lambda^{\frac12} e^{\epsilon (\xi_1\xi_2)^s}\!\!\! &0 \\[3pt]0& \!\!\!\lambda^{-\frac12} e^{-\epsilon(\xi_1\xi_2)^s} \end{smallmatrix}\right)\xi,\qquad
\sigma=\left(\begin{smallmatrix} 0\,&1\\[4pt] 1&\,0 \end{smallmatrix}\right).
\end{equation}
The surface \eqref{eq:MWnf} is also analytically equivalent to \cite[p.289]{Moser-Webster}
\begin{equation}\label{eq:MWnf2}
z_2=\begin{cases}
		z_1\bar z_1+\gamma\,(z_1^2+\bar z_1^2),\quad s=+\infty\\[6pt]
		z_1\bar z_1+\big(\gamma\pm(\RE z_2)^s\big)(z_1^2+\bar z_1^2),\quad s\geq 1.\end{cases}
\end{equation}
\FGRN

A complete classification in the limit elliptic case $\gamma=0$ was later obtained by X.~Huang \& W.~Yin \cite{Huang-Yin}, 
who constructed an infinite-dimensional space of formal normal forms, and proved that formally equivalent surfaces are analytically equivalent.

The parabolic case $\gamma=\frac{1}{2}$,  corresponding to $\lambda=1$, \Grn{is slightly different as there might be a whole curve of CR-singularities in $M$, but the Moser--Webster correspondence does nevertheless extend to this case.} Here the classification was described by P.~Ahern \& X.~Gong \cite{Ahern-Gong}
in terms of a functional modulus (cocycle) related to S.M.~Voronin's classification of germs of diffeomorphisms with unipotent linear part \cite{Ilyashenko}.

In the  non-exceptional hyperbolic case, $\gamma>\frac{1}{2}$ with $\lambda\notin e^{\pi i\Q}$, 
the formal classification was also provided by J.~Moser \& S.~Webster \cite{Moser-Webster} with formal normal form \eqref{eq:MWnf} except this time with
\[\epsilon=\begin{cases} 0,&s=+\infty,\\ \pm i,& s\geq1,\end{cases},\qquad \text{and }\ \rho(\xi)=\ov\xi,\]
which is also equivalent to \eqref{eq:MWnf2}.
%
However, the normalizing transformations in this case exhibit a small divisor problem and are in general divergent \cite{Moser-Webster,Gong1,Gong2,Gong3}.
In the case when $M$ is formally equivalent to the quadric $Q_{\gamma}$ and $\lambda$ satisfies a Diophantine condition, or more generally
a Brjuno type condition, then the existence of a convergent normalizing transformation was established by X.~Gong and L.~Stolovitch \cite{Gong1,Gong-Stolovitch_1}.
In the case when $M$ is not formally equivalent to the quadric $Q_{\gamma}$ a KAM-like phenomena arise for all non-exceptional $\lambda$ \cite{Gong2, Stolovitch-Zhao},
where an analytic conjugacy can be achieved between certain real analytic curves in $M$ and the hyperbolas $\{z_2=\omega=\big(\lambda^{\frac12} e^{\pm i\omega^s}+\lambda^{-\frac12} e^{\mp i\omega^s}\big)z_1\bar z_1+z_1^2+\bar z_1^2\}$, $\omega\in(\R,0)$, in \eqref{eq:MWnf} under a Diophantine type condition on the value of $\lambda^{\frac12} e^{\pm i\omega^s}$.

In this paper we are interested in the \emph{\textbf{exceptional hyperbolic case}}, that is when $\gamma\in\,\,]\frac{1}{2},\infty]$ 
and $\lambda,\lambda^{-1}$ \eqref{eq:lambda} are \emph{non-trivial complex roots of unity} of order $p$: 
\begin{equation}\label{eq:p}
	\lambda^{p}=1\quad\text{with}\quad p\geq 2. 
\end{equation}
In this case $\phi^{\circ p}=(\tau_1\circ\tau_2)^{\circ p}=\id+\hot$ and the dynamics of $\phi$ is of resonant parabolic type.
Nothing seems to have been known about normal forms, and formal or analytic classification in this situation.
We will work under an additional assumption that $M$ is \emph{\textbf{holomorphically flat}}, meaning that 
\Grn{it is contained in some Levi flat analytic real hypersurface of $\C^2$. Up to a biholomorphic change of coordinate \eqref{eq:transformationrule}, one can assume that
this hypersurface is  $\{\IM z_2=0\}$, i.e. that}
\begin{equation}\label{eq:flatM}
	M:\ \RE z_2=F(z_1,\bar  z_1),\quad \IM z_2=0.
\end{equation}
This means that $M$ is foliated by the family of real curves $\{\RE z_2=\const\}\cap M$.
The assumption of holomorphic flatness is equivalent to the existence of an \emph{analytic first integral} $H$ for the pair of involutions $(\tau_1,\tau_2)$
\begin{equation}\label{eq:firstintegral}
	H(z_1,w_1)=\gamma^{-1}z_1w_1+z_1^2+w_1^2+\hot,\qquad H\circ\tau_1=H\circ\tau_2=H,
\end{equation} 
which has a Morse point at 0 (Proposition~\ref{prop:flatness}).
In fact, if $M$ is in the form \eqref{eq:flatM} then $H(z_1,w_1):=F(z_1,w_1)=\bar F(w_1,z_1)$ is such first integral.

\Grn{In the elliptic case, every surface $M$ is holomorphically flat. This follows from the holomorphic conjugacy to the Moser--Webster normal form \eqref{eq:MWnf} for $0<\gamma<\frac12$ \cite[Theorem 1]{Moser-Webster}, and for $\gamma=0$ from the work of Huang \& Krantz \cite{Huang-Krantz}.
On the other hand in the hyperbolic case there are surfaces $M$ with any $\frac12<\gamma\leq\infty$ which are not holomorphically flat.}
Examples of such surfaces have been constructed by J.~Moser \& S.~Webster \cite{Moser-Webster}, E.~Bedford \cite{Bedford}, X.~Gong \cite{Gong3}, and others.
Furthermore, it has been known that holomorphic flatness alone is not enough to assure existence of convergent transformation to a formal normal in the hyperbolic case \cite{Gong1}. In fact, X.~Gong \cite[Theorem 1.3]{Gong2}, \cite[Theorem 1.1]{Gong3} and \cite[Theorem~1.2]{Gong23} shows that for each non-exceptional
$\gamma\in\,\,]\frac12,\infty[$ \Grn{and $s\geq 1$} there exists a holomorphically flat surface \eqref{eq:flatM} which is formally but not analytically equivalent to \eqref{eq:MWnf}. 
We will show that this is also true for exceptional $\gamma\in\,\,]\frac12,\infty]$.

\begin{thm}\label{thm:divergence}
	For any $\gamma\in\,\,]\tfrac12,\infty]$ and every $s\geq 1$, there exists a holomorphically flat manifold $M$ that is formally equivalent to the Moser--Webster normal form \eqref{eq:MWnf}, but not analytically.	
\end{thm}	

On the other hand, it has also been known \cite[Theorem 1.1]{Gong1}, and is easy to show, that

\begin{prop}\label{thm:convergence}
	If a manifold $M$ is formally equivalent to the quadric $Q_\gamma$ with exceptional $\gamma\in\,\,]\tfrac12,\infty]$,
	then it is analytically equivalent to it. In particular, it is holomorphically flat. 
\end{prop}

In view of Theorem~\ref{prop:MW0}, the formal classification of holomorphically flat surfaces $M$ of exceptional hyperbolic 
type is achieved in an implicit way by Theorem~\ref{thm:2}.
In particular, the triple of involutions $(\hattaunff1',\hattaunff2',\rhonf')$ \eqref{eq:taunf12} in formal normal form of the type (o), resp. (a), are associated to the quadric $Q_\gamma$, resp. the surface \eqref{eq:MWnf} with $\epsilon=\pm i$ and $s\geq 1$, see Section~\ref{sec:4oa}. Proposition~\ref{thm:convergence} is then a consequence of the finiteness of the group generated by $(\tau_1,\tau_2,\rho)$ and the linearity of the normal form $(\hattaunff1',\hattaunff2',\rhonf')$. 

\Grn{The formal type (b) of Theorem~\ref{thm:2} corresponds to a whole new formal type of surface.}
In this case we don't provide an explicit formal normal form of the surface.
Instead, we find a \emph{model surface} $\Mmod$ which is a representant of a larger \emph{model class} of surfaces, corresponding to the model class of $(\tau_1,\tau_2,\rho)$.
Theorem~\ref{thm:sectorial} on ``sectorial'' conjugacy between $(\phi,\tau)$ and its model $(\phimod,\taumod)$ has also its analogy
as a ``sectorial'' conjugacy between the Moser--Webster triple  $(\tau_1,\tau_2,\rho)$ and  its model $(\taumodf1',\taumodf2',\rhomod')$,
and can be rephrased directly as a ``sectorial'' equivalence between the complexified surfaces $\Cal M$ and $\CMmod$.

\begin{theorem}[``Sectorial'' equivalence]\label{thm:sectorialM}
Let $M$ be a  germ of a holomorphically flat manifold in $(\C^2,0)$ with an exceptional Bishop invariant $\gamma$,
$\Cal M$ its complexification and $(\tau_1^{\Cal M},\tau_2^{\Cal M},\rho^{\Cal M})$ the associated Moser--Webster triple of involutions acting on $\Cal M$.
Let $p\geq 2$ be the smallest positive integer such that $\lambda^{p}=1$, and let 
\[\Cal D^{\Cal M}=\Fix(\tau_1^{\Cal M})\cup\Fix(\tau_2^{\Cal M})\cup\Fix\big((\tau_1^{\Cal M}\circ\tau_2^{\Cal M})^{\circ p}\big)\] 
be a divisor in $\Cal M$.
Assume that $M$ is not formally equivalent to \eqref{eq:MWnf}.
Then there exist positive reals $\delta_1,\delta_2>0$, and a countable collection of cuspidal sectors\footnote{See Definition~\ref{def:sector}.} 
with vertex at $z_2=0$ covering together the disc $\{|z_2|<\delta_2\}$,
and for each such sector $S$:
\begin{itemize}
	\item there is a family of $4kp$ domains 
	$\big\{\Omega_{S,j}^{\Cal M}\big\}_{j=1,\ldots,4kp}$, covering $\big(\Cal M\sminus\Cal D^{\Cal M}\big)\cap\{|z_1|,|w_1|<\delta_1,\ z_2,w_2\in S\}$, 
	\item and a family of bounded analytic transformations of the form
\[\psi_{\Omega_{S,j}^{\Cal M}}(z,w)=\Big(\big(f_{\pi_1(\Omega_{S,j}^{\Cal M})}(z),\ \varphi(z_2)\big),\ \big(g_{\pi_2(\Omega_{S,j}^{\Cal M})}(w),\ \ov \varphi (w_2)\big)\Big),\]
defined	on the product domains $(z,w)\in \pi_1(\Omega_{S,j}^{\Cal M})\times\pi_2(\Omega_{S,j}^{\Cal M})$,\\
where $z_2\mapsto \varphi(z_2)$ is an analytic diffeomorphism on $\{|z_2|<\delta_2\}$,
\end{itemize}
that map $\Cal M$ to some complexified  model surface $\CMmod$ (see \eqref{eq:Mmodel} in Section~\ref{sec:4b})
 and $M=\Cal M\cap\Fix(\rho^{\Cal M})$ to $\Mmod=\CMmod\cap\Fix(\rho^{\CMmod})$.
\end{theorem}

As a consequence to Theorem~\ref{thm:2}, we also have:

\begin{prop}[Automorphism group]
	The group of formal automorphisms of a holomorphically flat surface $M$ that is not formally equivalent to $Q_\gamma$, with an exceptional Bishop invariant $\gamma$, is either trivial or isomorphic to $\Z_2$.
\end{prop}

\goodbreak
\newpage

\subsection{Organization of the paper}
\begin{itemize}[wide=0pt, leftmargin=\parindent]

	\item[-] Section~\ref{sec:3}: We obtain a formal classification of singular reversible integrable vector fields (Theorem~\ref{thm:1X}) and of parabolic reversible integrable diffeomorphisms (Theorem~\ref{thm:1}). 
	This is done through a construction of a formal normal form. 
	\item[-] Section~\ref{sec:4}:  We prove Theorem~\ref{thm:2} on formal classification of antiholomorphic parabolic reversible integrable diffeomorphisms. 
	\item[-] Section~\ref{sec:2}: We recall the basics of the Moser--Webster correspondence between singular CR-surfaces $M$ triples of involutions $(\tau_1,\tau_2,\rho)$, and derive the form of model surfaces.
	\item[-] Section~\ref{sec:5}: In \S\ref{sec:5a} we construct an explicit example of divergence in the case (a) $k=0$ (Theorem~\ref{thm:divergent}), a corollary of which is Theorem~\ref{thm:divergence}. In \S\ref{sec:Painleve} we provide a more detailed account of Example~\ref{example:Painleve} on the monodromy of the Sixth Painlev\'e equation. 
	\item[-] Section~\ref{sec:6}: In \S\ref{sec:6.1} we show existence of bounded ``sectorial'' holomorphic transformations to a \emph{model} diffeomeorphism in the formal type (b),(c) $k\geq1$. This is done through the construction of \emph{Fatou coordinates} \Grn{on certain domains, called \emph{Lavaurs domains}} in each level set of the first integral $h$. In these coordinates, the diffeomorphism reads as a translation depending on $h$. \Grn{We make sure that the construction depends well on the level $h$ and extends to the limit $h\to 0$}. 
	
	\Grn{In order to understand the form and topological organization of the Lavaurs domains,
	one needs to understand the dynamics of the model vector field, which is a rational vector field on each level set of $h$.
	On that purpose,  we consider its complex flow when time evolves along real lines in some direction $\theta$. 
	This amount to consider the \emph{real flow} of a family of holomorphic vector fields depending both on the level $h$ and on a parameter of rotation $\theta$. 
	We seek domains on each leaf $\{h=\const\}$ that are \emph{stable} as $(h,\theta)$ varies. 
	This is done in \S\ref{sec:6.2} using theory of the real-time dynamics of rational vector fields on $\CP^1$ which we recall.
	We then provide a precise, albeit a bit technical, construction of the Lavaurs domains and prove that they cover a full neighborhood of $0\in\C^2$ (Theorem~\ref{thm:covering}).}
	
	In \S\ref{sec:6.3} we prove Theorem~\ref{thm:sectorial} (Theorem~\ref{thm:cochain}) and describe the modulus of analytic classification in terms of a bounded cocycle Theorem~\ref{thm:analytic} (Theorem~\ref{thm:analyticclassification}.)
	
	\Grn{In \S\ref{sec:6.3compatibility} we discus certain ``compatibility conditions'' between the classifying cocycles over different sectors in the $h$-space.}
\end{itemize}

\goodbreak
\newpage

\section{Formal invariants for  parabolic integrable reversible diffeomorphisms}  \label{sec:3}

Main goal of this section is to prove Theorem~\ref{thm:1}.

\GRN
\begin{lemma}\label{lemma:phitau}
Let $(\phi,\tau)$ be a pair of a reversible diffeomorphism $\phi$ and its reversing reflection $\tau$, and $H=H\circ\phi=H\circ\tau$ a first integral of Morse type.
There exists an analytic change of coordinates under which $(\phi,\tau)$ and $H$ take the form
\begin{equation}\label{eq:phitau2}
		\phi(\xi)=\Lambda\xi+\hot(\xi),\qquad \tau(\xi)=\sigma\xi, \qquad H(\xi)=\xi_1\xi_2,  
\end{equation}	
with
\begin{equation*}
	\sigma=\left(\begin{smallmatrix}0&1\\[3pt]1&0\end{smallmatrix}\right),
	\qquad 	 \Lambda=\begin{cases}
		\left(\begin{smallmatrix}\lambda & 0 \\[3pt] 0 & \lambda^{-1}\end{smallmatrix}\right),\\[6pt]
		-\sigma.
	\end{cases}
\end{equation*}	
\end{lemma}

\begin{proof}
	Up to a linear change of variables, one can first assume that  $\tau(\xi)=\sigma\xi+\hot(\xi)$.
	The transformation $\Psi=\tfrac12\big(\id+\sigma\tau\big)$ is tangent to the identity (hence invertible) and such that $\Psi\circ\tau=\sigma\Psi$.
	
	So we can now also assume that $\tau(\xi)=\sigma\xi$. 
	If $\phi(\xi)=A\xi+\hot(\xi)$, then $\sigma A\sigma=A^{-1}$, and so if $u$ is an eigenvector of $A$ with eigenvalue $\lambda$, 
	then $\sigma u$ is an eigenvector of $A$ with eigenvalue $\lambda^{-1}$. 
 Hence if $\lambda\neq \pm1$, then the linear transformation $\Psi^{\circ(-1)}(\xi)=(u,\sigma u)\xi$ is $\sigma$-equivariant and such that $\Psi\circ\phi\circ\Psi^{\circ(-1)}=\Lambda\xi+\hot(\xi)$. Then also, up to a multiplicative constant, $H\circ\Psi^{\circ(-1)}=\xi_1\xi_2+\hot(\xi)$.
If $A$ has a double eigenvalue $\lambda=\pm1$, then $A=\pm I$ (the case of non-diagonalizable $A$ is excluded by the assumption on existence of Morse first integral).
Since $H\circ\sigma=H$, up to a multplicative constant, $H(\xi)=(\xi_1+b\xi_2)(\xi_2+b\xi_1)+\hot(\xi)$ with $b^2\neq1$ (since $\det D^2H(0)\neq0$), and we use the
the $\sigma$-equivariant linear transformation $\Psi(\xi)=\begin{pmatrix} 1&b\\ b&1\end{pmatrix}\xi$.
Finally, if $A$ has eigenvalues $\{1,-1\}$, then $A=\pm\sigma$. The case $A=\sigma$ would imply that $\tau_2(\xi):=\tau\circ\phi(\xi)=\xi+\hot(\xi)$ is an involution tangent to the identity, but there is no such involution except the identity itself (because the map $\Psi=\frac12(\id+\tau_2)$ would conjugate it to the identity, $\Psi\circ\tau_2=\id\circ\Psi$). So $A=-\sigma$ and we proceed to normalize the quadratic part of $H$ in the same way as in the case $A=\pm I$.

We finish by means of the $\sigma$-equivariant Morse lemma (Lemma~\ref{lemma:Morse} below with $\Cal G=\langle\sigma\rangle$), which allows to reduce $H(\xi)$ to its quadratic part $\xi_1\xi_2$.
\end{proof}
\FGRN

\begin{lemma}[Equivariant Morse lemma \protect{\cite[chap. 17.3]{AGV1}}]\label{lemma:Morse}~\\
	Let $H:(\C^2,0)\to(\C,0)$ be a formal/analytic germ with a non-degenerate critical point at 0 (Morse point), that is invariant with respect to a linear action of a compact group $\Cal G$ on $\C^2$. Then $H$ is reducible to its quadratic part by a formal/analytic change of variables tangent to identity and commuting with $\Cal G$.	
\end{lemma}

\subsection{Formal infinitesimal generator}\label{sec:3infgen}

Let us recall  that any formal diffeomorphism $\hat F(\xi)\in\hatDiff_{\id}(\C^2,0)$ 
has a unique infinitesimal generator, that is
a formal vector field $\hat\bX(\xi)$, vanishing at the origin and with a vanishing linear part, 
whose formal time-1-map is $\hat F(\xi)$ \cite[Theorem~3.17]{IlYa}.

The formal time-1-flow of $\hat\bX$ is a formal diffeomorphism of $(\C^2,0)$ defined by
\begin{equation}\label{eq:exp-def} \exp(\hat\bX)(\xi):=\bexp(\hat\bX).\xi=\begin{pmatrix}
		\bexp(\hat\bX).\xi_1\\\bexp(\hat\bX).\xi_2 \end{pmatrix} 
\end{equation} 
where
\begin{equation}\label{eq:exp-form}
 \bexp(\hat\bX).\hat f=\sum_{n=0}^{+\infty}\tfrac{1}{n!}\hat\bX^{.n}.\hat f,
\end{equation}
which is a well defined formal power series as the order of the $n$-times iterated derivative $\hat\bX^{.n}.\hat f$ growths with $n$:
\begin{equation}\label{eq:lowerorderbd}
		\ord_0 \hat\bX^{.n}.\hat f\geq \ord_0\hat f+n(\ord_0\hat\bX.\xi-1).
\end{equation} 
It satisfies 
\begin{equation}\label{eq:compo-exp} \hat f\circ\exp(\hat\bX)(\xi)=\bexp(\hat\bX).\hat f(\xi),\end{equation} for any formal germ $\hat f\in\C\llbracket\xi\rrbracket$,
and
\[ \tdd{t}\left[\hat f\circ\exp(t\hat\bX)\right]=\hat\bX.\hat f\big|_{\xi=\exp(t\hat\bX)}, \]
\cite[Chapter~3]{IlYa}.

The zeros of $\hat\bX$ are the same as the fixed points of $\hat F(\xi)=\exp(\hat\bX)(\xi)$, more precisely 
there exists a formal matrix valued function $\hat U(\xi)=\id+\hot$ such that 
$\hat\bX.\xi=\hat U(\xi)\cdot\big(\hat F(\xi)-\xi\big)$.
If one identifies formal vector fields with derivation operators on the space of formal series $\C\llbracket\xi\rrbracket$, then 
the infinitesimal generator
$\hat\bX(\xi)$ of $\hat F(\xi)$ is the same as the operator
\[ \hat\bX.\hat f=\log(\id+\Theta)=\sum_{n=1}^{+\infty}(-1)^{n-1}\tfrac{1}{n}\Theta^n(\hat f), \quad\text{where }\ \Theta(\hat f):=\hat f\circ\hat F-\hat f, \quad \hat f\in\C\llbracket\xi\rrbracket, \]
and $\Theta^n(\hat f)=\sum_{j=0}^n(-1)^{n-j}\binom{n}{j}\hat f\circ\hat F^j$ (note that $\Theta$ is linear as operator on $\C\llbracket\xi\rrbracket$).
From this it follows that the map $\hat F(\xi)$ and the vector field $\hat\bX(\xi)$ have the same formal first integrals: 
$\hat f\circ\hat F-\hat f=0$ if and only if $\hat\bX.\hat f=0$ for $\hat f\in\C\llbracket\xi\rrbracket$,\footnote{This would no longer be true for more general formal trans-series first integrals containing exponential terms.} 
and that if
$\hat F(\xi)$ commutes with another formal diffeomorphism $\hat G(\xi)$: $\hat F\circ\hat G=\hat G\circ\hat F$,
then $\hat G$ preserves its infinitesimal generator, $\hat\bX=\hat G^*\hat\bX$ 
(indeed one gets $\Theta^n(\hat G)=\left[\sum_{j=0}^n(-1)^{n-j}\binom{n}{j}F^{\circ j}\right]\circ\hat G=\Theta^n(\id)\circ\hat G$ from which 
$\hat\bX.\hat G=\hat\bX.\xi\big|_{\xi=\hat G}$).

\begin{remark}\label{remark:gevrey}
	If $F$ is analytic, and $\nu(F)+1=\min_{i=1,2}\ord_0 (\xi_i\circ F-\xi_i)$ is the order of tangency of $F$ to $\id$. Then \cite{BMLH} show that
	the formal infinitesimal generator $\hat\bX$ is of Gevrey order $\frac{1}{\nu(F)}$, meaning that if 
	$\hat\bX.\xi_i=\sum_{|\bm m|\geq 0}f_{i\bm m}\xi^{\bm m}$, then 
	$\sum_{|\bm m|\geq 0}\frac{f_{i\bm m}}{\Gamma(1+\frac{|\bm m|}{\nu(F)})}\xi^{\bm m}$ is convergent. 
\end{remark}

\subsection{Poincar\'e--Dulac formal normal form}

In the following let $\hat\bX$  be the formal infinitesimal generator of $\phi^{\circ p}$.

\begin{lemma}[Jordan decomposition]\label{lemma:decomposition}
There exists a formal decomposition 
\[ \phi=\hat\phi_s\circ\hat\phi_u=\hat\phi_u\circ\hat\phi_s,\qquad\text{with}\quad \hat\phi_s^{\circ p}=\id,\quad \hat\phi_u^{\circ p}=\phi^{\circ p}, \]
where $\hat\phi_u(\xi)=\xi+\hot\in\hatDiff_{\id}(\C^2,0)$ is the ``unipotent'' part, and 
$\hat\phi_s(\xi)=\Lambda\xi+\hot$ is the ``semisimple'' part.
If $\phi$ is reversible by $\sigma$, then so are $\hat\phi_u$ and $\hat\phi_s$.
\end{lemma}

\begin{proof}
Let $\hat\bX$  be the formal infinitesimal generator of $\phi^{\circ p}=\exp(\hat\bX)(\xi)$, and let $\hat\phi_u(\xi):=\exp(\frac1p\hat\bX)(\xi)$. Then 
$\hat\phi_u^{\circ p}=\phi^{\circ p}$, and $\hat\phi_u\circ\phi=\phi\circ\hat\phi_u$ since $\hat\bX=\phi^*\hat\bX$ (because $\phi$ commutes with $\phi^{\circ p}$). 
Let $\hat\phi_s=\phi\circ\hat\phi_u^{\circ(-1)}=\hat\phi_u^{\circ(-1)}\circ\phi$, then $\hat\phi_s(\xi)=\Lambda\xi+\hot$ and $\hat\phi_s^{\circ p}=\id$.
If $\hat\phi$ is reversed by $\sigma$, then so is $\hat\bX$, and therefore also $\hat\phi_u$ and $\hat\phi_s$.
\end{proof}

\begin{lemma}\label{lemma:phivsphip}
	Let $\hat\phi_s:\xi\mapsto\Lambda\xi+\hot$ be a formal/analytic diffeomorphism such that $\hat\phi_s^{\circ p}=\id$.
	Then $\hat\phi_s$ is formally/analytically linearizable. If furthermore  $\hat\phi_s$ is reversed by  $\sigma$, then there exists a  $\sigma$-equivariant formal/analytic transformation linearizing $\hat\phi_s$.
\end{lemma}

\begin{proof}
	Let  
	\[ \hat\Psi:=\tfrac1p\left(\id+\Lambda^{-1}\hat\phi_s+\ldots+\Lambda^{1-p}\hat\phi_s^{\circ(p-1)}\right)=\id+\hot, \]
	then $\hat\Psi\circ\hat\phi_s=\Lambda\hat\Psi$, i.e. $\hat\Psi$ is a linearizing transformation for $\hat\phi_s$. 
	If $\sigma\hat\phi_s\circ\sigma=\hat\phi_s^{\circ(-1)}$, then also $\sigma\Lambda\sigma=\Lambda^{-1}$,
	and one sees that $\sigma\hat\Psi\circ\sigma=\hat\Psi$ since $\Lambda^{p}=\id=\hat\phi_s^{\circ p}=\id$.
\end{proof}

\begin{prop}[Poincar\'e--Dulac formal normal form]\label{prop:PoincareDulac}
There exists a formal $\sigma$-equivariant change of coordinates 
$\hat\Psi(\xi)\in\hatDiff_{\id}(\C^2,0)$, $\sigma\hat\Psi=\hat\Psi\circ\sigma$,
that transforms the map $\phi(\xi)$ to a Poincar\'e--Dulac normal form
\[ \hat\phi:\xi\mapsto\hat\phi(\xi)=\Lambda\xi+\hot,\quad \hat\phi\circ\Lambda=\Lambda\hat\phi, \]
which is reversible by $\sigma$, $\sigma\hat\phi\circ\sigma=\hat\phi^{\circ(-1)}$,
and which brings the first integral to $h(\xi)=\xi_1\xi_2$.
\end{prop} 

\begin{proof}
Let $\hat\phi=\hat\phi_s\circ\hat\phi_u$ be the Jordan decomposition of $\phi$ of Lemma~\ref{lemma:decomposition}.
After a formal $\sigma$-equivariant change of coordinates of Lemma~\ref{lemma:phivsphip}, one can assume that $\hat\phi_s=\Lambda$, which means that 
$\hat\phi$ is in a Poincar\'e--Dulac normal form.

Let $\hat h(\xi)=\xi_1\xi_2+\hot$ be a formal $\sigma$-invariant first integral for $\hat\phi$.
Then up to replacing $\hat h(\xi)$ by $\tfrac1p\big(\hat h+\hat h\circ\Lambda+\ldots+\hat h\circ\Lambda^{p-1}\big)$
we may assume that $\hat h$ is $(\sigma,\Lambda)$-invariant.
Hence by the equivariant Morse lemma (Lemma~\ref{lemma:Morse}), there exists a $(\sigma,\Lambda)$-equivariant change of variables that brings $\hat h(\xi)$ to $\xi_1\xi_2$, while keeping $\hat\phi$ in a Poincar\'e--Dulac normal form. 
\end{proof}

\Grn{If $\hat\phi$ is in the Poincar\'e--Dulac normal form, then the formal infinitesimal generator $\hat\bX(\xi)$ of $\hat\phi^{\circ p}(\xi)=\exp(\hat\bX)(\xi)$ is
such that $\Lambda^*\hat{\bX}=\hat{\bX}=-\sigma^*\hat{\bX}$ and $\hat{\bX}.h=0$.
The problem is that the fixed point divisor $\Fix(\hat\phi^{\circ p})=\{\hat{\bX}=0\}$ is a priory purely formal,
but we need to ensure that it is analytic.} 
To achieve this, we will follow a different route:
\begin{itemize}[wide=0pt, leftmargin=\parindent]
	\item First we will analytically pre-normalize the original germ $\phi$ by repeating the above Poincar\'e--Dulac reduction ``modulo $(\phi^{\circ p}-\id)$''
	in order to obtain $(\sigma,\Lambda)$-invariant analytic divisor $\Fix(\hat\phi^{\circ p})$ (Proposition~\ref{prop:prepared}),
	after which the formal infinitesimal generator $\hat\bX(\xi)$ can be written in a prepared form \eqref{eq:prenormalX}.
	\item Then in a second step we will ``formally remove'' as many terms as possible in the formal infinitesimal generator $\hat\bX(\xi)$ of $\hat\phi^{\circ p}$ while preserving the analytic set $\Fix(\hat\phi^{\circ p})$ (Proposition~\ref{prop:FNFX}).
	\item And finally, we analytically deform the divisor $\Fix(\hat\phi^{\circ p})$ in order to further simplify the form of $\hat{\bX}(\xi)$ (Theorem~\ref{thm:FNFX1}).
\end{itemize}

\subsection{Prepared form}

Let $\phi(\xi)$ be as in Lemma~\ref{lemma:phitau}, reversed by $\sigma$ and with a first integral $h(\xi)=\xi_1\xi_2$.
In particular, $(\xi_1\circ\phi)\cdot(\xi_2\circ\phi)=\xi_1\xi_2$ implies that $\phi$ is of the form
\begin{equation}\label{eq:phi3}
	\phi(\xi)=\Lambda\left(\begin{smallmatrix} \xi_1\cdot(1+\hot)\\[3pt] \xi_2\cdot(1+\hot) \end{smallmatrix}\right).
\end{equation}
Assume that $\phi^{\circ p}\neq\id$, i.e. $\hat\bX\neq 0$ (otherwise $\phi$ would be 
analytically linearizable by Lemma~\ref{lemma:phivsphip}).
Denote
$\Cal I$ the ideal of $\C\{\xi\}=\Cal O(\C^2,0)$ generated by 
\begin{equation}\label{eq:f}
f:=\mfrac{\xi_1\circ\phi^{\circ p}-\xi_1}{\xi_1},
\end{equation}
which is the same as the ideal
generated by 
\begin{equation}\label{eq:f2}
\mfrac{\xi_2\circ\phi^{\circ p}-\xi_2}{\xi_2}=-\mfrac{\xi_1\circ\phi^{\circ p}-\xi_1}{\xi_1\circ\phi^{\circ p}}=-\mfrac{f}{1+f}.
\end{equation}
Let $\hat{\Cal I}$ denote the corresponding formal ideal in $\C\llbracket \xi\rrbracket$. 
We will denote $\xi\Cal I=\begin{pmatrix} \xi_1\Cal I \\ \xi_2\Cal I \end{pmatrix}$, the $\C\{\xi\}$-submodule of $\left(\C\{\xi\}\right)^2$,
and  $\xi\hat{\Cal I}$  its formalization.

\begin{lemma}\label{lemma:I}
The ideals $\Cal I$, $\hat{\Cal I}$ are invariant by composition with $\sigma$ and $\phi$.
\end{lemma}

\begin{proof}
Let $f$ be the generator \eqref{eq:f} of $\Cal I$. 
Then $\phi^{\circ p}=\xi+\left(\begin{smallmatrix}\xi_1f \\ -\xi_2\tfrac{f}{1+f}\end{smallmatrix}\right)$,
and using the Taylor expansion we see that $f\circ\phi^{\circ p}$ is also a generator of $\Cal I$, therefore $\Cal I\circ \phi^{\circ p}=\Cal I$.
Using \eqref{eq:f2} we have 
\[
f\circ\sigma= \mfrac{\xi_2\circ\phi^{\circ(-p)}-\xi_2}{\xi_2}=-\left(\mfrac{\xi_2\circ\phi^{\circ p}-\xi_2}{\xi_2}\right)\circ \phi^{\circ (-p)}\cdot\mfrac{\xi_2\circ\phi^{\circ(-p)}}{\xi_2}=\mfrac{f}{1+f}\circ\phi^{\circ(-p)}\cdot\mfrac{\xi_2\circ\phi^{\circ(-p)}}{\xi_2},
\]
and since $\mfrac{\xi_2\circ\phi^{\circ p}}{\xi_2}=\mfrac{1}{1+f}$, then 
\begin{equation}\label{eq:fsigma}
f\circ\sigma=\mfrac{f}{1+f}\circ\phi^{\circ(-p)}\cdot(1+f)\circ \phi^{\circ(-p)}=f\circ\phi^{\circ(-p)}.
\end{equation}
Since $\Cal I\circ \phi^{\circ p}= \Cal I$, we also have $\Cal I\circ \phi^{\circ(-p)}= \Cal I$, and hence $\Cal I\circ\sigma=\Cal I$.

\Grn{Assuming $\Lambda=\left(\begin{smallmatrix}\lambda&0\\0&\lambda^{-1}\end{smallmatrix}\right)$,} 
write $\xi_1\circ\phi=\lambda\xi_1\cdot(1+g(\xi))$ with $g(0)=0$, then
\begin{equation}\label{eq:fphi2}
\begin{aligned}
f\circ\phi&=\mfrac{(\xi_1\circ\phi^{\circ p})\cdot(1+g\circ\phi^{\circ p})}{\xi_1\cdot(1+g)}-1=
\mfrac{(1+f)\cdot(1+g\circ\phi^{\circ p})}{1+g}-1\\
&=\mfrac{f\cdot(1+g\circ\phi^{\circ p})+g\circ\phi^{\circ p}-g}{1+g}=f\cdot(1+O(\xi)) \in\Cal I,
\end{aligned}
\end{equation}
since $g\circ\phi^{\circ p}-g=f\cdot O(\xi)\in\Cal I$ using Taylor expansion.
Hence $\Cal I\circ\phi\subseteq\Cal I$, which also means that
$\Cal I=\Cal I\circ\phi^{\circ p} \subseteq \Cal I\circ\phi^{\circ(p-1)}\subseteq\cdots\subseteq \Cal I\circ\phi \subseteq \Cal I$, and hence $\Cal I\circ\phi=\Cal I$.

\Grn{The case when $\Lambda=-\sigma$ is similar, but this time
\begin{equation}\label{eq:fphi3}
		f\circ\phi=f\cdot(-1+O(\xi)) \in\Cal I.
\end{equation}	
}
\end{proof}

\begin{proposition}[Prepared form of $\phi$]\label{prop:prepared}
There exists an analytic change of coordinates $\Psi(\xi)\in\Diff_{\id}(\C^2,0)$ preserving $h=\xi_1\xi_2$ and commuting with $\sigma$,
such that in the new coordinate
\begin{equation}\label{eq:phiprepared}
 \phi(\xi)=\Lambda\xi\mod \Lambda\xi\Cal I, 
 \end{equation}
is linear modulo the ideal $\Cal I$, which now becomes $(\sigma,\Lambda)$-invariant: $f\in\Cal I\Leftrightarrow f\circ\Lambda\in\Cal I \Leftrightarrow f\circ\sigma\in\Cal I$.
\end{proposition}

\begin{proof}
By definition $\phi^{\circ p}(\xi)=\xi \mod\xi\Cal I$.
Let 
\[ \Psi(\xi)=\tfrac1{2p}\left(\Lambda^{p}\phi^{\circ(-p)}+\ldots+\Lambda^{1}\phi^{\circ(-1)}+\Lambda^{-1}\phi^{\circ 1}+\ldots+\Lambda^{-p}\phi^{\circ p}\right)(\xi)=\xi+\hot, \] 
then $\Psi\circ\sigma=\sigma\circ\Psi$, and 
\[\Lambda^{-1}\Psi\circ\phi-\Psi=\tfrac{1}{2p}\big(\id-\Lambda^{p}\phi^{\circ(-p)}+\Lambda^{-p-1}\phi^{\circ(p+1)}-\Lambda^{-1}\phi\big)\in\xi\Cal I\]
since $\Lambda^{-1}\left(\xi\Cal I\right)\circ\phi=\Lambda^{-1}\phi\Cal I=\xi\Cal I$ by Lemma~\ref{lemma:I} ($\phi$ is of the form \eqref{eq:phi3}) and hence
$\Lambda^{-j}\Psi\circ\phi^{\circ j}=\Psi \mod\xi\Cal I$, for $j\in\Z$.
Therefore $\tilde\phi(\xi):=\Psi\circ\phi\circ \Psi^{\circ(-1)}(\xi)=\Lambda\xi \mod\Lambda\xi\tilde{\Cal I}$, where $\tilde{\Cal I}:=\Cal I\circ \Psi^{\circ(-1)}$. 
From \ref{lemma:I}, we have $\Cal I\circ \phi=\Cal I$ so that 
$\tilde{\Cal I}\circ\tilde\phi=\Cal I\circ \phi\circ \Psi^{\circ(-1)}= \Cal I\circ \Psi^{\circ(-1)}=\tilde{\Cal I}.$ 
Therefore, we have $\tilde{\Cal I}=\tilde{\Cal I}\circ \tilde \phi=\tilde{\Cal I}\circ(\Lambda\xi)\mod \tilde{\Cal I}$, hence,  $\tilde{\Cal I}\circ\Lambda=\tilde{\Cal I}$.

The first integral $\tilde h:=h\circ\Psi^{\circ(-1)}$ then satisfies $\tilde h=\tilde h\circ\tilde\phi=\tilde h\circ\Lambda \mod\xi_1\xi_2\tilde{\Cal I}$
since $\tilde h=\xi_1\xi_2\cdot(1+\ldots)$ and  $\tilde\phi(\xi)=\Lambda\xi \mod\Lambda\xi\tilde{\Cal I}$.
A further transformation $\tilde\Psi(\xi):= \left(\mfrac{\tilde h(\xi)}{\xi_1\xi_2}\right)^{\frac12}\!\!\xi$ takes $\tilde h=(\xi_1\xi_2)\circ\tilde\Psi$ to $\xi_1\xi_2$, while preserving $\sigma$ and being $\Lambda$-equivariant modulo $\xi\tilde{\Cal I}$.
\end{proof}

\GRN
\begin{lemma}\label{lemma:P}
Suppose $\phi$ is as in Proposition~\ref{prop:prepared}.	
Then the $(\sigma,\Lambda)$-invariant ideal $\Cal I$ is generated by some uniquely determined $\sigma$-invariant germ
\begin{equation}\label{eq:P}
\begin{cases}h^sP(u,h)= 
		h^s\big(u^k+P_{k-1}(h)u^{k-1}+\ldots+P_0(h)\big),& \text{if }\ \Lambda=\left(\begin{smallmatrix}\lambda&0\\[3pt]0&\lambda^{-1}\end{smallmatrix}\right),\\[6pt]
		h^s\tilde P(\tilde u,h)(\xi_1\!+\!\xi_2)=h^s\big(\tilde u^{\tilde k}+
		\ldots+\tilde P_0(h)\big)(\xi_1\!+\!\xi_2),& \text{if }\ \Lambda=-\sigma,
		\end{cases}
\end{equation}
where 
\begin{equation}\label{eq:huv}
	h(\xi):=\xi_1\xi_2,\qquad \begin{cases} u(\xi):=\xi_1^p+\xi_2^p\\ \tilde u(\xi):=(\xi_1+\xi_2)^2.\end{cases} 
\end{equation}
are basic $(\sigma,\Lambda)$-invariant functions, and 
$\begin{cases} P_{k-1}(0)=\ldots=P_0(0)=0,\\ \tilde P_{\tilde k-1}(0)=\ldots=\tilde P_0(0)=0.\end{cases}$
\end{lemma}
\FGRN

\begin{proof}
Let $f$ \eqref{eq:f} be a generator of $\Cal I$. Since $\Cal I$ is invariant by $\sigma$, $f\circ\sigma=f\cdot V_1$ for some germ $V_1(\xi)$, satisfying $V_1\cdot(V_1\circ\sigma)=1$, hence $V_1(0)^2=1$. As $f\circ\sigma=f\circ\phi^{\circ(-p)}$ \eqref{eq:fsigma} with $\phi^{\circ(-p)}(\xi)=\xi+\hot(\xi)$, we see that in fact $V_1(0)=1$, hence
$f_1=\frac12(f+f\circ\sigma)=f \frac{1+V_1}{2}$ is a $\sigma$-invariant generator of $\Cal I$.

\Grn{
Similarly, as $\Cal I$ is invariant by $\Lambda$, $f_1\circ\Lambda=f_1\cdot V_2$ for some germ $V_2(\xi)$, 
and from \eqref{eq:fphi2}, resp. \eqref{eq:fphi3}, one sees that $V_2(0)=1$ if $\Lambda$ is diagonal, resp. $V_2(0)=-1$ for $\Lambda=-\sigma$.
Depending on the sign of $V_2(0)=\pm 1$, let  $f_2=\frac1p\Big(f_1\pm f_1\circ\Lambda+ \ldots(\pm 1)^{p-1}f_1\circ\Lambda^{p-1})\Big)$, then $f_2\circ\Lambda=\pm f_2$ and $f_2\circ\sigma=f_2$.
}

So if $\Lambda$ is diagonal, then $f_2$ is a $(\sigma,\Lambda)$-invariant, and as such
it can be expressed in a unique way as a germ of analytic function $u(\xi)$ and $h(\xi)$ \eqref{eq:huv}, which are functionally independent and generate the ring of $(\sigma,\Lambda)$-invariant functions (see e.g. \cite{GSS}).
Write $f_2(\xi)=h(\xi)^sg(u(\xi),h(\xi))$, where $s$ is maximal such that $h^s$ divides $f_2$, and let $k$ be the order of $g(u,0)$ in $u$.
By the Weierstrass preparation theorem \cite[chapter VII, §3, Proposition 6]{Bourbaki} with respect to the variables $(u,h)$ we can write
$g(u,h)= P(u,h)W(u,h)$ for a unique analytic Weierstrass polynomial $P(u,h)=u^k+P_{k-1}(h)u^{k-1}+\ldots+P_0(h)$ and some unity $W(u,h)$, $W(0,0)\neq 0$.

\Grn{If $\Lambda=-\sigma$ then $f_2$ is $\sigma$-invariant and $f_2\circ\Lambda=-f_2$, which means that $f_2$ can be expressed as a function of $(h,\xi_1+\xi_2)$ odd in $\xi_1+\xi_2$. 
So by a Weierstrass preparation theorem, it can be written as $f_2(\xi)=P(\xi_1+\xi_2,h)W(\xi_1+\xi_2,h)$ 
for a unique analytic Weierstrass polynomial odd in $\xi_1+\xi_2$ 
\begin{align*}
	P(\xi_1\!+\!\xi_2,h)&=(\xi_1\!+\!\xi_2)\tilde P\big((\xi_1\!+\!\xi_2)^2,h\big)\\
	&=(\xi_1\!+\!\xi_2)^{2\tilde k+1}+\tilde P_{\tilde k-1}(h)(\xi_1\!+\!\xi_2)^{2\tilde k-1}+\ldots+\tilde P_0(h)(\xi_1\!+\!\xi_2).
\end{align*}
}
\end{proof}

\GRN
To unify the discussion of the two cases: $\Lambda$ diagonal, and $\Lambda=-\sigma$, in the second case we write \eqref{eq:P} as 
\begin{equation}\label{eq:u-sigma}
P(u,h)=u\tilde P(u^2,h),\qquad u=\xi_1+\xi_2,	
\end{equation}
an odd polynomial of order $kp=2\tilde k+1$ in $u$, $P(-u,h)=-P(u,h)$.
\FGRN

\medskip
If $\hat\bX(\xi)$ is the formal infinitesimal generator for $\phi^{\circ p}$, then it can be written as
\begin{equation}\label{eq:genphip}
\hat\bX(\xi)=\frac{h^sP(u,h)}{\hat U(\xi)}\bE,
\end{equation}
where 
\begin{equation}\label{eq:bE}
\bE:=\xi_1\tdd{\xi_1}-\xi_2\tdd{\xi_2},
\end{equation}
and $\hat U(\xi)$ is a formal $\sigma$-invariant unity, $\hat U(0)\neq 0$, $\hat U\circ\sigma=\hat U$.

Rewriting $P(u,h)$ and $\hat U(\xi)$ as functions of the basic $\sigma$-invariants $(\xi_1+\xi_2,h)$,
then we can apply formal Weierstrass division theorem with respect to $(\xi_1+\xi_2,h)$ to write
\begin{equation}\label{eq:QU}
	\hat U(\xi)=Q(\xi)+P(u,h)\hat R(\xi), 
\end{equation}
for a unique $Q(\xi)= q_{pk-1}(h)(\xi_1+\xi_2)^{pk-1}+\ldots+ q_0(h)$, $ q_0(0)\neq 0$, and $\sigma$-invariant $\hat R(\xi)\in\C\llbracket x\rrbracket$.

\begin{lemma}\label{lemma:Q}~
The function $Q(\xi)$ above is analytic and $(\sigma,\Lambda)$-invariant, hence of the form
\[ Q(u,h)=Q_{k-1}(h)u^{k-1}+\ldots+Q_0(h),\quad\text{with }\ Q_0(0)\neq 0. \]
\end{lemma}

\begin{proof}
Let us first show that $Q$ is analytic.	
By definition $\phi^{\circ p}=\exp(\hat\bX)(\xi)=\sum_{n=0}^{+\infty}\tfrac   {1}{n!}\hat\bX^n.\xi$, 
with $\hat\bX=h^s\mfrac{P}{Q}\bE\mod h^sP^2\bE$,
and one has 
\[ \hat\bX^n.\xi=h^{ns}P(\bE.P)^{n-1}Q^{-n}\begin{psmallmatrix}1 & 0 \\[3pt] 0&-1 \end{psmallmatrix}\xi \mod h^sP^2\xi\quad\text{for all }\ n\geq 1, \]
where ``mod $h^sP^2\xi$'' means modulo the $\C\llbracket \xi\rrbracket$-submodule 
$\begin{pmatrix}h^sP^2\xi_1\C\llbracket\xi\rrbracket \\ h^sP^2\xi_2\C\llbracket\xi\rrbracket \end{pmatrix}$ of $(\C\llbracket\xi\rrbracket )^2$.
From this formally
\begin{equation}\label{eq:Q}
\left(\phi^{\circ p}(\xi)-\xi\right)=\mfrac{P}{\bE.P}\big(e^{h^s(\bE.P)Q^{-1}}-1\big) \begin{psmallmatrix}1 & 0 \\[3pt] 0&-1 \end{psmallmatrix}\xi \mod h^sP^2\xi.
\end{equation}
This may also be rewritten as
\begin{equation}\label{eq:Qinverse}
Q^{-1}=\mfrac{1}{h^s\bE.P}\log\left(1+h^s(\bE.P)\tilde f\right) \mod P,\quad\text{where }\ 
\tilde f=\mfrac{(\xi_1\circ\phi^{\circ p}-\xi_1)}{h^sP\xi_1},
\end{equation}
or 
$Q=\mfrac{h^s\bE.P}{\log\left(1+h^sf\bE.P\right)} \mod P$.
Note that $\tilde f$ is analytic since by definition $h^sP$ and $\mfrac{(\xi_1\circ\phi^{\circ p}-\xi_1)}{\xi_1}$ generate the same ideal $\Cal I$ of $\C\{\xi\}$.

We have $(\bE.P)\circ \sigma=-\bE.P$. Indeed, we have $\bE.P= (\bE.u)\partial_uP$ and $\bE.u=p(\xi_1^p-\xi_2^p)$.
As $Q=Q\circ\sigma$, we can then symmetrize the expression of $Q$ as
\begin{equation*}
Q=\mfrac{h^s\bE.P}{2\log\left(1+h^s\tilde f\bE.P\right)}-
\mfrac{h^s\bE.P}{2\log\left(1-h^s(\tilde f\circ\sigma)\bE.P\right)} \mod P,
\end{equation*}
which can now be expressed as an analytic function of $(h,\xi_1+\xi_2)$.
Therefore we obtain $Q$ as the remainder of the formal Weierstrass division of the term on the right side by $P$ (again with respect to the symmetric variables $(\xi_1+\xi_2,h)$). Since the formal Weierstrass division agrees with the analytic one, we see that $Q$ is indeed analytic. 

Now let us show that $Q\circ\Lambda=Q$.
According to Proposition~\ref{prop:prepared}, we have $\Lambda^{-1}\phi(\xi)=\xi\mod \xi\mathcal{I}$, where $\mathcal I$ is generated by $h^sP(u,h)$, i.e. $\Lambda^{-1}\phi(\xi)=\xi\mod h^sP\xi$. Since $h\circ\phi=h$, we can write 
\begin{equation}\label{eq:phiW}
\Lambda^{-1}\phi(\xi)=\begin{pmatrix}\xi_1+\xi_1 h^s PW \\ \xi_2-\xi_2 h^s P\mfrac{W}{1+h^sPW}\end{pmatrix}=
\begin{pmatrix}\xi_1+\xi_1 h^s PW \\ \xi_2-\xi_2 h^s PW\end{pmatrix} \mod h^{2s}P^2\xi, 
\end{equation}
where $W(\xi)\in\C\{\xi\}$ is a unity. We have $\,\hat\bX(\xi)=\mfrac{h^sP}{Q}\bE \mod h^sP^2\bE$.
We know that $\hat\bX(\xi)$ is invariant by $\phi(\xi)$, which means that $\hat\bX.\phi=\hat\bX.\xi\big|_{\xi=\phi}$.
On one side:
\[ \hat\bX.\phi=\mfrac{h^sP}{Q}\left(1+h^s(\bE.P)W\right)\Lambda\begin{psmallmatrix}1 & 0 \\[3pt] 0&-1 \end{psmallmatrix}\xi \mod h^sP^2\xi, \]
on the other side, expressing
$P\circ\phi=P\circ\Lambda+h^sPW\bE.(P\circ\Lambda)\cdot \mod h^sP^2$ by \eqref{eq:phiW}, $Q\circ\phi=Q\circ\Lambda\mod h^sP$, and $\phi=\Lambda\xi\mod h^sP\xi$, we obtain:
\Grn{
\begin{align*}
	 (\hat\bX.\xi)\circ\phi&=\mfrac{h^sP\circ\phi}{Q\circ\phi} \begin{psmallmatrix}1 & 0 \\[3pt] 0&-1 \end{psmallmatrix} \phi\\
&=\mfrac{h^s}{Q\circ\Lambda}\left(P\circ\Lambda+h^sPW\bE.(P\circ\Lambda)\right) \begin{psmallmatrix}1 & 0 \\[3pt] 0&-1 \end{psmallmatrix}\Lambda\xi \mod h^sP^2\xi, \\
&=\mfrac{h^sP}{Q\circ\Lambda}\left(1+h^s(\bE.P)W\right) \Lambda\begin{psmallmatrix}1 & 0 \\[3pt] 0&-1 \end{psmallmatrix}\xi \mod h^sP^2\xi, 
\end{align*}
since in all cases $(P\circ\Lambda)\cdot \begin{psmallmatrix}1 & 0 \\[3pt] 0&-1 \end{psmallmatrix}\Lambda=P\cdot\Lambda\begin{psmallmatrix}1 & 0 \\[3pt] 0&-1 \end{psmallmatrix}$. Hence $Q=Q\circ\Lambda\mod P$. 
}

Writing $Q\circ\Lambda=Q+PV$ then both sides of $\tfrac12\big(Q\circ\Lambda+Q\circ\Lambda^{-1}\big)=Q+\tfrac12P\big(V+V\circ\sigma\big)$
are $\sigma$-invariant, therefore can be written as functions of $(\xi_1+\xi_2,h)$, and by the uniqueness of the Weierstrass division by $P$ in the variables $(\xi_1+\xi_2,h)$ we have $Q=\tfrac12\big(Q\circ\Lambda+Q\circ\Lambda^{-1}\big)$.
This means that $Q=\sum_{j=0}^{kp-1}q_j(h)(\xi_1\xi_2)^j$ contains only the powers $j$ such that $\frac12(\lambda^j+\lambda^{-j})=1$, i.e. $j\in p\Z$. 
Hence $Q=Q\circ\Lambda$.
\end{proof}

\begin{corollary}\label{cor:infinitesimalgenerator}
Let $\phi^{\circ p}=\exp(\hat\bX)$ be reversed by $\sigma$ and with a first integral $h=\xi_1\xi_2$, and let $f:=\frac{\xi_1\circ\phi^{\circ p}-\xi_1}{\xi_1}$, and $s\in\Z_{\geq0}$ maximal such that $h^s$ divides $f$.
Then 
\[\hat\bX=\frac{f\log(1+\bE.f)}{\bE.f}\bE\mod h^{-s}f^2\bE.\]	
\end{corollary}

This is an analogue of a 1-dimensional formula of X.~Buff \& A.~Ch\'eritat \cite[\S 1]{Cheritat}.

\begin{proof}
If $\hat\bX=h^s\frac{P}{Q}\bE\mod h^sP^2\bE$, then by 	\eqref{eq:Q} $f=\mfrac{P}{\bE.P}\big(e^{h^s(\bE.P)Q^{-1}}-1\big)\mod h^sP^2=h^sP\,(Q(0)^{-1}+\hot)$,
and $\bE.f=\big(e^{h^s(\bE.P)Q^{-1}}-1\big)\mod h^sP$, hence $\frac{f\log(1+\bE.f)}{\bE.f}=h^s\frac{P}{Q}\mod h^sP^2$.
\end{proof}

\subsection{Analytic formal normal form}

The following lemma can be found for example in \cite[Proposition 2.2]{Teyssier} or \cite[Proposition 5.2]{Ribon-a}.

\begin{lemma}\label{lemma:X}
Let $\hat\bX_0,\ \hat\bX_1$ be two formal (resp. analytic) vector fields vanishing at the origin of $\C^2$, and assume there is $\hat\alpha(\xi)\in\C\llbracket\xi\rrbracket$ (resp. $\alpha(\xi)\in\C\{\xi\}$)
such that $\hat\bX_1=\mfrac{\hat\bX_0}{1+\hat\bX_0.\hat\alpha}$.
Then the formal (resp. analytic) flow map of
\begin{equation}\label{eq:YlemmaX}
\hat\bY=\dd{t}-\frac{\hat{\alpha}\hat\bX_0}{1+t\hat\bX_0.\hat\alpha}
\end{equation}
\[\hat\Psi^{\circ(-1)}(\xi)=\xi\circ\exp(\hat\bY)\big|_{t=0},\quad\text{with inverse }\ \hat\Psi(\xi)=\xi\circ\exp(-\hat\bY)\big|_{t=1},\]
conjugates $\hat\bX_0$ and $\hat\bX_1=\hat\Psi^*\hat\bX_0$.
Moreover
\begin{equation}\label{eq:form-exp} 
\hat\Psi(\xi)=\exp(t\hat\bX_0)\big|_{t=\hat\alpha(\xi)}.
\end{equation}

If furthermore $\hat\alpha=-\hat\alpha\circ\sigma=\hat\alpha\circ\Lambda$, then $\hat\Psi(\xi)$ is $(\sigma,\Lambda)$-equivariant.
\end{lemma}

\begin{proof}
On one hand, if $\hat\bX_t:=\mfrac{\hat\bX_0}{1+t\hat\bX_0.\hat\alpha}$ and $\hat\bY$ \eqref{eq:YlemmaX}
are considered as formal vector fields in $\xi,t$, then $[\hat\bX_t,\hat\bY]=0$, which means that the flow of $\hat\bY$ preserves the family $\hat\bX_t$: if  $\hat\Psi_s(\xi,t):=\xi\circ\exp(-s\hat\bY)(\xi,t)$, $\hat\Psi_s^*\hat\bX_{t-s}=\hat\bX_t$.
 
Moreover the map $\hat\Psi_t(x,t)$ satisfies $\hat\bY.\hat\Psi_t=0$:
\begin{align*}
	\tdd{t}\hat\Psi_t(\xi,t)&=\tdd{s}\hat\Psi_s(\xi,t)\big|_{s=t}+\tdd{t}\hat\Psi_s(\xi,t)\big|_{s=t}\\
	&=-\hat\bY.\xi\big|_{\xi=\hat\Psi_t}+\hat\bY.\hat\Psi_s\big|_{s=t}+\hat\alpha\hat\bX_t.\hat\Psi_s\big|_{s=t}=\hat\alpha\hat\bX_t.\hat\Psi_t.
\end{align*}
On the other hand, denoting
$\hat\Psi'_t(\xi)=\exp(s\hat\bX_0)(\xi)\Big|_{s=t\hat\alpha(\xi)},$
then, using the identities 
\[\tdd{s}\exp(s\hat\bX_0)(\xi)=\hat\bX_0.\xi\Big|_{\xi=\exp(s\hat\bX_0)(\xi)}=\hat\bX_0.\exp(s\hat\bX_0)(\xi),\]
we have
\[\hat\bX_0.\hat\Psi'_t=\Big(\hat\bX_0.(t\hat\alpha)\tdd{s} \exp(s\hat\bX_0)+
\hat\bX_0.\exp(s\hat\bX_0)\Big)\Big|_{s=t\hat\alpha(\xi)}=(1+t\hat\bX_0.\hat\alpha)\cdot \hat\bX_0.\xi\big|_{\xi=\hat\Psi'_t},\]
hence $\hat\bX_t.\hat\Psi'_t=\hat\bX_0.\xi\big|_{\xi=\hat\Psi'_t}$, i.e. $\hat\bX_t=(\hat\Psi'_t)^*\hat\bX_0$.
Moreover, $\tdd{t}\hat\Psi'_t=\hat\alpha\cdot\big(\hat\bX_0.\xi\big)\big|_{\xi=\hat\Psi'_t}$,
hence $\hat\bY.\hat\Psi'_t=0$.
Since $\hat\Psi_t$ and $\hat\Psi'_t$ satisfy the same partial differential equation with the same initial condition $\hat\Psi_0=\hat\Psi'_0=\id$,
they are both equal.

In particular, this means that the formal flow map $\hat\Psi(\xi):=\hat\Psi_1(\xi,1)=\hat\Psi'_1(\xi)$ is a well defined formal power series in $\xi$.
\end{proof}

Let $\phi$ be as in Proposition~\ref{prop:prepared}, and let $\hat\bX$ be the formal infinitesimal generator of $\phi^{\circ p}$.
According to \eqref{eq:genphip} and \eqref{eq:QU} it is in a \emph{prepared form}
\begin{equation}\label{eq:prenormalX}
\hat\bX(\xi)=
\begin{cases}
0,\\[6pt]
\mfrac{h^s}{\hat R(\xi)}\bE,\quad \hat R(0)\neq 0, &\text{if }\ k=0,\\[6pt] 
\mfrac{h^sP(u,h)}{Q(u,h)+P(u,h)\hat R(\xi)}\bE,\quad Q(0,0)\neq 0,& \text{if }\ k>0,
\end{cases}
\end{equation}
where $\bE$ is \eqref{eq:bE}, where $P(u,h)$, $Q(u,h)$ are analytic Weierstrass polynomials as in Lemma~\ref{lemma:P} and Lemma~\ref{lemma:Q},
and $\hat R(\xi)=\sum_{\bm m}r_{\bm m}\xi^{\bm m}\in\C\llbracket\xi\rrbracket$ is some formal $\sigma$-invariant germ.
Let $\hat\mu(h):=\sum_{j}r_{jj}h^j$. \label{page:mu}

\begin{proposition}\label{prop:FNFX}~\\
Suppose $\hat\bX$ is as in \eqref{eq:prenormalX}.  
Then there exists a formal $\sigma$-equivariant change of variables $\xi\mapsto\hat\Psi(\xi)=\xi+\hot$,
preserving $h$, that brings $\hat\bX(\xi)$ to
\begin{equation}\label{eq:FNFX}
\hat\bX(\xi)=
\begin{cases}
	0,\\
\mfrac{h^s}{\hat\mu(h)}\bE,\quad \hat\mu(0)=:c^{-1}\neq 0, &\text{if }\ k=0,\\[6pt]
\mfrac{h^sP(u,h)}{Q(u,h)+ P(u,h)\hat\mu(h)}\bE,\quad Q(0,0)=:c^{-1}\neq 0, & \text{if }\ k>0
\end{cases}
\end{equation}
where $\bE$ is \eqref{eq:bE}, with the same analytic Weierstrass polynomials $P(u,h),\ Q(u,h)$, and $2s+kp>0$. 
\end{proposition}

\begin{proof}
Write $\hat R(\xi)=\sum_{m_1,m_2}r_{m_1,m_2}\xi_1^{m_1}\xi_2^{m_2}$, and let $\hat\mu(h):=\sum_{m}r_{mm}h^m$.
Let 
\[ \begin{cases} 
	\hat\bX_1=\mfrac{1}{\hat R}\bE, \qquad \hat\bX_0=\mfrac{1}{\hat\mu}\bE, &\text{if }\ k=0,\\[6pt] 
	\hat\bX_1=\mfrac{P}{Q+P\hat R}\bE, \qquad \hat\bX_0=\mfrac{P}{Q+P\hat\mu}\bE,  &\text{if }\ k>0,
\end{cases} \] 
then
$\bX_1=\mfrac{\hat\bX_0}{1+\hat\bX_0.\hat\alpha}$, for $\bE.\hat\alpha=\hat R-\hat\mu$, i.e. 
$\hat\alpha=\sum_{m_1\neq m_2}\tfrac{r_{m_1,m_2}}{m_1-m_2}\xi_1^{m_1}\xi_2^{m_2}$ satisfies $\hat\alpha\circ\sigma=-\hat\alpha$. 
By Lemma~\ref{lemma:X} there exists a formal $\sigma$-equivariant transformation $\hat\Psi$ preserving $h$, such that $\hat\Psi^*\hat\bX_0=\hat\bX_1$. Then also $\hat\Psi^*(h^s\hat\bX_0)=h^s\hat\bX_1$.
\end{proof}

\begin{lemma}\label{lemma:Q0}
	In the case $k>0$, there exists an analytic  $(\sigma,\Lambda)$-equivariant change of variables $\xi\mapsto\hat\Psi(\xi)=\xi+\hot$,
	preserving $h$, after which the vector field $\hat\bX$ in the form \eqref{eq:FNFX} is such that $Q(u,0)=c^{-1}\in\C\sminus\{0\}$.
\end{lemma} 
\begin{proof}
	First consider the vector fields 
	\begin{align*}
		\bX_1&=\mfrac{P(u,0)\strut}{Q(u,0)+\hat\mu(0)P(u,0)}\bE=\mfrac{u^k\strut}{Q(u,0)+\hat\mu(0)u^k} (-1)^{j}pu\tdd{u},\\
		\bX_0&=\mfrac{P(u,0) }{Q(0,0)+\hat\mu(0)P(u,0) }\bE=\mfrac{u^k}{Q(0,0)+\hat\mu(0)u^k } (-1)^{j}pu\tdd{u},
	\end{align*}
	in the variable $u=\xi_{3-j}^p$ on the leaf $\{\xi_j=0\}$, $j=1,2$. 
	We then write $\bX_1=\mfrac{\bX_0}{1+\bX_0.\alpha}$ where
	\[\bE.\alpha(u)=(-1)^{j}pu\tdd{u}\alpha(u)=\tfrac{Q(u,0)-Q(0,0)}{u^k}=\tfrac{Q_1(0)}{u^{k-1}}+\ldots+\tfrac{Q_{k-1}(0)}{u},\]
	hence we set
	 $\alpha(u):=-\frac{(-1)^{j}}{pu^k}\Big[\tfrac{Q_1(0)}{(k-1)}u-\ldots-\tfrac{Q_{k-1}}{1}u^{k-1}\Big]$.
	The transformation $u\mapsto\psi(u):=u\circ\exp(-\bZ)\big|_{t=1}$, given by the flow of the vector field
	\begin{equation*}
		\bZ=\tdd{t}-\mfrac{\alpha\bX_0}{1+t\bX_0.\alpha }=
		\tdd{t}+\tfrac{\tfrac{Q_1(0)}{k-1}u+\ldots+\tfrac{Q_{k-1}(0)}{1}u^{k-1}}{(1-t)Q(0,0)+tQ(u,0)+\hat\mu(0)u^k \phantom{\big|}} u\tdd{u}, 
	\end{equation*}
is analytic, and by Lemma~\ref{lemma:X} it conjugates the vector fields $\bX_0$ and $\bX_1=\psi^*\bX_0$. 
	
	Considering the vector field $\hat\bX$ \eqref{eq:FNFX}, we have $\bE=p(\xi_1^p-\xi_2^p)\dd{u}$, where $u=\xi_1^p+\xi_2^p$, and 
	$(\xi_1^p-\xi_2^p)^2=u^2-4h^p$ agrees with $u^2$ on the zero level set $\{h=0\}$, hence $u^2\dd{u}$ agrees with $\tfrac1p(\xi_1^p-\xi_2^p)\bE$, 
	so we can as well replace the above vector field $\bZ$ in $(t,u)$ by the $(\sigma,\Lambda)$-equivariant
	\begin{equation*}
		\bY=\mfrac{\partial}{\partial t}+\mfrac{\tfrac{Q_1(0)}{k-1}+\ldots+\tfrac{Q_{k-1}(0)}{1}u^{k-2}}{(1-t)Q(0,0)+tQ(u,0)+\hat\mu(0)u^k \phantom{\big|}} \tfrac1p(\xi_1^p-\xi_2^p)\bE, 
	\end{equation*}
	in $(t,\xi)$.
	Let $\xi'=\Psi(\xi):=\xi\circ\exp(-\bY)\big|_{t=1}$, then the transform $\hat\bX'$ of $\hat\bX=\Psi^*(\hat\bX')$ in the variable $\xi'$, is such that 
	the restriction of $h^{-s}\hat\bX'$ to $\{h=0\}$ agrees with $\bX_0$,
	Namely, its polynomial $Q'(u',h)$ is such that $Q'(u',0)=c^{-1}=:Q(0,0)$.
	A further formal normalization of Proposition~\ref{prop:FNFX}, brings $\bX'$ to a new form \eqref{eq:FNFX}.
\end{proof}

Either of $\xi_1$ or $\xi_2$ defines a coordinate on each regular leaf $\{h=\const\neq 0\}$.
Define 
\begin{equation}\label{eq:FNFdual}
\hat\bX^{-1}=
\begin{cases}
h^{-s}\hat\mu(h)\bE^{-1}, &\text{if }\ k=0,\\[6pt]
\left(\mfrac{Q(u,h)}{h^sP(u,h)}+h^{-s}\hat\mu(h)\right)\bE^{-1},& \text{if }\ k>0,
\end{cases}
\end{equation}
where 
\[ \bE^{-1}=(-1)^{j-1}\tfrac{d\xi_j}{\xi_j}\mod\tfrac{dh}{h},\qquad j=1,2, \]
to be a formal 1-form dual to $\hat\bX$, defined modulo the formal forms vanishing along the foliation.	 

The fundamental group of each leaf $\{h=\const\neq 0\}$ is generated by a simple loop, encircling $0$ in the coordinate $\xi_j$ on the leaf, in positive direction if $j=1$, or negative direction if $j=2$.
Choosing this loop in a way that it either doesn't encircle any zero of $P$ on the leaf, or encircles them all, then 
the ``formal period'' of $\hat\bX^{-1}$ along this loop should be ``equal'' to $2\pi ih^{-s}\hat\mu(h)$.
While we won't give a precise meaning to the notion of a ``formal period'', we will show that it is a formal invariant.

\begin{lemma}\label{lemma:mu}
The formal meromorphic series $h^{-s}\hat\mu(h)$ is an invariant with respect to formal transformations $\xi\mapsto\hat\Psi(\xi)$ preserving the first integral $h$ and orientation
(i.e. such that $h\circ\hat\Psi=h$ and $\det D\hat\Psi(0)=1$).
\end{lemma}

\begin{proof}
For $h\neq 0$, let $y=\xi_1^p$, $\tfrac{h^p}{y}=\xi_2^p$. Then the restriction of $\hat\bX^{-1}$ on the leaf $h$ is
\begin{equation*}
\hat\bX^{-1}=
\begin{cases}
\tfrac1{p}\,h^{-s}\hat\mu(h)\tfrac{dy}{y}, &\text{if }\ k=0,\\[6pt]
\tfrac1{p}\,h^{-s}\mfrac{y^{k-1}Q(y+\tfrac{h^p}{y},\,h)}{y^{k}P(y+\tfrac{h^p}{y},\,h) \vphantom{\big|}}dy+\tfrac1{p}\,h^{-s}\hat\mu(h)\tfrac{dy}{y},& \text{if }\ k>0,
\end{cases}
\end{equation*}
where the term 
\[ \textstyle
	\mfrac{y^{k-1}Q(y+\frac{h^p}{y},\,h)}{y^{k}P(y+\frac{h^p}{y},\,h) \vphantom{\big|}}=
\mfrac{Q_{k-1}(h)(h^p+y^2)^{k-1}+\ldots+Q_0(h)y^{k-1}}{(h^p+y^2)^{k}+P_{k-1}(h)y(h^p+y^2)^{k-1}+\ldots+P_0(h)y^{k} \vphantom{\big|}}, \]
is  analytic at $y=0$ for each $h\neq 0$ fixed.
So in both cases $\tfrac1{p}\,h^{-s}\hat\mu(h)$ is the formal residue of the form $\hat\bX^{-1}$ at $y=0$.
If  $\xi\mapsto\hat\Psi(\xi)$ is a formal transformation preserving the first integral $h$ and orientation,
then 
$\xi_1\circ\hat{\Psi}=\xi_1\cdot\hat u(\xi)$ for some formal series $\hat u(\xi)$ in $\xi=(y^{\frac1p},hy^{-\frac1p})$ with $\hat u(0)\neq0$,
hence $(y\circ\hat\Psi)^n=y^n\cdot\hat u^{np}(\xi)$ for every positive integer $n$, and $\log(y\circ\hat\Psi)=\log y+ p\log\hat u(\xi)$, where $\log\hat u(\xi)$ is again a formal power series in $\xi=(y^{\frac1p},hy^{-\frac1p})$.
Therefore the formal residue of $d(y\circ\hat\Psi)^n$ is null for every integer $n$, and so is the formal residue of
$d\log(y\circ\hat\Psi)-d\log y$, which means that the pulled-back form $\hat\Psi^*\hat\bX^{-1}$ has the same formal residue as $\hat\bX^{-1}$ for every $h\neq 0$.
\end{proof}

\GRN
\begin{lemma}\label{lemma:automorphismofX}~
Assume $\hat\Psi\in\hatDiff_{\id}(\C^2,0)$ preserves the first integral $h=\xi_1\xi_2$ and a vector field $\hat\bX\neq0$ \eqref{eq:FNFX} with $2s+kp>0$.
Then $\hat\Psi(\xi)=\exp(h^{-s}\hat\beta(h)\hat\bX)$ for some formal power series $\hat\beta(h)$.
In particular, if $\hat\Psi$ is also $\sigma$-equivariant then it is the identity.
\end{lemma}

\begin{proof} 
	As $h\circ\hat\Psi=h$, so the infinitesimal generator of $\hat{\Psi}$ takes the form $\hat G(\xi)\bE$ for some $\hat G(\xi)\in\C\llbracket\xi\rrbracket$, and commutes with $\hat\bX$. This means that $\bE.\frac{\hat G(\xi)}{\frac{P(u,h)}{Q(u,h)+\hat\mu(h)P(u,h)}}=0$,
	hence $\hat G(\xi)=\hat\beta(h)\frac{P(u,h)}{Q(u,h)+\hat\mu(h)P(u,h)}$ for some formal series $\hat\beta(h)$. 
\end{proof}
\FGRN

The following formal normal form is somewhat similar to the normal form of Kostov \cite{Kostov} for parametric deformation of vector fields in one variable (see also 
\cite{Klimes-Rousseau3},
\cite[paragraph 5.5]{Ribon-f} and \cite{Rousseau, Rousseau-Teyssier}). These normal forms are essentially unique -- this is an analogue of \cite[Theorem 3.5]{Rousseau-Teyssier}.

\begin{thm}[Canonical formal normal for of the infinitesimal generator]~\label{thm:FNFX1}
\begin{enumerate}[wide=0pt, leftmargin=\parindent]
\item
	Let $\hat\bX$ be the vector field \eqref{eq:FNFX}.
	Then there exists a formal $(\sigma,\Lambda)$-equivariant change of variables $\hat\Psi(\xi)\in\hatDiff_{\id}(\C^2,0)$  preserving the foliation by $\{h=\const\}$ which brings $\hat\bX$ to the form
	\begin{equation}\label{eq:FNFX1}
		\hatbXnf(\xi)=
		\begin{cases}
		0,\\
		ch^s\bE, &\text{if }\ k=0,\\[6pt]
		ch^s\frac{P(u,h)}{1+c\hat\mu(h)\,P(u,h)}\bE,& \text{if }\ k>0,\quad \Lambda\ \text{diagonal},\\[6pt]
		ch^s\,\tilde P(\tilde u,h)(\xi_1\!+\!\xi_2)\bE,& \text{if }\ k=\tilde k+\frac12>0,\quad \Lambda=-\sigma,
		\end{cases}
	\end{equation}
	with analytic polynomial 
\[\begin{cases}
	P(u,h)=u^k+ P_{k-1}(h)u^{k-1}+\ldots+P_0(h),\quad P(u,0)=u^k,& u=\xi_1^p+\xi_2^p,\\	
	\tilde P(\tilde u,h)=\tilde u^{\tilde k}+ \tilde P_{\tilde k-1}(h)\tilde u^{\tilde k-1}+\ldots+\tilde P_0(h),\quad \tilde P(\tilde u,0)=\tilde u^{\tilde k},
	&\tilde u=(\xi_1+\xi_2)^2,
 	\end{cases}\]
possibly different than the one in Proposition~\ref{prop:FNFX},
	and $\hat\mu(h)$  (same as the one in Proposition~\ref{prop:FNFX}), and $c\in\C\sminus\{0\}$. Furthermore, in the cases $k>0$, $\hat\Psi$ is holomorphic in $(\mathbb{C}^2,0)$.
		
\item
	Assume that two formal vector fields of the form \eqref{eq:FNFX1} are equivalent by means of a formal transformation $\hat\Psi(\xi)\in\hatDiff_{\id}(\C^2,0)$  preserving the foliation by $\{h=\const\}$. Then the two vector fields are equal.
		
\item A general $\sigma$-equivariant formal transformation  $\hat\Psi(\xi)\in\hatDiff(\C^2,0)$  between two vector fields $\hatbXnf,\ \hatbXnf'\neq 0$ \eqref{eq:FNFX1}
	preserving the foliation by $\{h=\const\}$ is a linear transformation $\hat\Psi:\xi\mapsto\zeta\cdot\sigma^\epsilon\xi$, where 
	$\epsilon=0,1$, $\zeta^{kp+2s}=(-1)^{\epsilon}\frac{c'}{c}$.
\end{enumerate}
\end{thm}

\begin{proposition}\label{prop:normalization}
Let $\bX_0$, $\bX_1$ be two vector fields of the form
\begin{equation}\label{eq:pol_form2}
	\bX_t=\frac{P(u(\xi),y)}{c^{-1}+tZ(u(\xi),z)+\mu P(u(\xi),y))}\bE,	
\end{equation}
depending on parameters $(y,z,\mu)$ with 
\[	P(u,y)=(1+y_k)u^{k}+y_{k-1}u^{k-1}+\ldots+y_0,\quad Z(u,z)= z_{k-1}u^{k-1}+\ldots+z_1u,\]
where $u(\xi)=\xi_1^p+\xi_2^p$.
	Then there exists an analytic $(\sigma,\Lambda)$-equivariant $h$-preserving transformation $(\xi,y)\mapsto\big(\phi(\xi,y,z),\psi(y,z)\big)$ independent of $\mu$, tangent at identity at $(\xi,y)=0$, that transforms
	$\bX_1$ to $\bX_0$.
\end{proposition}

\begin{proof}
Let $\bX_t$, $t\in\C$, be as above \eqref{eq:pol_form2}.	
	We want to construct a family of transformations depending analytically on $t\in[0,1]$ between $\bX_0$ and $\bX_t$, defined by a flow of a vector field $\bY$ of the form
	\[
	\bY=\tdd{t}+\sum_{j=0}^{k}\omega_j(t,y,z)\tdd{y_j}+\mfrac{G(u(\xi),\,t,\,y,\,z)\cdot v(\xi)}{c^{-1}+tZ(u(\xi),\,z)+\mu P(u(\xi),\,y)}\mfrac{1}{p}\bE,
	\]
	where $v(\xi)=\xi_1^p-\xi_2^p$, $\tfrac{1}{p}\bE=v\tdd{u}=u\tdd{v}$, 
	for some unknown $\omega_j$ and $G$, such that $[\bY,\bX_t]=0$.
	This means (after multiplying the equation by $(c^{-1}+tZ+\mu P)^2$)
	\begin{equation}\label{eq:homological_eq}
	-ZP+(c^{-1}+tZ)\Omega+v^2G\tdd{u}P-v^2P\tdd{u}G-PGu=0,
	\end{equation}
	where 
	\[\Omega=\omega_0+\ldots+\omega_{k}u^{k}\quad\text{and }\ v^2=u^2-4h^p.\]
	We see that we can choose $G$ as a polynomial of order $k-2$ in $u$:
	\[
	G=g_0(t,y,z)+\ldots+g_{k-2}(t,y,z)u^k.
	\]
	Write $ZP=b_0(y,z)+\ldots+b_{2k-1}(y,z)u^{2k-1}$,
	then the equation \eqref{eq:homological_eq} takes the form of a non-homogeneous linear system for $(\omega,g)=(\omega_0,\ldots,\omega_k,g_0,\ldots,g_{k-2})$:
	\[
	A(t,h,y,z)\begin{pmatrix}\omega\\ g\end{pmatrix}=b(y,z),
	\]
	$b=\transp{(b_0,\ldots,b_{2k-1})}$.
	For $h=0$, $y=z=0$ the equation \eqref{eq:homological_eq} is
	\[
	c\,\Omega+(k-1)u^{k+1}G-u^{k+2}\tdd{u}G=0,
	\]
	hence
	\[
	A(t,0,0,0)=
	\begin{pmatrix} c^{-1}&& &&& \\[-6pt] &\!\!\ddots\!\!& &&& \\[-6pt] &&c^{-1} &&& \\ && &\!\!\!k\!-\!1\!\!\!&& \\[-6pt] && &&\!\!\ddots\!\!& \\[-6pt] &&&&&1 \end{pmatrix}.
	\]
	This means that $A(t,h,y,z)$ is invertible for $t$ from any compact in $\C$ if $|h|,\ |y|,\ |z|$ are small enough.
	Since $b(0,0)=0$, the constructed vector field $\bY(\xi,t,y,z)$ is such that $\bY(0,t,0,0)=\dd{t}$ and its flow is well-defined for all $|t|\leq 1$ as long as  $|h|,\ |y|,\ |u|$ are small enough.
\end{proof}

\begin{proof}[Proof of Theorem~\ref{thm:FNFX1}]~
	\begin{enumerate}[wide=0pt, leftmargin=0pt]
		\item \Grn{
		First we rescale by
			\begin{equation}\label{eq:rescaling} 
			\begin{cases}\xi\mapsto \hat\mu(h)^{-\frac{1}{2s}}\xi&\text{if }\ k=0,\\
				\xi\mapsto\left(\frac{Q_0(h)}{Q_0(0)}\right)^{-\frac{1}{2s+kp}}\xi &\text{if }\ k>0,\end{cases} 
		\end{equation}
		which changes $h$ but preserves the Pfaffian foliation $\{dh=0\}$.
		Then in the case $k>0$, we apply Lemma~\ref{lemma:Q0} and Proposition~\ref{prop:normalization} to bring $\hat\bX$ to the form $\frac{c\,h^sP}{1+c\,\hat{\mu}\,P}\bE$, where in the case of $\Lambda=-\sigma$ we use the convention \eqref{eq:u-sigma}.
				
		Moreover, in the case $\Lambda=-\sigma$, the vector fields $\hatbXnf=\frac{ch^sP}{1+c\hat{\mu}P}\bE$ and $\Lambda^*\hatbXnf=\frac{ch^sP}{1-c\hat{\mu}P}\bE$ are conjugated by
		$\Lambda^{-1}\hat\phi=\id+\hot$, hence by the point 2. of this theorem (proven just below) $\hatbXnf=\Lambda^*\hatbXnf$, meaning that $\hat\mu(h)=0$.
		}

		\item 
		For $k=0$ it is obvious. 
		For $k>0$, let $\xi\mapsto\hat\Psi(\xi)=\xi'$ be a formal transformation preserving the foliation $\Cal F$ such that $\hat\Psi^*\hatbXnf'=\hatbXnf$,
		where $\hatbXnf=h^s\frac{cP}{1+\hat\mu cP}\bE$, $\hatbXnf'=h^s\frac{c'P'}{1+\hat\mu'c'P'}\bE$ are two vector fields \eqref{eq:FNFX1}. 
		Then $h\circ\hat\Psi=e^{\hat a(h)}h$ for some $\hat a\in\C\llbracket h\rrbracket$, $\hat a(0)=0$, and the pullback of $\hatbXnf'$ by the transformation $\xi\mapsto e^{\frac12\hat a(h)}\xi$ is a vector field of the same form $h^s\frac{c'P'}{1+\hat\mu'c'P'}\bE$ except with $P'(u,h)=P_k(h)u^k+\ldots+P_0(h)$, $P_k(h)=1+\hot$.
		It will be enough to show that if two vector fields $\hatbXnf$, $\hatbXnf'$ are of this more general form and $\hat\Psi^*\hatbXnf'=\hatbXnf$ for some $\hat\Psi$ preserving $h=h\circ\hat\Psi$, then $\hatbXnf'=\hatbXnf$.
		
		By Lemma~\ref{lemma:mu} we know $\hat\mu=\hat\mu'$.
		Since $\sigma\hat\Psi\circ\sigma$ is another such transformation, by Lemma~\ref{lemma:automorphismofX} $\sigma\hat\Psi\circ\sigma=\hat\Psi\circ\exp(h^{-s}\hat\beta(h)\hatbXnf)$ for some formal power series  $\hat\beta(h)$. The transformation $\hat\Psi_1=\hat\Psi\circ\exp(\tfrac12 h^{-s}\hat\beta(h)\hatbXnf)$ has the same properties as $\hat\Psi$ and is $\sigma$-equivariant on top of that.
		It has the form $\hat\Psi_1(\xi)=\left(\begin{smallmatrix} e^{\hat\alpha(u,h)v} & 0 \\ 0 & e^{-\hat\alpha(u,h)v}\end{smallmatrix}\right)\xi$, where $v=\xi_1^p-\xi_2^p$,
		so the transformation equation becomes
		\[ \mfrac{cP}{1+\hat\mu cP}\cdot(1+\bE(\hat\alpha v))=\mfrac{cP'}{1+\hat\mu cP'}\circ\hat\Psi_1. \]
		For $h=0$ both vector fields $h^{-s}\hatbXnf$ and $h^{-s}\hatbXnf'$ are equal to $\frac{cu^k}{1+\hat\mu(0)cu^k}\bE$, and by Lemma~\ref{lemma:automorphismofX}
		the restriction of $\hat\Psi_1$ to $h=0$ is identity, i.e. $\hat\alpha(u,0)=0$.
		Let us assume that $\hat\alpha(u,h)=0\mod h^n$ for some $n\geq 0$ and show that it implies $\hat\alpha(u,h)=0\mod h^{n+1}$.
		We have
		\[ \mfrac{cP}{1+\hat\mu cP}\cdot(1+u\hat\alpha+ u^2\tdd{u}\hat\alpha)=\mfrac{cP'}{1+\hat\mu cP'}+ \hat\alpha u^2\tdd{u}\mfrac{cP'}{1+\hat\mu cP'} \mod h^{n+1}, \]
		from which
		\[ \mfrac{cP}{1+\hat\mu cP}-\mfrac{cP'}{1+\hat\mu cP'}=
		\mfrac{cu^{k+1}}{1+\hat\mu(0)cu^{k+1}}\big((k-1)\hat\alpha-u\tdd{u}\hat\alpha-k\hat\alpha\hat\mu(0)\mfrac{cu^{k}}{1+\hat\mu(0)cu^{k+1}} \big)\mod h^{n+1}, \]
		and
		\[ (P-P')=u^{k+1}\big((k-1)\hat\alpha-u\tdd{u}\hat\alpha\big) -\hat\mu(0)cu^{2k+1}\big(\hat\alpha+u\tdd{u}\hat\alpha\big) \mod h^{n+1}. \]
		Left side being a polynomial of order $k$ in $u$ means that both sides vanish modulo $h^{n+1}$.
		Applying a  modulo $h^{n+1}$ version of Lemma~\ref{lemma:automorphismofX}, we see that in fact $\hat\alpha=0\mod h^{n+1}$.
		
		\item A general transformation is a composition of its linear part and a transformation tangent to identity. The linear part must conjugate
		$c\,h^{-s}u^k\bE$ and $c'\,h^{-s}u^k\bE$.
	\end{enumerate}
\end{proof}

\begin{proof}[Proof of Theorem~\ref{thm:1}]
So far we have shown the existence of a formal $\sigma$-equivariant transformation $\xi\mapsto\hat\Psi(\xi)$
that conjugates $\phi^{\circ p}$ to
\[\hat\Psi\circ\phi^{\circ p}\circ\hat\Psi^{\circ(-1)}=\exp(\hatbXnf)(\xi),\]
where $\hatbXnf$ is as in Theorem~\ref{thm:FNFX1}.
Let $\hatphinf:=\hat\Psi\circ\phi\circ\hat\Psi^{\circ(-1)}$. It preserves the vector field $\hatbXnf=\hatphinf^*\hatbXnf$, and since $\hatbXnf=\Lambda^*\hatbXnf$, then also 
$\hatbXnf=(\Lambda^{-1}\hatphinf)^*\hatbXnf$, so by Lemma~\ref{lemma:automorphismofX},  
$\Lambda^{-1}\hatphinf=\exp(h^{-s}\hat\beta(h)\hatbXnf)$ for some $\hat\beta(h)$, and 
$\exp(\hatbXnf)=\hatphinf^{\circ p}=\exp(ph^{-s}\hat\beta(h)\hatbXnf)$ means that $\hatphinf=\Lambda\exp\big(\tfrac1p\hatbXnf\big)$.	
\end{proof}

\begin{proof}[Proof of Proposition~\ref{prop:prenormalphi}]
Proposition~\ref{prop:prepared} brings the formal infinitesimal generator $\hat\bX$ to the form \eqref{eq:prenormalX}.
Afterwards the analytic transformations of Lemma~\ref{lemma:Q0} and Proposition~\ref{prop:normalization}
and the rescaling \eqref{eq:rescaling} of the proof of Theorem~\ref{thm:FNFX1}, deform $P(u,h)$ so that the term $Q(u,h)$ becomes $c^{-1}=Q(0,0)$.
The respective prepared form \eqref{eq:prenormalX} means that 
\[\hat\bX=\hatbXnf\mod h^sP^2\bE=\bXmod\mod h^sP^2\bE.\]
And the formula \eqref{eq:Q} implies that 
\[\begin{aligned}
\phi(\xi)&=\Lambda\Big[I+\tfrac{P}{\bE.P}(e^{ch^s\bE.P/p}-1)\left(\begin{smallmatrix}
	1 & 0 \\ 0 & -1 \end{smallmatrix}\right)\Big]\xi\mod h^sP^2\xi\\
&=\hatphinf(\xi)\mod h^sP^2\xi\\
&=\phimod(\xi)\mod h^sP^2\xi.	
\end{aligned}\]
\end{proof}

As a bonus we obtain also analytic classification integrable reversible vector fields. 

\begin{thm}[Classification of integrable reversible vector fields]\label{thm:1X}~\\
	Let $\bX(\xi)$ be a germ of analytic (resp. formal) vector field in $(\C^2,0)$, with $\bX(0)=0$, which is reversed by $\sigma$ and has a first integral $h=\xi_1\xi_2$, 
	\[\sigma^*\bX=-\bX,\qquad \bX.h=0.\]
	Then $\bX$ is  conjugated by an analytic (resp. formal) tangent-to-identity transformation to one of the following vector fields 
	\[\bXnf=\begin{cases} 0,\\ 
		(c+ah^n)\big(\xi_1\tdd{\xi_1}-\xi_2\tdd{\xi_2}\big), \quad a\in\C,\ n\geq 1,\\
		c\,h^s\big(\xi_1\tdd{\xi_1}-\xi_2\tdd{\xi_2}\big), \quad s\geq1\\ h^s\frac{c\,P(u,h)}{1+c\,\mu(h)P(u,h)}\big(\xi_1\tdd{\xi_1}-\xi_2\tdd{\xi_2}\big),\quad s\geq 0,\end{cases}\]
	where $P(u,h)=u^k+P_{k-1}(h)u^{k-1}+\ldots+P_0(h)$, $u=\xi_1+\xi_2$, $k\geq1$, $P(u,0)=u^k$.
\end{thm}

The four cases are distinguished by the type of the zero divisor of $\bX$ and its position with respect to the invariant foliation by $\{h=\const\}$.

\begin{proof}
By the same formal reduction procedure as above with $\Lambda=I$, the only difference being that we allow also the case $k=0,\ s=0$. In this case the vector field $\hat\bX=\frac{1}{\mu(h)}\bE$ \eqref{eq:FNFX} with $\tfrac{1}{\mu(h)}=c+ah^n+\hot(h)$ is brought to $(c+ah^n)\bE$ by some scalar transformation $\xi\mapsto e^{\alpha(h)}\xi$ with $\alpha(0)=0$. If $\bX$ is analytic then all the transformations are analytic.
\end{proof}

\section{Formal normal form of Moser-Webster triples of involutions}  \label{sec:4}

\Grn{
Let $(\chi,\tau)$, be a pair of a reversible antiholomorphic diffeomorphism $\chi$ and its reversing reflection $\tau\in\Diff(\C^2,0)$: a holomorphic involution whose linear part of $\tau$ has eigenvalues $\{1,-1\}$,  
\[\tau^{\circ 2}=\id, \qquad\tau\circ\chi\circ\tau=\chi^{\circ (-1)}.\]
Let $\rho=\tau\circ\chi$, then $\rho$ is an antiholomorphic involution reversing $\chi$.
Assume that $\chi^{\circ 2p}(\xi)=\xi+\hot(\xi)$, $p\geq 1$ being minimal such integer, 
and assume that $H(\xi)$ is a first integral of Morse type for $(\chi^{\circ 2},\tau)$.
}

\begin{lemma}\label{lemma:chitau}
Let $(\chi,\tau)$ with first integral $H$ be as  above.
	There exists an analytic change of coordinates under which they take the form
	\begin{equation}\label{eq:chitau2}
		\chi(\xi)=\sigma\rho(\xi),\qquad \tau(\xi)=\sigma\xi,\quad\text{with first integral }\  h(\xi)=\xi_1\xi_2,
	\end{equation}
where $\rho(\xi)=\Lambda^{-\frac12}\ov\xi+\hot(\ov\xi)$, and $\Lambda^{\frac12}=\left(\begin{smallmatrix}\lambda^{\frac12} &0\\[3pt] 0& \lambda^{-\frac12}\end{smallmatrix}\right)=\ov\Lambda^{-\frac12}$.  
\end{lemma}
\begin{proof}
By Lemma~\ref{lemma:phitau} there exists an analytic coordinate $\xi$ in which
\[ \tau(\xi)=\sigma\xi,\qquad \chi^{\circ 2}(\xi)=\Lambda\xi+\hot(\xi),\qquad H(\xi)=\xi_1\xi_2.\]	
\Grn{Since $\tau_1=\tau$ is a reflection, then so is $\tau_2=\rho\circ\tau\circ\rho$, hence the case of $\Lambda=-\sigma$ cannot arise, so $\Lambda=\left(\begin{smallmatrix}\lambda &0\\[3pt] 0& \lambda^{-1}\end{smallmatrix}\right)$ is diagonal.}
If $\rho(\xi)=A\ov\xi+\hot(\bar \xi)$, then the relation $\rho^{\circ 2}=\id$ and the form of $(\sigma\rho)^{\circ 2}=\chi^{\circ 2}$ mean that
\[A\ov A=\id,\qquad \sigma A\sigma\ov A=\Lambda.\] 
Let $B=\Lambda^{\frac12}A$, then the above relations are equivalent to
\begin{equation}\label{eq:1764}
	\ov B=\Lambda^{-\frac12}B^{-1}\Lambda^{\frac12},\qquad B\sigma=\sigma B.
\end{equation}
The linear change of variables $\xi=\ov B^{-\frac12}\xi'$ transforms 
$\rho:\xi\mapsto A\ov\xi+\hot(\bar\xi)$ to 
$\xi'\mapsto =A'\ov\xi'+\hot(\bar\xi')$ where 
\[A'=\ov B^{\frac12}AB^{-\frac12}=\ov B^{\frac12}  \Lambda^{-\frac12}B^{\frac12}.\]
\GRN The relation $B\sigma=\sigma B$ means that $B=e^{i\alpha I+\beta\sigma}$ for some $\alpha,\beta\in\C$.
The relation \eqref{eq:1764} implies that
\begin{itemize}
	\item If  $\Lambda\neq\pm I$, then $\ov\alpha=-\alpha$, $\beta=0$,  hence $B^{\frac12}=\ov B^{-\frac12}$ commutes with $\Lambda^{\frac12}$.
	\item If  $\Lambda= I$, then $\ov\alpha=-\alpha$, $\ov\beta=-\beta$,  hence $B^{\frac12}=\ov B^{-\frac12}$ commutes with $\Lambda^{\frac12}=I$.
	\item If  $\Lambda=- I$, then $\ov\alpha=-\alpha$, $\ov\beta-\beta\in 2\pi i\Z$ and 
	$\Lambda^{\frac12}=\pm i\left(\begin{smallmatrix}1 &0\\[3pt] 0& -1\end{smallmatrix}\right)$, hence 
	$\ov B^{\frac12}=\pm \Lambda^{-\frac12}B^{-\frac12}\Lambda^{\frac12}$.
\end{itemize}
Therefore in all cases $A'=\pm\Lambda^{\frac12}$, but the matrix $\Lambda^{\frac12}$ is determined only up to a sign anyways.
\end{proof}
\FGRN

\GRN
\begin{proof}[Proof of Theorem~\ref{thm:antiholomorphicclassification}]
	If $(\chi,\tau)$ is of formal type (o), then it is analytically conjugated by a tangent-to-identity transformation to $(\chinf,\taunf)$ by means the transformation 
	\eqref{eq:Psichitau}.
	
	Assume $(\chi,\tau)$, $(\chi',\tau')$ are of formal type (a) or (b). Let $\Psi$ be analytic such that
	\[\chi'^{\circ 2}\circ\Psi=\Psi\circ\chi^{\circ 2},\qquad \tau'\circ\Psi=\Psi\circ\tau,\]
	and let $\tilde\Psi=\chi'\circ\Psi\circ\chi^{\circ(-1)}$.
	Then also
	\[\chi'^{\circ 2}\circ\tilde\Psi=\tilde\Psi\circ\chi^{\circ 2},\qquad \tau'\circ\tilde\Psi=\tilde\Psi\circ\tau,\]
	which means that 
	$\Psi^{\circ(-1)}\circ\tilde\Psi$ commutes with both $\chi^{\circ 2}$ and $\tau$.
	If $\Psi(\xi)=A\xi+\hot(\xi)$, then the matrix $A$ must commute with $\Lambda$, $\sigma$, and preserve $h=\xi_1\xi_2$ up to multiplicative constant since $\Psi$ must map between the unique leaf-wise invariant foliations $\{H=\const\}$ and $\{H'=\const\}$, hence
	$A=\begin{cases}aI,&\text{if }\ \Lambda\neq\pm I\\aI \text{ or } a\sigma,&\text{if }\ \Lambda=\pm I,\end{cases}$ for some $a\in\C^*$.
	If, as assumed, the matrix $A$ is real then $\xi\mapsto A\xi$ commutes with $\xi\mapsto\Lambda^{\frac12}\ov\xi$, the linear part of $\chi,\chi'$, therefore $\Psi^{\circ(-1)}\circ\tilde\Psi$ is tangent to identity.
	It follows from Theorem~\ref{thm:1} that the only element of $\hatDiff_{\id}(\C^2,0)$ that commutes with both $\chi^{\circ 2}$ and $\tau$ is the identity
	(indeed, up to a formal conjugacy one can assume that $(\chi^{\circ 2},\tau)$ are in the formal normal form $(\phinf,\sigma)$).
	Hence $\Psi^{\circ(-1)}\circ\tilde\Psi=\id$ and $\chi'\circ\Psi=\Psi\circ\chi$.
\end{proof}
\FGRN

\begin{theorem}[Formal normal form]\label{thm:normalform}
	The pair $(\chi,\tau)$ satisfying assumption of Theorem~\ref{thm:2} has a formal normal form $(\hatchinf,\hattaunf)$: 
	\begin{equation}\label{eq:chinormalform}
		\hatchinf(\xi)=\sigma\hatrhonf,\qquad
		\hatrhonf=\exp(-\tfrac{1}{2p}\hatbXnf)(\Lambda^{-\frac12}\bar\xi), \qquad 
		\hattaunf(\xi)=\sigma\xi,
	\end{equation}
	where $\hatbXnf$ \eqref{eq:FNFX1} is as in Theorem~\ref{thm:1} and furthermore satisfies $\hatbXnf=-\ov{(\Lambda^{\frac12})^*\hatbXnf}$.  
\end{theorem}

\begin{proof}
\Grn{In the case (o) of	Theorem~\ref{thm:1}, i.e. when $\chi^{\circ 2p}=\id$, and $(\chi,\tau)$ is as in Lemma~\ref{lemma:chitau}, one can construct a normalizing transformation by averaging over the finite group as
\begin{equation}\label{eq:Psichitau}.
	\Psi=\tfrac1{4p}\sum_{n=1}^{2p}\big(\chinf^{\circ n}\circ\chi^{\circ(-n)}+\taunf\circ\chinf^{\circ n}\circ\chi^{\circ(-n)}\circ\tau\big).
\end{equation}
}

Otherwise, by Theorem~\ref{thm:1}, after conjugation by a formal transformation, we may assume that 
\[\hat\chi^{\circ 2}(\xi)=(\sigma\hat\rho)^{\circ 2}= \hatphinf(\xi)=\Lambda\exp(\tfrac{1}{p}\hatbXnf)(\xi) \]
is in the formal normal form, with the formal vector field $\hatbXnf\neq 0$ \eqref{eq:FNFX1} satisfying $\hatbXnf=\Lambda^*\hatbXnf=-\sigma^*\hatbXnf$.
Let us show that 
\[\hat\rho(\xi)=\ov{\Lambda^{\frac12}\exp(\tfrac{1}{2p}\hatbXnf)(\xi)},\quad\text{and}\quad \hatbXnf=-\ov{(\Lambda^{\frac12})^*\hatbXnf},\]
which is equivalent to \eqref{eq:chinormalform}.
Write
\begin{equation}\label{eq:rho1}
	\hat\rho(\xi)=\ov{\Lambda^{\frac12}\hat\psi(\xi)},\qquad\text{for some}\quad\hat\psi(\xi)=\xi+\hot,
\end{equation}
then the identity $\hat\rho^{\circ 2}=\id$, is equivalent to
\begin{equation}\label{eq:ovpsi}
	\ov{\hat\psi}(\xi)=\Lambda^{\frac12}\hat\psi^{\circ(-1)}(\Lambda^{-\frac12}\xi).
\end{equation}	
Since $\hat\rho$ reverses $\hatphinf$ it also reverses the infinitesimal generator $\hatbXnf$	of $\hatphinf^{\circ p}$, hence by \eqref{eq:ovpsi} 
\begin{align*}
\exp(-t\hatbXnf)&=\hat\rho\circ\exp(t\hatbXnf)\circ\hat\rho=\ov{\Lambda^{\frac12}\hat\psi}\circ\ov{\exp(t\hatbXnf)}\circ(\Lambda^{\frac12}\hat\psi)\\
&=\hat\psi^{\circ(-1)}\circ\Lambda^{-\frac12}\circ\exp(t\ov{\hatbXnf})\circ\Lambda^{\frac12}\circ\hat\psi, \qquad\text{for all }\  t\in\R,
\end{align*}
which means that $\hat\psi^*\hatbXnf= -\ov{(\Lambda^{\frac12})^*\hatbXnf}$. As $u\circ\Lambda^{\frac12}=\pm u$, the vector field $-\ov{(\Lambda^{\frac12})^*\hatbXnf}$ is also of the form \eqref{eq:FNFX1}, and by the second point of Theorem~\ref{thm:FNFX1} $-\ov{(\Lambda^{\frac12})^*\hatbXnf}=\hat\psi^*\hatbXnf=\hatbXnf$.
By Lemma~\ref{lemma:automorphismofX} this means that
\[\hat\psi=\exp(h^{-s}\hat\beta(h)\hatbXnf)\]
for some formal power series $\hat\beta(h)$. In particular, $\hat\psi$ is $\Lambda$-equivariant and is reversed by $\sigma$.
We have also the identity 
\begin{align*}
\Lambda\exp(\tfrac{1}{p}\hatbXnf)&=\hat\phi=(\sigma\hat\rho)^{\circ 2}
=\sigma\ov\Lambda^{\frac12}\ov{\hat\psi}\circ(\sigma\Lambda^{\frac12}\hat\psi)=\sigma\hat\psi^{\circ(-1)}\circ\sigma\circ\Lambda\circ\hat\psi=\Lambda\hat\psi^{\circ 2},
\end{align*}
using \eqref{eq:ovpsi}, which means that $\hat\psi=\exp(\tfrac{1}{2p}\hatbXnf)$.
\end{proof}

We can now further transform the above formal normal form to bring $\hatrhonf$ to complex conjugation $\xi\mapsto\ov\xi$ which will
provide the normal form of Theorem~\ref{thm:2}.

\begin{proof}[Proof of Theorem~\ref{thm:2}]
Let $(\hattaunf,\hatrhonf)$ be in the normal form of Theorem~\ref{thm:normalform}.
A formal change of variables 
$\hat\Psi: \xi\mapsto\xi'=\Lambda^{\frac{1}{4}}\exp(\tfrac{1}{4p}\hatbXnf)(\xi)$,
conjugates $\hatrhonf(\xi)=\exp(-\tfrac{1}{2p}\hatbXnf)(\Lambda^{-\frac12}\bar\xi)$ to $\rhonf'(\xi')=\bar\xi'$:
\[ \hat\Psi\circ\hatrhonf(\xi)=\Lambda^{\frac14}\exp(-\tfrac{1}{4p}\hatbXnf)(\Lambda^{-\frac12}\bar\xi)=\bar\xi'=\rhonf'\circ\hat\Psi(\xi).\]
It conjugates
$\hatbXnf$ to $\hatbXnf'=(\Lambda^{-\frac14})^*\hatbXnf$,
and the involution $\hattaunf=\sigma$ to
\[ \hattaunf'(\xi')= \Lambda^{\frac14}\exp(\tfrac{1}{2p}\hatbXnf)(\Lambda^{\frac14}\sigma\xi')=\exp(\tfrac{1}{2p}\hatbXnf')(\Lambda^{\frac12}\sigma\xi').\]
A real dilatation $\xi\mapsto |c|^{-\frac1{kp+2s}}\xi$ sends $|c|\mapsto 2p$, i.e. $c=\pm 2ip$.

\Grn{By Theorem~\ref{thm:1}, 
	$\Cal Z(\hatbXnf', \, \sigma\Lambda^{\frac12},\,\Lambda)\subseteq\{\xi\mapsto e^{\frac{\pi ir}{kp+2s}}(\sigma\Lambda^{\frac12})^r\xi,\ r\in\Z_{2kp+4s}\}$.
But the only element commuting with $\rhonf:\xi\mapsto\ov\xi$ are for $\frac{r}{kp+2s}\in\Z$, i.e. either $\id$ or $\xi\mapsto -(\sigma\Lambda^{\frac12})^{kp}\xi$,
but the second map can be admissible only if $kp$ is even.
}
\end{proof}


\section{Surfaces and involutions}\label{sec:2}

\subsection{Moser--Webster correspondence}
The key point of J.~Moser \& S.~Webster's paper \cite{Moser-Webster} is Theorem~\ref{prop:MW0} which states 
that the formal/analytic classification of germs of surfaces $M$ \eqref{eq:M} agrees with that of the associated triple of involutions $(\tau_1,\tau_2,\rho)$ of $(\C^2,0)$.
This deserves to be explained here.

Let $\Cal M$ a germ of a complex surface in $(\C^4,0)$ of the form
\begin{equation}\label{eq:CM}
	\Cal M: z_2=F(z_1,w_1),\quad w_2=G(z_1,w_1),
\end{equation}
with $F,G$ higher order perturbations of $\gamma^{-1}z_1w_1+z_1^2+w_1^2$.

Two such surfaces $\Cal M$, $\Cal M'$ are equivalent if there exist a map $\psi:\Cal M\to\Cal M'$ which splits as $\psi(z,w)=\big(f(z),g(w)\big)$, i.e.
\begin{equation}\label{eq:fg}
\begin{tikzpicture}
	\node (A) {$(\Cal M,0)$};
	\node (B) [below=of A] {$(\Cal M',0)$};
	\node (C1a) [left=of A] {$(\C^2,0)$};
	\node (C2a) [right=of A] {$(\C^2,0)$};
	\node (C1b) [left=of B] {$(\C^2,0)$};
	\node (C2b) [right=of B] {$(\C^2,0)$};
	\draw[-stealth] (A)-- node[right] {\small $\psi$} (B);
	\draw[-stealth] (A)-- node [above] {\small $\pi_1$} (C1a);
	\draw[-stealth] (A)-- node [above] {\small $\pi_2$} (C2a);
	\draw[-stealth] (B)-- node [above] {\small $\pi_1$} (C1b);
	\draw[-stealth] (B)-- node [above] {\small $\pi_2$} (C2b);
	\draw[-stealth] (C1a)-- node [right] {\small $f$} (C1b);
	\draw[-stealth] (C2a)-- node [right] {\small $g$} (C2b);
\end{tikzpicture}
\end{equation}
where $\pi_1(z,w)=z$ and $\pi_2(z,w)=w$.

Given a germ of a complex surface \eqref{eq:CM} one associates to it a pair of involutions  $(\tau_1^{\Cal M},\tau_2^{\Cal M})$ acting on $\Cal M$, 
such that $\pi_j\circ\tau_j^{\Cal M}=\pi_j$, $j=1,2$,
which in the local coordinate $(z_1,w_1)$ is identified with a pair of involutions
$(\tau_1, \tau_2)$ of $(\C^2,0)$, such that
\begin{align*}
	z_1\circ\tau_1&=z_1, \quad F\circ\tau_1=F,\\
	w_1\circ\tau_2&=w_1, \quad G\circ\tau_2=G,\\
\end{align*}
of the form \eqref{eq:tau1tau2}.

The complex surface $\Cal M$ comes from a complexification of a real surface $M$ \eqref{eq:M} 
if and only if 
\[\rho\circ(F,G)=(F,G)\circ\rho\]
for the antiholomorphic involution
\begin{equation}\label{eq:rho}
	\rho: \begin{pmatrix} z_1 \\ w_1 \end{pmatrix}\mapsto \begin{pmatrix} \bar w_1 \\ \bar z_1 \end{pmatrix},
\end{equation}
that is if $F(z_1,w_1)=\bar G(w_1,z_1)$. In this case $\rho$ 
conjugates $\tau_1$ with $\tau_2=\rho\circ\tau_1\circ\rho$, and the intertwined triple of involutions $(\tau_1,\tau_2,\rho)$ is called a Moser--Webster triple.

\begin{proposition}[Moser \& Webster \cite{Moser-Webster}]~\label{prop:MW}
	\begin{enumerate}[wide=0pt, leftmargin=\parindent]
		\item 	
		Two complex surfaces $\Cal M$ and $\Cal M'$ \eqref{eq:CM} with $\gamma^{-1}=\gamma'^{-1}\in\C$
		are  equivalent by means of a formal (resp. analytic) transformation $(z',w')=\big(f(z),g(w)\big)$,
		if and only if the associated pairs of involutions $(\tau_1,\tau_2)$ and $(\tau_1',\tau_2')$ are conjugated by an element of the group $\widehat\Diff(\C^2,0)$ (resp. $\Diff(\C^2,0)$).
		
		There is a bijective correspondence between the formal (resp. analytic) equivalence classes of the surfaces $\Cal M$ with given $\gamma^{-2}\neq 4$ and  pairs of holomorphic reflections $(\tau_1,\tau_2)$ with $\tr D\phi(0,0)=\gamma^{-2}-2$, where $D\phi(0,0)$ is the matrix of the linear part of $\phi=\tau_1\circ\tau_2$ at the origin.
		
		\item
		Two  real surfaces $M$ and $M'$ \eqref{eq:M} with $\gamma=\gamma'\in\,\,]0,+\infty]$
		are  equivalent by means of a formal (resp. analytic) transformation $z'=f(z)$
		if and only if the  triples of involutions $(\tau_1,\tau_2,\rho)$ and $(\tau_1',\tau_2',\rho')$ 
		associated with the complexifications  $\Cal M$ and $\Cal M'$ of $M$ and $M'$
		are conjugated by an element of the group $\widehat\Diff(\C^2,0)$ (resp. $\Diff(\C^2,0)$).
		
		There is a bijective correspondence between the formal (resp. analytic) equivalence classes of the real surfaces $M$ with given $\gamma\neq\frac12$ and Moser--Webster triples $(\tau_1,\tau_2,\rho)$, with $\tr D\phi(0,0)=\gamma^{-2}-2$ where $D\phi(0,0)$ is the matrix of the linear part of $\phi=\tau_1\circ\tau_2$ at the origin.
	\end{enumerate}
\end{proposition}

Given the importance of this correspondence, and to keep the paper relatively self-contained, we provide a proof based on \cite{Moser-Webster}.
It rests on the following simple observations.

\begin{enumerate}[label=(\roman*), wide=0pt, leftmargin=\parindent]
	\item \label{item:hilbertbasis1} 
	\textit{The functions 
		\[z_1\quad \text{and}\quad z_2=F(z_1,w_1)=\gamma^{-1}z_1w_1+z_1^2+w_1^2+\hot\] 
		form a functionally independent set of generators of the ring of formal/analytic germs invariant by $\tau_1$.\footnote{In \cite{GSS} such set is called a Hilbert basis of the ring.}		
		This means that any formal/analytic $\tau_1$-invariant germ
		can be expressed uniquely as a formal/analytic function of $(z_1,z_2)=\big(z_1,F(z_1,w_1)\big)$.
		In particular, any other pair $\big(Z_1(z_1,w_1),Z_2(z_1,w_1)\big)$ of generators of the ring of $\tau_1$-invariant germs is related to $(z_1,z_2)$ by a formal/analytic diffeomorphism. Similarly, $(w_1,w_2)=\big(w_1,G(z_1,w_1)\big)$  are generators of the ring of formal/analytic $\tau_2$-invariant germs.}
	
	
	\item \label{item:hilbertbasis2}
	\textit{For any formal/analytic complex reflection\footnote{Involution $\tau$ whose linear part $D\tau(0,0)$ has eigenvalues $\{1,-1\}$.} $\tau$ of $(\C^2,0)$,  
		and for any non-degenerate formal/analytic germ $s:(\C^2,0)\to(\C,0)$, such that its derivative at the origin $Ds(0,0)$ is not an eigenvector for $D\tau(0,0)$,
		the pair $\big(s+s\circ\tau,\ s\cdot(s\circ\tau)\big)$ forms a functionally independent set of generators of the ring of formal/analytic $\tau$-invariant germs.}
	

\item \label{item:hilbertbasis3}
\Grn{\textit{If $T_1,T_2\in\GL_2(\C)$ are linear reflections, $T_1^2=T_2^2=I$, such that $T_1T_2$ is conjugated to 
	$\Lambda=\left(\begin{smallmatrix}\lambda &0\\[3pt]0&\lambda^{-1}\end{smallmatrix}\right)\neq I$, then up to a conjugation 
	$T_1=\left(\begin{smallmatrix}1&0\\[3pt]-\gamma^{-1}&-1\end{smallmatrix}\right)$, 
	$T_2=\left(\begin{smallmatrix}-1&-\gamma^{-1}\\[3pt]0&1\end{smallmatrix}\right)$, where $\gamma^{-2}=\lambda+\lambda^{-1}+2$.}	}
\end{enumerate}

\GRN
\begin{proof}[Proof of Proposition~\ref{prop:MW}]
Let us prove only point 1, point 2 is similar.
 
Suppose the pair $(\tau_1,\tau_2)$ is associated to $\Cal M$ \eqref{eq:CM}, and $(\tau_1',\tau_2')$ to $\Cal M'$.	
If $\tau_j'=\Psi\circ\tau_j\circ\Psi^{\circ(-1)}$, $j=1,2$, are conjugated by a transformation $(z_1',w_1')=\Psi(z_1,w_1)$,
then both $\big(z_1,\, F(z_1,w_1)\big)$ and $\big(z_1'\circ\Psi,\, F'\circ\Psi\big)$ are basic $\tau_1$-invariant functions, therefore there exist a diffeomorphism
$f\in\Diff(\C^2,0)$ such that 
\[(z_1',\, F')\circ\Psi(z_1,w_1)=f\big(z_1,\, F(z_1,w_1)\big).\]
Similarly also
\[(w_1',\, G')\circ\Psi(z_1,w_1)=g\big(w_1,\, G(z_1,w_1)\big),\]
for some $g\in\Diff(\C^2,0)$. Then $\big((z_1',z_2'),\, (w_1',w_2')\big)=\big(f(z_1,z_2),\, g(w_1,w_2)\big)$ is a biholomorphic map between $\Cal M$ and $\Cal M'$.

Conversely, if $(z,w)\mapsto (z',w')=\big(f(z),g(w)\big)$ is a biholomorphic map between two surfaces $\Cal M$ and $\Cal M'$, then its restriction to $\Cal M$ conjugates the associated pairs of involutions of the surfaces,
$(f,g)\circ(\tau_1^{\Cal M},\tau_2^{\Cal M})=(\tau_1^{\Cal M'},\tau_2^{\Cal M'})\circ(f,g)$.
Therefore in the local coordinates $(z_1,w_1)$ on $\Cal M$ and $(z_1',w_1')$ on $\Cal M'$ the map 
\[(z_1,w_1)\mapsto(z_1',w_1')=\Psi(z_1,w_1):= \big(z_1\circ f(z_1,F(z_1,w_1)),\ w_1\circ g(w_1,G(z_1,w_1))\big)\]
conjugates $(\tau_1,\tau_2)$ and $(\tau_1',\tau_2')$.

Given a pair of reflections $\tau_1,\tau_2\in\Diff(\C^2,0)$, if $\lambda+\lambda^{-1}\neq 2$, one can assume that $\tau_{j}=T_j+\hot$, $j=1,2$, where $T_j$ 
are as in the assertion \ref{item:hilbertbasis3}.
Then the functions 
\[s_1(z_1,w_1):=\lambda_1^{\frac12} z_1+w_1,\quad s_2(z_1,w_1):=z_1+\lambda_2^{\frac12}w_1,\quad \text{where}\quad \lambda_j^{\frac12}+\lambda_j^{-\frac12}=\gamma^{-1},\]
satisfy the assumption of the assertion \ref{item:hilbertbasis2}. Letting
\begin{equation*}
	\begin{aligned}
		\tilde z_1&=\tfrac{1}{\lambda_1^{\frac12}-\lambda_1^{-\frac12}}\Big(s_1+s_1\circ\tau_1\Big),&\quad \tilde z_2&=-s_1\cdot(s_1\circ\tau_1),\\
		\tilde w_1&=\tfrac{1}{\lambda_2^{\frac12}-\lambda_2^{-\frac12}}\Big(s_2+s_2\circ\tau_2\Big),&\quad \tilde w_2&=-s_2\cdot(s_2\circ\tau_2),
	\end{aligned}
\end{equation*}
then $\tilde z_1=z_1+\hot(z_1,w_1)$, $\tilde w_1=w_1+\hot(z_1,w_1)$, and
$\tilde z_2,\tilde w_2=z_1^2+w_1^2+\gamma^{-1}z_1w_1+\hot(z_1,w_1)$.
This means that $(\tilde z_1,\tilde z_2,\tilde w_1,\tilde w_2)$ defines a germ of a surface $\tilde{\Cal M}$ in $(\C^4,0)$ of the form \eqref{eq:CM},
parametrized by $(z_1,w_1)$.
The induced action of $\tau_1$, resp. $\tau_2$, on $\tilde{\Cal M}$ fixes $(\tilde z_1,\tilde z_2)$, resp. $(\tilde w_1,\tilde w_2)$, and therefore it is
precisely $\tau_1^{\tilde{\Cal M}}$, resp. $\tau_2^{\tilde{\Cal M}}$.
%
\end{proof}
\FGRN

\begin{cor}[Group of automorphisms]
	\begin{enumerate}[wide=0pt, leftmargin=\parindent]
		\item 	
		For $\gamma^{-1}\in\C$, the group of formal/analytic transformations \eqref{eq:fg} that preserve the complex surfaces $\Cal M$ \eqref{eq:CM} 
		is isomorphic to the group of formal/analytic diffeomorphisms commuting with the associated pair of involutions $(\tau_1,\tau_2)$.
		
		\item
		For $\gamma\in\ ]0,+\infty]$, the group of formal/analytic transformations $z'=f(z)$ that preserve the real surface $M$ \eqref{eq:M} is isomorphic to the group of formal/analytic diffeomorphisms commuting with the associated Moser--Webster triple of involutions $(\tau_1,\tau_2,\rho)$.
	\end{enumerate}
\end{cor}
\begin{proof}
	Clearly an automorphism of $\Cal M$ of the form \eqref{eq:fg} commutes with the pair $(\tau_1^{\Cal M}, \tau_2^{\Cal M})$.
	Conversely, if $\Psi\in\Diff(\C^2,0)$ commutes with $(\tau_1,\tau_2)$, then the map $\psi:(z,w)\mapsto\big((z_1,F),(w_1,G)\big)\circ\Psi(z_1,w_1)$ commutes with
	$(\tau_1^{\Cal M}, \tau_2^{\Cal M})$ and by the assertion \ref{item:hilbertbasis1} splits as $\psi(z,w)=\big(f(z),g(w)\big)$.
	
\end{proof}

\GRN
\begin{definition}
	A surface $\Cal M$ \eqref{eq:CM} is \emph{holomorphically flat} if up to a biholomorphic change of coordinates \eqref{eq:fg} it lies in the complex hyperplane $\{z_2=w_2\}$.	
\end{definition}
\FGRN

 Holomorphic flatness of the surface is known to correspond to the existence of a non-constant first integral for the pair of involutions (cf. \cite[Proposition 5.2]{Gong1}). More precisely: 

\begin{proposition}\label{prop:flatness}
	Let $\Cal M$ be a surface \eqref{eq:CM} associated with a pair of involutions $(\tau_1,\tau_2)$. Then the following are equivalent:
	\begin{enumerate}[wide=0pt, leftmargin=\parindent]
		\item The diffeomorphism $\phi=\tau_1\circ\tau_2$ has an analytic first integral 
		\begin{equation}\label{eq:H}
			H(z_1,w_1)=\gamma^{-1}z_1w_1+z_1^2+w_1^2+\hot(z_1,w_1).
		\end{equation}
		\item The pair of involutions $(\tau_1,\tau_2)$ has an analytic first integral \eqref{eq:H}. 
		\item The surface $\Cal M$ is holomorphically flat.	
	\end{enumerate}	
	
	If furthermore, $\Cal M=\{z_2=F(z_1,w_1),\ w_2=\ov F(w_1,z_1)\}$ is a complexification of a real surface $M$, then additionally the above are also equivalent to:
	\begin{enumerate}[wide=0pt, leftmargin=\parindent]\setcounter{enumi}{3}
		\item The pair of involutions $(\tau_1,\tau_2)$ has an analytic first integral \eqref{eq:H}, such that $H=\ov{H\circ\rho}$.	
		\item There exists an analytic function $ H(z)=z_2+\hot:(\C^2,0)\to (\C,0)$ whose restriction on the surface $M$ takes real values, $H:M\to(\R,0)$.
	\end{enumerate}	
\end{proposition}

\begin{proof}
	\textit{(1) $\Rightarrow$ (2):}	
	If $H$ \eqref{eq:H} is a first integral for $\phi$, then $H'=\frac12(H+H\circ\tau_1)$ is again of the form \eqref{eq:H} and satisfies $H'\circ\tau_1=H'=H'\circ\tau_2$.
	
	\textit{(2) $\Rightarrow$ (1):}	Obvious.	
	
	\textit{(2) $\Rightarrow$ (3):}	
	If $H$ \eqref{eq:H} is a first integral for $(\tau_1,\tau_2)$, then by being $\tau_1$-invariant it takes the
	form $H(z_1,w_1)=H_1(z_1,F(z_1,w_1))$ for some unique $H_1(z)=z_2+\hot$, and by  being $\tau_2$-invariant $H(z_1,w_1)=H_2(w_1,G(z_1,w_1))$ for $H_2(w)=w_2+\hot$. 	After the change of coordinate 
	$(z,w)\mapsto (z',w')=\Big((z_1,H_1(z)),\ (w_1,H_2(w))\Big)$ the surface $\Cal M$ takes the flat form $\Cal M\subset\{z_2'=w_2'\}$.
	
	\textit{(3) $\Rightarrow$ (2):}
	Conversely, $H(z_1,w_1):=F(z_1,w_1)=G(z_1,w_1)$ is a first integral for $(\tau_1,\tau_2)$.
	
	\textit{(2) $\Rightarrow$ (4):}	
	Let $H$ \eqref{eq:H} be a first integral for $(\tau_1,\tau_2)$, and let $H'(z_1,w_1)=\frac12\big(H(z_1,w_1)+\ov{H\circ\rho(z_1,w_1)}\big)$, then the restriction of $H'$ to $\{\bar z_1=w_1\}$ 	is such that $H'(z_1,w_1)=\ov{H'}(w_1,z_1)$. Since the hyperplane $\{\bar z_1=w_1\}\subset\C^2$ is totally real, the identity is true on a full neighborhood of $0\in\C^2$.
	
	\textit{(4) $\Rightarrow$ (5):}	
	Write $H(z_1,w_1)=H_1(z_1,F(z_1,w_1))=H_2(w_1,G(z_1,w_1))$ as in \textit{``(2) $\Rightarrow$ (3)''}. Since now $G(z_1,w_1)=\ov F(w_1,z_1)$ and $H(z_1,w_1)=\ov H(w_1,z_1)$, 
	then also $H_2(z)=\ov H_1(z)$, and the restriction of $H_1(z)$ to $M=\Cal M\cap\{\bar z=w\}$ takes real values.
	
	\textit{(5) $\Rightarrow$ (3):}
	As in \textit{``(2) $\Rightarrow$ (3)''}, the transformation  $z\mapsto z'=(z_1,H(z))$, $w\mapsto w'=(w_1,\ov H(w))$ does the job.
\end{proof}

\FGRN

\subsection{Model surfaces}

The Moser--Webster triple of involutions $(\tau_1,\tau_2,\rho)$ associated to a holomorphically flat $M$ have the form \eqref{eq:tau1tau2}.
The holomorphic transformation of Lemma~\ref{lemma:chitau} takes the form
\begin{equation}\label{eq:A} 
	\xi=\tfrac{1}{2}\left(A^{-1}+\sigma A^{-1}\tau_1\right)\begin{pmatrix}z_1\\w_1\end{pmatrix}=
	A^{-1}\begin{pmatrix}z_1\\w_1\end{pmatrix}+\hot,
\end{equation}
where $A= \frac{1}{\lambda^{\frac12}-\lambda^{-\frac12}}\left(\begin{smallmatrix} -1 & -1\\[3pt] \lambda^{\frac12} & \lambda^{-\frac12} \end{smallmatrix}\right)$, \
$A^{-1}= \left(\begin{smallmatrix} \lambda^{-\frac12} & 1  \\[3pt]  -\lambda^{\frac12}& -1 \end{smallmatrix}\right)$. 
Afterwards, the formal transformations to the formal normal form \eqref{eq:chinormalform} are tangent to the identity, while the further transformation to the formal normal form of Theorem~\ref{thm:2} in a variable $\xi'$
is of the form $\xi=C\Lambda^{-\frac14}\xi'+\hot(\xi')$, where $C\in\R_{>0}$.
Hence the resulting formal normalizing transformation is of the form
\[\begin{pmatrix}z_1\\w_1\end{pmatrix}=
\frac{C}{\lambda^{\frac12}-\lambda^{-\frac12}}\begin{pmatrix} -\lambda^{-\frac14} & -\lambda^{\frac14} \\ \lambda^{\frac14} & \lambda^{-\frac14} \end{pmatrix}\xi'+\hot(\xi'),\qquad
z_2,w_2=-C^2h+\hot.\]

\subsubsection{Normal form surface in the formal cases (o) and (a)} \label{sec:4oa}

\GRN
\begin{proposition}
Let $(\taunff1',\taunff2',\rhonf')$ be the formal normal form \eqref{eq:taunf12} of a Moser--Webster triple in the formal cases (o) or (a) of Theorem~\ref{thm:2},
namely
\[\taunff1'(\xi')=\begin{pmatrix}\alpha(h)\,\xi_2' \\ \alpha(h)^{-1}\xi_1'\end{pmatrix},\qquad 
\taunff2'(\xi')=\begin{pmatrix}\alpha(h)^{-1}\xi_2' \\ \alpha(h)\,\xi_1'\end{pmatrix},\qquad
\rhonf'(\xi')=\ov\xi',\]
where
\[\alpha(h)=\ov \alpha(h)^{-1}=\begin{cases} (o)\ \lambda^{\frac12}, \\[6pt] (a)\  \lambda^{\frac12} e^{\pm ih^s}.\end{cases} \]
Its associated surface takes the form
\begin{equation}\label{eq:fnfsurface}
\Mnf:\quad	 z_2'=2\RE(\alpha(-z_2'))\,z_1'\bar z_1'+z_1'^2+\bar z_1'^2.
\end{equation}
The original surface holomorphically flat $M$ \eqref{eq:M} is equivalent to $\Mnf$ by a formal transformation of the form $z'=\hat f(z)=\big(Cz_1+\hot(z),\ C^2z_2+\hot(z_2)\big)$, for some $C>0$.
\end{proposition}
\FGRN

\begin{proof}
Let 
\begin{align*}
z_1'&=\mfrac{-\alpha^{-\frac12}}{\alpha-\alpha^{-1}}\big(\xi_1'+\xi_1'\circ\taunff1'\big)
=\mfrac{-1}{\alpha-\alpha^{-1}}\big(\alpha^{-\frac12}\xi_1'+\alpha^{\frac12}\xi_2'\big),\\
w_1'&=\mfrac{\alpha^{-\frac12}}{\alpha-\alpha^{-1}}\big(\xi_2'+\xi_2'\circ\taunff2'\big)
=\mfrac{1}{\alpha-\alpha^{-1}}\big(\alpha^{\frac12}\xi_1'+\alpha^{-\frac12}\xi_2'\big),
\end{align*}
from which,
\[
\xi_1'=\alpha^{-\frac12} z_1'+ \alpha^{\frac12}w_1',\qquad
\xi_2'=-\alpha^{\frac12}z_1'-\alpha^{-\frac12} w_1',
\]
and let
\[z_2'=-\alpha^{-1}\xi_1'\cdot(\xi_1'\circ\taunff1')=-h, \qquad w_2'=-\alpha^{-1}\xi_2'\cdot(\xi_2'\circ\taunff2')=-h.\]
Then $\begin{pmatrix}z'(\xi')\\w'(\xi')\end{pmatrix}=\begin{pmatrix}\bar w'(\xi')\\ \bar z'(\xi')\end{pmatrix}$ and
\begin{equation}\label{eq:fnfcomplexsurface}
z_2'=w_2'=(\alpha+\alpha^{-1})z_1'w_1'+z_1'^2+w_1'^2
\end{equation}
where $\alpha=\alpha(-z_2')$.
The Moser--Webster triple of involutions becomes
\begin{equation*}
\taunff1': \left(\begin{smallmatrix}z_1'\\[3pt] w_1'\end{smallmatrix}\right)\mapsto \left(\begin{smallmatrix} z_1'\\[3pt] -w_1'-(\alpha+\alpha^{-1})\,z_1'\end{smallmatrix}\right),\qquad
\taunff2': \left(\begin{smallmatrix}z_1'\\[3pt] w_1'\end{smallmatrix}\right)\mapsto \left(\begin{smallmatrix} -z_1'-(\alpha+\alpha^{-1})\,w_1'\\[3pt] w_1'\end{smallmatrix}\right),
\end{equation*}
and $\rho:\left(\begin{smallmatrix} z_1' \\[3pt] w_1' \end{smallmatrix}\right)\mapsto \left(\begin{smallmatrix} \bar w_1' \\[3pt] \bar z_1' \end{smallmatrix}\right)$,
which is precisely the Moser--Webster triple associated to the surface $\Mnf$ \eqref{eq:fnfsurface}.
\end{proof}

As is shown in \cite{Moser-Webster}, for $\gamma\neq+\infty$ it is possible to transform analytically the surface \eqref{eq:fnfsurface} to the formal normal form \eqref{eq:MWnf2} of Moser--Webster.

\subsubsection{Model surface in the case formal (b)}\label{sec:4b}

Theorem~\ref{thm:2}\,(b) gives a formal normal form $(\hattaunff1',\hattaunff2',\rhonf')$ \eqref{eq:taunf12} of the Moser--Webster triple in terms of a normal form of its infinitesimal generator $\hatbXnf'$, but it doesn't give an explicit expression of the involutions. 
Likewise, the Moser--Webster correspondence (Proposition~\ref{prop:MW}) provides only an implicit construction of the corresponding normal form surface.
\GRN
Instead of finding a formal normal form surface $\Mnf$ whose associated Moser--Webster triple would be analytically conjugated to $(\hattaunff1',\hattaunff2',\rhonf')$,
our goal more modest: we shall derive an explicit form of some surface $\Mmod$ in any given model class.

By Definition~\ref{def:model}, two Moser--Webster triples belong to  the same model class if the respective generators $\hatbXnf'$ of their
formal normal forms agree modulo $h^sP'^2\bE$, where $\{h^sP'=0\}=\{\hatbXnf'=0\}=\Fix(\phinf^{\circ p})$. 
We will therefore calculate the  formal normal form surface modulo $h^sP'^2$. Discarding in the end ``mod $h^s P'^2$''  part, we obtain another surface in the same model class. 

\begin{proposition}
Consider any model Moser--Webster triple
\[	\taumodf1'=\exp(\tfrac{1}{2p}\bXmod')\circ(\Lambda^{\frac12}\sigma), \quad \taumodf2'=(\sigma\Lambda^{\frac12})\circ\exp(\tfrac{1}{2p}\bXmod'), \quad \rhomod'=\bar\xi,\] 
where $\bXmod'=\pm 2ip\,h^s P'\bE$, with $P'(u',h)=\ov{P'}(u',h)$, $u'=\xi_1^p+\lambda^{\frac{p}2}\xi_2^p$, as in Theorem~\ref{thm:2}.
It belongs to the same model class as the Moser--Webster triple of the surface $\Mmod$ whose complexification takes the form
\begin{multline}\label{eq:Mmodel}
	\CMmod:\quad
	z_2=w_2=	z_1^2+w_1^2+(\lambda^{\frac12}\!+\!\lambda^{-\frac12})z_1w_1 \\
	-(-z_2)^s \tilde P(z,w)\Big(a\,(z_1^2+w_1^2)+ bz_1w_1 \Big)\tilde R(z,w),
\end{multline}
where $\tilde P(z,w)=P'\big(\tilde u(z_1,w_1),-z_2\big)$ with
\begin{equation}\label{eq:tildeu}
	\tilde u=\big(\lambda^{-\frac14}z_1+\lambda^{\frac14}w_1\big)^p+(-1)^p\lambda^{\frac{p}{2}}\big(\lambda^{\frac14}z_1+\lambda^{-\frac14}w_1\big)^p,
\end{equation}
$a=2\tfrac{\lambda^{\frac12}+\lambda^{-\frac12}}{\lambda^{\frac12}-\lambda^{-\frac12}}$, \ 
$b= \tfrac{(\lambda^{\frac12}+\lambda^{-\frac12})^2+4}{\lambda^{\frac12}-\lambda^{-\frac12}}$, \
and \ $\tilde R(0,0)=\pm i$. 
\end{proposition}

\FGRN

\begin{proof}
Since $h\circ\taumodf1'=\taumodf2'=h$ and $\taumodf1'(\xi)=\ov{\taumodf2'}(\xi)$, one can write
\[\taumodf1'(\xi)=\begin{pmatrix}\ov\alpha(\xi)^{-1}\xi_2 \\ \ov\alpha(\xi)\,\xi_1\end{pmatrix},\qquad \taumodf2'(\xi)=\begin{pmatrix}\alpha(\xi)^{-1}\xi_2 \\ \alpha(\xi)\,\xi_1\end{pmatrix}, \]
where $\ov\alpha\circ\taumodf1'=\ov\alpha$, \ $\alpha\circ\taumodf2'=\alpha$.
Similarly to Section~\ref{sec:4oa} let us define 
\begin{align*}
	z_1&=\mfrac{-\ov\alpha^{\frac12}}{\lambda^{\frac12}-\lambda^{-\frac12}} \big(\xi_1+\xi_1\circ\taumodf1'\big)
	=\mfrac{-1}{\lambda^{\frac12}-\lambda^{-\frac12}} \big(\ov\alpha^{\frac12}\xi_1+\ov\alpha^{-\frac12}\xi_2\big),\\
	w_1&=\mfrac{\alpha^{-\frac12}}{\lambda-\lambda^{-1}} \big(\xi_2+\xi_2\circ\taumodf2'\big)
	=\mfrac{1}{\lambda^{\frac12}-\lambda^{-\frac12}} \big(\alpha^{\frac12}\xi_1+\alpha^{-\frac12}\xi_2\big),\\[6pt]
	z_2&=-\ov{\alpha}\xi_1\cdot(\xi_1\circ\taumodf1')=-h,\\
	w_2&=-\alpha^{-1}\xi_2\cdot(\xi_2\circ\taumodf2')=-h.
\end{align*}
From \eqref{eq:Q}
\begin{align*} 
\alpha(\xi)&=\lambda^{\frac12}\big(1+2h^sP'(\xi)t(\xi)\big)\mod h^sP'^2, &	t(\xi)&:=\tfrac{1}{2h^s\bE.P'}(e^{\tfrac{ch^s}{2p}\bE.P'}-1),\\
\ov\alpha(\xi)&=\lambda^{-\frac12}\big(1+2h^sP'(\xi)\bar t(\xi)\big)\mod h^sP'^2, &	\bar t(\xi)&:=\tfrac{1}{2h^s\bE.P'}(e^{-\tfrac{ch^s}{2p}\bE.P'}-1),
\end{align*}
where $t\circ(\sigma\Lambda^{\frac12})=-\ov t$. Therefore
\begin{equation*}
	\left(\begin{matrix} z_1 \\ w_1 \end{matrix}\right)=\psi(\xi)= 
	\mfrac{1}{\lambda^{\frac12}-\lambda^{-\frac12}}
	\left(\begin{smallmatrix}  -\lambda^{-\frac14}(1+h^sP'\bar t) & -\lambda^{\frac14}(1-h^sP'\bar t) \\[3pt] \lambda^{\frac14}(1+h^sP't) & \lambda^{-\frac14}(1-h^sP't) \end{smallmatrix}\right)\xi\mod h^sP'^2\xi.
\end{equation*}
From which also
\begin{equation*}
\xi=\Big(1-\mfrac{\lambda^{\frac12}+\lambda^{-\frac12}}{\lambda^{\frac12}-\lambda^{-\frac12}}\{h^sP'(t-\ov t)\}\Big)\!\!	
\left(\begin{smallmatrix}   \lambda^{-\frac14}(1-\{h^sP't\}) & \lambda^{\frac14}(1-\{h^sP'\bar t\}) \\[3pt] 
	-\lambda^{\frac14}(1+\{h^sP't\}) &-\lambda^{-\frac14}(1+\{h^sP'\bar t\}) \end{smallmatrix}\right)\!\!\left(\begin{matrix} z_1 \\ w_1 \end{matrix}\right)\!\!\!
\mod h^sP'^2,
\end{equation*}
where the bracket $\{h^sP't\}$ is functions of $\xi$ and need to be composed with $\psi^{\circ(-1)}$.
Therefore
{\small
\begin{align*}
	z_2&=
	\begin{multlined}[t][.7\displaywidth]
	\Big(1-2\mfrac{\lambda^{\frac12}+\lambda^{-\frac12}}{\lambda^{\frac12}-\lambda^{-\frac12}}\{h^sP'(t\!-\!\ov t)\}\Big)
\Big(z_1^2+w_1^2+z_1w_1\big(\lambda^{\frac12}\!+\!\lambda^{-\frac12}+(\lambda^{\frac12}\!-\!\lambda^{-\frac12})\{h^sP'(t\!-\!\ov t)\}\big) \Big)\\
\mod h^sP'^2
\end{multlined}
\\
&=
\begin{multlined}[t][.7\displaywidth]
z_1^2+w_1^2+(\lambda^{\frac12}\!+\!\lambda^{-\frac12})z_1w_1 - \{h^sP'(t\!-\!\ov t)\}\Big(2\mfrac{\lambda^{\frac12}+\lambda^{-\frac12}}{\lambda^{\frac12}-\lambda^{-\frac12}}(z_1^2+w_1^2)+
\mfrac{(\lambda^{\frac12}+\lambda^{-\frac12})^2+4}{\lambda^{\frac12}-\lambda^{-\frac12}}z_1w_1 \Big)\\
\mod h^sP'^2,
\end{multlined}
\end{align*}
}
where we still need to calculate $\{h^sP'(t\!-\!\ov t)\}\circ\psi^{\circ(-1)}$.
Let $A:=D\psi(0)$, then $A^{-1}=\left(\begin{smallmatrix} \lambda^{-\frac14} & \lambda^{\frac14} \\[3pt] -\lambda^{\frac14} & -\lambda^{-\frac14} \end{smallmatrix}\right)$, 
and denote 
$\tilde u(z_1,w_1)=u'(\xi)\big|_{\xi= A^{-1}\left(\begin{smallmatrix} z_1 \\ w_1 \end{smallmatrix}\right)}$ which equals \eqref{eq:tildeu}.
Expressing
\begin{align*}
A^{-1}\psi(\xi)-\xi&= \mfrac{h^sP'}{\lambda^{\frac12}-\lambda^{-\frac12}}
\begin{pmatrix}  \lambda^{\frac12} t-\lambda^{-\frac12}\ov t & \ov t-t \\ \ov t-t & \lambda^{-\frac12}t-\lambda^{\frac12}\ov t \end{pmatrix}\xi \mod (h^sP')^2\\
&=h^sP'\Big[\tfrac{t+\ov t}{2}I + \tfrac{(\lambda^{\frac12}+\lambda^{-\frac12})(t-\ov t)}{2(\lambda^{\frac12}-\lambda^{-\frac12})} J -\tfrac{t-\ov t}{\lambda^{\frac12}-\lambda^{-\frac12}}\sigma \Big]\xi \mod (h^sP')^2,
\end{align*}
where $J=\left(\begin{smallmatrix}1&0\\0&-1\end{smallmatrix}\right)$, we can develop
\[\tilde u\circ\psi= u' + (\tfrac{\partial u'}{\partial \xi_1},\tfrac{\partial u'}{\partial \xi_2})\big(A^{-1}\psi(\xi)-\xi\big)  \mod h^sP'^2
=u'+ h^sP'r  \mod h^sP'^2,\]
where
\[r(\xi)=p\tfrac{t+\ov t}{2}(\xi_1^p+\lambda^{\frac{p}2}\xi_2^p) + p\tfrac{(\lambda^{\frac12}+\lambda^{-\frac12})(t-\ov t)}{2(\lambda^{\frac12}-\lambda^{-\frac12})} (\xi_1^p-\lambda^{\frac{p}2}\xi_2^p) 
-p\tfrac{t-\ov t}{\lambda^{\frac12}-\lambda^{-\frac12}}h(\xi_1^{p-2}+\lambda^{\frac{p}2}\xi_2^{p-2}).\]
Therefore $\{h^sP'(t\!-\!\ov t)\}\circ\psi^{\circ(-1)}=(-z_2)^s\tilde P\tilde R \mod z_2^{s}\tilde P^2$, where
\[\tilde P=P'\big|_{u'=\tilde u(z_1,w_1),\ h=-z_2},\quad 
\tilde R=\frac{t-\ov{t}}{1+h^s\tfrac{\partial P'}{\partial u'} r}\Big|_{\xi=A^{-1}\left(\begin{smallmatrix} z_1 \\ w_1 \end{smallmatrix}\right)},\]
 are obtained by composition with the linear substitution  $\xi=A^{-1}\left(\begin{smallmatrix} z_1 \\ w_1 \end{smallmatrix}\right)$.
Finally, by discarding the terms ``$\text{mod}\ h^sP'^2$'' we obtain the \emph{complex model surface} \eqref{eq:Mmodel}.
\end{proof}

\subsection{Proof of Theorem~\ref{thm:sectorialM}}

\com{I've deleted the previous generally but somewhat vaguely formulated proposition on sectorial equivalence of surfaces, and replaced it by the following proof.}

\begin{proof}[Proof of Theorem~\ref{thm:sectorialM}]
\Grn{Let $(\tau_1,\tau_2,\rho)$ be the Moser--Webster triple associated to a holomorphically flat surface $M=\{z_2=F(z_1,\ov z_1)=\ov F(\ov z_1,z_1)\}$, and let
$(\tau_1',\tau_2',\rho')$ be another Moser--Webster triple in the same model class (e.g. the model itself) and $M'=\{z_2'=F'(z_1',\ov z_1')=\ov F'(\ov z_1',z_1')\}$ its associated holomorphically flat surface.
}

Let $\{\Omega \mapsto \Psi_{\Omega}\}$ be a conjugating cochain between $(\tau_1,\tau_2)$ and $(\tau_1',\tau_2')$
	\[\Psi_{T_j(\Omega)}\circ\tau_j=\tau_j'\circ \Psi_{\Omega},\qquad T_j=D\tau_j(0,0),\quad j=1,2,\]
whose existence follows from Theorem~\ref{thm:sectorial},
that maps between the first integrals as $F'\circ\Psi_{\Omega}(z_1,w_1)=\varphi\circ F(z_1,w_1)$ for some $\varphi\in\Diff(\C,0)$ independent of $\Omega$.

The 2-sheeted projection $\pi_1:(z,w)\mapsto z$ provides a local $\tau_1$-invariant coordinate $z=\big(z_1,F(z_1,w_1)\big)$ on $\Cal M\sminus\Fix(\tau_1^{\Cal M})$.
Hence, if $\Omega^{\Cal M}$ is a simply connected domain in $\Cal M\sminus\Fix(\tau_1^{\Cal M})$ corresponding to a domain $\Omega$ in $\C^2\sminus\Fix(\tau_1)$,
then $z_1'\circ\Psi_{\Omega}(z_1,w_1)=f_{\pi_1(\Omega^{\Cal M})}(z_1,F(z_1,w_1))$ for some function $f_{\pi_1(\Omega^{\Cal M})}(z_1,z_2)$ on $\pi_1(\Omega^{\Cal M})$, and the relation $\tau_1'\circ\Psi_{\Omega}=\Psi_{T_1(\Omega)}\circ\tau_1$ means that it is well defined.
Similarly, $w$ is a local $\tau_2$-invariant coordinate on $\Cal M\sminus\Fix(\tau_2^{\Cal M})$, and $w_1'\circ\Psi_{\Omega}=g_{\pi_2(\Omega^{\Cal M})}(w_1,\ov F(w_1,z_1))$ for some function $g_{\pi_2(\Omega^{\Cal M})}(w_1,w_2)$ on $\pi_2(\Omega^{\Cal M})$, is a well defined map.
The restriction of the product map 
$\Big(\big(f_{\pi_1(\Omega^{\Cal M})},\varphi\big),\ \big(g_{\pi_2(\Omega^{\Cal M})},\ov\varphi\big)\Big)$
to the set $\Omega^{\Cal M}\subset\Cal M$ then agrees with the lifting $\psi_{\Omega^{\Cal M}}:\Omega^{\Cal M}\to\Cal M'$ of $\Psi_{\Omega}:\Omega\to\C^2$.
\end{proof}

\goodbreak

\section{The formal type (a) $k=0$}\label{sec:5}

\subsection{Topological obstructions to convergence}\label{sec:5a}

\begin{theorem}\label{thm:divergent}
For any $\lambda\in\{|\lambda|=1\}$ and $s\geq 1$, there exists a Moser--Webster triple of involutions $(\tau_1,\tau_2,\rho)$  that is formally equivalent to
\begin{equation}\label{eq:taunfa}
\taunff1'(\xi)=\left(\begin{smallmatrix} 0 & \alpha(h) \\[3pt] \alpha(h)^{-1} & 0\end{smallmatrix}\right)\xi,\qquad 
\taunff2'(\xi)=\left(\begin{smallmatrix} 0& \alpha(h)^{-1} \\[3pt] \alpha(h) & 0\end{smallmatrix}\right)\xi,\qquad 
\rhonf'(\xi)=\bar\xi,
\end{equation}
where 
\[\alpha(h)=\lambda^{\frac12} e^{\pm ih^s},\]
but not analytically.
In particular, when $\lambda\in e^{\pi i \Q}$ is a root of unity, then \eqref{eq:taunfa} is the formal normal form of Theorem~\ref{thm:2}\,(a). 
\end{theorem}

Whenever $h$ is such that $\alpha(h)^{2n}=1$ for some $n\in\Z_{>0}$ then the restriction of the germ 
$\phinf'^{\circ n}=(\taunff1'\circ \taunff2')^{\circ n}=\left(\begin{smallmatrix} \alpha(h)^{2n} &0 \\[3pt] 0& \alpha(h)^{-2n}\end{smallmatrix}\right)\xi$ to the corresponding leaf $\{h=\const\}$
is equal to identity, that is, it satisfies the property:
\begin{center}
$(\mathcal P)$\hspace{1.5cm}$\phinf'$ is periodical of equal period on the whole leaf.
\end{center}

The set of all periodical leaves corresponds exactly to the values of $h$ such that $\tfrac12\arg\lambda \pm h^s\in \pi\Q$,
in particular it accumulates at the origin with $2s$ equidistributed asymptotic directions $e^{\frac{i\ell\pi}{s}}\mathbb{R}^+$, $\ell=0,\ldots, 2s-1$. Any germ $\phi$ that is analytically equivalent to $\phinf'$ needs to have the same property.

We will construct a triple $(\tau_1,\tau_2,\rho)$ that is formally equivalent to the above normal form $(\taunff1',\taunff2',\rhonf')$,
but has no periodic leaves other than the level set $\{h=0\}$.
The basic idea is to construct $\tau_1,\tau_2$ that extend on the whole $\C^2$ as algebraic maps 
and consider them on the compactification $(\CP^1)^2$.
Since each leaf of the foliation $\Cal F$ except of $\{\xi_1=0\}$ (resp. $\{\xi_2=0\}$) accumulates to the point $\xi=(\infty,0)$ (resp. $(0,\infty)$) in the compactification,
it is enough to show that the map $\phi=\tau_1\circ\tau_2$ has no leaves consisting of periodic points near this point other than 
$\{\xi_2=0\}$ (resp. $\{\xi_1=0\}$).

\begin{proof}[Proof of Theorem~\ref{thm:divergent}]
Let 
\[\tau_1(\xi):=\phi^{\circ\frac12}\circ\sigma(\xi),\qquad \tau_2(\xi):=\sigma\phi^{\circ\frac12}(\xi),\qquad \rho(\xi)=\bar\xi, \]
where
$\phi^{\circ\frac12}:=\sigma\psi\circ\sigma\circ\psi^{\circ(-1)}$, and where $\psi$ is an analytic germ such that $h\circ\psi=h$ and $\bar\psi=\sigma\psi\circ\sigma$, which will then mean that 
$\rho\circ\tau_1\circ\rho=\tau_2=\sigma\circ\tau_1\circ\sigma$.

Namely, let us take
\[\psi(\xi):=\begin{pmatrix}\frac{\alpha^{-\frac12}+ah^s\xi_1}{1+\bar ah^s\alpha^{-\frac12}\xi_2}\xi_1\\[6pt]
\frac{\alpha^{\frac12}+\bar ah^s\xi_2}{1+ah^s\alpha^{\frac12}\xi_1}\xi_2\end{pmatrix},\qquad a\in\C.\]
Then $\psi(\xi)=\left(\begin{smallmatrix}\alpha(h)^{-\frac12} &0 \\[3pt] 0& \alpha(h)^{\frac12}\end{smallmatrix}\right)\xi\mod h^s\Cal J\cdot\xi$, where $\Cal J$ is the ideal of analytic functions vanishing at $\xi=0$. 
Hence $\phi(\xi)=\left(\begin{smallmatrix}\alpha(h)^2 & 0\\[3pt] 0 & \alpha^{-2}(h)\end{smallmatrix}\right)\xi\mod h^s\Cal J\cdot\xi$.
If $\lambda$ is  a root of unity, then $\phi(\xi)$ is of formal type (a) of Theorems~\ref{thm:1} and~\ref{thm:2}: 
indeed $\frac{\xi_1\circ\phi-\xi_1}{\xi_1}=\alpha(h)^{2p}-1 \mod  h^s\Cal J=\pm 2iph^s \mod h^s\Cal J$ hence $k=0$. 
This means that the Moser--Webster triple $(\tau_1,\tau_2,\rho)$ is formally equivalent to \eqref{eq:taunfa}.
If $|\lambda|=1$ but $\lambda$ is not a root of unity then the same follows from the proof of \cite[Theorem 3.4]{Moser-Webster}.

Near the point $\xi=(\infty,0)$, the restriction of $\psi$ to each leaf $h=\const\neq 0$ acts on the local coordinate $\xi_1^{-1}=\frac{\xi_2}{h}$ as 
\[\psi:\xi_1^{-1}\mapsto \xi_1^{-2}\frac{1+\bar a\alpha^{-\frac12} h^{s+1}\xi_1^{-1}}{ah^s+\alpha^{-\frac12}\xi_1^{-1}}=
\frac{1}{ah^s}\xi_1^{-2}+O(\xi_1^{-3}),\]
and it follows that 
\[\phi^{\circ\frac12}=\bar{\psi}\circ\psi^{\circ(-1)}: \xi_1^{-\frac12}\mapsto \frac{a}{\bar a}\xi_1^{-\frac12}+O(\xi_1^{-1}) \]
is analytic in $\xi_1^{-\frac12}$ near $\xi_1^{-\frac12}=0$.
Choosing $a\in\C$ such that $\frac{1}{2\pi i}\log\frac{a}{\bar a}$ is irrational then no iterate of $\phi^{\circ\frac12}$ is equal to identity on any leaf
except of the local leaf $\{h=0\}=\{\xi_2=0\}$, on which $\phi^{\circ\frac12}\big|_{\{h=0\}}:\xi_1^{-1}\mapsto\lambda^{-\frac12}\xi_1^{-1}$.
Hence, $\phi$ does not have the property $(\mathcal P)$.
\end{proof}

\GRN

\subsection{Example of convergence: Monodromy of the Sixth Painlev\'e equation}\label{sec:Painleve}

Here we provide a more detailed account of Example~\ref{example:Painleve}.

	Following the works of Okamoto, the Sixth Painlev\'e equation can be expressed in the form of a non-autonomous Hamiltonian system
	\[\tfrac{dq}{dt}=\tdd{p}H(q,p,t;\kappa),\quad \tfrac{dp}{dt}=-\tdd{q}H(q,p,t;\kappa),\] 
	depending on some parameter $\kappa\in\C^4$.
	The solutions define a singular foliation in the $(q,p,t)$-space, transverse to the fibration $(q,p,t)\mapsto t$ away from the singular fibers $t=0,1,\infty$.
	All the solutions are endlessly meromorphically continuable (this is the Painlev\'e property), and Okamoto has shown that the foliation allows a 
	semi-compactification $\Cal M(\kappa)$  fibered over the $t$-space $\CP^1$, in which all non-vertical leaves are coverings of  $\CP^1\sminus\{0,1,\infty\}$.  
	As a consequence each loop in $\pi_1(\CP^1\sminus\{0,1,\infty\},t_0)$, gives rise to a Poincar\'e return map, a.k.a. \emph{nonlinear monodromy map}, acting as a symplectic isomorphism of the fiber $\Cal M_{t_0}(\kappa)$ above $t_0$ (called \emph{Okamoto's space of initial conditions}).
	
	It is well known (see e.g. \cite{Inaba-Iwasaki-Saito}) that by the Riemann--Hilbert correspondence, the Okamoto's space of initial condition $\Cal M_{t_0}(\kappa)$ is isomorphic to the minimal desingularization of the cubic surface $\Cal R_{t_0}(\theta)=\{x\in\C^3: F(x,\theta)=0\}$,
	\[ F(x,\theta)= x_1x_2x_3+x_1^2+x_2^2+x_3^2-\theta_1x_1-\theta_2x_2-\theta_3x_3+\theta_4,\]
	where $\theta=\theta(\kappa)\in\C^4$, known as the \emph{character variety} (of representations $\pi_1(\CP^1\sminus\{0,1,\infty,t_0\})\to\SL_2(\C)$).
	Under this correspondence, the nonlinear monodromy map associated to a simple loop around either of the singularities takes the form
	\begin{equation*}
		\begin{aligned}
			\phi_{ij}:	  x_i&\mapsto x_i-F_i+x_kF_j, \\
			x_j&\mapsto x_j-F_j,	 \\
			x_k&\mapsto x_k,   
		\end{aligned}	
	\end{equation*} 
	where $F_i(x,\theta)=\tdd{x_i}F(x,\theta)$,
	and can be expressed as a composition of two involutions $\phi_{ij}=\tau_i\circ\tau_j$, 
	\begin{equation*}
		\begin{aligned}
			\tau_{i}:	  x_i&\mapsto x_i-F_i, & \qquad \tau_{j}:	  x_i&\mapsto x_i,\\
			x_j&\mapsto x_j, & x_j&\mapsto x_j-F_j,	 \\
			x_k&\mapsto x_k, & 	x_k&\mapsto x_k, 
		\end{aligned}	
	\end{equation*} 
	$\{i,j,k\}=\{1,2,3\}$, both preserving $x_k$.
	The pair $(\tau_i,\tau_j)$ has up to 4 fixed points $x^*$ on the surface $\Cal R_{t_0}(\theta)$, the solutions of $F_i(x^*,\theta)=F_j(x^*,\theta)=0=F(x^*,\theta)$, 
	which are non-singular points of $\Cal R_{t_0}(\theta)$ under the condition that $F_k(x^*,\theta)\neq 0$.
	Let $\alpha(x_k)=\big(\alpha_i(x_k),\ \alpha_j(x_k),\ \alpha_k(x_k)\big)$
	\[\alpha_i(x_k)=\frac{\theta_j x_k-2\theta_i}{x_k^2-4}, \quad \alpha_j(x_k)=\frac{\theta_i x_k-2\theta_j}{x_k^2-4}, \quad \alpha_k(x_k)=x_k\]
	be the solutions of $F_i(\alpha(x_k),\theta)=F_j(\alpha(x_k),\theta)=0$,
	then $x^*=\alpha(x_k^*)$ where
	\[({x_k^*}^2-4)({x_k^*}^2-\theta_k x_k^*+\theta_4)+\theta_i^2+\theta_j^2-\theta_i\theta_j x_k^*=0.\]
	In the local coordinates $(y_i,y_j)=\big(x_i-\alpha_i(x_k),\ x_j-\alpha_j(x_k)\big)$ near the point $x^*$ the surface $\Cal R_{t_0}(\theta)$ takes the form
	\[ 0=F(x,\theta)-x_iF_i(\alpha(x_k),\theta)-x_jF_j(\alpha(x_k),\theta)=y_iy_jx_k+y_i^2+y_j^2+ F(\alpha(x_k),\theta),\]
	and the two involutions become
	\begin{equation*}
		\begin{aligned}
			\tau_{i}: y_i&\mapsto -y_i-y_jx_k,  & \qquad \tau_{j}: y_i&\mapsto y_i,\\
			y_j&\mapsto y_j, 					&				   y_j&\mapsto -y_j-y_ix_k,	 
		\end{aligned}	
	\end{equation*} 
	with first integral 
	\begin{equation}\label{eq:hPVI}
		h=y_iy_jx_k+y_i^2+y_j^2= -F(\alpha(x_k),\theta)=-F_k(x^*,\theta)\cdot(x_k-x_k^*)+O\big((x_k-x_k^*)^2\big).
	\end{equation}
	Both $\tau_i,\tau_j$ are linear maps depending on $x_k$, therefore if ${x_k^*}\neq\pm 2$ then the pair $(\phi_{ij},\tau_i)$  is analytically conjugated by a linear change of variable to the formal normal form $(\phinf,\taunf)$ \eqref{eq:phinf}. By comparing the traces of $\phinf$ and $\phi_{ij}$
	\[\lambda e^{\tilde ch^s}+\lambda^{-1} e^{-\tilde ch^s}=-2+x_k^2=-2+\big({x_k^*}-\tfrac{h}{F_k(x^*,\theta)}+O(h^2)\big)^2,\] 
	using \eqref{eq:hPVI}, namely  this means that 
	\[s=1,\qquad \lambda+\lambda^{-1}=-2+{x_k^*}^2,\qquad (\lambda-\lambda^{-1})\,\tilde c=\tfrac{-2x_k^*}{F_k(x^*)}.\] 
	Depending on the parameter $\theta(\kappa)$ and the point $x^*\neq\pm2$, the multiplier $\lambda\neq 1$ may or may not be a root of unity, 
	but the local analytic conjugacy of the pair $(\phi_{ij},\tau_i)$, and therefore of the corresponding nonlinear monodromy of Painlev\'e VI, to the normal form $(\phinf,\taunf)$ \eqref{eq:phinf} exists in any case as long as the condition $({x_k^*}^2-4)F_k(x^*)\neq0$ is satisfied.

\FGRN

\section{Analytic classification in the case $k>0$}\label{sec:6}

The rest of the paper is devoted to a description of the modulus of analytic classification in the formal cases (b) and (c) of Theorem~\ref{thm:1}.

Rather than working within each formal equivalence class, \emph{we will work in the larger \textbf{model class}} \Grn{(Definition~\ref{def:model})} 
\emph{consisting in the case (b) of the union of the formal classes over all possible invariants $\hat\mu(h)$.} 
The reason for such approach is that instead of working with the a priori purely formal normal form $\hatbXnf=h^s\mfrac{cP}{1+\hat\mu cP}\bE$ we will work with the analytic model
$\bXmod=h^scP\bE$. 
Conveniently, also the description of the dynamics of $\bXmod$ upon which the domains of normalization of Theorem~\ref{thm:sectorial} will be constructed in Section~\ref{sec:6.2} is slightly more simple than that of $\hatbXnf$ (that is if $\hatbXnf$ was convergent).

\subsection{Construction of normalizing transformations to a  model}\label{sec:6.1}

Let $\phi(\xi)=\Lambda\xi+\hot$, $\tau(\xi)=\sigma\xi=\left(\begin{smallmatrix} 0&1\\1&0 \end{smallmatrix}\right)\xi$ and $h(\xi)=\xi_1\xi_2$ be in 
the prenormal form of Proposition~\ref{prop:prenormalphi}.
Assume that $\phi$ is of the formal type (b) or (c) of Theorem~\ref{thm:1}, 
and let $\hatphinf=\Lambda\exp(\frac1p \hatbXnf)$ with
\[\hatbXnf=\frac{ch^sP(u,h)}{1+\hat\mu(h) cP(u,h)}\bE,\qquad\text{where}\quad 
\begin{cases}
	\text{(b)}\  \hat{\mu}(h)\in\C\llbracket h\rrbracket,\\
	\text{(c)}\  \hat{\mu}(h)=0,\\
\end{cases}
\]
and
\[\begin{cases}
	\text{(b)}\ P(u,h)=u^k+P_{k-1}(h)u^{k-1}+\ldots+P_0(h),& u=\xi_1^p+\xi_2^p,\\
	\text{(c)}\ P(u,h)=u^{2\tilde k+1}+\tilde P_{\tilde k-1}(h)u^{2\tilde k-1}+\ldots+\tilde P_0(h)u,& u=\xi_1+\xi_2,\\
\end{cases}
\]
be the formal normal form,
and let $\phimod=\Lambda\exp(\frac1p \bXmod)$ with
\begin{equation}\label{eq:bY}
	\bXmod:=ch^sP(u,h)\bE=h^s\bY,\qquad \bY:=cP(u,h)\bE.
\end{equation}
be the associated model \Grn{(Definition~\ref{def:model}).}
Denoting $\hat\bX$ \eqref{eq:prenormalX} the formal infinitesimal generator of $\phi^{\circ p}=\exp(\hat\bX)$, then
since $\phi$ is in the prenormal form
\[\hat\bX=ch^sP(u,h)\bE \mod h^sP(u,h)^2\bE,\]
and
 \begin{equation}\label{eq:adaptedmodel}
	\phi^{\circ p}=\exp(ch^sP(u,h)\bE)\mod h^sP(u,h)^2\xi,  
\end{equation}
which by Corollary~\ref{cor:infinitesimalgenerator} is equivalent to 
\begin{equation}\label{eq:flog}
	\frac{f\log(1+\bE.f)}{\bE.f}=ch^sP\mod h^sP^2,\quad\text{where }\ f:=\frac{\xi_1\circ\phi^{\circ p}-\xi_1}{\xi_1}.
\end{equation}

We denote $\bY_h$ the restriction  of $\bY$ on the leaf $\{h=\xi_1\xi_2=\const\}$.
In the local coordinate $\xi_1$ on the leaf it can be written as
\begin{align}\label{eq:bYh}
\bY_h&=c\,P_h(\xi_1)\,\xi_1\tdd{\xi_1},\quad\text{where }\ 
P_h(\xi_1):=\begin{cases}\text{(b)}\ P(\xi_1^p+\tfrac{h^p}{\xi_1^{p}},h),& kp>0\\ \text{(c)}\ P(\xi_1+\tfrac{h}{\xi_1},h),& kp=2\tilde k+1, \end{cases}
\end{align}
and $\xi_1^{kp}P_h(\xi_1)$ is a polynomial in $\xi_1$ of order $2kp$.
The vector field $\bY_h$ is a rational in $\xi_1\in\CP^1$, with poles at $\xi_1=0,\infty$ of orders $kp-1$, and depends analytically on the parameter $h$. 
For $h\neq 0$ it is reversed
by the involution 
\begin{equation}\label{eq:sigma1}
\sigma:\xi_1\mapsto\frac{h}{\xi_1}.
\end{equation}

Consider a neighborhood of the origin
\begin{equation}\label{eq:B}
B=\big\{|\xi_1|,|\xi_2|<\delta_1,\ |h|<\delta_2\big\}, 
\end{equation}
for some $\delta_1,\delta_2>0$, with $\delta_2$ small enough so that all zeros of $\bY_h$ lie inside $B_h$ for all $|h|<\delta_2$,
where $B_{h}=B\cap\{h=\const\}$.
In the coordinate $\xi_1$, when $h\neq 0$, $B_h$ is identified with the annulus 
\begin{equation}\label{eq:Bh}
B_h=\Big\{\tfrac{|h|}{\delta_1}<|\xi_1|<\delta_1\Big\}.
\end{equation}
The level set $B_0$ has two irreducible components $\{|\xi_1|<\delta_1\}$ and $\{|\xi_2|<\delta_1\}$ symmetric one to the other by the involution $\sigma$ \eqref{eq:sigma1}.

\emph{Our goal is to construct normalizing transformations $\Psi_{\Omega}$ on some ramified domains \Grn{$\Omega=\coprod_h\Omega_h$ in $B=\coprod_h B_h$,}
such that }
\[ \Psi_{\Omega}\circ\phi^{\circ p}=\phimod^{\circ p}\circ\Psi_{\Omega}. \]
We shall construct such $\Psi_{\Omega}$ \Grn{leaf-by-leaf} treating $h$ as a parameter.
This will be easier to do in  the \emph{rectifying coordinate} of $\bY$ defined as follows:
\Grn{ Let 
\[\bY_h^{-1}=\mfrac{d\xi_1}{c\,\xi_1\,P_h(\xi_1)}\] 
be a meromorphic 1-form dual to $\bY_h$ on each leaf $\{h=\const\}$, and let  
\begin{equation}\label{eq:bt}
\bt_h(\xi)=\int \bY_h^{-1}
\end{equation}
be the a priori multivalued rectifying coordinate for $\bY_h$ on $\{h=\const\}\sminus\{P_h=0\}$,
\[\bY_h=\tfrac{\partial}{\partial \bt_h} \qquad\text{and}\qquad \bt_h\circ\phimodh^{\circ p}=\bt_h\circ\exp(h^s\bY_h)=\bt_h+h^s, \]
defined up to addition of some constant $C(h)$. 
We denote $\bt(\xi)=\int \bY^{-1}$ a choice of the rectifying coordinate for $\bY$ depending analytically on $h$.
}

\subsubsection{Construction of a Fatou coordinate}
For each $h$ the Riemann surface of $\bt_h(\xi)$ is a covering surface of the punctured leaf $\{h=\const\}\sminus\{P_h=0\}$, where $\{P_h=0\}$ denotes the fixed point set of $\bY_h$ \eqref{eq:bYh}.
When endowed with the coordinate $\bt_h$ it becomes a \emph{translation surface}\footnote{A \emph{translation surface} is a Riemann surface with an atlas whose transition maps are affine translations. Equivalently, it is a Riemann surface endowed with a holomorphic abelian differential -- in our case $d\bt_h$.}
containing \emph{saddle points}\footnote{A \emph{saddle point} (or conical singularity) of a translation surface is a point whose neighborhood is a topological disc of angular opening $2m\pi$ for some $m\in\Z_{>0}$. See e.g. \cite{Wright}. }
of angle $2kp\pi$ situated over the points $\xi_1=0,\infty$, \Grn{which are poles of order $kp$ of $\bt_h$}.
By restricting to $B_h$, the image $\bt_h(B_h)$ acquires  \emph{holes} around the saddle points corresponding to either of the complements
$\{|\xi_1|\geq\delta_1\}$ and $\{|\xi_2|=\frac{|h|}{|\xi_1|}\geq\delta_1\}$ of $B_h$ in the leaf $\{h=\const\}$.

Denote $\phi_h^{\circ p}$, resp. $\phimodh^{\circ p}=\exp(\bXmodh)=\exp(h^s\bY_h)$, the restriction  of $\phi^{\circ p}$, resp. $\phimod^{\circ p}$ on $B_h$.
We want to construct a normalizing transformation $\Psi_{\Omega_h}$ on some ramified domain\footnote{More precisely, the ramified domain $\Omega_h$ may be understood as a domain in the covering surface of $B_h\sminus\{P_h=0\}$ on which $\bt_h$ lives.} 
 $\Omega_h\subseteq B_h\setminus\{P_h=0\}$,
depending locally analytically on $h$,
such that
\[\phimodh^{\circ p}\circ\Psi_{\Omega_h}=\Psi_{\Omega_h}\circ\phi_h^{\circ p}.\]
This will be obtained by constructing a \emph{Fatou coordinate}
$\bT_{\Omega_h}$ for $\phi_h^{\circ p}$ on $\Omega_h$ 
satisfying\footnote{Usually a Fatou coordinate is defined as conjugating $\phi_h^{\circ p}$ to translation by $1$. However in our situation, as we need to keep track of the dependency in $h$, it is natural to conjugate to translation by $h^s$. We still call it Fatou coordinate.}
\[ \bT_{\Omega_h}\circ\phi_h^{\circ p}=\bT_{\Omega_h}+h^s. \]
\Grn{Then $\Psi_{\Omega_h}$ will be defined by the identity}
\[\bT_{\Omega_h}=\bt_h\circ\Psi_{\Omega_h}.\]

The construction of the Fatou coordinate $\bT_{\Omega_h}$ is more-less the same as that used by Voronin \cite{Voronin} as well as \cite{MRR,Ribon-a,Shi}:
first construct a quasi-conformal conjugacy between $\bt_h\circ\phi_h^{\circ p}$ and $\bt_h\mapsto\bt_h+h^s$ on a strip in $\bt_h(B_h)$, afterwards correct it to a holomorphic one using Ahlfors--Bers theorem, and then extend it to a bigger domain through iteration of $\phi_h^{\circ p}$.
The shape of the domain will depend only on the position of the strip in the surface $\bt_h(B_h)$ with respect to its holes.

\begin{definition}~\label{def:stable}
\begin{itemize}[wide=0pt, leftmargin=\parindent]
	\item[-] Let $\theta\in\R$. A \emph{real trajectory} of the vector field $e^{i\theta}h^s\bY_h$ through a point $\xi$ is the curve $t\mapsto\exp(te^{i\theta}h^s\bY_h)(\xi)$, $t\in\R$, which corresponds to the line $\bt_h(\xi)+e^{i\theta}h^s\R$ in the coordinate $\bt_h$.
	It is the same as a the trajectory of $\bXmodh=h^s\bY_h$ as the complex time evolves in the direction $\theta$.
	\item[-] For a given $h$, an angle $\theta\in\,\,]0,\pi[$ is called \emph{stable relative to $B_h$} if for all $\xi\in B_h$ the real trajectory of  $e^{i\theta}h^s\bY_h$ through $\xi$ stays in $B_h$ for either all positive (forward time $t>0$) or all negative (backward time $t<0$) time,
	i.e. it is not allowed to leave $B_h$ in both directions.
		
	In this case the forward, resp. backward, limit of the trajectory is necessarily one of the equilibrium points of $\bY_h$ in $B_h$ (more on this in Section~\ref{sec:6.2}).
\end{itemize}
\end{definition}

\begin{definition}\label{def:admissible}
For a fixed $h$, small $\delta_3>0$ (to be precised later), and a stable angle $\theta\in\,\,]\delta_3,\pi-\delta_3[$, 
let $\Gamma_h$ be a real trajectory of $e^{i\theta}h^s\bY_h$ through some non-singular point $\xi_*(h)\in B_h$,
such that both $\Gamma_h,\phi^{\circ p}(\Gamma_h)\subset B_h$, and denote $\Sigma_h\subset B_h$ the region bounded by the two curves $\Gamma_h$ and $\phi^{\circ p}(\Gamma_h)$.
Let $\bt_h(\Sigma_h)$ be its bijective image by a branch of $\bt_h$ in the translation surface $\bt_h(B_h)$ (see Figure \ref{figure:lavaursdomain}) -- we call it an \emph{admissible strip}. 
\end{definition}

The existence of a stable $\theta$ and of an admissible strip will be proved in Section~\ref{sec:6.2}.\label{page:admissiblestrip}
Since the curve $\Gamma_h$ is transversal to the real flow of the model vector field $\bXmodh=h^s\bY_h$, the region $\Sigma_h$
is also ``transversal'' to the dynamics of $\phimodh^{\circ p}$ and therefore also to $\phi_h^{\circ p}$.

\begin{figure}[t]
	\centering
	\includegraphics[width=0.99\textwidth]{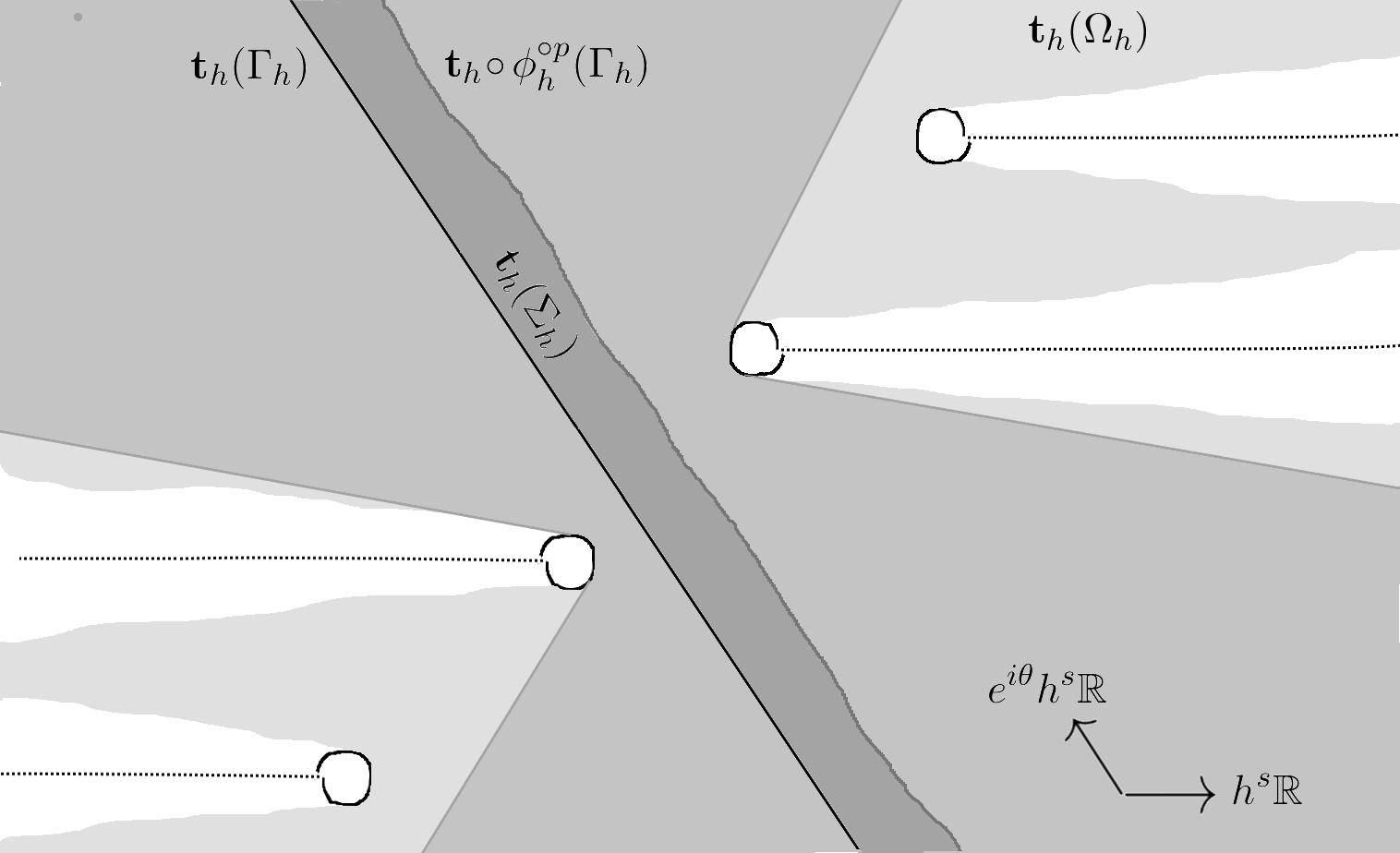}
	\caption{An admissible strip $\bt_h(\Sigma_h)$ in dark gray, and the saturated Lavaurs domain $\bt_h(\Omega_h)$ in light gray. In the
		medium gray is the maximal subdomain of $\bt_h(\Omega_h)$ spanned by all the lines $t+e^{i\theta'}h^s\R$ of varying slope $\theta'\in\,\,]\delta_3,\pi-\delta_3[$ contained inside $\bt_h(\Omega_h)$.
		This is a ``naive version'' of the Lavaurs domain which will be constructed later in Section~\ref{sec:6.2}: 
		such subdomain is easier to understand, doesn't depend on $\phi^{\circ p}$ but only on $\phimod^{\circ p}$, and the constructed conjugacy
		$\Psi_{\Omega_h}$ will be bounded on it. }
	\label{figure:lavaursdomain}
\end{figure}

\begin{definition}\label{def:Omegasaturation}
\Grn{	Let $\bt_h(\Sigma_h)$ be an admissible strip and let us define a ramified domain $\Omega_{h}$ in $B_h\sminus\{P_h(\xi_1)=0\}$ 
	as the ``saturation'' of $\Sigma_h$ through iteration of $\phi_h^{\circ p}$. This is more clear in the coordinate $\bt_h$:
 	Let $\Phi_h:=\bt_h\circ \phi_h^{\circ p}\circ\bt_h^{\circ(-1)}$, and define
 	\[\bt_h(\Omega_h):=\{t: \exists n\in\Z,\ \Phi_h^{\circ n}(t)\in\bt_h(\Sigma_h)\ \&\ 
 	\forall l=0,\ldots,n,\ \Phi_h^{\circ l}(t)\in \bt_h(B_h)\}, \]
 	see Figure~\ref{figure:lavaursdomain}.
 	Then the ramified domain $\Omega_h$, which we call a \emph{saturated Lavaurs domain}, is defined as a simply connected bijective preimage of $\bt_h(\Omega_h)$ in 
 	a (covering surface) of $B_h\sminus\{P_h(\xi_1)=0\}$.
  }
\end{definition}

\com{I've left out the dependence of $\Omega_h$ on $\theta$ in the notation, because it depends on $\Sigma_h$ which depends not only on $\theta$ but also on $\xi_*$ and $\delta_1,\delta_2,\delta_3$ and $\phi_h^{\circ p}$, but in the end neither of it matters under a small perturbation.}

\begin{proposition}[Existence of a Fatou coordinate]\label{prop:Fatou}
There are constants $\delta_1,\delta_2,\delta_3>0$ (as in \re{eq:B} and \rd{def:admissible}), such that for 
each $h\neq 0$ \Grn{and a saturated Lavaurs domain $\Omega_{h}$ (Definition~\ref{def:Omegasaturation}): }	
\begin{enumerate}[wide=0pt, leftmargin=\parindent]
\item
On the domain $\Omega_h$ there exists an analytic function $\bT_{\Omega_h}:\Omega_h\to \C$ that is a Fatou coordinate for $\phi_h^{\circ p}$
\begin{equation}\label{eq:Fatourelation}
	\bT_{\Omega_h}\circ\phi_h^{\circ p}=\bT_{\Omega_h}+h^s, 
\end{equation}
and such that 
\begin{enumerate}
	\item $\bT_{\Omega_h}$ is univalent (injective) on $\Omega_{h}$,
	\item $\lim_{\bt_h(\xi)\to\pm\infty\cdot e^{i\theta}h^s}\IM\big(h^{-s}\bT_{\Omega_h}(\xi)\big)=\pm\infty$,
	where the limit is taken along any trajectory of $e^{i\theta}h^s\bY_h$ in $\Omega_{h}$, corresponding to a line $\bt_h\in t_0\pm  e^{i\theta}h^s\R_{>0}$ in $\bt_h(\Omega_h)$, 
	\item $\bT_{\Omega_h}\circ\bt_h^{\circ(-1)}$ has a moderate growth\footnote{At most polynomial-like in $\bt_h$.} when $\bt_h\to\pm\infty\cdot e^{i\theta}h^s$ along any such line.
\end{enumerate}

\item
If $\,\bT'_{\Omega_h}:\Omega_h\to \C$ is another analytic function satisfying \eqref{eq:Fatourelation} with either of the properties \textit{(a), (b), (c)}
then $\bT'_{\Omega_h}-\bT_{\Omega_h}=C(h)$ is constant on the leaf.
\end{enumerate} 
\end{proposition} 

Our proof below follows the same general strategy as \cite{Voronin,Shi,MRR,Ribon-a,Rousseau}.

\begin{lemma}\label{lemma:Delta}
Denote
\begin{equation}\label{eq:Delta}
	h^s\Delta:=\bt\circ\phi^{\circ p}-\bt\circ\phimod^{\circ p}=\bt\circ\phi^{\circ p}-\bt-h^s,
\end{equation}
meaning that $\phi^{\circ p}=\exp\big((1+s)\bXmod\big)\big|_{s=\Delta}$. 
Then $\Delta(\xi)=cP(u,h)U(\xi)$ for some analytic germ $U(\xi)$, with $U(0)=-\hat{\mu}(0)$, where $\hat\mu(h)$ is the formal invariant of $\phi^{\circ p}$.
\end{lemma}

\begin{proof}
 	This follows from the implicit function theorem.
 	Let 
 	\begin{align*}
 		F(t,\xi):\!&=\xi_1\circ \exp((1+t)\bXmod) - \xi_1\circ \exp(\bXmod)\\
 		&= t\bXmod.\big[\sum_{n=1}^{+\infty}\tfrac{1}{n!}(1+t+\ldots+t^{n-1})\bXmod^{.(n-1)}.\xi_1\big], 
 	\end{align*}
 	see \eqref{eq:exp-form}.
 	The identity \eqref{eq:Delta} is equivalent to $\xi_1\circ\phi^{\circ p}-\xi_1\circ\phimod^{\circ p}=F(\Delta,\xi)$.
 	 As $\bXmod=h^scP\bE$, we see that the function  $G(U,\xi):=\frac{1}{c^2h^sP^2\xi_1}F(cPU,\xi)=U+\hot(U,\xi)$
 	is analytic in the variables $(U,\xi)$.
 	By the assumption \eqref{eq:adaptedmodel}, write  
 	\[\xi_1\circ\phi^{\circ p}-\xi_1\circ\phimod^{\circ p}=c^2h^sP^2\xi_1\cdot\tilde r(\xi)\] 
 	for some analytic germ $\tilde r(\xi)$.
 	Then the implicit equation $G(U(\xi),\xi)=\tilde r(\xi)$ has a unique analytic solution $U(\xi)$ with $U(0)=\tilde r(0)$.
 	
  	Let $\hat{\bX}=\frac{h^scP}{1+cP\hat R}\bE$ be the formal infinitesimal generator of $\phi^{\circ p}=\exp(\hat{\bX})$.
 	Then $\frac{\hat\bX.\xi_1-\bXmod.\xi_1}{\xi_1 h^sc^2P^2}=-\hat R(0)+\hot=-\hat{\mu}(0)+\hot$,
 	and one can see that
 	\[\tilde r(\xi)=\frac{\xi_1\circ\phi^{\circ p}-\xi_1\circ\phimod^{\circ p}}{\xi_1 h^sc^2P^2}=\sum_{n=1}^{+\infty}\frac1{n!}\frac{\hat\bX^{.n}.\xi_1-\bXmod^{.n}.\xi_1}{\xi_1 h^sc^2P^2}=-\hat{\mu}(0)+\hot, \]
 	therefore $U(0)=\tilde r(0)=-\hat{\mu}(0)$.	
\end{proof}

\begin{lemma}\label{lemma:omega} 
\Grn{	Let $\xi_*$, $\theta$ and $\Sigma_h$ be as in Definition~\ref{def:admissible}, and let us choose a determination of $\bt_h(\xi)$ \eqref{eq:bt} such that $\bt_h(\xi_*(h))=0$, and hence the left boundary line of the admissible strip $\bt_h(\Sigma_h)$ is $\bt_h(\Gamma_h)=e^{i\theta}h^s\R$.}
    Let $\bt_h(\tilde\Sigma_h)=e^{i\theta}h^s\R+[0,h^s[$ be the semi-closed strip between the lines $\bt_h(\Gamma_h)=e^{i\theta}h^s\R$ on the left (included) and $\bt_h\circ\phimodh^{\circ p}(\Gamma_h)=h^s+e^{i\theta}h^s\R$ on the right (excluded).
	Then the map $\omega_1:\bt_h(\tilde\Sigma_h)\to\bt_h(\Sigma_h)$ defined by 
	\[t=h^s\left(r+e^{i\theta}q\right)\mapsto\omega_1(t,\ov t)=t+rh^s\Delta_h\circ\bt^{\circ(-1)},\quad r\in[0,1[,\ q\in\R,\]
	where $\Delta_h$ is the restriction of \eqref{eq:Delta} to $\{h=\const\}$, is a smooth bijection.
	Moreover, let
	\begin{equation}\label{eq:Beltrami}
	 \bmu(t):=\frac{\dd{\ov t}\omega_1}{\dd{t}\omega_1},\quad\text{then}\quad |\bmu\circ\bt_h|=O(|P_h|), 
	\end{equation}
	that is, for every $\delta_3>0$ there exist $\delta_1,\delta_2>0$ from the definition of $B$ \eqref{eq:B} such that for every $\theta\in[\delta_3,\pi-\delta_3]$ and every admissible strip $\bt_h(\Sigma_h)\subset\bt_h(B_h)$ the above is satisfied uniformly.
\end{lemma}

\begin{proof}
Writing $h^{-s}t=r+e^{i\theta}q$, then $q=\frac{\IM(h^{-s}t)}{\sin\theta}$, 
$r=\RE(h^{-s}t)-\IM(h^{-s}t)\cotg\theta$, therefore
\[\omega_1(\bt_h,\ov\bt_h)=\bt_h\,(1+\tfrac{\Delta_h}{2}(1+i\cotg\theta))+\ov\bt_h\ov h^{-s}h^s \tfrac{\Delta_h}{2}(1-i\cotg\theta),\]
from which
\[ \bmu\circ\bt_h=\frac{\dd{\ov\bt_h}\omega_1(\bt_h,\ov\bt_h)}{\dd{\bt_h}\omega_1(\bt_h,\ov\bt_h)}
=\frac{\frac{\ov h^{-s}h^s}2(1-i\cotg\theta)\Delta_h}{1+\frac{1}{2}(1+i\cotg\theta)\Delta_h+rh^s\bY_h.\Delta_h}, \]
since $\tfrac{\partial}{\partial \bt_h}=\bY_h$.
By Lemma~\ref{lemma:Delta}, $|\Delta_h|=O(|P_h|)$ and $|\bY_h.\Delta_h|=O(|P_h|)+O(|\bE.P_h|)$, from which $|\bmu\circ\bt_h|=O(|P_h|)$.
Let us show that $\omega_1$ is a bijection $\bt_h(\tilde\Sigma_h)\to\bt_h(\Sigma_h)$.
It sends segments parallel to $[0,h^s[$ in $\bt_h(\tilde\Sigma_h)$ to segments in $\bt_h(\Sigma_h)$ of slope close to $\arg h^s$ 
(if $|P|$ is small enough then the difference of the angles, which is at most $\arcsin|\Delta_h|$, can be assumed $\ll\frac{\delta_3}{2}$). 
According to \eqref{eq:Delta}, the curve $\bt_h\circ \phi_h^{\circ p}\circ\bt_h^{\circ(-1)}(e^{i\theta}\R)$ is $|h^s\Delta_h\circ\bt_h^{\circ(-1)}|$-close to $h^s+e^{i\theta}h^s\R$, and it can be assumed that the argument of its tangent is close to $\theta+\arg h^s$  (if $|P_h|$, and therefore $|\Delta_h|$, is small enough then it can be assumed the arguments differ by $\ll\frac{\delta_3}{2}$). 
It follows that the images of two different segments by $\omega_1$ don't intersect, and thus $\omega_1$ is bijective.
\end{proof}

\begin{proof}[Proof of Proposition~\ref{prop:Fatou}]
1) 
The map $t\mapsto z=e^{\frac{2\pi i}{h^s}t}$ identifies the strip $\bt_h(\tilde\Sigma_h)$ with $\C^*=\CP^1\sminus\{0,\infty\}$. 
Let $\tilde\omega_1(z):=\omega_1(\tfrac{h^s}{2\pi i}\log z)$ be a map $\C^*\to \bt_h(\Sigma_h)$ with a discontinuity along the spiral that is the preimage of the line $e^{i\theta}h^s\R$. 
By \eqref{eq:Beltrami}, its Beltrami constant
\[ \tilde\bmu(z):=\frac{\dd{\ov z}\tilde\omega_1}{\dd{z}\tilde\omega_1}=\frac{z}{\ov z}\bmu(\tfrac{h^s}{2\pi i}\log z),\qquad \tilde\bmu(0)=\tilde\bmu(\infty)=0, \]
can be (up to restricting the size of the polydisc $B$) assumed small enough so that its essential supremum is $\|\bmu\|_\infty=\sup_{t\in\bt_h(\tilde\Sigma_h)}|\bmu|<1$.
So by the Ahlfors--Bers theorem, there exists a unique quasi-conformal map $\tilde\omega_2:\CP^1\to\CP^1$ fixing the points $0,1,\infty$ such that $\frac{\dd{\ov z}\tilde\omega_2}{\dd{z}\tilde\omega_2}=\tilde\bmu(z)$ almost everywhere.
Let 
$\omega_2(t)=\frac{h^s}{2\pi i}\log\tilde\omega_2(e^{2\pi ih^{-s}t})$ be the lifting of $\tilde{\omega}_2$ by $t=\tfrac{h^s}{2\pi i}\log z$, and put
\[ \bT_{\Omega_h}:=\omega_2\circ\omega_1^{\circ(-1)}\circ\bt_h.\]
Then the map $\omega_2\circ\omega_1^{\circ(-1)}=\tfrac{h^s}{2\pi i}\log\big(\tilde\omega_2\circ\tilde\omega_1^{\circ(-1)}\big)$
is then a holomorphic conjugacy from $\bt_h(\Sigma_h)$ to some curved strip $\bT_{\Omega_h}(\Sigma_h)$ bounded by $\bT_{\Omega_h}(\Gamma_h)$ and $\bT_{\Omega_h}(\Gamma_h)+h^s$, 
and is such that $\lim_{t\to \pm\infty\cdot e^{i\theta}h^s} \RE\big(\frac{2\pi i}{h^s}\,\omega_2\circ\omega_1^{\circ(-1)}(t)\big)=\pm\infty$.
Moreover $\bT_{\Omega_h}$ is such that $\bT_{\Omega_h}\circ\phi_h^{\circ p}=\bT_{\Omega_h}+h^s$ on the left boundary $\Gamma_h$, and  
it has a well-defined analytic continuation to the domain $\bt_h(\Omega_h)$ of Definition~\ref{def:Omegasaturation} satisfying the Fatou relation 
$\bT_{\Omega_h}(\xi)=\bT_{\Omega_h}\circ\phi_h^{\circ np}(\xi)-nh^s$, 
whenever $\phi_h^{\circ np}(\xi)\in\Sigma_h$ for some $n\in\Z$, and all the iterates $\phi_h^{\circ lp}(\xi)\in B_h$, $l\in[0,n]\subset\Z$. 

By the construction $\bT_{\Omega_h}$ satisfies the properties \textit{(a), (b), (c)} on $\Sigma_h$, and the analytic extension to $\Omega_h$ through iteration preserves them as well.

2) The map $\psi:t\mapsto \bT'_{\Omega_h}\circ\bT_{\Omega_h}^{\circ(-1)}(t)$, it commutes with the translation $t\mapsto t+h^s$ and therefore extends to an analytic map
on $\C$, with $\psi(t)-t$ periodic of period $h^s$. If $\bT_{\Omega_h}'$ is univalent on $\Omega_{h}$, then $\psi:\C\to\C$ is univalent as well, hence $\psi(t)=t+C(h)$.
If $\lim_{t\to \pm\infty\cdot e^{i\theta}h^s}\RE\big(\frac{2\pi i}{h^s}\,\bT_{\Omega_{h}}'(t)\big)=\pm\infty$ then also
$\lim_{t\to \pm\infty\cdot e^{i\theta}h^s}\RE\big(\frac{2\pi i}{h^s}\,\psi(t)\big)=\pm\infty$, and $\psi$ is a lift by $z=e^{\frac{2\pi i}{h^s}t}$ of a diffeomorphism of $\CP^1$ in the $z$-coordinate fixing $0$ and $\infty$, hence of a rotation by some multiplicative constant $e^{\frac{2\pi i}{h^s}C(h)}$, meaning that 
$\psi:t\mapsto t+C(h)$. 
If $\bT_{\Omega_{h}}'$ has a moderate growth in $\bt_h$, then so does $\psi(t)-t$ which can be written as a Fourier series $\sum_{n\in\Z}C_n e^{\frac{2\pi i nt}{h^s}}$, and the moderate growth at both $t\to+e^{i\theta}\infty$ and $t\to-e^{i\theta}\infty$ means that $C_n=0$ for all $n\neq 0$.
\end{proof}

\begin{proposition}\label{prop:Fatou-behavior}
The Fatou coordinate $\bT_{\Omega_h}$ of Proposition~\ref{prop:Fatou} satisfies	
\begin{itemize}
	\item[(i)] $\displaystyle \lim_{\bt_h\to\pm\infty\cdot e^{i\theta}h^s}\bT_{\Omega_h}-\bt_h-\hat\mu(0)\log\xi_1\in\C$,
	\item[(ii)] $\displaystyle \lim_{\bt_h\to\pm\infty\cdot e^{i\theta}h^s}\tdd{\bt_h}\big(\bT_{\Omega_h}-\bt_h\big)=0,$
\end{itemize}
both limits considered along any line $\bt_h\in t_0\pm  e^{i\theta}h^s\R_{>0}$ in $\bt_h(\Omega_h)$ with $\theta\in\,\,]\delta_3,\pi-\delta_3[$.
\end{proposition}

\begin{proof}
First let us prove \textit{(i)}. 
\Grn{Let $\bt_h(\Sigma_h)$ and $\bt_h(\tilde\Sigma_h)$ be as in Lemma~\ref{lemma:omega}.}
Denote $\check\bt_h=\bt_h+\hat\mu(0)\log\tfrac{\xi_1}{h^{\frac12}}$
and 
\[h^s\check\Delta_h=\check\bt_h\circ\phi_h^{\circ p}-\check\bt_h-h^s=h^s\Delta_h+\hat\mu(0)\log\frac{\xi_1\circ\phi_h^{\circ p}}{\xi_1},\]
then by Lemma~\ref{lemma:Delta} $\check\Delta(\xi)=cP(\xi)\check U(\xi)$ with $\check U(0)=0$.
In the proof of Proposition~\ref{prop:Fatou}, we can replace $\bt_h$ by $\check\bt_h$ and $\Delta_h$ by $\check\Delta_h$ and thanks to the uniqueness property we obtain the same Fatou coordinate up to a constant.

We need to calculate the limit $\lim_{\check\bt_h\to\pm \infty\cdot e^{i\theta}h^s}\bT_{\Omega_h}-\check\bt_h$.
We have 
\begin{equation}\label{eq:limt}
	\begin{aligned}
		\bT_{\Omega_h}-\check\bt_h&=\tfrac{h^s}{2\pi i}\log\mfrac{\tilde\omega_2\circ \tilde\omega_1^{\circ(-1)}(\check\bt_h)}{\tilde\omega_1^{\circ(-1)}(\check\bt_h)}+\tfrac{h^s}{2\pi i}\log\tilde\omega_1^{\circ(-1)}(\check\bt_h)-\check\bt_h\\
		&=\tfrac{h^s}{2\pi i}\log\tfrac{\tilde\omega_2(\check\bz_h)}{\check\bz_h}\big|_{\check\bz_h=\tilde\omega_1^{\circ(-1)}(\check\bt_h)}+
		\omega_1^{\circ(-1)}(\check\bt_h)-\check\bt_h,
	\end{aligned}
\end{equation}
where $\check\bz_h=e^{\frac{2\pi i}{h^s}\check\bt_h}$.
As $|\omega_1^{\circ(-1)}(\check\bt_h)-\check\bt_h|=O(|h^sP_h|)$, this means that 
the limit equals to 
\[\lim_{\check\bt_h\to+ \infty\cdot e^{i\theta}h^s}(\bT_{\Omega_h}-\check\bt_h)=\lim_{z\to0} \tfrac{h^s}{2\pi i}\log\tfrac{\tilde\omega_2(z)}{z}=
\tfrac{h^s}{2\pi i}\log\tdd{z}\tilde\omega_2\big|_{z=0},\] 
resp. 
\[\lim_{\check\bt_h\to- \infty\cdot e^{i\theta}h^s}(\bT_{\Omega_h}-\check\bt_h)=\lim_{z\to\infty} \tfrac{h^s}{2\pi i}\log\tfrac{\tilde\omega_2(z)}{z}=
-\tfrac{h^s}{2\pi i}\log\tdd{y}\tfrac{1}{\tilde\omega_2(y^{-1})}\big|_{y=0}.\]
We have $\tilde\bmu(0)=\tilde\bmu(\infty)=0$, we need to check that $\tilde\omega_2(z)$ is conformal at $z=0,\infty$ and therefore that
the derivative $\tdd{z}\tilde\omega_2(z)$ exists and is finite at $z=0,\infty$.
Let us look at $z=0$ only. By \cite[Theorem 7.1]{Lehto-Virtanen} the sufficient condition is that the integral
\[\iint\displaylimits_{|z|<e^{-R}}\tfrac{|\tilde\bmu(z)|}{1-|\tilde\bmu(z)|}\tfrac{d\RE(z)d\IM(z)}{|z|^2}
=(2\pi)^2\!\!\!\!\iint\displaylimits_{\substack{0<\RE(th^{-s})<1,\\ \IM(th^{-s})>\frac{R}{2\pi}}}\tfrac{|\bmu(t)|}{1-|\bmu(t)|}{ d\RE(th^{-s})d\IM(th^{-s})}
\]
is finite for some $R>0$.
This is equivalent (for $h^s\neq 0$) to $\int|\check\Delta_h||\bY_h^{-1}|=\int|\check U_h|\bE_h^{-1}$ to be bounded in $\Sigma_h$, which is in fact satisfied since for $h\neq 0$ the point $\check\bz_h=0$ correspond to some singularity $\xi=a_h\neq 0$, and for $h=0$ we have $\check U_0(0)=0$.

Let us now prove \textit{(ii)}. 
Let $a(h)$ be the zero of $P_h$ \eqref{eq:bYh} such that $\xi_1\to a(h)$ as $\bt_h\to\pm\infty\cdot e^{i\theta}h^s$.
Assume first that $a(h)$ is simple.
Let $\nu_a=2\pi i\res_{\xi_1=a(h)}\bY_h^{-1}$ and let $\mon_a:\xi_1\mapsto a(h)+e^{2\pi i}(\xi_1-a(h))$ be the monodromy operator of analytic continuation along a simple positive loop around $a(h)$,
hence $\bt_h\circ\mon_a=\bt_h+\nu_a(h)$ (and also $\check\bt_h\circ\mon_a=\check\bt_h+\nu_a(h)$ since $a(h)\neq 0$). 
The map $\bT':=\bT_{\Omega_h}\circ\mon_a-\nu_a$ is a Fatou coordinate for $\phi_h^{\circ p}$ on the shifted domain $\mon_a(\Omega_h)$,
which has a nonempty intersection with $\Omega_h$ on the Riemann surface of $\bt_h$ (i.e. $\bt_h(\Omega_h)$ and $\bt_h(\Omega_h)+\nu_a$ intersect).
The map $t\mapsto\bT'\circ\bT_{\Omega_h}^{\circ(-1)}(t)$ a diffeomorphism that commutes with the translation $t\mapsto t+h^s$,
and such that $\lim_{t\to \pm\infty\cdot e^{i\theta}h^s}\bT'\circ\bT_{\Omega_h}^{\circ(-1)}(t)-t=0$ by (i).
Therefore  $\bT'\circ\bT_{\Omega_h}^{\circ(-1)}(t)-t$  is represented by a convergent Fourier series $\sum_{n\in\pm\Z_{> 0}}C_n e^{\frac{2\pi i nt}{h^s}}$ and
\[g(\bt_h):=\bT_{\Omega_h}\circ\mon_a-\nu_a-\bT_{\Omega_h}=\sum_{n\in\pm\Z_{>0}}C_n e^{\frac{2\pi i n\bT_{\Omega_h}}{h^s}}\]
defined on a neighborhood of the ray $\bt_h\in t_0\pm e^{i(\theta+s\arg h)}\R_{>0}$ is exponentially flat in $\bt_h$.
The Cauchy integral
\[G(t):=\tfrac{1}{2\pi i}\int_{t_0\pm e^{i(\theta+s\arg h)}\R_{>0}}\tfrac{g(r)}{t-r}dr, \]
satisfies $G(t+\nu_a)-G(t)=g(t)$.
Hence the function $F=\bT_{\Omega_h}-\bt_h-G\circ\bt_h$ is such that
$F\circ\mon_a=F$, i.e. it is univalent on a neighborhood of $\xi_1=a(h)$.
Moreover it is bounded and therefore
$\lim_{\bt_h\to\pm e^{i\theta}h^s\infty}\tdd{\bt_h}F=0$ since $\tdd{\bt_h}=\bY_h$ vanishes at $a_h$. Hence
\begin{align*} 
\lim_{\bt_h\to\pm e^{i\theta}h^s\infty}\tdd{\bt_h}(\bT_{\Omega_h}-\bt_h)
&=\lim_{\bt_h\to\pm e^{i\theta}h^s\infty}\tdd{t}G(t)\\
&=\lim_{t\to\pm e^{i\theta}h^s\infty}-\tfrac{1}{2\pi i t^2}\int_{t_0\pm e^{i(\theta+s\arg h)}\R_{>0}}\tfrac{g(r)}{(1-\frac{r}{t})^2}dr=0.
\end{align*}

The case of a multiple zero $a(h)$ is done similarly, except in this case several different domains $\Omega_h$ are necessary to cover sectorially the neighborhood of $a(h)$. The argument is essentially just a variation on Ramis--Sibuya theorem (e.g. \cite[Theorem 1.3.2.1]{Malgrange}).

\end{proof}

\begin{proposition}\label{prop:Fatou-dependence}
\begin{enumerate}[wide=0pt, leftmargin=\parindent]
	\item The Fatou coordinate $\bT_{\Omega_h}$ and the domain $\Omega_h$ are independent of a small variation of admissible $\theta$ and of the  point $\xi_*(h)$ (the two of them determining $\Sigma_h$, p.~\pageref{page:admissiblestrip}) as long as the strip $\bt_h(\Sigma_h)$ stays admissible, except for an addition of some constant $C_0(h)\in\C$ to $\bT_{\Omega_h}$.
	\item The Fatou coordinate $\bT_{\Omega_h}$ depends analytically/continuously on $h$ as long as the point $\xi_*(h)$ depends analytically/continuously on $h$ and the angle $\theta$ varies continuously in $h$.
	\item When $h\to 0$ with an asymptotic direction (i.e. so that $\arg(h)$ has a limit), and an admissible strip $\bt_h(\Sigma_h)$ has a limit, then the construction of the function $\bT_{\Omega_h}$ extends to the limit.
	
	In particular, in the the case $s>0$ when a limit $h\to 0$ of admissible strip is just a line, one has $\tdd{\bt_0}(\bT_{\Omega_0}-\bt_0)=\frac{1}{1+\Delta_0}-1$ with $\Delta_0(\xi_1)=\Delta(\xi_1,0)$ \eqref{eq:Delta}, 
	\Grn{i.e. 
		\[\bT_{\Omega_0}-\bt_0-\hat\mu(0)\log\xi_1=\int\Big(\tfrac{-U_0}{1+cP_0U_0}-\hat\mu(0)\Big)\tfrac{d\xi_1}{\xi_1}\quad\text{is analytic on $B_0$,}\] 
		where $U_0(\xi_1)=U(\xi_1,0)=-\hat\mu(0)+\hot(\xi_1)$ (Lemma~\ref{lemma:Delta}).}
	\end{enumerate}	
\end{proposition}

In order to prove this proposition, we need first some preliminary considerations.

The Fatou coordinate $\bT_{\Omega_h}$ conjugates $\phi_h^{\circ p}$ to translation by $h^s$ on the admissible strip $\bt_h(\Sigma_h)$, which is roughly of width $\sim\sin\theta\cdot|h^s|$.
When $s=0$, the construction is uniform when $h\to 0$. On the other hand, if $s>0$, then as $h\to 0$ the admissible strip  shrinks to the line $\bt_0(\Gamma_0)=e^{i(\theta+s\arg h)}\R$, and the translation by $h^s$ degenerates to identity, so one can no longer extend by iteration.
The remedy is to approach $h\to 0$ along a sequence $h_n=n^{-\frac1s}h_0$ and consider instead the iterates $\phi_{h_n}^{\circ np}$.
The Fatou coordinate for $\phi_{h_n}^{\circ np}$ is the same as the one for $\phi_{h_n}^{\circ p}$ (by uniqueness), except now it is defined on an admissible strip of non-vanishing width $\sim\sin\theta\cdot|h_0^s|$, and conjugates $\phi_{h_n}^{\circ np}$ to translation by $nh_n^s=h_0^s$.

\begin{lemma}\label{lemma:hn}
For any given $h_0$ let $h_n=n^{-\frac1s}h_0$ and define
\[\varphi_n(\xi_1):=\xi_1\circ\phi^{\circ np}(\xi_1,\tfrac{h_n}{\xi_1}),\qquad \xi_1\in B_{h_n}.\]
Then the sequence of diffeomorphisms $\varphi_n(\xi_1)$ converges locally uniformly on $B_0\sminus\{0\}$ to 
$\varphi_\infty(\xi_1)=\exp\big(h_0^s(1+\Delta_0)\bY_0\big)(\xi_1)$ where $\Delta_0(\xi_1)=\Delta(\xi_1,0)$ \eqref{eq:Delta}.
\end{lemma}

\begin{proof}
Denoting $\Delta_{h_n}(\xi_1)=\Delta(\xi_1,\tfrac{h_n}{\xi_1})$, we have by iteration of \eqref{eq:Delta}
\[\bt_{h_n}\circ\varphi_n(\xi_1)=\bt_{h_n}+h_n^s\sum_{j=0}^{n-1}\big(1+\Delta_{h_n}\!\!\circ\xi_1\circ\phi^{\circ jp}(\xi_1,\tfrac{h_n}{\xi_1})\big).\] 
Consider the vector field
$nh_n^s(1+\Delta_{h_n}\!\!\circ\bt_{h_n}^{\circ(-1)}(t))\tdd{t}$ on the translation surface $\bt_{h_n}(B_{h_n})$. The approximation of its flow at a time $\frac{j}{m}$ by the Euler method with step-size $\frac{1}{m}$ is
$y_{h_n,m}^j(t)$, where $y_{h_n,m}^0(t)=t$ and 
\[y_{h_n,m}^{j+1}(t)=y_{h_n,m}^{j}(t)+\tfrac{1}{m} nh_n^s\big(1+\Delta_{h_n}\!\!\circ\bt_{h_n}^{\circ(-1)}(y_{h_n,m}^{j}(t))\big),\quad j=0,\ldots,m-1.\]
In particular $\bt_{h_n}\circ\varphi_n=y_{h_n,n}^n\circ\bt_{h_n}$ is the time 1 approximation with step $\tfrac1n$.
By well-known results on locally uniform convergence of the Euler approximation to the actual solution (in our case the flow of $h_n^s(1+\Delta_{h_n}\!\!\circ\bt_{h_n}^{\circ(-1)}(t))\tdd{t}$ on the time interval $[0,1]$),
see e.g. \cite[Theorems~6.2.2~\&~4.5.2]{Hubbard-West}, and its uniform dependence on a parameter, we have
\begin{equation*}
	\begin{array}{rcl}
		y_{h_n,m}^m & \xrightarrow{m\to+\infty} & \exp\big(nh_n^s(1+\Delta_{h_n}\!\!\circ\bt_{h_n}^{\circ(-1)}(t))\tdd{t}\big) 	\\[6pt]
		{}^{n\to+\infty}\Big\downarrow  && \qquad\quad\Big\downarrow{}^{n\to+\infty} \\[6pt]
		y_{0,m}^m & \xrightarrow{m\to+\infty} & \exp\big(h_0^s(1+\Delta_{0}\!\circ\bt_{0}^{\circ(-1)}(t))\tdd{t}\big). 	\\
	\end{array}	
\end{equation*}
Therefore $\bt_{h_n}\circ\varphi_n=y_{h_n,n}^n\circ\bt_{h_n}\to \exp\big(h_0^s(1+\Delta_{0})\tdd{\bt_{0}}\big)$ as $n\to+\infty$.
\end{proof}

\begin{proof}[Proof of Proposition~\ref{prop:Fatou-dependence}]
\begin{enumerate}[wide=0pt, leftmargin=\parindent]
	\item Moving the point $\xi_*(h)$ and varying the admissible angle $\theta$ moves the strip $\bt_h(\Sigma_h)\subset \bt_h(B_h)$ in a continuous way. Every (partial) orbit of $\bt_h\circ \phi_h^{\circ p}\circ\bt_h^{\circ(-1)}$ in $\bt_h(B_h)$ that hits $\bt_h(\Sigma_h)$, will also hit the moved strip, which means that they both give rise to the same domain $\bt_h(\Omega_h)$.
	
	\item Treating $h$ as parameter, the coordinate $\bt_h$ such that $\bt_h(\xi_*(h))=0$, and the map $\omega_1$ depend 
	analytically/continuously on $h$, and so does the Beltrami constant $\tilde\bmu$. Then also the solution $\tilde\omega_2$ of the Beltrami equation depends analytically/continuously  on $h$, see e.g. \cite[Theorem~7.6]{Carleson-Gamelin}, and therefore so does $\bT_{\Omega_h}$ as well. (See \cite[Appendix]{Shi} for a more detailed argument).
		
	\item
		In the case $s=0$ this follows from the previous point. 
	In the case when $s>0$, we consider a sequence $h_n$ as in Lemma~\ref{lemma:hn} and the maps $\varphi_n$
	in place of $\phi_h^{\circ p}$. Then the admissible strip for $\varphi_{n}$ and the associated Beltrami function $\bmu$ \eqref{eq:Beltrami} have both well defined limits as $n\to\infty$. Therefore also $\bT_{\Omega_{0}}=\lim_{n\to+\infty}\bT_{\Omega_{h_n}}$ is the Fatou coordinate for 
	$\varphi_\infty=\exp\big(h_0^s(1+\Delta_{0})\tdd{\bt_{0}}\big)$ which is (up to a constant) necessarily equal to
	$\bT_{\Omega_{0}}=h_0^s\int \tfrac{1}{1+\Delta_{0}} d\bt_{0}$. 
	\end{enumerate}		
\end{proof}

\subsubsection{Normalizing transformation to the model}

Let $S$  be the maximal domain in the $h$-space over which one can choose an admissible strip $\bt_h(\Sigma_h)$ in a continuous fashion \Grn{(varying its position and angle $\theta$)} and thus
construct the domain $\Omega_h$ (more details about this in Section~\ref{sec:6.2}).
We will denote 
\[ \Omega=\coprod_{h\in S}\Omega_h, \]
a ramified domain in the $\xi$-space with ramification locus at a subset of $\{P(\xi)=0\}\cup\{h=0\}$.
Let $\bT_{\Omega}$ be a Fatou coordinate for $\phi^{\circ p}$ on $\Omega$ constructed in Propositions~\ref{prop:Fatou}~\&~\ref{prop:Fatou-dependence},
and let 
\begin{equation}\label{eq:alpha}
\alpha_{\Omega}(\xi):=\bT_{\Omega}(\xi)-\bt(\xi),\qquad \xi\in\Omega.
\end{equation}
Since $\bT_{\Omega}=\bt+\alpha_{\Omega}=\exp(t\tdd{\bt})\big|_{t=\alpha_{\Omega}(\xi)}$,
the map 
\begin{equation}\label{eq:LavaursPsi}
\Psi_{\Omega}(\xi)=\exp(t\bY)(\xi)\big|_{t=\alpha_{\Omega}(\xi)},
\end{equation}
is analytic on $\Omega$ and such that
\[ \bT_{\Omega}=\bt\circ\Psi_{\Omega},\qquad\text{and}\qquad
\Psi_{\Omega}\circ\phi^{\circ p}=\exp(h^s\bY)\circ\Psi_\Omega, \]
therefore it is a normalizing transformation for $\phi^{\circ p}$ which conjugates it to the model $\phimod^{\circ p}=\exp(h^s\bY)$. 
Let us stress that in general the transformation $\Psi_{\Omega}$ is multivalued in $\xi$, since the domain $\Omega$ is ramified.

\begin{theorem}\label{thm:normalizngtransformation}
\begin{enumerate}[wide=0pt, leftmargin=\parindent]
\item The normalizing transformation $\Psi_{\Omega}$ \eqref{eq:LavaursPsi} is analytic on $\Omega$, and 
	such that $\Psi_{\Omega}(\xi)-\xi=O(P\xi)$ is bounded and tends to identity along any complete real trajectory of $e^{i\theta}h^s\bY_h$ in $\Omega_h$.
\item If $\Psi_{\Omega}'$ is another analytic conjugating transformation on $\Omega$, $\Psi_{\Omega}'\circ\phi^{\circ p}=\phimod^{\circ p}\circ\Psi_{\Omega}'$, 
that is bounded and tends to identity along some complete real trajectory of $e^{i\theta}h^s\bY_h$ in $\Omega_h$, 
for each $h\in S=h(\Omega)$, then $\Psi_{\Omega}'=\exp(C_{\Omega}(h)\bY)\circ\Psi_{\Omega}$ for some analytic $C_{\Omega}(h)$ on $S$.
\end{enumerate} 
\end{theorem}

\begin{proof}[Proof of Theorem~\ref{thm:normalizngtransformation}]
1) Writing $\bT_{\Omega}=\big(\bt+\hat\mu(0)\tfrac12\log\tfrac{\xi_1}{\xi_2}\big)+\big(\alpha_\Omega(\xi)-\hat\mu(0)\tfrac12\log\tfrac{\xi_1}{\xi_2}\big)$, 
where $\bt+\hat\mu(0)\tfrac12\log\tfrac{\xi_1}{\xi_2}$ is the rectifying coordinate of the vector field $\frac{cP}{1+\hat{\mu}(0)cP}\bE$,
and where by Proposition~\ref{prop:Fatou-behavior} the function $\alpha_\Omega(\xi)-\hat\mu(0)\tfrac12\log\tfrac{\xi_1}{\xi_2}$ is bounded along any trajectory of $e^{i\theta}h^s\bY_h$.
We have 
\[\Psi_{\Omega}(\xi)=\exp\Big(t\tfrac{cP}{1+\hat{\mu}(0)cP}\bE\Big)(\xi)\Big|_{t=\alpha_\Omega(\xi)-\hat\mu(0)\tfrac12\log\tfrac{\xi_1}{\xi_2}},\]
and for every $R>0$ there exist $\delta_1>0$ such that
 $(t,\xi)\mapsto \exp\big(t\frac{cP}{1+\hat{\mu}(0)cP}\bE\big)(\xi)$ is analytic on the polydisc $|t|<R$, $|\xi|<\delta_1$.
Hence the result.

2) Since $\bT_{\Omega}':=\bt\circ\Psi_{\Omega}'$ is another Fatou coordinate	on $\bt(\Omega)$ with $\bT_{\Omega}'$ having a moderate growth 
when $\bt\to \infty$ along the trajectory on both ends, then by Proposition~\ref{prop:Fatou} $\bT_{\Omega}'-\bT_{\Omega}=C_\Omega(h)$.
\end{proof}

\subsubsection{``Sectorial'' holomorphic infinitesimal generator}
From Lemma~\ref{lemma:X} and \eqref{eq:form-exp} together with \eqref{eq:LavaursPsi}, we have that the vector field 
\begin{equation}\label{eq:LavaursX}
	\bX_{\Omega}:=h^s(\Psi_\Omega)^*\bY=\frac{h^s\bY}{1+\bY.\alpha_\Omega}=\frac{h^scP}{1+cP\bE.\alpha_{\Omega}},
\end{equation}
where $1+\bY.\alpha_\Omega=\dd{\bt}\bT_\Omega\neq 0$, is a bounded infinitesimal generator for $\phi^{\circ p}$ on $\Omega$, called
\emph{Lavaurs vector field},
\[ \phi^{\circ p}=\exp\big(\bX_{\Omega_S}\big), \]
and
\[ h^s\bX_{\Omega}^{-1}-\bY^{-1}=d\alpha_{\Omega}\mod\tfrac{dh}{h}. \]
Since $\bT_\Omega$ is unique up to addition of some $C(h)$, $h\in S$, the Lavaurs vector field $h^s\bX_\Omega$ \eqref{eq:LavaursX} on $\Omega$ is uniquely defined.

We recall that a germ of of holomorphic diffeomorphism tangent to identity at the origin, its fixed point, such as our $\phi^{\circ p}$, has a formal infinitesimal generator $\hat\bX$ (see Section~\ref{sec:3infgen}), albeit it may not have any  infinitesimal generator holomorphic at the origin. The Lavaurs vector field defined above serves as \emph{holomorphic ``sectorial'' infinitesimal generator}. It is defined uniquely on a ``sectorial'' Lavaurs domain 
and is asymptotic to the formal infinitesimal generator as next proposition shows.
As such, it can be seen as a ``sectorial'' realization of the formal infinitesimal generator $\hat\bX$ of $\phi^{\circ p}$.
This is analogical to the 1-dimensional situation, where the formal infinitesimal generator of a parabolic diffeomorphism of $(\C,0)$ is Borel summable on sectors (petals)
\cite{Ecalle,Ecalle2,Voronin,Malg-diffeo}.

\begin{proposition}\label{prop:asymptotic}
Let $\hat\bX=h^s\frac{cP(u,h)}{1+cP(u,h)\hat R(\xi)}\bE$ be the formal infinitesimal generator of $\phi^{\circ p}$, with $\hat R(\xi)=\sum r_{\bm m}\xi^{\bm m}$, and let $\bX_\Omega$ \eqref{eq:LavaursX} be the Lavaurs vector field. 
Let $\tilde{\Omega}\subseteq\Omega$ be a subdomain on which $\bX_\Omega$ is uniformly bounded, and assume that $0\in\partial{\tilde\Omega}$. 
Then
$\bX_\Omega$ is uniformly asymptotic to $\hat\bX$ on $\tilde{\Omega}$, i.e. the function $\bE.\alpha_\Omega(\xi)$ has a formal asymptotic expansion $\hat R(\xi)$ uniformly when $\tilde\Omega\ni\xi\to 0$.
\end{proposition}

\begin{proof}	
For $n\in\Z_{>0}$, let $j^{(n)}\hat\bX$, resp. $j^{(n)}\phi^{\circ p}$, denote the $n$-jet of the formal vector field $\hat\bX$, resp. of the analytic diffeomorphism $\phi^{\circ p}$, with respect to the variable $\xi$.
Then formally,
\begin{equation}\label{eq:jX}
	 j^{(n)}\phi^{\circ p}=j^{(n)}\exp\big(j^{(n)}\hat\bX\big)
\end{equation}
and since both sides are analytic, and since $\phi^{\circ p}=\exp(\bX_{\Omega})(\xi)$ for $\xi\in\tilde\Omega$, this means that
\[ \exp(\bX_{\Omega})(\xi)-\exp(j^{(n)}\hat\bX)(\xi)\in\xi\Cal J_{\tilde\Omega}^n,\]
where $\Cal J_{\tilde\Omega}$ denotes the ideal $\xi_1\Cal B_{\tilde\Omega}+\xi_2\Cal B_{\tilde\Omega}$ of the ring $\Cal B_{\tilde\Omega}$ of bounded analytic functions on $\tilde\Omega$.
We want to conclude that
\[ \bZ_{\Omega,n}:=\bX_{\Omega}(\xi)-j^{(n)}\hat\bX(\xi)\in\Cal J_{\tilde\Omega}^n\cdot\bE.\]

Denote $F_t:=\exp(t\bX_{\Omega})\circ\exp(-tj^{(n)}\hat\bX)$, then 
\begin{align*}
	\tdd{t}F_t\big|_{\xi=F_t^{\circ(-1)}}&=\bX_{\Omega}.\xi-j^{(n)}\hat\bX.\exp(t\bX_{\Omega})\big|_{\xi=\exp(-t\bX_{\Omega})}\\
	&=\bX_{\Omega}.\xi - \exp(-t\bX_{\Omega})^*\big(j^{(n)}\hat\bX\big).\xi\\
	&=\exp(-t\bX_{\Omega})^*\bZ_{\Omega,n}.\xi,
\end{align*}
hence
\[F_{r}=\xi\circ\exp\big(r\tdd{t}+r\exp(-t\bX_{\Omega})^*\bZ_{\Omega,n}\big)\big|_{t=0}\]
is the time-$r$-flow of the non-autonomous vector field $\exp(-t\bX_{\Omega})^*\bZ_{\Omega,n}$.

We know that  $\bZ_{\Omega,n}$ is a bounded multiple of $h^sP\bE$, and in particular it belongs to $\Cal J_{\tilde\Omega}\cdot\bE$.
From \eqref{eq:jX} we also know that $F_1(\xi)=\xi\mod \xi\Cal J_{\tilde\Omega}^n$,
and since
\[F_1= \sum_{m=0}\tfrac{1}{m!}\big(\tdd{t}+\exp(-t\bX_{\Omega})^*\bZ_{\Omega,n}\big)^m.\xi\,\Big|_{t=0}=\xi+\bZ_{\Omega,n}.\xi+\ldots,  \]
is a convergent sum, in which if $\bZ_{\Omega,n}\in\Cal J_{\tilde\Omega}^l\cdot\bE$ for some $l\geq 1$ then the terms ``$\ldots$'' belong to $\Cal J_{\tilde\Omega}^{l+1}\cdot\xi$, this means that $\bZ_{\Omega,n}\in\Cal J_{\tilde\Omega}^n\cdot\bE$.
\end{proof}

\subsection{Dynamics of the vector field $\bY$ and the Lavaurs domains}\label{sec:6.2}

The goal of this section is to:
	\begin{itemize}[wide=0pt, leftmargin=\parindent]
	\item[-] show that a stable $\theta$ and an admissible strip $\bt_h(\Sigma_h)$ indeed exist for all $|h|<\delta_2$ for some $\delta_2>0$, and that the collection of Lavaurs domains $\Omega$ covers all $B$ (Theorem~\ref{thm:covering}). 
	\item[-] understand the organization of these domains $\Omega_h$ in $B_h$, and how they depend on $h$ and $\theta$ and the position of the strip 
	$\bt_h(\Sigma_h)$ in $\bt_h(B_h)$.
\end{itemize}
The modulus of analytic classification will be described afterwards in Section~\ref{sec:6.3}.

We will replace the saturated Lavaurs domains $\Omega_h$ of Section~\ref{sec:6.1} by slightly smaller domains,
called simply ``Lavaurs domains'' the form of which will depend only on the dynamics of the model vector field $h^s\bY$, and on the constants $\delta_1,\delta_2,\delta_3$ 
(defining size's constraints on $|\xi_1|$, $|h|$ and on the variation of the angle $\theta$)
but not on $\phi^{\circ p}$. The basic idea would be to construct these new domains so that their $\bt_h$-image would be spanned by all the lines $t_0+e^{i\theta}h^s\R$ of varying angle $\theta$ in  $\bt_h(\Omega_h)\subset\bt_h(B_h)$, see Figure~\ref{figure:lavaursdomain}. 
The preimage of each such line $t_0+e^{i\theta}h^s\R$ is nothing else then a complete real-time trajectory of the vector field $e^{i\theta}h^s\bY_h$ in $\Omega_h\subset B_h$, \Grn{which is why we need to have some understanding of their organization.
The exact construction of the Lavaurs domains in Section~\ref{sec:6.2omega} will be slightly more technical in order to ensure they cover $B_h$ in a uniform way, but the rough idea stays the same.}

So let $\bY=cP\bE$ be the rational vector field \eqref{eq:bY}, and $\bY_h=cP_h\bE$ \eqref{eq:bYh} its restriction to a leaf $\{h=\const\}$.
In the coordinate $\xi_1$ on a leaf with $h\neq 0$, the vector field $\bY_h$ is rational in $\xi_1\in\CP^1$, with poles at $\xi_1=0,\infty$,
and with coefficients $P_j(h)$ depending analytically on $h$, $|h|<\delta_2$.
Up to restricting $\delta_1^2\gg\delta_2>0$, we can assume that all the $2kp$ zeros of $\bY_h$ (counted with multiplicity) lie inside the annulus
\[ B_h=\Big\{\frac{|h|}{\delta_1}<|\xi_1|<\delta_1\Big\} \qquad\text{for}\qquad |h|<\delta_2, \]
while the poles of $\bY_h$ lie outside of the annulus, symmetrically in the outer and inner components of its complement.
The limit vector field 
\begin{equation}\label{eq:bY_0}
\bY_0=pc\,\xi_1^{(k+1)p}\tdd{\xi_1^p},
\end{equation}
on the irreducible component $\{\xi_2=0,\ 0\leq|\xi_1|<\delta_1\}$ of $B_0$,
is obtained by the merging of $2kp$ zero points (counted with multiplicity) with the pole of order $kp-1$ at the origin into a zero point of order $kp+1$.
The situation on the other irreducible component $\{\xi_1=0,\ 0\leq|\xi_2|<\delta_1\}$ of $B_0$ is symmetric by means of $\sigma$.

We want to understand the topological organization of the real trajectories in the rotating family of vector fields
\begin{equation}\label{eq:rotated-vf}
e^{i\theta}h^s\bY_h,\qquad\theta\in\R,
\end{equation}
in dependence on the parameter $h$ and the angle $\theta$.
We will first consider it globally as a family of rational vector fields on $\CP^1$, and later look at their restriction to $B_h$.

Let us recall some basis properties of the real dynamics of rational vector fields on $\CP^1$
(see e.g. \cite{Benziger,Brickman-Thomas,GGJ,Hajek, Klimes-Rousseau2, Strebel, Tomasini, Tahar}).

\begin{figure}[t]
\centering
\begin{subfigure}{.4\textwidth} \includegraphics{tikz-endsa.pdf} \caption{$h=0$} \label{figure:endsa} \end{subfigure}
\qquad
\begin{subfigure}{.4\textwidth} \includegraphics{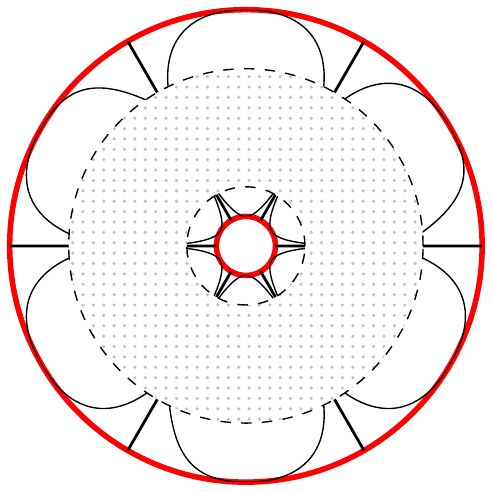} \caption{$h\neq 0$} \label{figure:endsb} \end{subfigure}	
\caption{(a) The real phase portrait of $\bY_0$ on the irreducible component $\{\xi_2=0\}$ of $B_0$, showing the $2kp$ ``outer'' sepal zones (here $kp=3$).
		(b) The real phase portrait of $\bY_h$ near the boundary of $B_h$, showing the ends of $2kp$ ``outer half-zones'' (see Section~\ref{sec:6.2halfzones}) and $2kp$ ``inner half-zones'' (here $kp=3$). The middle part (dotted) is where all the equilibria are situated and where the global organization of the phase portrait is determined.	}
\label{figure:ends}
\end{figure}

\subsubsection{Critical points.}\label{sec:6.2.1}
Let $a(h)$ be a zero (an \emph{equilibrium} point) of $\bY_h$, and denote
\begin{equation}\label{eq:nu_a}
\nu_a(h)=2\pi i\,\res_{\xi_1=a(h)}\bY_h^{-1},
\end{equation}
the \emph{period} (also called \emph{dynamical residue}) of $\bt_h$ \eqref{eq:bt} around it.

For each $h,\theta$, the real-time-flow curves of  the vector field $e^{i\theta}h^s\bY_h$ define a real-analytic singular foliation on $\CP^1$, also known as the 
\emph{real phase portrait} of $e^{i\theta}h^s\bY_h$, or the \emph{horizontal foliation} of the meromorphic differential $e^{-i\theta}h^{-s}\bY_h^{-1}$. 

Near a \emph{simple equilibrium} the vector field $e^{i\theta}h^s\bY_h$ is locally biholomorphically equivalent to 
$e^{i\theta}h^s\frac{2\pi i}{\nu_a}z\tdd{z}$. Its real dynamics is
\begin{itemize}
	\item[-] \emph{attractive} if $\IM(e^{-i\theta}h^{-s}\nu_a)<0$,
	\item[-] \emph{repulsive} if $\IM(e^{-i\theta}h^{-s}\nu_a)>0$,
	\item[-] \emph{center} if $e^{-i\theta}h^{-s}\nu_a\in\R$, \
$		\begin{cases}
		\textit{odd (counter-clockwise)}:&  e^{-i\theta}h^{-s}\nu_a>0,\\
		\textit{even (clockwise)}:& e^{-i\theta}h^{-s}\nu_a<0.
	\end{cases}$
\end{itemize}	

Near a \emph{multiple equilibrium} (parabolic point) of a multiplicity $m+1$ the vector field is locally biholomorphically equivalent to 
$e^{i\theta}h^s\frac{z^{m+1}}{1+\frac{\nu_a}{2\pi i}z^m}\tdd{z}$ (cf. e.g. \cite[Theorem 5.25]{IlYa}). Its real phase portrait exhibits $2m$ sepal zones, 
consisting of asymptotically closed trajectories (i.e. those whose both positive and negative time limit is the equilibrium),
separated alternatingly by $m$ 
\emph{attractive} and $m$ \emph{repulsive directions} (such as in Figure~\ref{figure:endsa}).

\begin{figure}[t]
	\centering
	\includegraphics[width=0.3\textwidth]{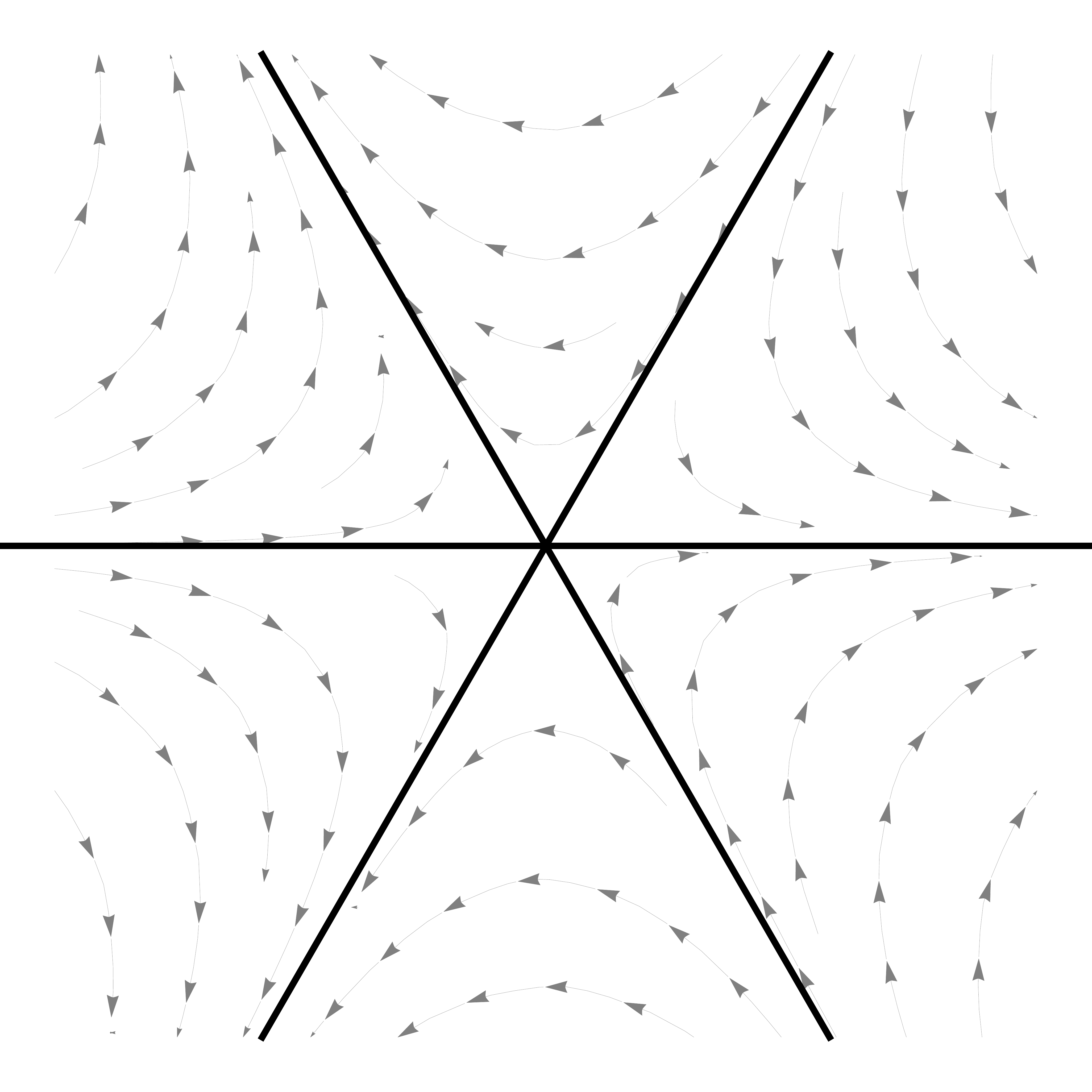}
	\caption{The real phase portrait near a pole of order $kp-1$ with $2kp$ ends (here $kp=3$).}
	\label{figure:pole}
\end{figure}

The vector field $e^{i\theta}h^s\bY_h$ \eqref{eq:bYh} has two poles for $h\neq 0$: an \emph{outer pole} at $\xi_1=\infty$ and an \emph{inner pole} at $\xi_1=0$,
while for $h=0$ there is one pole on each irreducible component of $B_0$.
Each pole is of order $kp-1$, and so the vector field has $2kp$ \emph{separatrices}\footnote{Separatrices are trajectories having a pole in its closure and passing through it in a finite time. The local model for a pole of order $n-1$ is $\frac{1}{z^{n-1}}\frac{\partial}{\partial z}$ for $n\in \mathbb{N}^*$. It has $2n$ separatrices at $0$, namely $z_{\ell}(t):=e^{\frac{i\pi\ell}{n}}(t-t_0)^{1/n}$, $t-t_0\in\mathbb{R}_+^*$, $\ell=0,\ldots, 2n-1$.} emanating from the pole, alternatingly \emph{incoming} and \emph{outgoing},
separating $2kp$ sectors, called \emph{ends}, on which the real flow is of hyperbolic form (Figure~\ref{figure:pole}). \label{page:ends}  
The ends are either
\[ \begin{cases}
	\textit{odd end}& \hskip-6pt \Leftrightarrow\ \text{the flow at the corner is ``counter-clockwise''},\\
	\textit{even end}& \hskip-6pt \Leftrightarrow\ \text{the flow at the corner is ``clockwise''.}
\end{cases} \]
A separatrix can be either landing at an equilibrium point, or it coincides with another separatrix to form a 
\emph{homoclinic} or \emph{heteroclinic} connection (depending on whether the trajectory goes back to the same pole or to a different one).
The time $\bt_h$ \eqref{eq:bt} it takes to arrive at the pole from a regular point is always finite.

The following well known result concerns real flow of rational vector fields on $\mathbb{CP}^1$:

\begin{proposition}
\begin{enumerate}[wide=0pt, leftmargin=\parindent]
\item A rational vector field on $\CP^1$ has no limit cycles: all non-periodic trajectories 
are landing in the positive/negative time limit at an equilibrium point, or they reach a pole as a separatrix. 
Every periodic trajectory belongs to a maximal open domain consisting of periodic trajectories (all of the same period), the boundary of which is formed by a union of homoclinic or heteroclinic separatrices. (See \cite[Theorem~3]{Hajek}.)
\item Each attractive/repulsive simple equilibrium is the landing point of at least one separatrix, and each direction separating sepals of a parabolic equilibrium is tangent to at least one landing separatrix.	
\item Counted with multiplicity, the number of zeros minus the number of poles surrounded by a periodic orbit is equal $1$. (This is just a Poincar\'e--Hopf index theorem.)
\end{enumerate}
\end{proposition}

\begin{figure}[t]
\begin{subfigure}[t]{.4\textwidth} \includegraphics[width=0.45\textwidth]{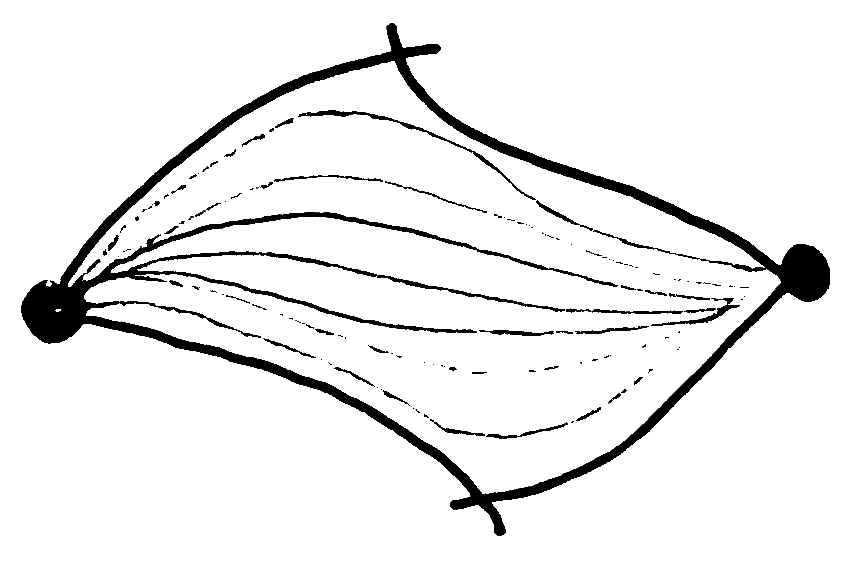}\includegraphics[width=0.45\textwidth]{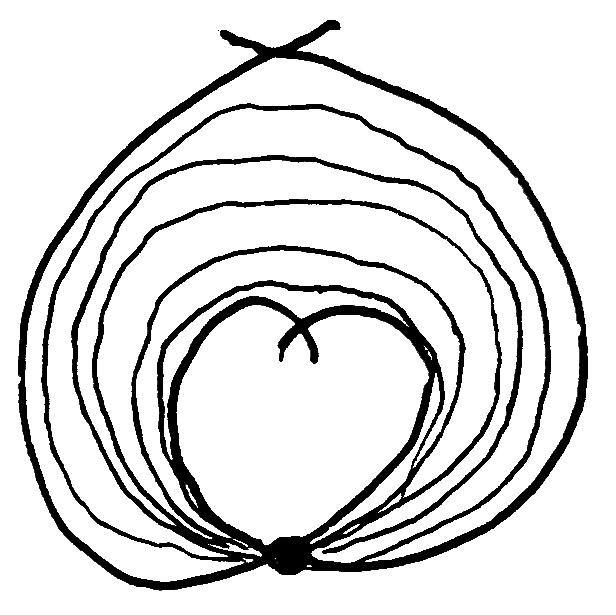}	\caption{$\alpha\omega$-zone} \label{figure:zonex-a} \end{subfigure}	
\begin{subfigure}[t]{.16\textwidth} \includegraphics[width=\textwidth]{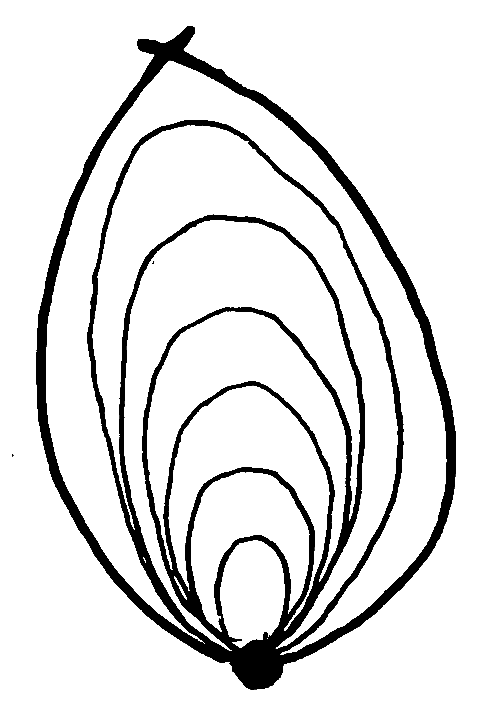}	\caption{Sepal zone} \label{figure:zonex-b} \end{subfigure}
\ 	
\begin{subfigure}[t]{.18\textwidth} \includegraphics[width=\textwidth]{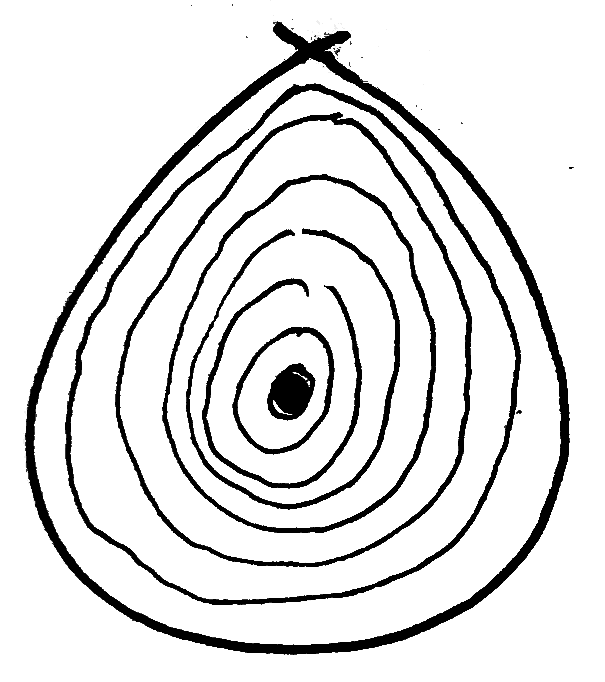}	\caption{Center zone} \label{figure:zonex-c} \end{subfigure}
\ 
\begin{subfigure}[t]{.2\textwidth} \includegraphics[width=\textwidth]{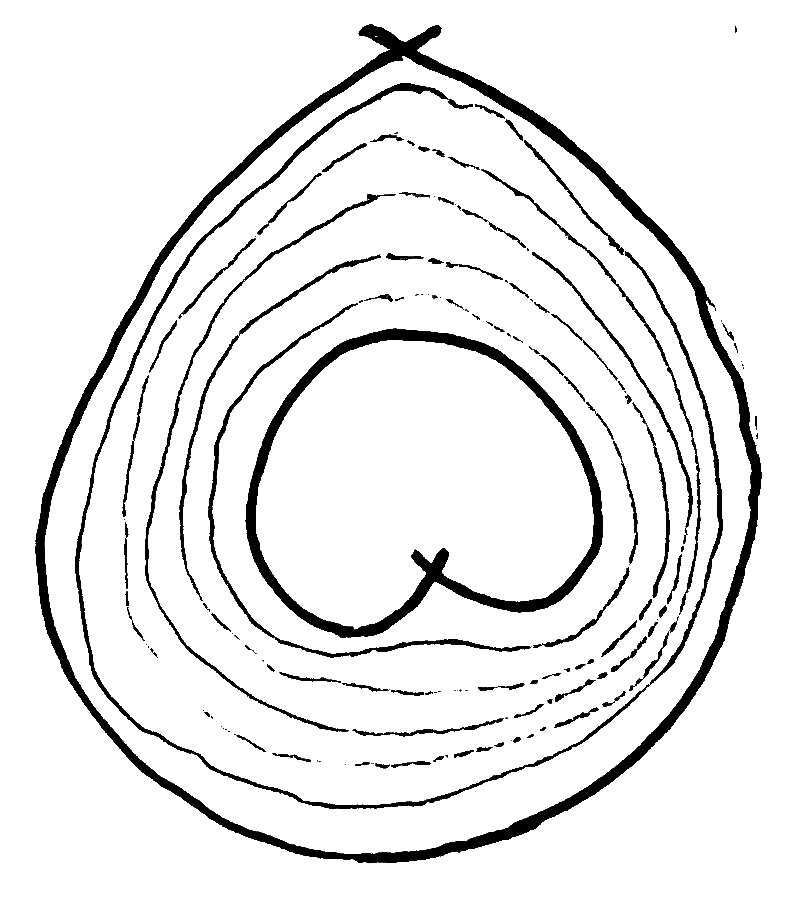}	\caption{Annular zone} \label{figure:zonex-d} \end{subfigure}
\caption{The different kind of zones in the $\xi_1$-coordinate. }
\label{figure:zonex}
\end{figure}

\subsubsection{Zone decomposition.}
The \emph{separatrix graph} is defined as the closure in $\CP^1$ of the union of all separatrices of all poles.
The connected components in $\CP^1$ of the complement of the separatrix graph are called \emph{zones of $e^{i\theta}h^s\bY_h$ in $\CP^1$}.
In each zone all the trajectories are homotopic by a homotopy fixing the equilibria. 
There are four types of zones that can occur for a rational vector field according to their form in the $\bt_h$-coordinate (Figure~\ref{figure:zone}):
\begin{enumerate}
	\item \emph{$\alpha\omega$-zone:} All trajectories share the same $\alpha$-limit and the same $\omega$-limit  which are either two different equilibria (simple or multiple) or the same multiple equilibrium, and the closure of the zone is not contractible by a homotopy fixing the equilibria (Figure~\ref{figure:zonex-a}).
 	In the $\bt_h$-coordinate it corresponds to an open infinite strip parallel to $e^{i\theta}h^s\R$ (Figure~\ref{figure:zone-a}). 
	
	\item \emph{Odd/even sepal zone:} All trajectories share the same $\alpha$-limit and the $\omega$-limit which are the same multiple equilibrium,
	and the closure of the zone is contractible by a homotopy fixing the equilibrium (Figure~\ref{figure:zonex-b}).
 	In the $\bt_h$-coordinate it corresponds to an open half-plane $e^{i\theta}h^s\H^{\pm}$ (Figure~\ref{figure:zone-b}). 
		
	\item \emph{Odd/even center zone:} It consists of periodic counter-clockwise/clockwise periodic orbits of the same period 
	$\pm e^{-i\theta}h^{-s}\nu_a>0$ around a center equilibrium point $a$ -- the zone has the form of a pierced disc  (Figure~\ref{figure:zonex-c}).
 	In the $\bt_h$-coordinate it corresponds to a quotient of a half-plane $e^{i\theta}h^s\H^\pm$ by $\nu_a\Z$ (Figure~\ref{figure:zone-c}). 
	
	\item \emph{Annular zone} (\emph{periodic annulus}:) It consists of periodic orbits of the same period $\pm e^{-i\theta}h^{-s}\nu>0$ 
	(where $\nu$ is equal to the sum of the periods of the equilibria encircled by the trajectory) -- the zone has the form of an annulus (Figure~\ref{figure:zonex-d}).
 	In the $\bt_h$-coordinate it corresponds to a quotient of an open infinite strip parallel to $\nu\R$ by $\nu\Z$ (Figure~\ref{figure:zone-d}). 
\end{enumerate}

\begin{figure}[t]
\centering	
\begin{subfigure}[t]{.2\textwidth}  \includegraphics{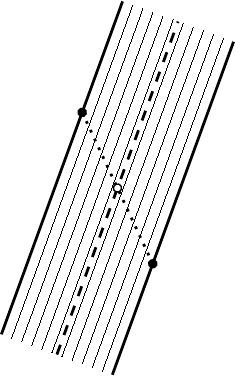} \caption{$\alpha\omega$-zone} \label{figure:zone-a}	\end{subfigure}	
\quad
\begin{subfigure}[t]{.2\textwidth}  \includegraphics{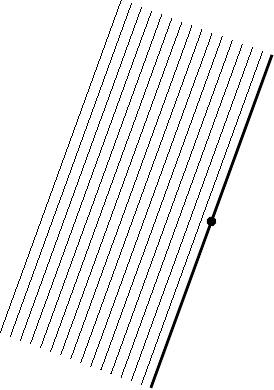}	\caption{sepal zone (odd)} \label{figure:zone-b} \end{subfigure}	
\quad
\begin{subfigure}[t]{.2\textwidth}  \includegraphics{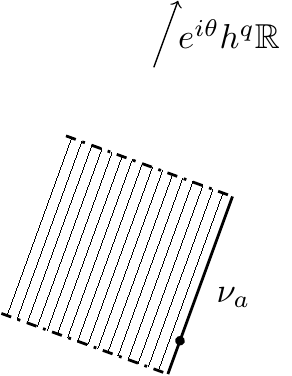} \caption{center zone (odd)} \label{figure:zone-c} \end{subfigure}	
\quad
\begin{subfigure}[t]{.2\textwidth}  \includegraphics{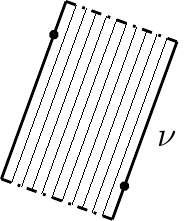} \caption{annular zone} \label{figure:zone-d} \end{subfigure}
\caption{The different types of zones in the $\bt_h$-coordinate. Generically there is only one end of a pole on each boundary line (as in the picture), but non-generically can be more. In (a): the gate trajectory (dashed) of the $\alpha\omega$-zone passes through the center-point (white point) of the transversal (dotted) joining the two ends (black points). In (c) and (d): the dash-dotted lines are glued together forming a cylinder. In (d): there are $\Z$-many saddle connections/transversals on the cylinder between the two ends (not depicted).}
\label{figure:zone}
\end{figure}

\begin{figure}[t]
	\centering	
	\begin{subfigure}[t]{.4\textwidth}  \includegraphics[width=\textwidth]{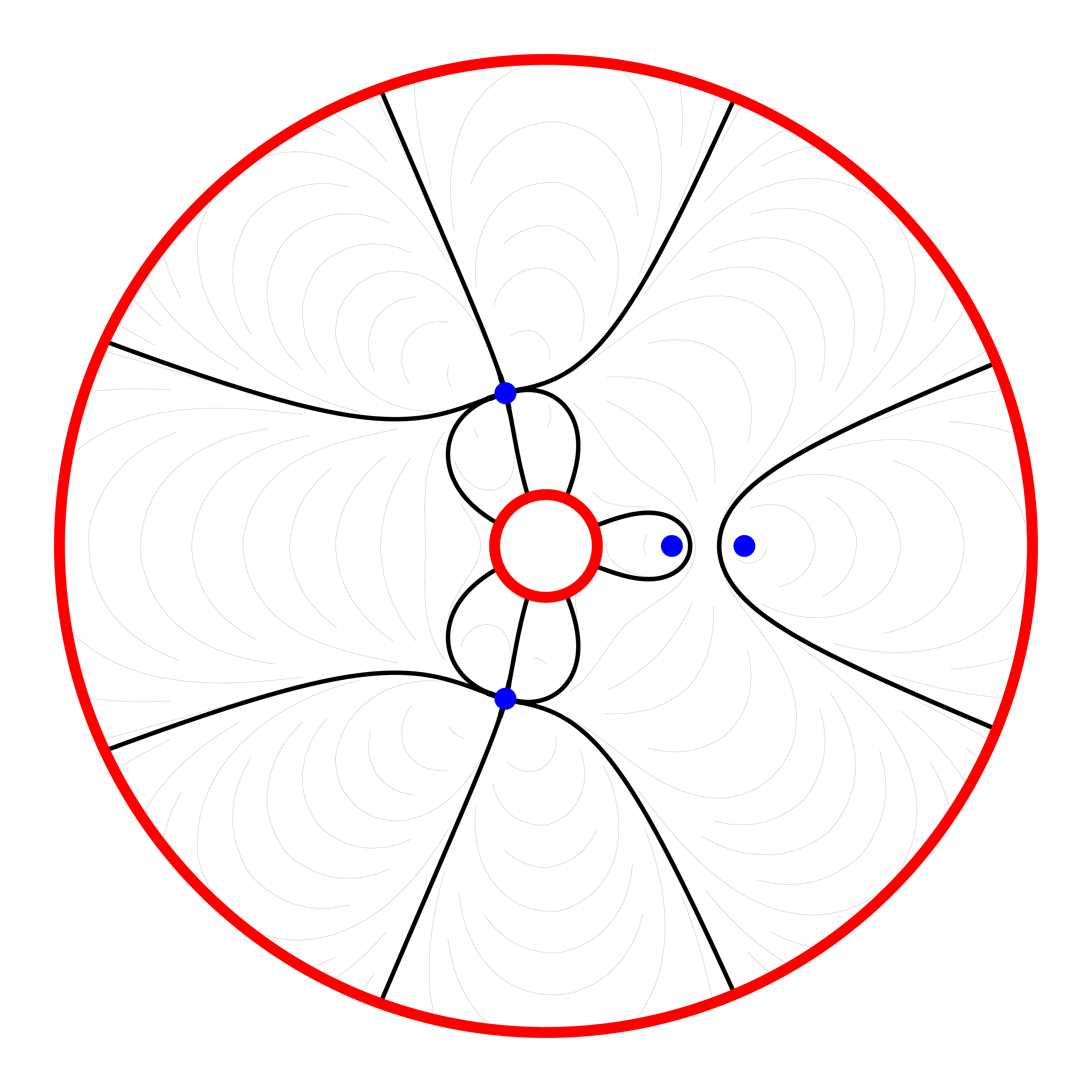} \caption{} \label{figure:vf1}	\end{subfigure}	
	\qquad
	\begin{subfigure}[t]{.4\textwidth}  \includegraphics[width=\textwidth]{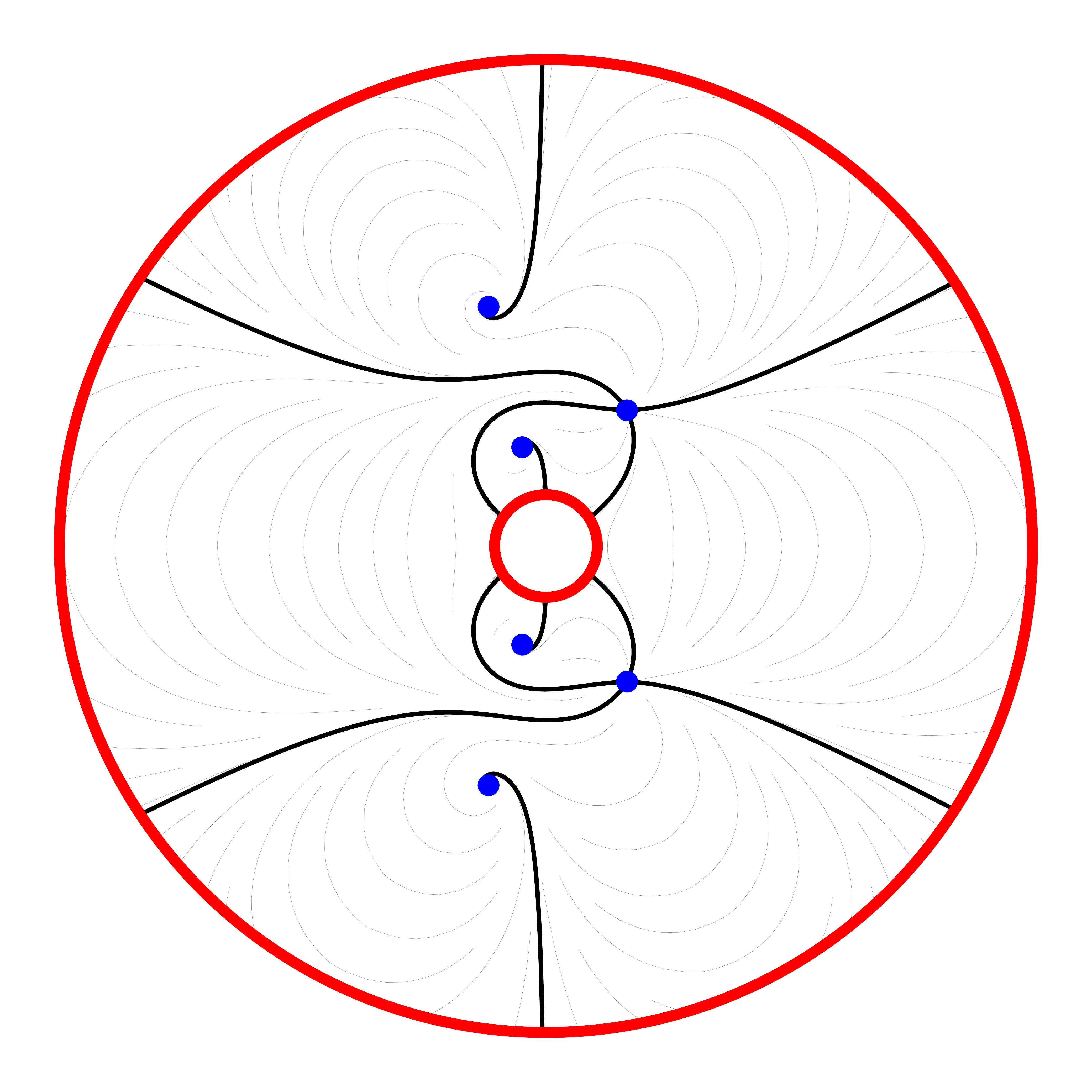}	\caption{} \label{figure:vf2} \end{subfigure}	
	\caption{Examples of the zone decomposition of a real phase portrait (separatrices in black). 
		\newline (a) \protect{$\bY_h=i\big(\xi_1+\frac{h}{\xi_1}-2\big)\big(\xi_1+\frac{h}{\xi_1}+\frac12\big)^3\xi_1\frac{\partial}{\partial\xi_1}$ with $h=0.95$: the phase portrait exhibits 2 center zones, 8 sepal zones and 2 $\alpha\omega$-zones.}
		\newline (b) \protect{$\bY_h=i\big(\big(\xi_1+\frac{h}{\xi_1}\big)^3-1\big)\xi_1\frac{\partial}{\partial\xi_1}$ with  $h=0.95$: all 6 zones are $\alpha\omega$-zones.}}
	\label{figure:vf}
\end{figure}

\begin{lemma}
The images by $\sigma:\xi_1\mapsto\frac{h}{\xi_1}$ and by 
$\begin{cases}\textit{(b)}\ \Lambda:\xi_1\mapsto\lambda\xi_1,\\
\textit{(c)}\ \sigma\Lambda=-I:\xi_1\mapsto-\xi_1,
\end{cases}$ of a separatrix of $e^{i\theta}h^s\bY_h$ are again separatrices of $e^{i\theta}h^s\bY_h$ (up to an orientation). Hence $\sigma$ and $\Lambda$ map zones to zones.	
\end{lemma}
\begin{proof}
	Trivial.
\end{proof}

\begin{remark}
	While the zone decomposition of $e^{i\theta}h^s\bY_h$ is preserved by the action of $\tau_1=\sigma$, in general this is not true for $\tau_2=\sigma\Lambda\exp(h^s\bY)$ unless $\theta\in\pi\Z$. 
\end{remark}

\subsubsection{The rotating family and rotational stability.}

\begin{definition}\label{def:saddleconnections}
All the possible homoclinic or heteroclinic connections that appear in the rotating family  $e^{i\theta}h^s\bY_h$ for some $\theta\in\R$ are called \emph{saddle connections}.	
They correspond to oriented straight segments (geodesic segments) on the translation surface of $\bt_h$ between the saddle points corresponding to the poles.	
\end{definition}

The \emph{period map} is given by integration of the form $d\bt_h=\bY_h^{-1}$ along paths between saddles.
Namely, to each saddle connection $\gamma$, one associates its \emph{period} along $\gamma$:\label{page:period}
\[ \nu_\gamma=\int_\gamma \bY_h^{-1}\neq 0. \]
The period along the symmetric saddle connection $\sigma(\gamma)$ is the same, $\nu_{\sigma(\gamma)}=\nu_\gamma$.
(On the other hand, if $\xi_1=a$ is an equilibrium, then the dynamical residue \eqref{eq:nu_a} satisfies  $\nu_{\sigma(a)}=-\nu_a$ since $\sigma$ reverses orientation.)

\begin{definition}\label{def:rotstable}
	The vector field $e^{i\theta}h^s\bY_h$ is called \emph{rotationally stable} if it contains no homoclinic or heteroclinic separatrix connections. We call such pair of values $(h,\theta)$ \emph{stable}.
\end{definition}

The idea is that homoclinic/heteroclinic connections disappear under a small perturbation of the angle of rotation $\theta$ in \eqref{eq:rotated-vf} while landing separatrices are stable. This means that the topological organization of the phase portrait of a rotationally stable vector field doesn't change under a small change of $\theta$.

\begin{proposition}[Muciño-Raymundo, Valero \cite{Mucino-Valero}]\label{prop:rotstable}
The set of $\theta\in[0,\pi]$ for which $e^{i\theta}h^s\bY_h$ (with fixed $h$) is \emph{not} rotationally stable 
is at most countable with only finitely many accumulation points at exactly those $\theta$ for which $e^{i\theta}h^s\bY_h$
has an annular zone.

In fact, for each annular zone that appears in the rotating family for some angle $\theta$ and each pair of ends on opposite boundaries of the zone, there are $\Z$-many saddle connections between these ends inside the zone, and there is only a finite number\footnote{In our case at most $kp(2kp-1)$ counted without orientation.} of other saddle connections. 
\end{proposition}

\Grn{
In our case, there can be no more then one annular zone that appears in the rotating family (and therefore it is necessarily invariant by $\Lambda$ and $\sigma$):
There needs to be at least one pole on each boundary component of the zone and we have only two poles, $0$ and $\infty$,
hence if two different zones appeared in the rotating family for two different angles $\theta_1\neq\theta_2\mod\pi\Z$, they would have to intersect.
However this is not possible since each trajectory in the first annular zone is closed (periodic), dividing $\CP^1$ into two parts,
and all the trajectories of the other rotated field that cross it do so with the same angle $\theta_1-\theta_2$, so none of them can be closed.}
Therefore, in the Proposition~\ref{prop:rotstable} there is at most one value of $\theta$ modulo $\pi\Z$ to which the rotationally unstable angles can accumulate.

\begin{figure}[t]
\begin{subfigure}{.3\textwidth} \includegraphics[width=\textwidth]{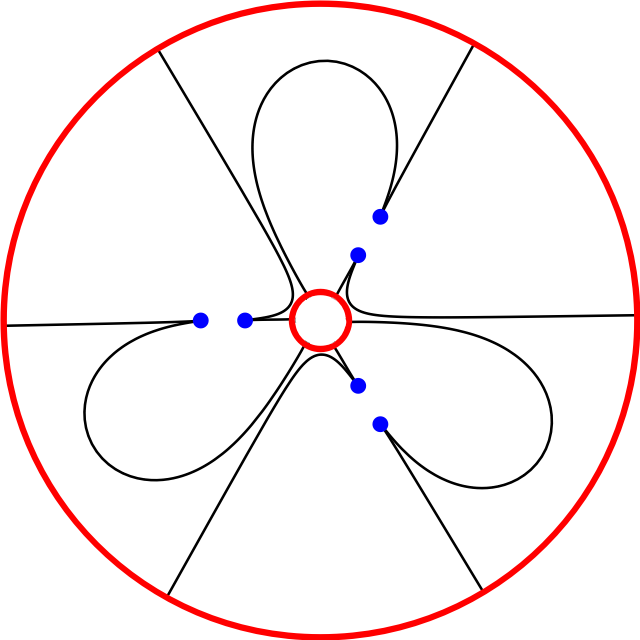} \caption{} \label{figure:vf3-a}\end{subfigure}
\quad
\begin{subfigure}{.3\textwidth} \includegraphics[width=\textwidth]{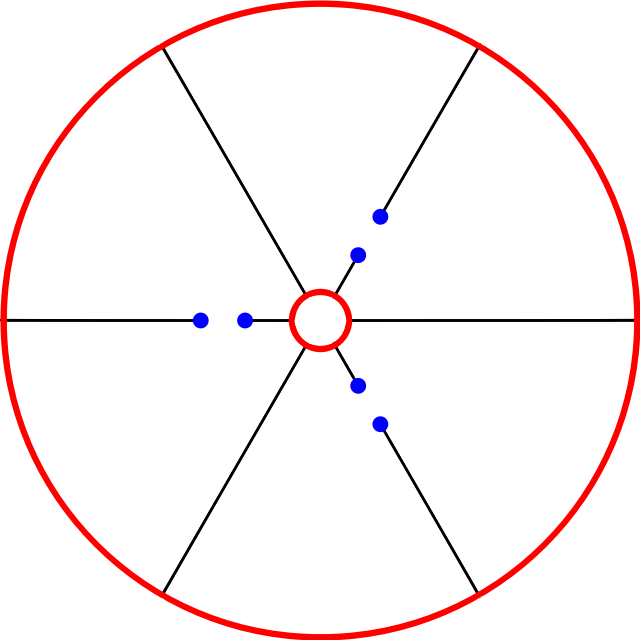} \caption{$\theta=\frac{\pi}{2}$} \end{subfigure}
\quad
\begin{subfigure}{.3\textwidth} \includegraphics[width=\textwidth]{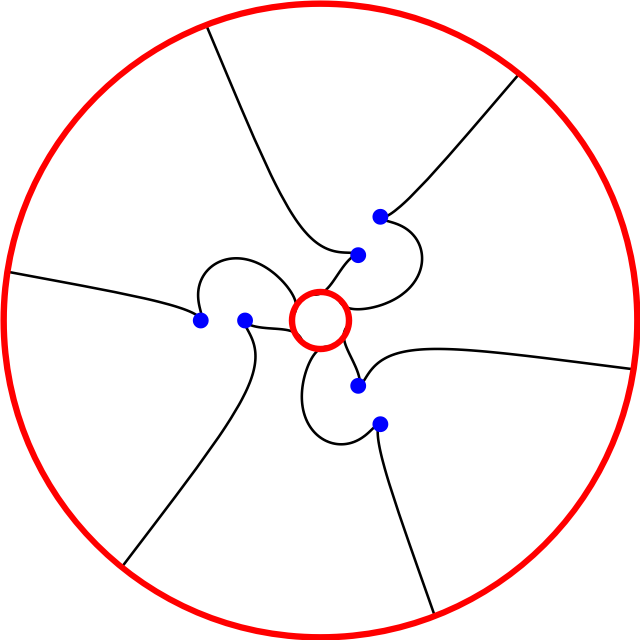} \caption{} \label{figure:vf3-c} \end{subfigure}
\\[12pt]
\begin{subfigure}{.3\textwidth} \includegraphics[width=\textwidth]{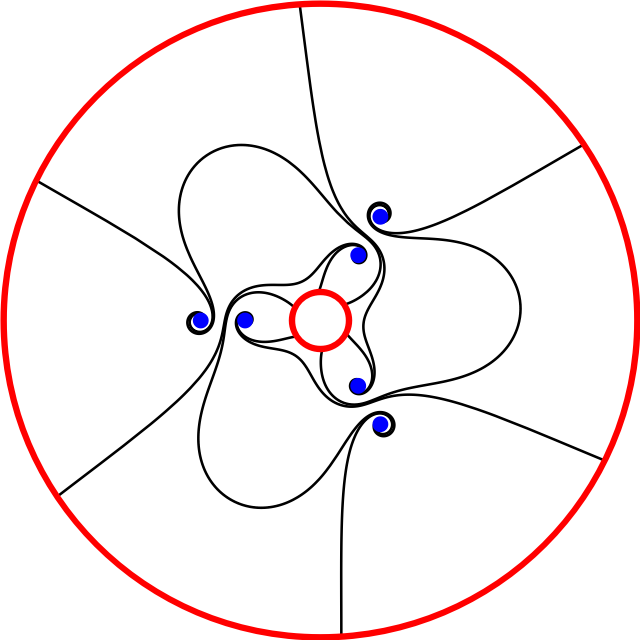} \caption{} \end{subfigure}
\quad		
\begin{subfigure}{.3\textwidth} \includegraphics[width=\textwidth]{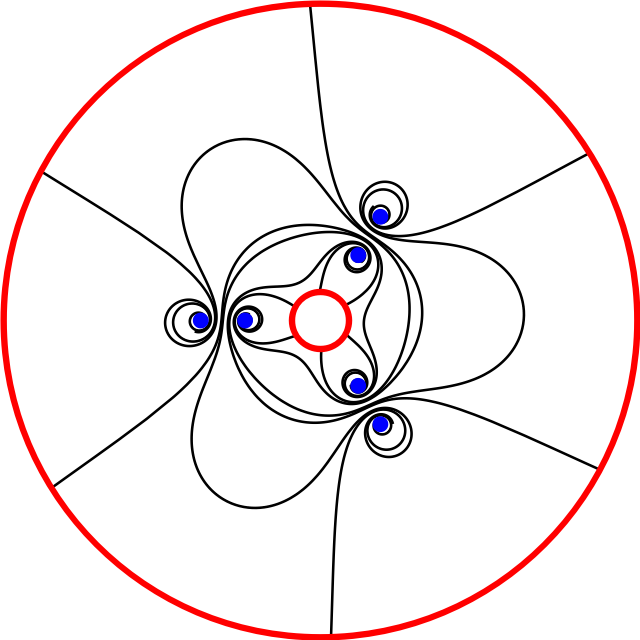} \caption{} \end{subfigure}	
\quad
\begin{subfigure}{.3\textwidth} \includegraphics[width=\textwidth]{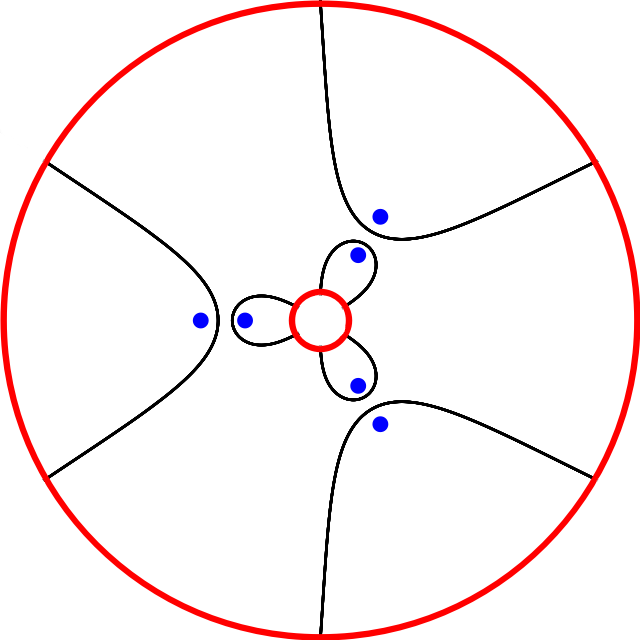} \caption{$\theta=\pi$} \label{figure:vf3-f} \end{subfigure}		
\caption{Example of separatrices in the real phase portrait of a rotating family $e^{i\theta}h^s\bY_h$ in the case $k=1$ and $p=3$ (Example~\ref{example:k=1}) for a fixed $h>0$ and with $c=i$, $P_0(h)>0$, according to some increasing values of $\theta$.
There is an infinite sequence of bifurcations at unstable values of $\theta\in[0,\pi]$ which accumulate towards the angle $0,\pi$ a the annular zone.
In the figure, a bifurcation occurs in the transition from (a) to (c), from (c) to (d), from (d) to (e). The portrait (b) is rotationally unstable with 3 heteroclinic separatrices, and so is (f) which exhibits an annular zone and 6 center zones.}
\label{figure:vf3}
\end{figure}

\begin{example}[$k=1$]\label{example:k=1}
For $k=1$, the vector field is of the form $\bY=c\big(u+P_0(h)\big)\bE$, i.e.
\[\bY_h=pc\big(\xi_1^{2p}+P_0(h)\xi_1^p+h^p\big)\tdd{\xi_1^p}.\]
Here if $P_0(h)^2\neq4h^p$ then an annular zone always appears in the rotating family for some value of $\theta$ (see Section~\ref{sec:6.2example}), although it could potentially be degenerated to a (poly)cycle.
An example of the dependence of the real phase portrait of $e^{i\theta}h^s\bY_h$ on $\theta$ is shown in Figure~\ref{figure:vf3}.
\end{example}

\begin{definition}
The only kinds of zones that a rotationally stable vector field can have are 
	\begin{enumerate}
		\item sepal zones with one end\footnote{Recall that an \emph{end} refers to one of the  hyperbolic sectors at a pole, see page \pageref{page:ends}.}  at a pole (Figure~\ref{figure:zonex-b}),
		\item $\alpha\omega$-zones with exactly two ends (at either two distinct poles or at one same pole) (Figure~\ref{figure:zonex-a}),
	\end{enumerate}
\Grn{as any other kind of zone would have on its boundary either a homoclinic or a heteroclinic connection (which goes against Definition~\ref{def:rotstable}).}
Let us call such zones \emph{rotationally stable}. 
\end{definition}

\subsubsection{Half-zone decomposition.}\label{sec:6.2halfzones}
\Grn{
Rotationally stable $\alpha\omega$-zones have 2 ends, while rotationally stable sepal zones have only 1 end each.
What can happen in a parametric family is that at a limit some singularities merge, and an $\alpha\omega$-zone splits into a pair of sepal zones: 
the bifurcation happens near the singular locus, away from the boundary, so ends are not affected.
To get a unified description that passes well to the limit, we shall split each rotationally stable $\alpha\omega$-zone into a pair of ``half-zones'': each having one end.
There is a canonical way to do so.
}

\begin{definition}\label{def:gates}
A rotationally stable $\alpha\omega$-zone contains exactly one saddle connection\footnote{See Definition~\ref{def:saddleconnections}.} 
which joins its two ends -- \Grn{in the coordinate $\bt_h$ it corresponds to the finite segment joining the saddle points on the two boundary lines,} 
see Figure~\ref{figure:zone-a}. It is called a \emph{transversal} of the zone.
Since  the transversal has a finite period, it has a well defined \emph{midpoint} (note that the transversal, and therefore also the midpoint, is independent of a small variation of the angle $\theta$).  
Following \cite{HLR}, the trajectory through the midpoint is called \emph{gate}.
In the coordinate $\bt_h$, the $\alpha\omega$-zone is an infinite strip and the gate divides it to two parallel strips of half the width (see Figure~\ref{figure:k1}).

The \emph{gate graph} of $e^{i\theta}h^s\bY_h$ for a stable $\theta$ is the closure of the union of the unoriented gates of all $\alpha\omega$-zones.\footnote{It is a topological graph in $\C^*$ with vertices at the equilibria and the gates as edges.}
\end{definition}

\begin{definition}\label{def:halfz}
The	\emph{half-zones} of a rotationally stable $e^{i\theta}h^s\bY_h$ are the connected components of the complement of the union of the separatrix graph and the gate graph. They are either 
\begin{enumerate}
	\item a rotationally stable sepal zone (bounded by 2 separatrices), or
	\item a half of a rotationally stable  $\alpha\omega$-zone cut in two halves by the gate trajectory (bounded by 2 separatrices and the gate).
\end{enumerate}

Thus each half-zone has exactly one end located at either the outer pole $\xi_1=\infty$, or the inner pole $\xi_1=0$.  
Accordingly, the half-zone are called either \emph{outer} or \emph{inner}. 
\end{definition}

To each outer half-zone there is a symmetric by $\sigma$ inner half-zone.
The terminology clearly depends on the choice of the coordinate $\xi_1$ on $B_h$: choosing $\xi_2$ instead would lead to opposite naming, but this is not important.

\begin{proposition}\label{prop:gate}
For a stable $(h,\theta)$, $h\neq 0$, the gate graph 
\begin{enumerate}
	\item is connected and homotopic to a simple non-contractible loop in the leaf $\{h=\const\}$,
	\item is preserved by $\sigma:\xi_1\mapsto\frac{h}{\xi_1}$ \eqref{eq:sigma1} and by 
	$\begin{cases}\text{(b)}\ \Lambda:\xi_1\mapsto\lambda\xi_1,\\ 
		\text{(c)}\ \sigma\Lambda=-I:\xi_1\mapsto-\xi_1, 
	\end{cases}$
	\item contains all the equilibria, 
	and in the formal case (b) also all the fixed points $\xi_1=\pm\lambda^{\frac{n}2} h^{\frac12}$ of $\Lambda^{n}\sigma=\left(\begin{smallmatrix}0&\lambda^n \\ \lambda^{-n}&0\end{smallmatrix}\right)$, $n\in\Z_{p}$, which are either equilibria or  midpoints of $\alpha\omega$-zones.
	. 
	\com{I don't remember if we use this statement anywhere. 
		However these points are some natural special points that can be considered. For example if one asks whether there is a $\sigma$-invariant $\alpha\omega$-zone, the answer is yes if and only if the fixed point of $\sigma$ is not an equilibrium (hence it is a midpoint of such zone). }
\end{enumerate}
\end{proposition}

\begin{proof}
1) The union of the closures of the outer half-zones is a contractible neighborhood of $\infty$, while that of inner half-zones is a contractible neighborhood of $0$. The intersection of these two neighborhoods is a simple loop consisting of gates between outer and inner half-zones. 
Connectedness follows from the rotational stability.

2) Clear.

3) \Grn{In the formal case (b), the fixed point  $\xi_1=\pm\lambda^{\frac{n}2}h^{\frac12}$  of $\Lambda^{n}\sigma$ cannot belong to neither the outer nor the inner component of the complement of the gate graph, as these two get swapped by $\Lambda^{n}\sigma$, 
hence it lies on the gate graph. If it is not an equilibrium, then it lies on the gate of some $\alpha\omega$-zone invariant by $\Lambda^{n}\sigma$.
	The unique transversal of this zone is then also mapped onto itself by $\Lambda^{n}\sigma$, and the fixed point belongs to it: the transversal divides the zone into two halves which are swapped by $\Lambda^{n}\sigma$. Hence it is the midpoint of the transversal. Likewise in the formal case (c).}
\end{proof}

\begin{cor}
	For every stable $(h,\theta)$, $h\neq 0$, there is always at least $p$ of $\alpha\omega$-zones with one end at $0$ and other end at $\infty$.
\end{cor}
\begin{proof}
\Grn{By point 1) of Proposition~\ref{prop:gate}	the gate graph must contain at least one gate, which then has $p$ distinct images by $\Lambda^n$, $n=0,\ldots,p-1$, (resp. by $\sigma\Lambda=-I$ in the formal case (c)), which are also gates.}
\end{proof}

An advantage of working with half-zones rather then zones is that half-zones are in a 1--1 correspondence with the ends, and as such they have a well defined limit when $h\to 0$.

\begin{proposition}\label{prop:halfz-param}
	The outer half-zones of $e^{i\theta}h^s\bY_h$ depend continuously on $(h,\theta)$ on each connected component of
	\[ \{(h,\theta): |h|<\delta_2,\ \theta\in\R/\pi\Z,\ e^{i(\theta+s\arg(h))}\bY_h\ \text{is rotationally stable}\}, \]
	and  have a limit as outer half-zones of $e^{i(\theta_0+s\tilde\theta_0)}\bY_0$ when $(h,\theta)\to (0\cdot e^{i\tilde\theta_0},\theta_0)$ (i.e. when $h\to 0$  with an asymptotic direction, $\arg(h)\to \tilde\theta_0$) within this connected component. 
	Note that the vector field $\bY_0$ \eqref{eq:bY_0} is defined on the irreducible component $\{\xi_2=0,\ |\xi_1|<\delta_1\}$ of $B_0$,
	and as such all its half-zones are outer. A symmetric statement by means of $\sigma$ is true for the inner half-zones, which have a limit in the other irreducible component $\{\xi_1=0,\ |\xi_2|<\delta_1\}$ of $B_0$. 
\end{proposition}

\begin{proof}
	The outer half-zones are delimited by separatrices of the pole at $\xi_1=\infty$ and by gates. Both depend continuously on $(h,\theta)$ as long as $e^{i(\theta+s\arg(h))}\bY_h$ is rotationally stable:
	the only way separatrices can undergo a bifurcation is through the appearance of homoclinic/heteroclinic connection.\footnote{In fact, each equilibrium possesses a neighborhood that is a trapping domain: any trajectory of $e^{i(\theta+s\arg(h))}\bY_h$ for any $\theta$ that crosses its boundary will end at the equilibrium, and the boundary of the maximal domain is a union of saddle connections \cite{Klimes-Rousseau2}. This implies that landing points of separatrices are stable by small change of $\theta$.}
At the limit when $h\to  0$ the equilibria merge with the pole $\xi_1=0$, and the gate graph degenerates to this point as well. 
Note that the limit vector field $e^{i(\theta_0+s\tilde\theta_0)}\bY_0$ is rotationally stable for any $\theta_0+s\tilde\theta_0$.
The separatrices of  $\xi_1=\infty$ persist to the limit and so do the outer half-zones.
\end{proof}

Another advantage of working with half-zones is that on each of them there is a unique determination of the flat coordinate $\bt_h$ \eqref{eq:bt} that vanishes at the end of the zone.

\begin{lemma}\label{lemma:t}
On each half-zone there is a uniquely defined translation coordinate $\bt_h=\int\bY_h^{-1}$ that has a vanishing limit at its end.
When two neighboring half-zones are halves of the same $\alpha\omega$-zone, then their coordinates differ on the gate trajectory by the 
period of the transversal\footnote{Definition~\ref{def:gates}} of the zone. This system of coordinates depends analytically on the parameter $h$, in particular it passes to the limit when $h\to 0$.
\end{lemma}

\begin{proof}
For an outer half-zone one takes $t_h(\xi_1)=\int_{\infty}^{\xi_1}\bY_h^{-1}$, and for an inner half-zone one takes $t_h(\xi_1)=\int_{0}^{\xi_1}\bY_h^{-1}$.
\end{proof}

\subsubsection{Enlargement of half-zones through stable rotation.}\label{sec:6.2domains}

Let $\{|h|<\delta_2\}$ be a small neighborhood of the origin in the $h$-space.
It will be useful to consider its \emph{polar blow-up}, that is to identify $h$ with $|h|\cdot e^{i\arg h}$, and $|h|=0$ is a circle of asymptotic directions.
\Grn{By Proposition~\ref{prop:rotstable} the set of $(h,\theta)$ for which $e^{i(\theta+s\arg(h))}\bY_h$ is rotationally unstable consists of at most countably many curves
$h\mapsto\theta(h)$. Lemma~\ref{lemma:asymptotic-h} below shows that these curves extend asymptotically at the limit to the circle $|h|=0$. }

Let $K$ be one of the at most countably many connected components of the complementary set
\begin{equation}\label{eq:K}
\{(h,\theta): |h|<\delta_2,\ \theta\in\,\,]\delta_3,\pi-\delta_3[,\ e^{i(\theta+s\arg(h))}\bY_h\ \text{rotationally stable}\}, 
\end{equation}
where  $\delta_3>0$ is some arbitrarily small fixed constant (cf. Lemma~\ref{lemma:omega}), and $h$ is understood as being from the polar blow-up.
\Grn{And let us assume that the closure of $K$ intersects the circle $|h|=0$. }

Let
\[S\subseteq \{|h|<\delta_2\}\ \text{ be the projection of $K$ into the $h$-plane}\]
by $(h,\theta)\mapsto h$ (possibly a ramified set defined on the covering surface of the polar blowup of $(\C,0)$).
For each $h\in S$ denote $K_h=\{\theta: (h,\theta)\in K\}$ the associated open interval of stable angles, hence
$K=\coprod_{h\in S}\{h\}\times K_h$.

For any $(h,\theta)\in K$ let $Z_{h,\theta}$ be an outer or an inner half-zone of $e^{i(\theta+s\arg(h))}\bY_h$, depending continuously on $(h,\theta)\in K$,
and let
\begin{equation}\label{eq:domainZ}
Z_h=\bigcup_{\theta\in K_h}Z_{h,\theta},
\end{equation}
be their union over $\theta\in K_h$, defined in way that it is simply connected, i.e.  on the covering surface of the pierced leaf $\{h=\const\}\sminus\{P=0\}$
identified with $\C^*\sminus\{P_h=0\}$ in the coordinate $\xi_1$.
It is  an \emph{enlargement of the half-zone $Z_{h,\theta}$ through stable rotation} (see Figure~\ref{figure:k2}).
In another words, $\bt_h(Z_h)$ is the union of the maximal strips in $\bt_h(\C^*\sminus\{P_h=0\})$ of direction $\theta$ varying continuously with $\theta\in K_h$.
For $|h|=0$, put 
\[Z_{0\cdot e^{i\arg h}}=\bigcup_{\theta\in K_{0\cdot e^{i\arg h}}}Z_{{0\cdot e^{i\arg h}},\theta},\qquad Z_0=\bigcup_{0\cdot e^{i\arg h}\in S}Z_{0\cdot e^{i\arg h}}.\]

\begin{proposition}
	When $S\ni h\to 0\cdot e^{i\tilde\theta}$ with an asymptotic direction, $\arg h\to\tilde{\theta}$, then each outer domain $Z_h$ \eqref{eq:domainZ} tends to $Z_{0\cdot e^{i\tilde\theta}}$,
	which is a subdomain of $Z_0$.
\end{proposition} 

\begin{proof}
	This is a consequence of the definition and the continuity of the zones on $(h,\theta)\in K$ which extends to $|h|=0$ by Lemma~\ref{lemma:asymptotic-h} below.	
\end{proof}

Define
\begin{equation} \label{eq:Z_S}
	Z_S=\coprod_{h\in S}\{h\}\times Z_h, 
\end{equation}
as the corresponding ramified domain in the whole $\xi$-space, i.e. a subdomain of a covering surface of $\C^2\sminus(\{P=0\}\cup\{h=0\})$.

These domains $Z$ represent a ``global version'' of the Lavaurs domains that we want to construct and understand. 
In order to obtain the actual Lavaurs domains we will later ``localize'' the above construction to a neighborhood $B$ \eqref{eq:B} of the origin.

\begin{remark}\label{remark:covering}
	Let us remark that for each $(h,\theta)$ fixed the union of all the \Grn{different} half-zones $Z_{h,\theta}$ of $e^{i(\theta+s\arg(h))}\bY_h$ covers all except of the closure of the separatrices and of the gates (Definition~\ref{def:gates}). 
	This means that for a given $K,S$ as above, and each fixed $h\in S$, the different domains $Z_h$ \eqref{eq:domainZ} overlap and cover together the whole leaf $\{h=\const\}$, identified with $\C^*$, except of the equilibria and of the midpoints of $\alpha\omega$-zones.
	The exclusion of the midpoints has been a bit arbitrary consequence of our definition of half-zones: we could have instead defined the half-zones of an $\alpha\omega$-zones in such a way that would overlap a bit along the gate. So the midpoints can be considered as covered as well.\footnote{Let us note that in the proof of Theorem~\ref{thm:sectorialM} we will actually want to exclude the fixed points $\xi_1=\lambda^{\frac{n}2}\pm h^{\frac12}$ of $\Lambda^{n}\sigma$, $n\in\Z_{p}$, (resp. $\xi_1=\pm ih^{\frac12}$ of $\sigma\Lambda=-I$ in the formal case (c)), which, unless equilibria, are midpoints by Proposition~\ref{prop:gate}.} 
\end{remark}

\begin{remark}
	When crossing the bifurcation locus, which consists of those $(h,\theta)$ for which $e^{i\theta}h^s\bY_h$ is rotationally unstable, the topological organization of the phase portrait changes. This change may affect only some of the half-zones while other may persist unchanged. It would be therefore natural to consider for each half-zone its maximal ``chamber'' in the $(h,\theta)$-space over which it is defined (i.e. over which it evolves continuously). However, we find it simpler for the discussion to always stop at the bifurcation locus for all the half-zones, whether they bifurcate or not.
\end{remark}

\subsubsection{Asymptotic behavior when $h\to 0$.}

We recall \re{eq:bYh}, that 
\begin{align*}
	\bY_h&=cp\left((\xi_1^p\!+\tfrac{h^p}{\xi_1^p})^k+P_{k-1}(h)(\xi_1^p\!+\tfrac{h^p}{\xi_1^p})^{k-1}+\ldots+P_{0}(h)\right)\xi_1^p\tdd{\xi_1^p}\\
	&=cp\,\xi_1^{-kp}P_h(\xi_1)\xi_1^p\tdd{\xi_1^p}.
\end{align*}

\begin{lemma}\label{lemma:asymptotic-h}
	\begin{enumerate}
		\item The $\bt_h$-length $|\nu_\gamma(h)|=\left|\int_\gamma \bY_h^{-1}\right|$ of any saddle connection $\gamma(h)$ in the rotational family $e^{i\theta}h^s\bY_h$ is uniformly of order at least \Grn{$|h|^{-\check s k}$} when $h\to 0$, where
		\[ \check s=\min\left(\left\{\frac{\ord_0 P_j(h)}{(k-j)}\mid j=0,\ldots, k-1\right\}\cup\left\{\frac{p}2\right\}\right). \]
		\item  For each fixed $\theta$, the rotationally unstable values of $h$ form at most countable number of real curves, each approaching $|h|=0$ with an asymptotic direction.
	\end{enumerate}	
\end{lemma}	 

\begin{proof}
\textit{1. }\ 		The proof is the same as in \cite[Lemma~4.7]{HLR}. First let us note that all the equilibria are situated in some ring 
		\[\tfrac{1}{L}|h|^{p-\check s}<|\xi_1|^p<L|h|^{\check s}\quad \text{for some } \ L>0.\]
		Indeed, if $u_*(h)=\xi_{1,*}(h)^p+\frac{h^p}{\xi_{1,*}(h)^p}$ is a root of $P(u,h)=u^k+\sum_{j=0}^{k-1}P_j(h)u^j$, then 
		$|u_*|^k\leq \sum_{j=0}^{k-1}|P_j||u_*|^j\leq k\max_j|P_j||u_*|^j$,
		and there exist some $C>0$ such that $|u_*|\leq  \Grn{ \max_j|kP_j|^{\frac1{k-j}} } \leq C|h|^{\check s}$ on the disc $\{|h|<\delta_1\}$,
		which means that $|\xi_{1,*}|^{2p}\leq |u||\xi_{1,*}|^p+|h^p|\leq \max\{ 2C|h|^{\check s}|\xi_{1,*}|^p,\ 2|h|^p\}$ and 
		\[|\xi_{1,*}|^p\leq\Grn{ \max\{2C|h|^{\check s},\sqrt 2|h|^{\frac{p}{2}}\} } 
		\leq |h|^{\check s}\max\{2C,\sqrt2\delta_2 ^{\frac{p}2-\check s}\}=|h|^{\check s}L,\]
		 and similarly also  $\frac{|h|^p}{|\xi_{1,*}|^p}\leq |h|^{\check s}L$.
		
		Any saddle connection must pass through this ring (a saddle connection between two different poles $0$, $\infty$ has no other choice than cross the ring,
		while a saddle connection with poles on its both ends equal encircles at least one equilibrium on each side since otherwise it would be contractible in $\C^*\smallsetminus\{P_h(\xi_1)=0\}$ which is not possible).
		Up to a $\sigma$-symmetry (which preserves periods up to a sign), we may assume that the saddle connection emanates from the pole $\xi_1=\infty$. Then its $d|\xi_1|^p$-length inside the ring of ``double the size'' 
		\begin{equation}\label{eq:ring}
			\tfrac{1}{2L}|h|^{p-\check s}<|\xi_1|^p<2L|h|^{\check s}
		\end{equation} 
	is at least $L|h|^{\check s}$, while its speed is
		\Grn{$\big|\bY_h.\xi_1^p\big|=cp\,|P(u,h)||\xi_1|^p\leq 2cp\,\check CL|h|^{\check s(k+1)}$, as the maximum of $|P(u,h)|$ in the ring \eqref{eq:ring} can be majorated by $\check C|h|^{k\check s}$ for some $\check C>0$,
		hence the $\bt_h$-time the saddle connection spends in \eqref{eq:ring} exceeds $\frac{1}{2|cp|\check C}|h|^{-\check sk}$. }

\textit{2. }\  The  period $\nu_\gamma(h)$ (page~\pageref{page:period})  along any closed loop $\gamma(h)$ of $e^{i\theta}h^s\bY_h$ can be expressed as a sum of residues at roots of $P_h$ (see \re{eq:bt}) and therefore has an expansion in terms of a meromorphic Puiseux power series in $h$. Hence the set on which
		$\nu_\gamma(h)\in e^{i\theta}h^s\R$ has a form of a singular real analytic subvariety on a finite covering of $(\C,0)$, and each branch has an asymptotic tangent at $0$.
		
		Similarly, let us show that the period $\nu_\gamma(h)$ along any heteroclinic saddle connection of $e^{i\theta}h^s\bY_h$  between $0$ and $\infty$ is 
		of the form $\nu_\gamma(h)=\zeta_0(h)+\zeta_1(h)\log(h)$ where $\zeta_0,\zeta_1$ are meromorphic Puiseux power series.
		Indeed, $\bt_h=\int\bY_h^{-1}$ takes the form $R_h(\xi_1)+\tfrac{1}{2\pi i}\sum_{\{P_h(a_j(h))=0\}}\nu_{a_j}(h)\log(\xi_1-a_j(h))$ where $R_h(\xi_1)$ is a rational function of $\xi_1$ with coefficients Puiseux series in $h$, and with poles at zeros of $P_h$, hence the term $\big[R_h(\xi_1)\big]_{\xi_1=0}^\infty$ is a Puiseux series, while the term 
		\begin{align*}
		\Big[\tfrac{1}{2\pi i}\sum_{a_j}&\nu_{a_j}(h)\log(\xi_1-a_j(h))\Big]_{\xi_1=0}^\infty=\\
		&=	-\tfrac{1}{2\pi i}\sum_{a_j}\nu_{a_j}(h)\log(-a_j(h)) + \tfrac{1}{2\pi i}\lim_{\xi_1\to\infty}\sum_{a_j}\nu_{a_j}(h)\log(\xi_1-a_j(h))\\
		&= 	-\tfrac{1}{2\pi i}\sum_{a_j}\nu_{a_j}(h)\log(-a_j(h)),
		\end{align*} 
	\Grn{as $\lim_{\xi_1\to\infty}\sum_{a_j}\nu_{a_j}(h)\log(\xi_1-a_j(h))=\lim_{\xi_1\to\infty}\sum_{a_j}\nu_{a_j}(h)\log(1-\frac{a_j(h)}{\xi_1})=0$ since $\sum_{a_j}\nu_{a_j}(h)=0$,}
	 hence one gets a meromorphic  Puiseux series possibly in combination with $\log(h)$.  
		Anyway, each branch of the set where $\nu_\gamma(h)=\zeta_0(h)+\zeta_1(h)\log(h)\in e^{i\theta}h^s\R$ is again a real curve with tangent at $0$ since $\nu(h)$ is of negative order in $h$.
		
\end{proof}

\begin{figure}[t]
\centering
\includegraphics{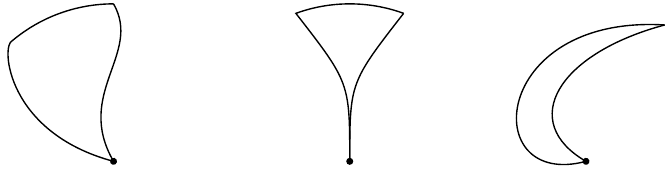}	
\caption{Examples of cuspidal sectors. The one in the middle has zero angular opening contrary to the others.}
	\label{figure:cuspidal}
\end{figure}

\begin{definition}\label{def:sector}
	A \emph{cuspidal sector} with vertex at $h=0$ is a simply connected planar domain $S$ bounded by two real analytic curves, each of which has an asymptotic tangent at the vertex, and by an arc of a fixed radius (see Figure~\ref{figure:cuspidal}).
	The \emph{angular opening} of the cuspidal sector is the angle between the tangent rays of the two bounding curves at the vertex.	
	
	\emph{We shall consider the vertex as included in the cuspidal sector, $0\in S$.} \Grn{That way the limit situation at $h=0$ will be part of our description.}
\end{definition}

\Grn{\emph{Let us stress that a covering of a neighborhood of 0 by a collection of cuspidal sectors of positive angular openings may not contain any finite subcovering.}
	This is strikingly different from the case of usual sectorial coverings. The reason is that there might be directions of rays who have no initial segment covered by a single cuspidal sector from the collection.}

\begin{cor}
	The domains $S$ of Section~\ref{sec:6.2domains} in the $h$-space have a form of cuspidal sectors at the origin of a positive angular opening $\geq\frac{\pi-2\delta_3}{s+\tilde sk}$
	for some $\tilde s\geq \check s$.
\end{cor} 	

\begin{proof}
For fixed $\theta$ the set of rotationally unstable values of $h$ has the form of union of curves $h^{-s}\nu_\gamma(h)\in e^{i\theta}\R$, where $\nu_\gamma(h)$ is the period of some saddle connection $\gamma$, $\ord_0 \nu_\gamma(h)^{-1}\geq \check sk$ by Lemma~\ref{lemma:asymptotic-h}.
\Grn{Asymptotically $\theta=\arg\big(h^{-s}\nu_\gamma(h)\big) \sim \big(s+\ord_0\nu_\gamma(h)^{-1}\big)\arg(h)$ as $|h|\to 0$, hence allowing $\theta\in\ ]\delta_3,\pi-\delta_3[$ to vary in interval of length $\pi-2\delta_3$ translates to varying the asymptotic direction $\arg(h)$ in an interval of length 
	$\frac{\pi-2\delta_3}{s+\ord_0\nu_\gamma(h)^{-1}}\geq\frac{\pi-2\delta_3}{s+\tilde sk}$ where 
	$\tilde sk=\max_{\gamma}\ord_0\nu_\gamma(h)^{-1}\geq \check sk$.
}
\end{proof}

\subsubsection{Example~\ref{example:k=1} continued (case $k=1$).} \label{sec:6.2example}

\begin{figure}[t]
\begin{subfigure}[t]{.5\textwidth} \includegraphics{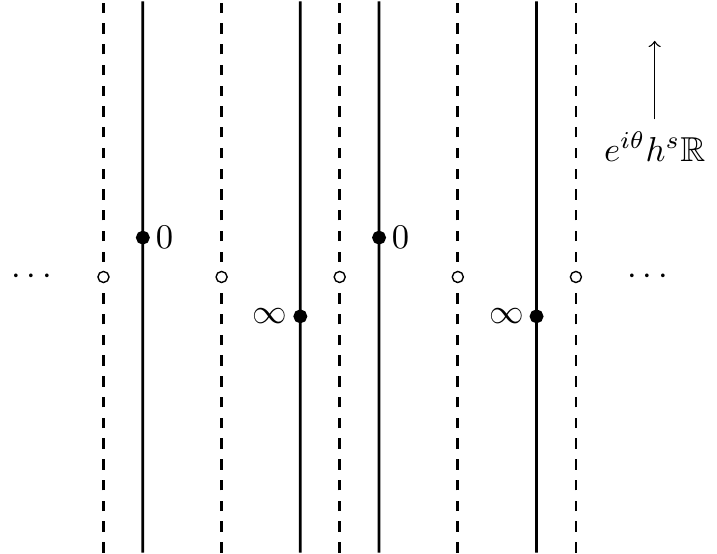}	\caption{$P_0(h)^2\neq 4h^p$} \label{figure:k1a} \end{subfigure}
\quad
\begin{subfigure}[t]{.3\textwidth} \includegraphics{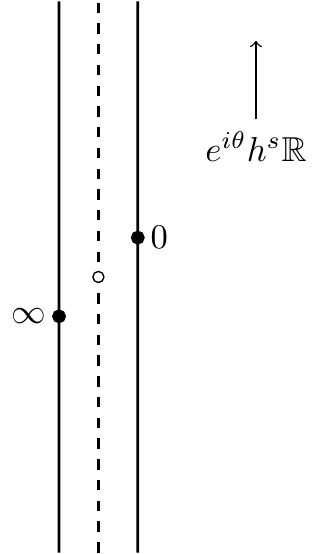} \caption{$P_0(h)^2=4h^p\neq 0$} \label{figure:k1b} \end{subfigure}
\quad
\begin{subfigure}[t]{.1\textwidth} \includegraphics{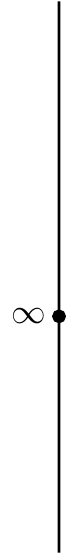}  \caption{$h=0$} \label{figure:k1c} \end{subfigure}
\caption{The translation surface of $\bt_h$ for $k=1$. The flow is parallel to $e^{i\theta}h^s\R$ (which in the figure is vertical directed upwards).
		The thick lines are the separatrices of the saddle points corresponding to $y=0,\infty$ (black points) for $e^{i\theta}h^s\bY_h$ , 
		determining (a) 2 $\alpha\omega$-zones modulo translation by $\nu_{1,2}\Z$, resp. (b) 2 sepal zones and 1  $\alpha\omega$-zone, resp. (c) 2 sepal zones.
		The dashed lines are the gates passing through the midpoints (white points) and further dividing the real phase portrait into (a), (b) 4 half-zones: 2 outer half-zones attached to the point $\infty$ and 2 inner half-zones attached to the point $0$ modulo translation by $\nu_{1,2}\Z$, resp. (c) 2 outer half-zones attached to the point $\infty$.
		The angle $\theta$ is unstable whenever some pair of saddle points ``$0$'' and ``$\infty$'' lie on the same line $e^{i\theta}h^s\R$.
		In the case (a) there are countably many unstable directions $\theta$, accumulating to $\arg(h^{-s}\nu_{1,2})\mod \pi\Z$, while in (b) there is only one unstable direction modulo $\pi\Z$, and in (c) all directions are stable.
	}
\label{figure:k1}
\end{figure}

Let us look at the case $k=1$, $\bXmod=ch^s\big(u+P_0(h)\big)\bE$, in detail. We have 
\[\bY_h=pc\big(y^2+P_0(h)y+h^p\big)\tdd{y},\quad\text{where }\ y=\xi_1^p,  \]
with equilibria at $y=a_{j}(h)$, $j=1,2$
\[a_{j}(h)=\tfrac{1}{2}\big(-P_0(h)+(-1)^{j-1}\sqrt{P_0(h)^2-4h^p}\big).\]
While the vector field $e^{i\theta}h^s\bY_h$ in the coordinate $y$ is analytic on all $\CP^1$, we still need to treat the points $y=0,\infty$ as poles of order 0, 
that is to consider the trajectory through the point as a union of one incoming and one outgoing separatrix. 

The rectifying coordinate $\bt$ in this case is quite simple,
\[\bt_h=\int\bY_h^{-1}=\begin{cases}
\frac{1}{pc(a_1-a_2)}\log\frac{y-a_1}{y-a_2}, & a_1(h)\neq a_2(h),\\
-\frac{1}{pc(y-a)}, & a_1(h)=a_2(h)=:a(h),
\end{cases}\]
and the translation surface of $\bt_h$ is $\C$ (Figure~\ref{figure:k1}). 
The dynamical residue \eqref{eq:nu_a} at $a_{j}(h)$ is 
\begin{equation}\label{eq:nu12}
\nu_{j}(h)=\frac{(-1)^{j-1} 2\pi i }{pc(a_{1}(h)-a_2(h))}=\frac{(-1)^{j-1}2\pi i}{pc\sqrt{P_0(h)^2-4h^p}}, \quad a_1(h)\neq a_2(h).
\end{equation}
For $h\neq 0$ the period of a heteroclinic connection between $0$ and $\infty$, determined modulo $\nu_{j}\Z$,
is
\begin{equation}\label{eq:nu08}
\nu_{0\infty}(h)=\int_\infty^0\bY^{-1}=
\begin{cases}
\frac{1}{pc(a_1-a_2)}\log\frac{a_1}{a_2}=\frac{\nu_{1}}{2\pi i}\log\frac{a_1}{a_2}, & a_1\neq a_2,\\
\frac{1}{pca}, & a_1=a_2=:a.
\end{cases}
\end{equation}

\begin{figure}
\begin{subfigure}{.45\textwidth} \includegraphics[scale=0.8]{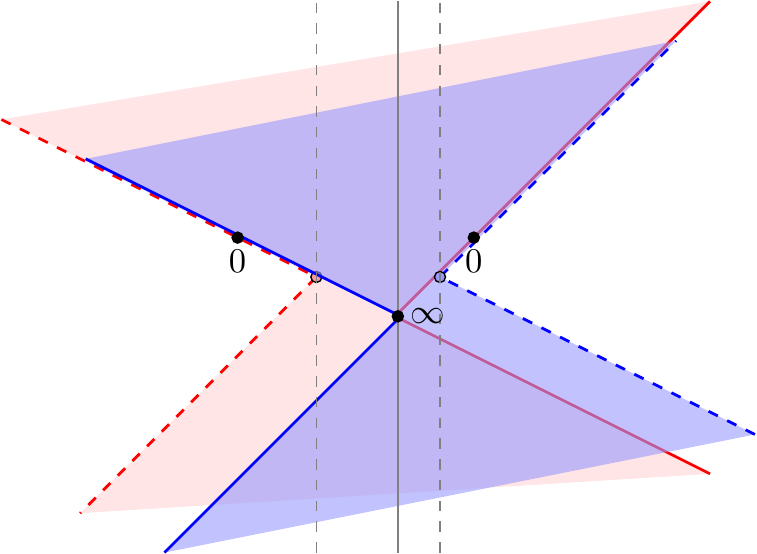} \caption{$P_0(h)^2\neq 4h^p$} \label{figure:k2a} \end{subfigure}
\quad
\begin{subfigure}{.45\textwidth} \includegraphics[scale=0.8]{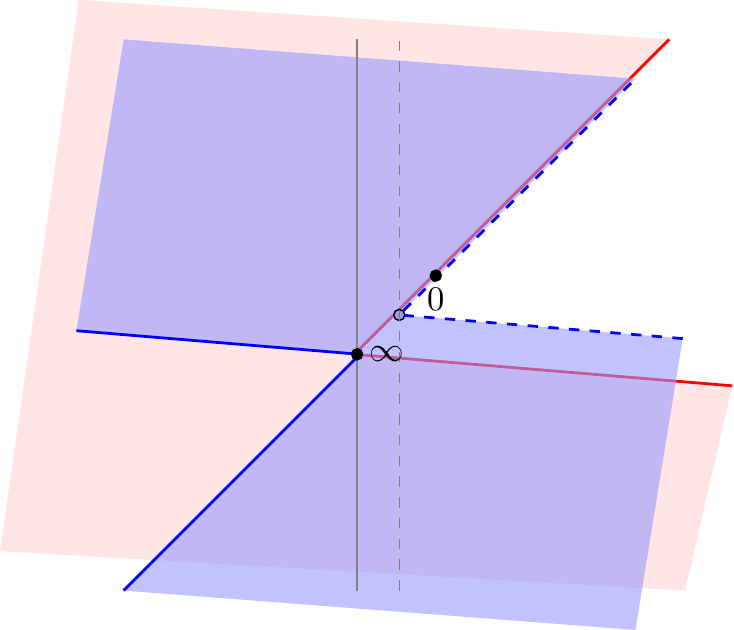} \caption{$P_0(h)^2=4h^p\neq 0$} \label{figure:k2b} \end{subfigure}
\\
\centering
\begin{subfigure}{.5\textwidth} \includegraphics[scale=0.8]{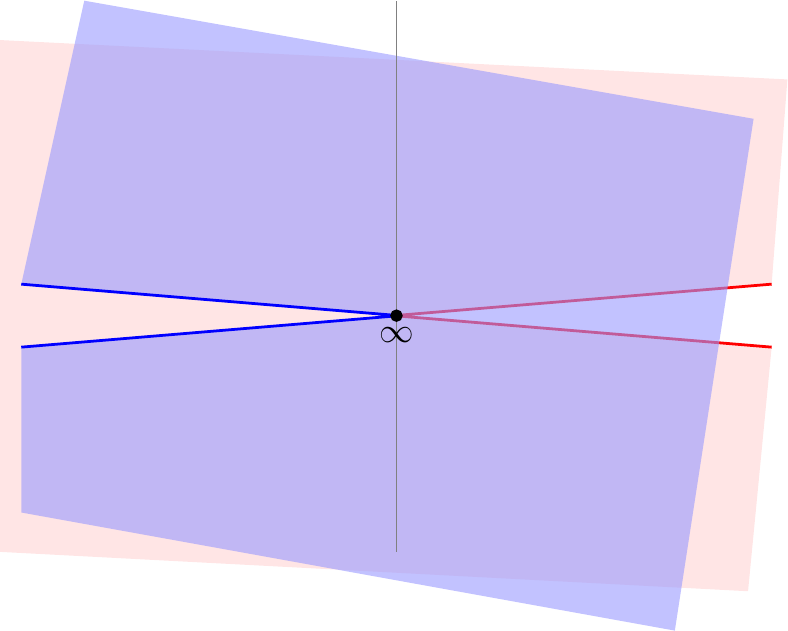} \caption{$h=0$} \label{figure:k2c} \end{subfigure}
	\caption{The rotational enlargement $Z_h$ of a pair of outer half-zones (i.e. those attached to the saddle point $\infty$) on the translation surface of $\bt_h$ for $k=1$. The gray vertical lines, corresponding to a pair of separatrices and to gates for a fixed $\theta$ can vary their direction as long as the angle $\theta\in\,\,]\delta_3,\pi-\delta_3[$ stays stable. The forbidden direction $h^s\R$ is horizontal in this figure. Compare with Figure~\ref{figure:k1}.}
	\label{figure:k2}
\end{figure}

Let us look at the set of unstable $(h,\theta)$ in the case $P_0(h)^2\neq 4h^p$, i.e. when $a_1(h)\neq a_2(h)$ are simple equilibria.
If $(h,\theta)$ is such that $\nu_{j}\in e^{i\theta}h^s\R$, that is if $\theta=\arg\nu_{1,2}(h)-s\arg h\mod\pi\Z$, then both equilibria are simultaneously centers,
one bounded by the homoclinic saddle connection of $y=\infty$ (the line $y=-\frac{P_0}{2}+i\sqrt{P_0^2-4h^p}\R$) and the other bounded by
the homoclinic saddle connection of $y=0$ (\Grn{which is the image of the previous one by $\sigma:y\mapsto \frac{h^p}{y}$, hence a circle in the $y$-coordinate through the points $0$ and $-\frac{2h^p}{P_0}$ on its diameter).}
The zone between these two homoclinic saddle connections is an annular zone for this $\theta$ (see Figure~\ref{figure:vf3-f} in coordinate $\xi_1=y^{1/3}$).
In the coordinate $\bt_h$ (Figure~\ref{figure:k1a}) this annulus corresponds to the strip bounded by two parallel lines (horizontal in Figure~\ref{figure:k2a}), one containing all the saddle points ``$0$'', the other containing all the saddle points ``$\infty$''.
There is an infinite number of saddle connections between ``$0$'' and ``$\infty$'' 
of different unstable directions $\theta$, all contained inside this strip. 
The set of their directions accumulates for each fixed $h$ towards $\arg\nu_{1,2}(h)-s\arg h\mod\pi\Z$. 

Now let us look at the asymptotic behavior of the set of unstable $(h,\theta)$ when $h\to 0$ radially (with fixed $\arg h$).
There are three kinds of situations that can arise depending on the asymptotic behavior of the simple roots $a_1(h)$ and $a_2(h)$, which have Puiseux series representation in $h$.
Let 
\[\check s=\min\big\{\ord_0 P_0(h),\tfrac{p}{2}\big\},\qquad \tilde s=\tfrac1{2} \ord_0\big(P_0(h)^2-4h^p\big),\] 
and consider the weighted blow-up $z=h^{-\tilde sp}\big(y+\frac12 P_0(h)\big)$.
It transforms $h^s\bY_h$ into
\[h^s\bY_h=pch^{s+\tilde s}\big(z^2-D(h)\big)\tdd{z},\]
where $D(h)=\tfrac14h^{-2\tilde s}\big(P_0(h)^2-4h^p\big),\ D(0)\neq 0$.
Then $\nu_{j}=(-1)^{j-1} h^{-\tilde s}\tfrac{\pi i }{pc\sqrt{D(h)}}$, and the direction to which the set of unstable directions accumulate for each $h\neq 0$ fixed is
\begin{equation}\label{eq:accumulation}
\arg\big(h^{-s}\nu_{j}(h)\big)=\tfrac{\pi}{2} -\arg(c)-(s+\tilde s)\arg h-\tfrac12\arg D(h)\mod\pi\Z.
\end{equation}
\begin{figure}[t]
\renewcommand\thesubfigure{\roman{subfigure}}
\centering 
\begin{subfigure}{.3\textwidth} \includegraphics{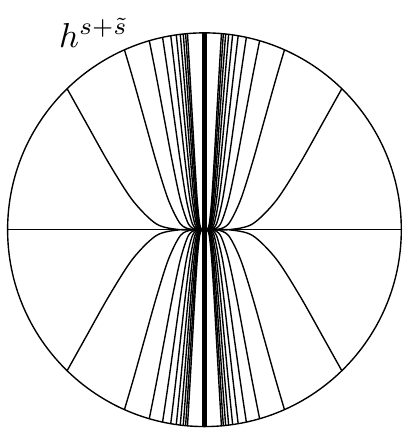} \caption{$\check s=\tilde s<\frac{p}{2}$} \label{figure:k3b} \end{subfigure}
\quad
\begin{subfigure}{.3\textwidth} \includegraphics{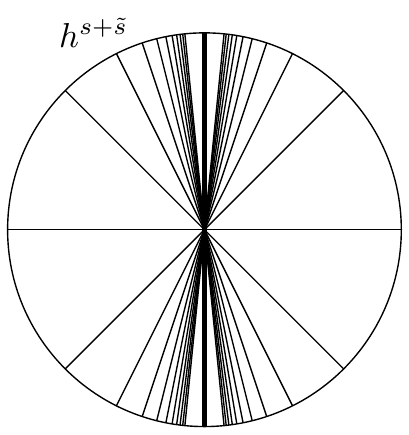} \caption{$\check s=\tilde s=\frac{p}{2}$} \label{figure:k3a} \end{subfigure}
\quad
\begin{subfigure}{.3\textwidth} \includegraphics{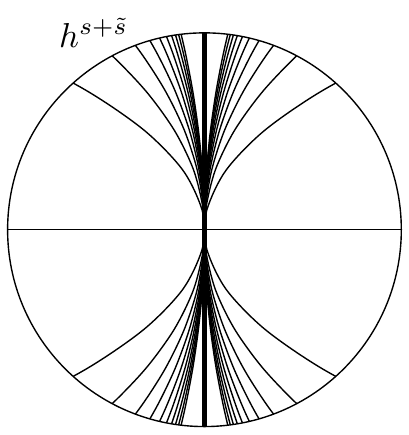} \caption{$\check s=\frac{p}{2}<\tilde s$} \label{figure:k3c} \end{subfigure}
	\caption{For $k=1$, the values of $h^{s+\tilde s}$ for which $e^{i(\theta+s\arg h)}\bY_h$ is rotationally unstable for a fixed value of $\theta$
		(figure for $P_0(h)=\ov P_0(h)$, $c=-i$, $\theta=\frac{\pi}{2}$).
	The sectors $S$ are enlargements of the above sectors of stability by varying $\theta\in\,\,]\delta_3,\pi-\delta_3[$, which asymptotically corresponds to enlargement by a rotation $h^{s+\tilde s}\mapsto e^{i\theta}h^{s+\tilde s}$.}	\label{figure:k3}
\end{figure}	
\begin{enumerate}[label=(\roman*), wide=0pt, leftmargin=\parindent]
	\item If $\check s=\tilde s<\frac{p}{2}$, then $\ord_0 a_1(h)=\check s\neq p-\check s=\ord_0 a_2(h)$, and 
	\GRN
	by \eqref{eq:nu08}: $\frac{\nu_{0\infty}(h)}{\nu_{1}(h)}=\frac{1}{2\pi i}\log\frac{a_1(h)}{a_2(h)}\sim\frac{2\check s-p}{2\pi i}\log(h)$  tends to infinity as $h\to 0$, hence $|\nu_{0\infty}(h)|$ grows faster than $|\nu_{1,2}(h)|$.
	This means that for each saddle connection $\nu_\gamma(h)=\nu_{0\infty}(h)\mod\nu_{1}\Z$ between $0$ and $\infty$: 
		\begin{align*}
			\lim_{|h|\to 0}\arg \big(h^{-s}\nu_\gamma(h)\big)&=\lim_{|h|\to 0}\arg\big(h^{-s}\nu_{0\infty}(h)\big)\mod\pi\Z\\
			&=	-(s+\tilde s)\lim\arg(h)-\arg(c\sqrt{D(0)})\mod\pi\Z.
		\end{align*}
Therefore the associated branch of the set of unstable $(h,\theta)$ has the same asymptotic limit $(0\cdot e^{i\arg(h)},\theta)$ with
\[ \theta+(s+\tilde s)\arg(h) \in -\arg(c)-\tfrac12\arg D(0)\mod\pi\Z,\]
which is perpendicular to the limit $h\to 0$ of the accumulation direction \eqref{eq:accumulation}. See Figure~\ref{figure:k3b}.	
\FGRN

	\item If $\check s=\tilde s=\frac{p}{2}$ ($p$ has to be even), then $\lim_{h\to 0}\frac{a_1(h)}{a_2(h)}\in\C^*\sminus\{1\}$, which means that both $\nu_{1,2}(h)$ and $\nu_{0\infty}(h)$ grow with the same rate $h^{-\frac{p}{2}}$ when $h\to 0$ and asymptotically in the same relative position:
	$\lim_{h\to 0}\frac{\nu_{0\infty}(h)}{\nu_{1}(h)}=\lim_{h\to 0}\frac{1}{2\pi i}\log\frac{a_1(h)}{a_2(h)}\in\C^*$,
	\Grn{and the same is true for all saddle connections $\nu_\gamma(h)=\nu_{0\infty}(h)\mod\nu_{1}\Z$ between $0$ and $\infty$.}  See Figure~\ref{figure:k3a}.
		
	\item If $\check s=\frac{p}{2}<\tilde s$  ($p$ has to be even), then 
	$\lim_{h\to 0}\frac{a_1(h)}{a_2(h)}=1$, and $\lim_{h\to 0}\frac{\nu_{0\infty}(h)}{\nu_{1}(h)}=\lim_{h\to 0}\frac{1}{2\pi i}\log\frac{a_1(h)}{a_2(h)}=0$,
	which means that $|\nu_{j}(h)|$ grows faster than $|\nu_{0\infty}(h)|$  when $h\to 0$.
	In this case, not only the set of unstable directions accumulate to \eqref{eq:accumulation} for each $h$, 
	but also when $h\to0$ radially then each particular branch of the set of unstable $(h,\theta)$ tends asymptotically to the same radial limit 
		\[ \theta+(s+\tilde s)\arg(h) \in \tfrac{\pi}{2}-\arg(c)-\tfrac12\arg D(0)\mod\pi\Z.\]
	See Figure~\ref{figure:k3c}.
\end{enumerate}

\subsubsection{Local zones relative to $B_h$.}\label{sec:6-localzones}
The same theory as we have sketched on the preceding pages can be adapted to restriction of the vector field $e^{i\theta}h^s\bY_h $ to the domain $B_h$ \eqref{eq:Bh}. 
In this case the role of the poles $\xi_1=0$ and $\xi_1=\infty$ is played by the two whole exterior discs $\CP^1\sminus B_h=\{|\xi_1|\leq\frac{|h|}{\delta_1}\}\cup \{|\xi_1|\geq\delta_1\}$.
Their images by $\bt_h$ in the Riemann surface of $\bt_h$ are referred to as \emph{holes} in the surface \Grn{(Figure~\ref{figure:lavaursdomain}).}

\begin{definition}\label{def:rotstableBh}
	The vector field $e^{i\theta}h^s\bY_h$ is \emph{rotationally stable relative to $B_h$} if no trajectory through a point $\xi_*\in B_h$ escapes $B_h$ in both positive and negative time.
	\Grn{Such pair $(h,\theta)$ is then also called \emph{stable relative to $B_h$} if no trajectory of $e^{i\theta}h^s\bY_h$ escapes $B_h$ in both positive and negative time (cf. Definition~\ref{def:stable}).}
\end{definition}

\begin{lemma}\label{lemma:rotstableB} 
	For $|h|<\delta_2$, with $\delta_2$ small enough, the vector field $e^{i\theta}h^s\bY_h$ is rotationally stable relative to $B_h$ if and only if 
	no straight line $t_*+e^{i\theta}h^s\R$ in the Riemann surface of $\bt_h$ intersects two different holes.
\end{lemma}
 
\begin{lemma}\label{lemma:tangency} 
For $|h|<\delta_2$, with $\delta_2$ small enough, the vector field $e^{i\theta}h^s\bY_h$ has exactly $2kp$ tangency points with the outer, resp. inner, boundary circle of $B_h$.	
\end{lemma}

\begin{proof}
\Grn{Let $\xi_1=\delta_1\zeta$ with $|\zeta|=1$ be a point on the outer boundary circle of $B_h$ at which  
 the vector field $e^{i\theta}h^s\bY_h$ is tangent to the circle, and therefore perpendicular to the radial direction of $\xi_1$,}
so the equation of tangency is $e^{i(\theta+s\arg h+\arg c)}P_h(\delta_1\zeta)\in i\R$. 
This is equivalent to
\[e^{i(\theta+s\arg h+\arg c)}P_h(\delta_1\zeta)=-e^{-i(\theta+s\arg h+\arg c)}\ov{P_h}(\delta_1\zeta^{-1}),\]
as $\ov\xi_1=\delta_1\zeta^{-1}$,
which after a multiplication by $\zeta^{kp}$ becomes a polynomial equation of order $2kp$ for $\zeta$,
depending analytically on $(h,\ov h)$, so it cannot have more then $2kp$ roots.
\Grn{At the same time we know there are at least $2kp$ such tangencies, since there are $2kp$ separatrices which cross the boundary circle alternatingly entering and leaving, hence there must be at least one point of tangency in between each two separatrices.}  	
\end{proof}

\begin{proof}[Proof of Lemma~\ref{lemma:rotstableB}]
The only thing to prove is that for  $|h|<\delta_2$ small enough, no  straight line $t_*+e^{i\theta}h^s\R$ in the Riemann surface of $\bt_h$ 
intersects the same hole twice. We know that for $\delta_2$ small, the real phase portrait of $e^{i\theta}h^s\bY_h$ near the boundary of $B_h$ looks like in 
Figure~\ref{figure:ends} with the outer boundary of $B_h$ intersecting the $2kp$ outer ends in $2kp$ arcs. 
If a trajectory enters and leaves $B_h$ through two different arcs 
then the period $\nu_\gamma$ along a curve $\gamma$ consisting of the intersection of the trajectory with $B_h$ extended on each side towards the pole within the given end
is non-null, therefore the two ends correspond to two different holes in the Riemann surface of $\bt_h$, which goes against the assumptions of Definition~\ref{def:rotstableBh}.
So assume that the trajectory slices twice through a boundary arc associated to the same end. By a simple topological consideration this means that there would have to be at least 2 points of tangency between the arc and the vector field. But by Lemma~\ref{lemma:tangency} the total number of tangencies between the outer boundary circle $\{|\xi_1|=\delta_1\}$ and the vector field is $2kp$ for all $h$, and therefore each of the $2kp$ arc has exactly one such point. Symmetrically for the inner boundary circle.
\end{proof}

\begin{definition}
The role of the separatrix graph is played by the set of all trajectories of $e^{i\theta}h^s\bY_h$ that escape $B_h$ in either positive or negative time.
\Grn{We call the connected components of its complement \emph{local zones relative to $B_h$} -- they are spanned by complete real trajectories inside $B_h$.} 

\Grn{An \emph{inner/outer end of a local zone} is (a neighborhood of) the point where it touches the inner/outer boundary of $B_h$,
i.e. one of the $2kp$ tangency points between  $e^{i\theta}h^s\bY_h$ and the boundary (Lemma~\ref{lemma:tangency}).} 
A sepal (resp. $\alpha\omega$-) zone relative to $B_h$ of a rotationally stable vector field relative to $B_h$ has exactly one (resp. two) such tangencies. 
A local $\alpha\omega$-zone is split to two \emph{local half-zones} following the same gate trajectory as before:
if $|h|$ is small enough then 
the width of the $\alpha\omega$-zone, which tends to infinity as $|h|\to 0$ (Lemma~\ref{lemma:asymptotic-h}), is substantially larger than the diameters of the holes, which are bounded, thus the gate trajectory of each local $\alpha\omega$-zone is contained in $B_h$ and therefore in the zone. 
\end{definition}

\begin{remark}\label{rem:size-hole}
In the construction of the Fatou coordinate on a Lavaurs domain $\Omega_{h}$ (page \pageref{page:admissiblestrip}), we need not only a line $t_*+e^{i\theta}h^s\R$ inside 
$\bt_h(B_h)$, the existence of which is assured by the relative rotationally stability (Definition~\ref{def:rotstableBh}), but a whole admissible strip of $\bt_h(\Sigma_h)$ of width $\sim \sin\theta \cdot h^s$ contained in $\bt_h(B_h)$. 
The size of the holes is uniformly bounded when $h\to 0$ and is commensurate to $\delta_1^{-kp}$ 
\Grn{(since the restriction of $\bY_h^{-1}$ to the outer complement of $B_h$ is uniformly bounded and tends to $\bY_0^{-1}=\frac{d\xi_1}{c\,\xi_1^{kp+1}}$, and likewise on the inner complement),
and their distance tends to infinity wit order at least $|h|^{-\check sk}$ by Lemma~\ref{lemma:asymptotic-h},
so one can easily arrange by changing $\delta_1$ tiny little bit that for each local zone there is always an admissible strip as well.}
\end{remark}

\subsubsection{Lavaurs domains revisited.}\label{sec:6.2omega}

The naive idea of construction of the Lavaurs domains would be to follow the same construction as in \ref{sec:6.2domains} of enlarging the relative half-zones through stable variation of $\theta\in\,\,]\delta_3,\pi-\delta_3[$, to obtain again a collection of sectors $S$ in the $h$-space, and for each $S$ a collection of $2kp$ inner and $2kp$ outer  domains relative to $B_h$, which are simply connected and possibly ramified over the equilibrium set $\{P(\xi)=0\}$.
However, the union of such constructed inner and outer domains associated with sectors $S$ may fail to cover any neighborhood $\tilde B_h\subset B_h$ of a fixed radius $\tilde\delta_1$ (independent of $h$ when $S\ni h\to 0$).
Luckily, this issue can be remedied by extending these domains through iteration by $\phi^{\circ p}$ as in Definition~\ref{def:Omegasaturation}. 
There is no need to extend them beyond the boundaries of the respective ``global'' domains $Z_h$ of Section~\ref{sec:6.2domains}, since the union of these domains $Z_h$ already covers $\C^*=B_h$ (Remark~\ref{remark:covering}).

But we shall proceed in the opposite way, and construct the Lavaurs domains $\Omega_h$ in two steps:
\begin{itemize}
	\item construct a ``global'' domain $Z_h$ the same way as in Section~\ref{sec:6.2domains} but with variation of the angle $\theta\in K_h$ that is stable relative to $B_h$, 
	\item restrict this domain $Z_h$ to $B_h$ and remove also all the points that may potentially not be accessible through the iteration by $\phi^{\circ p}$ from the relative domain: in the coordinate $\bt_h$, these inaccessible points lie in a ``shade'' of the hole corresponding to the end of $Z_h$. 
\end{itemize}

\begin{figure}[t]
	\centering
	\includegraphics{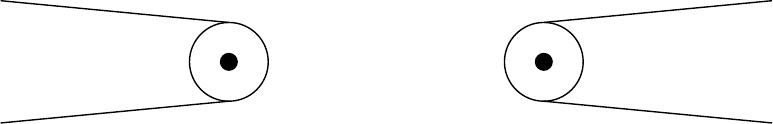}
	\caption{The negative shade (left) and the positive shade (right) of a hole in the coordinate $\bt_h$ (picture with $s\arg h=0\mod\pi\Z$).}
	\label{figure:shade}
\end{figure}  

\begin{definition}\label{def:shade}
For a hole $H_h\subset \bt_h(\CP^1\sminus B_h)$, its \emph{positive}, resp. \emph{negative}, \emph{shade} is the set of points $\bigcup_{t_*\in H_h}\{t:-\delta_3\leq\arg(h^{-s}(t-t_*))\leq\delta_3\}$, resp.  $\bigcup_{t_*\in H_h}\{t:-\delta_3\leq\arg(h^{-s}(t_*-t))\leq\delta_3\}$ (see Figure~\ref{figure:shade}).	
\end{definition}

The following construction of the Lavaurs domains is best understood from Figure~\ref{figure:k4}.

\begin{definition}[Lavaurs domains]\label{def:LavaursOmega}
	Given $\delta_1,\delta_2,\delta_3>0$, let $B_h$ be \eqref{eq:Bh} for $|h|<\delta_2$, let $K$ be a connected component of
	\begin{equation}\label{eq:Krelative}
	\{(h,\theta): |h|<\delta_2,\ \theta\in\,\,]\delta_3,\pi-\delta_3[,\ e^{i(\theta+s\arg(h))}\bY_h\ \text{rot. stable w.r.t. }B_h \},
	\end{equation}
	and let $S\subseteq \{|h|<\delta_2\}$ be the image of $K$ by the projection $(h,\theta)\mapsto h$ (possibly a ramified set defined on the covering surface of $\C^*$).
	For $h\in S$ let $K_h:=\{\theta: (h,\theta\in K)\}$ be the maximal open interval of stable angles relative to $B_h$.
	For any $(h,\theta)\in K$ let $Z_{h,\theta}$ be an outer or inner half-zone of $e^{i(\theta+s\arg(h))}\bY_h$ (i.e. global, relative to $\C^*$), depending continuously on $(h,\theta)\in K$,
	and let
	\begin{equation}\label{eq:domainZrelative}
	Z_h=\bigcup_{\theta\in K_h}Z_{h,\theta}.
	\end{equation}
	By the construction, each $\bt_h(Z_{h,\theta})$ intersects the same unique hole $H_h$, corresponding to the end of the zone.
	Depending whether the this hole is on the right or the left of $\bt_h(Z_{h,\theta})$ with respect to the direction  $e^{i\theta}h^s\R_{>0}$,
	let $\Omega_{S,h}$ be the complement in $Z_h$ \eqref{eq:domainZrelative} of the positive/negative shadow of the hole $H_h$ (Figure \ref{figure:k4}).
	Finally, let 
	\[\Omega_S=\coprod_{h\in S}\Omega_{S,h}.\]
	We call $\Omega_S$ a \emph{Lavaurs domain}.
\end{definition}

\begin{figure}
\begin{subfigure}{.45\textwidth} \includegraphics[scale=0.8]{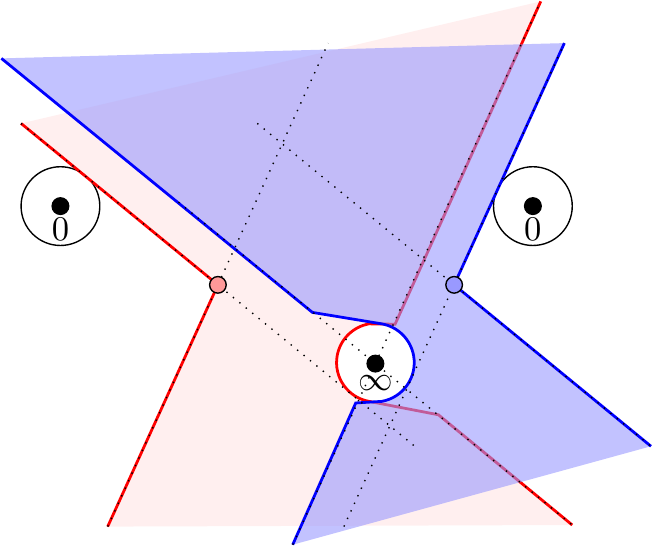} \caption{$P_0(h)^2\neq 4h^p$} \label{figure:k4a} \end{subfigure}
\quad
\begin{subfigure}{.45\textwidth} \includegraphics[scale=0.8]{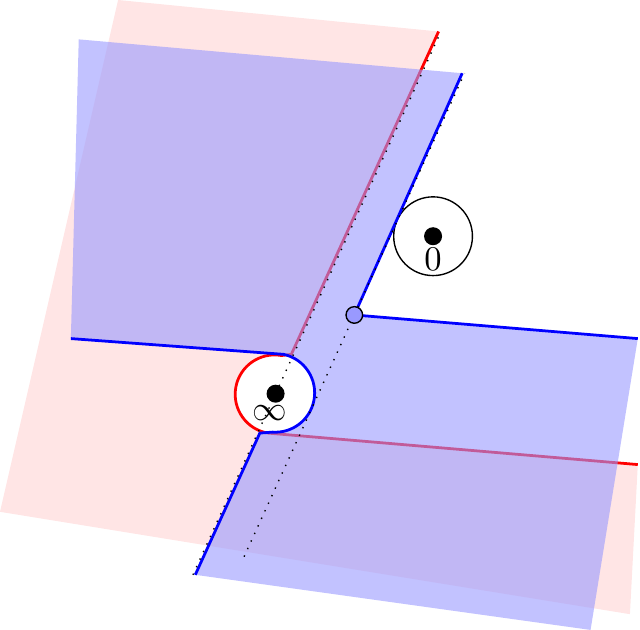} \caption{$P_0(h)^2=4h^p\neq 0$} \label{figure:k4b} \end{subfigure}
\\
\centering
\begin{subfigure}{.5\textwidth} \includegraphics[scale=0.8]{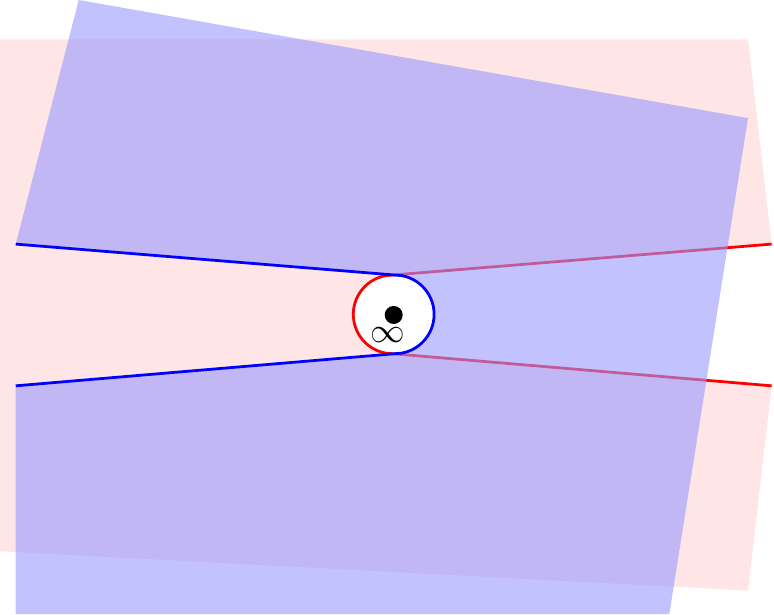} \caption{$h=0$} \label{figure:k4c} \end{subfigure}
\caption{Image of a pair of outer Lavaurs domains on the translation surface of $\bt_h$ for $k=1$.
		Compare with Figure~\ref{figure:k2}.}
	\label{figure:k4}
\end{figure}

\begin{lemma}
For $\delta_2$ small enough, each of the above Lavaurs domains $\Omega_S$ of Definition~\ref{def:LavaursOmega} is contained inside one of the saturated Lavaurs domains of Definition~\ref{def:Omegasaturation}.
\end{lemma}

\begin{proof}
	The distance between two holes (corresponding to images of the complement of $B_h$ in the leaf $\{h=\const\}$ by $\bt_h$) on the same sheet of the surface $\bt_h(B_h)$ tends to infinity when $|h|\to 0$ (Lemma~\ref{lemma:asymptotic-h})
	so if $|h|$ is small enough then the $\bt_h$-image of the saturated Lavaurs domain contains an admissible strip $\bt_h(\Sigma_h)$. Moreover, $|\phi_h^{\circ np}-\bt_h|\leq n\sup|\Delta|$,
	hence if $\sup|\Delta|<\sin\delta_3$ then all the points in $\bt_h(\Omega_{S,h})$ are accessible from $\bt_h(\Sigma_h)$ through iteration.
\end{proof}

\begin{theorem}[Covering theorem]\label{thm:covering}
	For any $\delta_1,\delta_2,\delta_3>0$ sufficiently small there exist $\tilde{\delta_1},\tilde{\delta_2}$, such that 
	the collection of the cuspidal sectors $S$ cover a disc $\{|h|<\tilde\delta_2\}$, and for each given sector $S$ and for all $h\in S$ the $4kp$ associated Lavaurs domains $\Omega_{S,h}$ cover together 
	$\tilde B_h\sminus\{P_h=0\}$, $\tilde B_h=\{\frac{|h|}{\tilde\delta_1}<|\xi_1|<\tilde\delta_1\}$, i.e. the union of the $\Omega_S$'s covers 
	$\tilde B_S\sminus\{P=0\},\qquad \tilde B_S:=\coprod_{h\in S}\tilde B_h$. See Figure~\ref{figure:sectors}.
\end{theorem}

\begin{proof}
Whenever the sector $S$ is non-trivial, the interval $K_h$ of variation of $\theta$ is non-empty for $h\in S$, and the union of the $2kp$ inner and $2kp$ outer domains $Z_h$ 
\eqref{eq:domainZrelative} covers \Grn{the whole leaf $\{h=\const\}\sminus\{P_h=0\}$} (cf. Remark~\ref{remark:covering}).	
Therefore, the Lavaurs domains $\Omega_{S,h}$ cover all $B_h\sminus\{P_h=0\}$ except of the points whose image in the coordinate $\bt_h$ lie in the intersection of the positive and the negative shade of a hole.  
Each hole is contained in some disc of uniformly bounded radius $\leq R\sim \frac{1}{kp|c|}\delta_1^{-kp}$ \Grn{(cf. Remark~\ref{rem:size-hole}),}
and it is easy to see that the intersection of the positive and the negative shades of the hole is therefore contained in the disc of radius $\frac{R}{\cos\delta_3}$
(Figure~\ref{figure:shade}), which in turn lies in some bigger hole for some $\tilde\delta_1<\delta_1$ for all $|h|<\delta_2$.

What we need to show is that the different sectors $S$ cover together some disc $\{|h|<\tilde\delta_2\}$,
i.e. that for each $h$ there is at least one rotationally stable direction $\theta$ relative to $B_h$.	
Since the holes have uniformly bounded radii $\leq R$, and their distance tends to infinity with rate at least $|h|^{-\check sk}$ (Lemma~\ref{lemma:asymptotic-h}), their effect on the local zones and on the relative rotational stability is smaller the smaller $|h|$ is.	
By Proposition~\ref{prop:rotstable} there are two kinds of saddle connections in the rotating family: countably many of those that lie inside an annular domain (if such annulus appears, then in our situation it is unique), and a finite number of other ones.
In the surface of $\bt_h$, the periodic annulus correspond to an infinite strip with a periodic series of poles on each boundary line, with a hole around each pole.
If $|h|$ is small enough, the size of the holes is small compared to the distance between the holes, and there are plenty of ways to choose a direction $\theta$ transverse to the strip such that no line $t_*+e^{i\theta}h^s\R$ intersects two holes, \Grn{meaning that $\theta$ is stable relative to $B_h$ (Definition~\ref{def:stable})}: in fact the Lebesgue measure of the set of bad directions can be made as small as one likes by restricting $\delta_2$.
The finite number of unstable directions corresponding to the other saddle connections gives rise to an additional set of intervals of instability relative to $B_h$,
but again the Lebesgue measure of them can be made small, and therefore one can always find an angle $\theta$ that avoids them too.
\end{proof}

\begin{proposition}\label{prop:boundedness}
	Let $\Omega_{S}$ be one of the domains of Theorem~\ref{thm:covering}, and $\Psi_{\Omega_{S}}$ the normalizing transformation \eqref{eq:LavaursPsi} on $\Omega_S$. Then $\Psi_{\Omega_{S}}$ is bounded on $\Omega_{S}$.
\end{proposition}

\begin{proof}
Let $\Omega_S$ be one of the domains of Theorem~\ref{thm:covering} as constructed in Definition~\ref{def:LavaursOmega}, and let $\tilde\Omega_S\subseteq\Omega_S$
be the localized version of Section~\ref{sec:6-localzones} of the domain $Z_S$ \eqref{eq:Z_S} of Section~\ref{sec:6.2domains}, both associated to the same connected component $K$ of \eqref{eq:Krelative} consisting of those $(h,\theta)$ that are stable with respect to given neighborhood $B$ of $0\in\C^2$.
So $\bt_h(\tilde\Omega_{S,h})$ can expressed as a union of open strips of varying angles $\theta\in K_h$ in the surface $\bt_h(B_h)$.
The claim is that it is enough to show that $\Psi_{\Omega_{S}}$ is bounded on $\tilde\Omega_{S}$.
Indeed, the width of the complement $\bt_h(\Omega_{S,h})\sminus \bt_h(\tilde\Omega_{S,h})$ is uniformly bounded. Therefore if $s=0$, it takes only a finite number of iterations by $\phi^{\circ p}$ to extend $\Psi_{\Omega_{S}}$ from $\tilde\Omega_S$ to $\Omega_S$ and the boundedness is preserved.
If $s\geq 1$, then one divides the $h$-space into concentric rings and iterates on each by $\phi^{\circ np}$, $n\in \Z_{>0}$, as in Lemma~\ref{lemma:hn}. 

So let us show that $\Psi_{\Omega_{S}}$ is bounded on $\tilde\Omega_{S}$. Up to restricting the radius $\delta_1$ of $B$ a bit, one can assume that the set $K$ is compact,
and therefore so is its projection $S$ and each interval $K_h$, $h\in S$, and that $\bt_h(\tilde\Omega_{S,h})$ is also closed. Hence the compact closure of $\tilde\Omega_S$ in the $\xi$-space consist of $\tilde\Omega_{S}$ and of (a part of) the divisor $\{P(\xi)=0\}$ (note that $0\in S$ is included by definition).
All we need to show is that $\Psi_{\Omega_{S}}$ extends continuously to $\{P(\xi)=0\}$ as identity. But this follows from the form of $\tilde\Omega_{S}$ and Theorem~\ref{thm:normalizngtransformation}.
\end{proof}

\subsection{Modulus of analytic classification}\label{sec:6.3}

 \subsubsection{Normalizing cochains, outer cocycles and analytic classification.}\label{sec:6.3.1}

 Constructed in the previous section, Theorem~\ref{thm:covering}, we have a covering of a neighborhood of 0 in the $h$-space by a collection of at most countably many 
 cuspidal sectors (shortly just \emph{sectors}). Over each sector $S$ we have a family of $2kp$ outer and $2kp$ inner Lavaurs domains, $\Omega_S=\coprod_{h\in S}\Omega_{S,h}$, covering together each local leaf $B_h$ \eqref{eq:Bh}, $h\in S$, and hence the domain
 \[B_S=\coprod_{h\in S} B_h,\]
 in the $\xi$-space.
 \Grn{This covering is $(\sigma,\Lambda)$-invariant: if $\Omega_S$ is a domain from this covering then so are its images $\Lambda^n(\Omega_S)$ and $\sigma\Lambda^n(\Omega_S)$ for all $n=1,\ldots,p$.}
 
 By the results of Section~\ref{sec:6.1} and Proposition~\ref{prop:boundedness}, on each of these domains $\Omega_S$ there is a bounded normalizing transformation $\Psi_{\Omega_{S}}$ such that
 \[\Psi_{\Omega_{S}}\circ\phi^{\circ p}=\phimod^{\circ p}\circ\Psi_{\Omega},\qquad \phimod^{\circ p}=\exp(h^s\bY).\]

  For a given sector $S$, and $h\in S$, a pair of neighboring Lavaurs domains $\Omega_{S,h}$, $\Omega_{S,h}'$ (i.e. corresponding to neighboring half-zones) can have two kinds of intersections (Figure~\ref{figure:intersections}): 
  \begin{itemize}
  	\item an intersection corresponding to a separatrix, going from an equilibrium point to outer (resp. inner) boundary, called \emph{outer intersection} (resp. \emph{inner intersection}),
  	\item an intersection corresponding to a gate between two halves of the same $\alpha\omega$-zone, going from one equilibrium to another, called \emph{gate intersection}.
  \end{itemize}

\begin{figure}[t]
\begin{subfigure}[t]{0.5\textwidth} \hskip-12pt \includegraphics[scale=0.8]{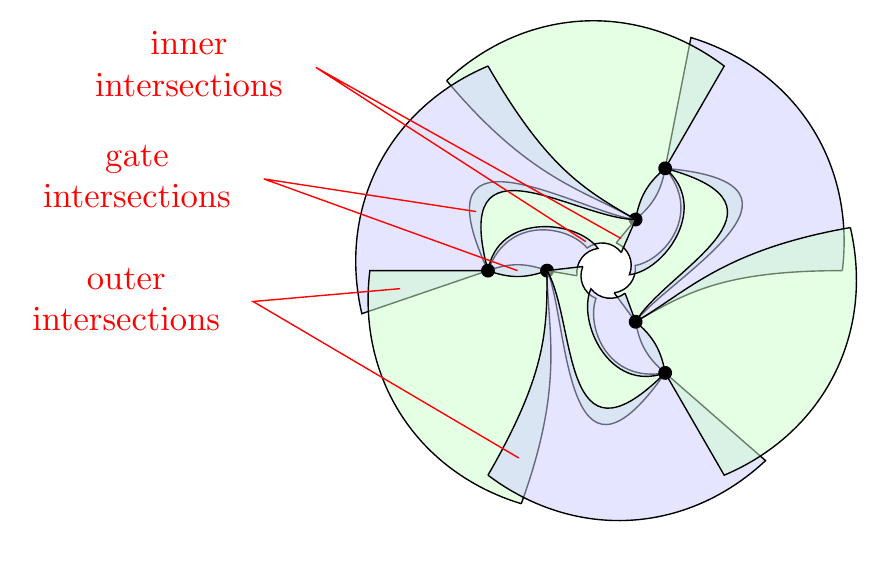} \caption{}\label{figure:intersections}\end{subfigure}
\qquad
\begin{subfigure}[t]{0.4\textwidth}  \includegraphics[scale=0.8]{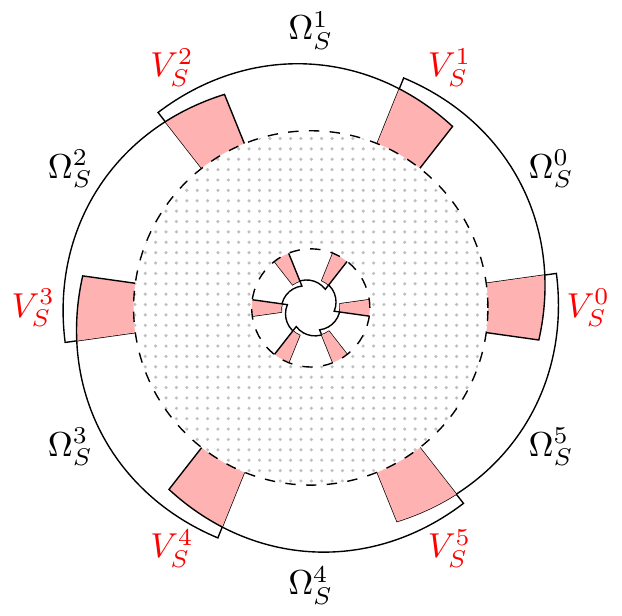} \caption{}\label{figure:cocycle}\end{subfigure}
\caption{(a) Example (with $k=1$, $p=3$) of a covering of $B_h\sminus\{P_h=0\}$ by the Lavaurs domains $\Omega_S$ and their intersections.  
	 (b) The outer and inner Lavaurs domains near the boundary of $B_h$, and the corresponding outer and inner intersections (in red). Inside the middle ring (dotted) the organization of the domains may be complicated; this is where the gate intersections are situated.}
\end{figure}

 \begin{theorem}[Existence of an equivariant normalizing cochain]\label{thm:cochain} 
 	For each of the sectors $S$, on the associated covering of $B_S\sminus\{P=0\}$ by
 	the $4kp$ Lavaurs domains $\Omega_S$, there exists a \emph{normalizing cochain} $\{\Omega_S\mapsto\Psi_{\Omega_S}\}$ 
 	consisting of bounded analytic transformations such that\footnote{ \Grn{Note that in \textit{(1)} and \textit{(2)} the way the conjugation by a normalizing cochain works is by composition with maps associated to different domains from the same covering on the two sides of the identies.}}
  	\begin{enumerate}[label=(\arabic*)]
 		\item  it conjugates $\phi$ and the model $\phimod=\Lambda\exp(\tfrac1p h^s\bY)$ 
 		\[\Psi_{\Lambda^{n}(\Omega_S)}\circ\phi^{\circ n}=\phimod^{\circ n}\circ\Psi_{\Omega_S}, 
 		\quad n=1,\ldots, p, \] %
 		\item  it is $\sigma$-equivariant, that is
 		\[\Psi_{\sigma(\Omega_S)}\circ\sigma=\sigma\Psi_{\Omega_S}, \] 
 		\item  \Grn{if $\Omega_S$ and $\Omega_S'$ share a gate intersection then}
 		\[\Psi_{\Omega_S}=\Psi_{\Omega_S'}\qquad \text{on the gate intersection.}\] 
 	\end{enumerate}

    This normalizing cochain is unique up to left composition with some cochain $\{\Omega_S\mapsto\exp(C_{\Omega_S}(h)\bY)\}$ of flow maps of $\bY$:
 	\begin{equation}\label{eq:changecochain}
 	\Psi_{\Omega_S}'=\exp(C_{\Omega_S}(h)\bY)\circ\Psi_{\Omega_S},
 	\end{equation}
 	where each $C_{\Omega_S}$ is bounded analytic on $S$ and
\begin{enumerate}[label=(\arabic*')]
	\item  $C_{\Lambda^{n}(\Omega_S)}=C_{\Omega_S}$, $n=1,\ldots,p$, 
	\item  $C_{\sigma(\Omega_S)}=-C_{\Omega_S}$,
	\item  $C_{\Omega_S}=C_{\Omega_S'}$ \Grn{whenever $\Omega_S$ and $\Omega_S'$ share a gate intersection.}
\end{enumerate}
\end{theorem}

 \begin{proof}
 	\Grn{Let us prove it in the formal case (b) when $\Lambda=\left(\begin{pmatrix}\lambda&0\\0&\lambda^{-1}\end{pmatrix}\right)$ is diagonal. In the formal case (c) one needs to replace $\Lambda=-\sigma$ by $\sigma\Lambda=-I$ and $\phi$, $\phimod$ by $\sigma\phi$, $\sigma\phimod$ in some of the arguments.}
 	
 	The existence of bounded normalizing transformations $\Psi_{\Omega_S}$ such that $\Psi_{\Omega_S}\circ\phi^{\circ p}=\exp(h^s\bY)\circ\Psi_{\Omega_S}$ have been proved in Section~\ref{sec:6.1} and in Proposition~\ref{prop:boundedness}.
 	Let us show that they can be chosen so that they satisfy the conditions \textit{(1)--(3)}.
 	
 	First of all, let us note there is no potential conflict between the conditions \textit{(1)} and \textit{(2)}, since rotation by $\Lambda$ maps outer domains to outer domains and inner to inner, while $\sigma$ switches between inner and outer.
 	So we divide the domains into their $(\sigma,\Lambda)$-orbits: there is $2k$ of them each consisting of $2p$ domains, and on each orbit we ensure the conditions 
 	\textit{(1)} \& \textit{(2)}.
 	Now we consider a graph structure on the space of orbits where two orbits are connected by an edge if a domain in one share a gate intersection with a domain in the other. This partitions the space of orbits into components, which can be only of the following types:
 	\begin{itemize}
 		\item[-] cycle of order 0: a single orbit whose domains are of sepal type, i.e. have no gates,
 		\item[-] cycle of order 1: a single orbit whose $2p$ domains are organized into $p$ pairs sharing $p$ gates,
 		\item[-] cycle of order 2: two orbits with $2p$ gates between the $2p$ domains of one orbit and the $2p$ domains of the other orbit. 
 	\end{itemize} 
 	Let us show that the condition \textit{(3)} can be satisfied on each component in either of the above cases.
 
 	\begin{itemize}
 	\item[-] 	There is nothing to show in the case of a cycle of order 0. 	
 	\item[-] In the case of cycle of order 1, let  $\Omega_S$ and $\tilde\Omega_S$ be two different domains in the same orbit sharing a gate intersection. It is impossible that
 	$\tilde\Omega_S=\Lambda^{n}\Omega_S$ for some $n\in\Z_{p}\sminus\{0\}$ since if two outer, resp. inner, domains share a gate then they need to be of opposite parity
 	in the cyclic ordering of outer, resp. inner, domains (also the ends at $\infty$, resp. $0$, of the corresponding half-zones of the same $\alpha\omega$-zone have opposite parities, see p.~\pageref{page:ends}). 
 	Therefore $\tilde\Omega_S=\sigma\Lambda^{n}\Omega_S$ for some $n\in\Z_{p}$.
 	Assuming the conditions \textit{(1)} \& \textit{(2)}, let $\Psi_{\Omega_S}$ and $\Psi_{\tilde\Omega_S}=\sigma\exp(\frac{n}{p}h^s\bY)\circ\Psi_{\Omega_S}\circ\phi^{\circ(-n)}\circ\sigma$ be the bounded normalizing transformations.
 	Since $\Omega_S$ and $\tilde\Omega_S$ share a gate intersection (i.e. are basically two halves of the same domain), and $\Psi_{\Omega_S}(\xi)=\Psi_{\tilde\Omega_S}(\xi)\mod P\xi$ (Theorem~\ref{thm:normalizngtransformation}), this means that
 	$\Psi_{\tilde\Omega_S}=\exp(C(h)\bY)\circ\Psi_{\Omega_S}$ for some $C(h)$ bounded analytic on $S$ (by Propositions~\ref{prop:Fatou} and~\ref{prop:boundedness}),
 	which means we can simply replace $\Psi_{\Omega_S}$ by $\exp(\frac12C\bY)\circ\Psi_{\Omega_S}$ and $\Psi_{\tilde\Omega_S}$ by 
 	$\exp(-\frac12C(h)\bY)\circ\Psi_{\tilde\Omega_S}=\sigma\exp(\frac{n}{p}h^s\bY)\circ[\exp(\frac12C(h)\bY)\circ\Psi_{\Omega_S}]\circ\phi^{\circ(-n)}\circ\sigma$.
 	We make the same change on the whole $\Lambda$-orbit which fixes the problem on all the $p$ gate intersections.
 	\item[-] The case of cycle of order 2 is easy: one takes a normalizing transformation on one of the orbits satisfying \textit{(1)} \& \textit{(2)}, and extends it by \textit{(3)} to the other orbit.
 	\end{itemize} 
 	 	
 	If $\{\Omega_S\mapsto\Psi_{\Omega_S}\}$ and $\{\Omega_S\mapsto\Psi'_{\Omega_S}\}$ are two normalizing cochains then
 	by Proposition~\ref{prop:Fatou} there exists a unique cochain $\{\Omega_S\mapsto C_{\Omega_S}(h)\}$ of bounded analytic maps on $S$ such that \eqref{eq:changecochain}. If they both satisfy \textit{(2)} then
 	\begin{align*}
 	\exp(C_{\sigma(\Omega_S)}\bY)\circ\Psi_{\sigma(\Omega_S)}=
 	{\Psi'}_{\sigma(\Omega_S)}=\sigma{\Psi'}_{\Omega_S}\circ\sigma&=	\sigma\exp(C_{\Omega_S}\bY)\circ\Psi_{\Omega_S}\circ\sigma\\ &=\exp(-C_{\Omega_S}\bY)\circ\Psi_{\sigma(\Omega_S)},
 	\end{align*}
 	hence  \textit{(2')}: $C_{\sigma(\Omega_S)}=-C_{\Omega_S}$.
 	Similarly for \textit{(1')} and \textit{(3')}.
\end{proof}

The point \textit{(2)} of Theorem~\ref{thm:cochain} means that the normalizing cochain is fully determined by the \emph{outer normalizing cochain} consisting of the normalizing transformations associated to the outer Lavaurs domains.
 
Let us label the outer Lavaurs domain in a counterclockwise cyclic order as
 \[\Omega_{S}^0,\ldots,\Omega_{S}^{2kp-1},\]
and the outer intersections as (see Figure~\ref{figure:cocycle})
\[V_{S}^0,\ldots,V_{S}^{2kp-1},\qquad V_{S}^j\subseteq\Omega_{S}^j\cap\Omega_{S}^{j-1}. \]

\begin{definition}[Outer cocycle]\label{def:outercocycle}
On the outer intersections $V_{S}^j$, $j\in\Z_{2p}$, we have transition maps between the normalizing transformations
\begin{equation}\label{eq:psitransition}
	\psi_{V_{S}^j}= \Psi_{\Omega_{S}^{j-1}}\circ\Big(\Psi_{\Omega_{S}^{j}}\Big)^{\circ(-1)}, 
\end{equation}
which commute with the model map: 
\[\psi_{V_{S}^j}\circ\exp(\bXmod)=\exp(\bXmod)\circ\psi_{V_{S}^j}.\]
By \textit{(1)} of Theorem~\ref{thm:cochain}, they satisfy 
\begin{equation}\label{eq:cocyclesymmetry}
\begin{cases}\text{(b)}\ \psi_{\Lambda^{n}(V_{S})}=\phimod^{\circ n}\circ\psi_{V_{S}^j}\circ\phimod^{\circ(-n)},& 
	\Lambda=\left(\begin{smallmatrix}\lambda & 0 \\ 0 &\lambda^{-1}	\end{smallmatrix}\right),\\
\text{(c)}\ \psi_{-I(V_{S})}=\sigma\phimod\circ\psi_{V_{S}^j}\circ\big(\sigma\phimod\big)^{\circ(-1)},& 
\sigma\Lambda=-I.
\end{cases}
\end{equation}
The collection of these transition maps is defined up to conjugation by a cochain of flow maps
$\{\exp(C_{\Omega_{S}^0}(h)\bY),\ldots,\exp(C_{\Omega_{S}^{2kp-1}}(h)\bY) \}$, where $C_{\Omega_{S}^j}$ are bounded analytic on $S$:
\begin{equation}\label{eq:conjugationpsi}
	\psi_{V_{S}^j}'=\exp(C_{\Omega_{S}^{j-1}}(h)\bY)\circ\psi_{V_{S}^j}\circ\exp(-C_{\Omega_{S}^{j}}(h)\bY). 
\end{equation}
We call the equivalence class of the collection $\{\psi_{V_{S}^0},\ldots,\psi_{V_{S}^{2kp-1}}\}$ by this conjugacy an \emph{outer cocycle} associated to the sector $S$.
\end{definition}

\begin{remark}
It is natural to ask that the conjugating cochain $\{\exp(C_{\Omega_{S}^j}(h)\bY)\}$ should also be subject to conditions \textit{(1')--(3')} of Theorem~\ref{thm:cochain}. 
One can show that if two collections of transition maps $\{\psi_{V_{S}^j}\}$ and $\{\psi'_{V_{S}^j}\}$, associated to
two normalizing cochains $\{\Psi_{\Omega_{S}^j}\}$ and $\{\Psi'_{\Omega_{S}^j}\}$ satisfying conditions \textit{(1)--(3)} of Theorem~\ref{thm:cochain}, are conjugated,
then the conjugating cochain $\{\exp(C_{\Omega_{S}^j}(h)\bY)\}$ is unique and satisfies the conditions \textit{(1')--(3')}.
\end{remark}

\begin{definition}[Centralizer of the model]\label{def:centralizer}
Let $ \Cal Z^{\sigma,\Lambda}(\bXmod)$ be the group of analytic $(\sigma,\Lambda)$-equivariant transformations preserving $\bXmod=h^s\bY$.
By Theorem~\ref{thm:FNFX1}, it is a subgroup of the multiplicative group of transformations $\xi\mapsto e^{\frac{r\pi i}{2s+kp}}\sigma^r\xi$, $r\in\Z_{4s+2kp}$,
and agrees with the group of analytic $\sigma$-equivariant transformations preserving $\phimod$.
Note that if $\Lambda^2\neq\ I$ then $r$ must be even for $\sigma^r$ to commute with $\Lambda$. 
\end{definition}

Each element $e^{\frac{r\pi i}{2s+kp}}\sigma^r$ of the group $\Cal Z^{\sigma,\Lambda}(\bXmod)$ acts on the whole collection of the Lavaurs domains by a mutation:
sending a Lavaurs domain $\Omega_S$ over sector $S$ to a Lavaurs domain $\Omega'_{S'}=e^{\frac{r\pi i}{2s+kp}}\sigma^r(\Omega_S)$ over sector $S'=e^{\frac{2r\pi i}{2s+kp}}S$.
Its action on the collection of outer cocycles is the following
\begin{equation}\label{eq:mutation}
\psi_{V_{S}}\mapsto \sigma^r\psi_{e^{\frac{r\pi i}{2s+kp}}(V_S)}\circ\sigma^r.
\end{equation}

\begin{theorem}[Analytic classification]\label{thm:analyticclassification}~
 Two $\sigma$-reversible germs $\phi,\phi'=\Lambda\xi+\hot$ with a first integral $h$ from the same model class are:	
\begin{enumerate}[wide=0pt, leftmargin=\parindent]	
\item analytically equivalent by a $\sigma$-equivariant transformation that is tangent to the identity if and only if for every sector $S$ their associated outer cocycles are equal.
\item analytically equivalent by a $\sigma$-equivariant transformation  if and only if 
there exists an element of the centralizer $\Cal Z^{\sigma,\Lambda}(\bXmod)$ such that the collection of outer cocycles of $\phi$ is equal to the image of that of $\phi'$ by the action \eqref{eq:mutation}.
\end{enumerate}
\end{theorem}

The proof will be given in Section~\ref{sec:6proof}.

\subsubsection{Fourier representation of a cocycle.}\label{sec:6.3.2}

For an outer domain $\Omega_{S}^j$ let the coordinate $\bt$ be as in Lemma~\ref{lemma:t}, i.e. one that  vanishes at the outer pole $\xi_1=\infty$, 
which is the same one for all the outer domains,
and 
let $\bT_{\Omega_{S}^j}$ defined by
\[\bt\circ\Psi_{\Omega_{S}^j}=\bT_{\Omega_{S}^j}\]
be the Fatou coordinate for $\phi^{\circ p}$, $\bT_{\Omega_{S}^j}\circ\phi^{\circ p}=\bT_{\Omega_{S}^j} +h^s$.
Define $ \beta_{V_{S}^j}(t,h)$ by
\begin{equation}\label{eq:beta}
 t+\beta_{V_{S}^j}(t,h) :=\bT_{\Omega_{S}^{j-1}}\circ\Big(\bT_{\Omega_{S}^j}\Big)^{\circ(-1)}=\bt\circ\psi_{V_{S}^j}\circ\bt^{\circ(-1)},
\end{equation}
where $\psi_{V_{S}^j}$ is the transition map \eqref{eq:psitransition}.
Then $\beta_{V_{S}^j}(t,h)=\beta_{V_{S}^j}(t+h^s,h)$ is $h^s$-periodic and therefore can be expressed by its \emph{Fourier series}
\begin{equation}\label{eq:betafourier}
\beta_{V_{S}^j}(t,h)=\sum_{n\in\Z}\beta_{V_{S}^j}^{(n)}(h)e^{\frac{2\pi i nt}{h^s}}. 
\end{equation} 
Since $\beta_{V_{S}^j}(t,h)$ has at most a moderate growth when $t\to\infty$ in $\bt(V_{S}^j)$, the above sum is only either over $n\in\Z_{\geq0}$ or over $n\in\Z_{\leq0}$, depending whether the equilibrium to which the intersection domain $V_S^j$ is attached is attractive ($\IM(\frac{t}{h^s})\to+\infty$) or repulsive ($\IM(\frac{t}{h^s})\to-\infty$).

\begin{lemma}
	The Fourier coefficients $\beta_{V_{S}^j}^{(n)}(h)$ of a cocycle are bounded analytic on $S$ (i.e. bounded analytic on $S^*:=S\sminus\{0\}$ and continuous on $S$).
\end{lemma}
\begin{proof}
	Follows from the analyticity of the construction over the sector $S$ (cf. Proposition~\ref{prop:Fatou-dependence}.)	
\end{proof}

The effect of conjugation \eqref{eq:conjugationpsi} of the cocycle on $\beta_{V_{S}^0},\ldots,\beta_{V_{S}^{2kp-1}}$ is:
\begin{equation}\label{eq:conjugationbeta}
\beta_{V_{S}^j}'(t,h)=\beta_{V_{S}^j}(t-C_{\Omega_{S}^{j}}(h),h)+C_{\Omega_{S}^{j-1}}(h)-C_{\Omega_{S}^{j}}(h),
\end{equation}
in particular
\[{\beta_{V_{S}^j}'}^{(0)}(h)=\beta_{V_{S}^j}^{(0)}(h)+C_{\Omega_{S}^{j-1}}(h)-C_{\Omega_{S}^{j}}(h).\]
Since $\bt\circ\phimod=\bt+\frac{h^s}{p}$, the rotational symmetry \eqref{eq:cocyclesymmetry} means that 
\[
\begin{cases}\text{(b)}\ \beta_{\Lambda^{n}(V_{S}^j)}(t+\tfrac{n}{p}h^s,h)=\beta_{V_{S}^j}(t,h),& 
	\Lambda=\left(\begin{smallmatrix}\lambda & 0 \\ 0 &\lambda^{-1}	\end{smallmatrix}\right),\\
	\text{(c)}\ \beta_{-I(V_{S}^j)}(-t-\tfrac{1}{2}h^s,h)=-\beta_{V_{S}^j}(t,h),& 
	\sigma\Lambda=-I.
\end{cases}
\]

\begin{remark}
	Following Martinet \& Ramis \cite{MR}, in the coordinate $\bz=e^{\frac{2\pi i}{h^s}\bt}$ the outer cocycle maps 
	\[\bz\circ\psi_{V_{S}^j}\circ\bz^{\circ(-1)}=\bz\, e^{\frac{2\pi i}{h^s}\sum_{n}\beta_{V_{S}^j}^{(n)}(h)\bz^n}\]
	can be interpreted as a parametric family of analytic germs of diffeomorphisms of $(\CP^1,0)$, resp. $(\CP^1,\infty)$, called \emph{horn maps}, that serve as gluing maps of a string of $2kp$ spheres $\CP^1$ identified at their points $0$, resp. $\infty$.
	The inner cocycle defines another symmetric string of $2kp$ spheres. The spheres in the two strings are for $h\neq 0$ further identified through gate maps
	(which are the period shifts $\bt\mapsto\bt+\nu_\gamma(h)$ of Lemma~\ref{lemma:t} acting as $\bz\mapsto\bz\, e^{\frac{2\pi i}{h^s}\nu_\gamma(h)}$) to give a global representation of the orbit space of $\phi^{\circ p}$, albeit one that can be fairly complicated.
\end{remark}

\subsubsection{Proof of Theorem~\ref{thm:analyticclassification}.}\label{sec:6proof}

\begin{proposition}\label{prop:automorphisms}
	The group $\Cal Z_{\id,S}(\phi^{\circ p})$ of bounded analytic diffeomorphisms on $B_S=\coprod_{h\in S}B_h$ that are tangent to identity and commute with $\phi^{\circ p}$ is either:
	\begin{itemize}
		\item[-] continuous, if and only if \Grn{$\phi^{\circ p}=\exp(\bX)$ is embeddable in the flow of a vector field $\bX$ analytic on a neighborhood of $0\in\C^2$,} 
		in which case
		\[ \Cal Z_{\id,S}(\phi^{\circ p})=\{\exp(h^{-s} C(h)\bX):\ C(h)\ \text{bounded analytic on $S$}\},\] 
		or
		\item[-] discrete and equal to 
		\[ \Cal Z_{\id,S}(\phi^{\circ p})=\{\Theta^{\circ m}:\ m\in\Z\},\]
		where $\Theta$ is an analytic germ on a full neighborhood of $0\in\C^2$ such that $\phi^{\circ p}=\Theta^{\circ n}$ for some $n\in\Z_{>0}$.
	\end{itemize} 
	The subgroup $\Cal Z_{\id,S}^\sigma(\phi^{\circ p})$ of $\sigma$-equivariant diffeomorphisms tangent to identity is trivial.
\end{proposition}

\begin{proof}
	Let $F\in \Cal Z_{\id,S}(\phi^{\circ p})$, $F\circ \phi^{\circ p}=\phi^{\circ p}\circ F$.
	Given the Fatou coordinate $\bT_{\Omega_{S}^j}$ for $\phi^{\circ p}$  on an outer Lavaurs domain $\Omega_{S}^j$, then
	$\bT_{\Omega_{S}^j}\circ F$ is also a Fatou coordinate for $\phi^{\circ p}$ on $\Omega_{S}^j$ with at most moderate growth, so by Proposition~\ref{prop:Fatou}, 
	\begin{equation}\label{eq:3805}
		\bT_{\Omega_{S}^j}\circ F=\bT_{\Omega_{S}^j}+C_{\Omega_{S}^j}(h),
	\end{equation}
	for some $C_{\Omega_{S}^j}(h)$ bounded analytic on $S$.
	Similarly on a neighboring Lavaurs domain $\Omega_{S}^{j-1}$, so on the intersection $V_{S}^j\subseteq \Omega_{S}^{j}\cap\Omega_{S}^{j-1}$ we have by \eqref{eq:3805}, \eqref{eq:beta} and \eqref{eq:conjugationbeta}
	\[\beta_{V_{S}^j}(t,h)=\beta_{V_{S}^j}(t-C_{\Omega_{S}^j},h)+C_{\Omega_{S}^{j-1}}(h)-C_{\Omega_{S}^j}(h).\]
	This means that $C_{\Omega_{S}^{j-1}}(h)=C_{\Omega_{S}^j}(h)=:C(h)$ is such that 
	$\beta_{V_{S}^j}^{(l)}(h)e^{\frac{2\pi i l C(h)}{h^s}}=\beta_{V_{S}^j}^{(l)}(h)$ for all $l$,
	i.e. 
	\begin{itemize}
		\item[-] either $\beta_{V_{S}^j}^{(l)}=0$ for all $l\in\Z\sminus\{0\}$, meaning that $\psi_{V_{S}^j}=\id$,
		\item[-] or $C(h)=h^s\frac{m}{n}$ for some $m,n$ relatively prime, and $\beta_{V_{S}^j}^{(l)}=0$ for all $l\in\Z\sminus n\Z$,
		meaning that $\psi_{V_{S}^j}$ commutes with $\exp(\frac{1}{n}\bXmod)$.
	\end{itemize}
	And this has to be true for all $j=0,\ldots,2kp-1$.
	
	In the first case, if $\psi_{V_{S}^j}=\id$ for all $j$, then the Lavaurs vector fields $\bX_{\Omega^j_{S}}$ \eqref{eq:LavaursX} 
	glue up together as one analytic vector field $\bX_S$ on $B_S$, such that $\phi^{\circ p}=\exp(\bX_S)$.
	If $\tilde S$ is another sector with a nontrivial intersection $\tilde S\cap S$, then  on this intersection the Lavaurs vector fields 	$\bX_{\tilde\Omega_{\tilde S}^j}$ associated to domains of $\tilde S$ must agree with $\bX_S$ due to their uniqueness, and therefore they glue up to $\bX_{\tilde S}=\bX_S$. 
	In the end this means that there is just one analytic Lavaurs vector field $\bX$ on a full neighborhood of $0\in\C^2$, such that $\phi^{\circ p}=\exp(\bX)$.
	
	In the second case, if $\psi_{V_{S}^j}$ commutes with $\exp(\frac{1}{n}\bXmod)$ for each index $j$, then the different maps 
	$\Theta_{\Omega_S^j}:=\exp(\frac{1}{n}\bX_{\Omega_{S}^{j}})=\Psi_{\Omega_{S}^{j}}^{\circ(-1)}\circ\exp(\frac{1}{n}\bXmod)\circ\Psi_{\Omega_{S}^{j}}$ glue up together to a single map $\Theta_S$ defined on $B_S$, such that $\Theta_S^{\circ n}=\phi^{\circ p}$ and $\Theta_S^{\circ m}=F$.
	Now again, if $\tilde S$ is another sector with nontrivial intersection $\tilde S\cap S$, then  on this intersection the Fatou coordinates on the
	domains associated to $\tilde S$ satisfy $\bT_{\tilde\Omega_{\tilde S}^j}\circ \Theta=\bT_{\tilde\Omega_{\tilde S}^j}+\frac{1}{n}h^s$,
	which means that also the outer cocycle $\psi_{\tilde V_{\tilde S}^j}$ commutes with $\exp(\frac{1}{n}\bXmod)$, and that  
	$\Theta_{\tilde\Omega_{\tilde S}^j}:=\exp(\frac{1}{n}\bX_{\tilde\Omega_{\tilde S}^{j}})$ agrees for all $j$ defining a single map $\Theta_{\tilde S}$ on $B_{\tilde S}$, and $\Theta_{\tilde S}=\Theta_S$ on $B_{S\cap\tilde S}$.
	In the end this means that the different $\Theta_S$ glue up to a single analytic $\Theta$ on a full neighborhood of $0\in\C^2$, such that $\Theta^{\circ n}=\phi^{\circ p}$.
	
The triviality of $\sigma$-equivariant elements follows from Lemma~\ref{lemma:automorphismofX}.
\end{proof}

\begin{proof}[Proof of Theorem~\ref{thm:analyticclassification}]
	Assume first that $\phi$ and $\phi'= G^{\circ(-1)}\circ\phi\circ G$ are equivalent by means of a $\sigma$-equivariant analytic transformation $G$.
	Let $G_0$ be the linear part of $G$. Then $G_0$ must commute with both $\sigma$ and $\Lambda$ (which is the linear part of both $\phi$, $\phi'$), and also to preserve the vector field $\bXmod$: in fact since both the infinitesimal generators of $\phi^{\circ p}$ and $\phi'^{\circ p}$ are formally conjugated to some $\hatbXnf=h^s\frac{cP}{1+\hat\mu cP}\bE$, then also $G_0^*\hatbXnf$ is formally conjugted to $\hatbXnf$, and since $G_0$ is linear and $G_0^*\hatbXnf$ is ``of the same form'' as $\hatbXnf$, then by Theorem~\ref{thm:FNFX1} $G_0^*\hatbXnf=\hatbXnf$, and therefore also $G_0^*\bXmod=\bXmod:=h^scP\bE$. 
	Hence $G_0\in\Cal Z^{\sigma,\Lambda}(\bXmod)$.
	Therefore if $\{\Omega_S\mapsto\Psi_{\Omega_S}\}$ is a normalizing cochain for $\phi$, $\Psi_{\Omega_S}\circ\phi^{\circ p}=\phimod^{\circ p}\circ\Psi_{\Omega_S}$
	satisfying the conditions \textit{(1)-(3)} of Theorem~\ref{thm:cochain},	
	then $\{\Omega_S\mapsto\Psi_{G_0(\Omega_S)}\circ G\}$ is a normalizing cochain for $\phi'$ which obviously gives rise to the same cocycle as $\{\Omega_S\mapsto\Psi_{\Omega_S}\}$ except on domains transported by $G_0$.
	And if $\{\Omega_S\mapsto\Psi'_{\Omega_S}\}$ is another normalizing cochain for $\phi'$ as in Theorem~\ref{thm:cochain}, then $G_0^{-1}\circ\Psi_{\Omega_S}\circ G\circ{\Psi'_{\Omega_S}}^{\circ(-1)}(\xi)=\xi+\mod P\xi$ commutes with $\phimod^{\circ p}=\exp(h^s\bY)$,
	and by Theorem~\ref{thm:normalizngtransformation},
	$G_0^{-1}\circ\Psi_{\Omega_S}\circ G\circ{\Psi'_{\Omega_S}}^{\circ(-1)}=\exp\big(C_{\Omega_S}(h)\bY\big)$ for some $C_{\Omega_S}(h)$ bounded analytic on $S$.
	Therefore the two cocycles are conjugated.
		
	Conversely, let $\phi^{\circ p}$, $\phi'^{\circ p}$ be two germs with the same model $\exp(h^s\bY)$, and assume the outer cocycles  $\{\psi_{V_{S}^0},\ldots,\psi_{V_{S}^{2kp-1}}\}$ and  $\{\psi'_{V_{S}^0},\ldots,\psi'_{V_{S}^{2kp-1}}\}$ associated to the normalizing cochains of Theorem~\ref{thm:cochain} are conjugated by \eqref{eq:conjugationpsi}.
	Then the corrected normalizing cochain $\{\Omega_S\mapsto\exp(-C_{\Omega_S}(h)\bY)\circ\Psi'_{\Omega_S}\}$ for $\phi'^{\circ p}$ also satisfies the conditions of Theorem~\ref{thm:cochain}, and defines the same cocycle as $\{\Omega_S\mapsto\Psi_{\Omega_S}\}$. 
	Similarly for the action of an element of $\Cal Z^{\sigma,\Lambda}(\bXmod)$. 
	So we can assume that the two cocycles are equal and $\psi_{V_{S}}=\psi'_{V_{S}}$ for all intersections $V_S$.
	Then the composition $G_{\Omega_{S}}(\xi):=\Psi_{\Omega_{S}}^{\circ(-1)}\circ\Psi'_{\Omega_{S}}(\xi)$ of the two normalizing cochains glue up together on the union of the outer domains, and by the $\sigma$-symmetry also on the union of the inner domains, and by the condition \textit{(3)} of Theorem~\ref{thm:cochain} the two agree also on the gate intersections.
	Therefore the cochain $\{\Omega_S\mapsto G_{\Omega_S}\}$ glues up together to a single bounded analytic $\sigma$-equivariant transformation
	$G_S$ on $B_S$, tangent to identity, and such that $G_S\circ\phi'^{\circ p}=\phi^{\circ p}\circ G_S$.
	Now if $G_S$ and $G_{\tilde S}$ are two such conjugating transformations above two different sectors $S$, $\tilde S$ with a nontrivial intersection, then 
	$F_{S,\tilde S}:=G_S\circ G_{\tilde S}^{\circ(-1)}$ defined on $B_{S\cap\tilde S}$ is $\sigma$-equivariant and commutes with $\phi^{\circ p}$, so according to Proposition~\ref{prop:automorphisms} it is equal to identity.
	Hence all the transformations $G_S$ over different sectors $S$ glue up to a single analytic $\sigma$-equivariant transformation $G$ on a full neighborhood of $0\in\C^2$.
\end{proof}

\subsubsection{Sectorial realization of the formal invariant $\hat\mu(h)$.}\label{sec:6.3.4}

By choosing to work in a model class (Definition~\ref{def:model}) we have forgotten about the formal invariant $\hat\mu(h)$
\Grn{in the formal case (b) (in the formal case (c) $\hat\mu(h)=0$).}
Let us show that we can not only recover this
formal invariant from each one of the classifying cocycles, but moreover we also obtain its sectorial realization $\mu_S(h)$ over each of the (cuspidal) sectors $S$.

For an outer Lavaurs domain $\Omega_{S}^j$ let a \emph{gate path} $\gamma_j$ be the oriented path between the two (possibly equal) equilibria to which the domain is attached, formed by a part of the positively oriented boundary  $\Omega_{S}^j$ that lies in the gate intersection domain of  $\Omega_{S}^j$ (if  $\Omega_{S}^j$ is sepal then  $\gamma^j$ is defined to be trivial, i.e. constant path consisting of the equilibrium only). 
Let $\alpha_{\Omega_{S}^j}=\bt\circ\Psi_{\Omega_{S}^j}-\bt$ be as in \eqref{eq:alpha} and let $\bX_{\Omega_{S}^j}$ be the associated Lavaurs vector field \eqref{eq:LavaursX}, then define
\begin{equation}\label{eq:gateintegral}
 \eta_{\Omega_{S}^j}(h):=\int_{\gamma_j}(h^s\bX_{\Omega_{S}^j,h}^{-1}-h^s\bXmodh^{-1})=\int_{\gamma_j}\bE.\alpha_{\Omega_{S}^j,h}\bE^{-1}=\big[\alpha_{\Omega_{S}^j}\big]_{\gamma_j},
\end{equation}
be an integral over the gate path, independent of the choice of the Fatou coordinate.

\begin{proposition}\label{prop:muLavaurs}
Given a sector $S$ and the associated cochain of normalizing transformations, let $ \eta_{\Omega_{S}^j}(h)$ be the gate integrals \eqref{eq:gateintegral},
and let $\beta_{V_{S}^j}^{(0)}(h)$ be the constant terms in the Fourier representation of the cocycle \eqref{eq:betafourier}.
Then
\begin{equation}\label{eq:muLavaurs} 
\tfrac{1}{2\pi i}\sum_{j=0}^{2kp-1} \eta_{\Omega_{S}^j}(h)=\tfrac{1}{2\pi i}\sum_{j=0}^{2kp-1}\beta_{V_{S}^j}^{(0)}(h):= \mu_S(h).
\end{equation}
Clearly this sum $\mu_S(h)$ is an invariant of the cocycle with respect to the conjugation \eqref{eq:conjugationbeta}.

By the $\Lambda$-symmetry of the cocycle \eqref{eq:cocyclesymmetry}, 
\Grn{in the formal case (b)
\begin{equation}\label{eq:muSa}
	\tfrac{1}{2\pi i}\sum_{j=N}^{N+2k-1} \eta_{\Omega_{S}^j}(h)=\tfrac{1}{2\pi i}\sum_{j=N}^{N+2k-1}\beta_{V_{S}^j}^{(0)}(h)=\tfrac1p \mu_S(h)\quad\text{for any }\ N\in\Z_{2kp},
\end{equation}
while in the formal case (c) $\mu_S(h)=0$.
}
\end{proposition}

\begin{proof}
We have 
$\bt+\alpha_{\Omega_{S}^{j-1}}=\bt\circ\Psi_{\Omega_{S}^{j-1}}=\bt\circ\psi_{V_{S}^j}\circ\Psi_{\Omega_{S}^{j}}=\big(\bt+\beta_{V_{S}^j}\circ\bt\big)\circ\Psi_{\Omega_{S}^{j}} =\bt+\alpha_{\Omega_{S}^j} +\beta_{V_{S}^j}\circ(\bt+\alpha_{\Omega_{S}^j})$, hence
$\beta_{V_{S}^j}\circ(\bt+\alpha_{V_{S}^j}) =\alpha_{\Omega_{S}^{j-1}}-\alpha_{\Omega_{S}^j}$. 
Evaluating at the limit at the equilibrium $a_j(h)$ to which the intersection  $V_{S}^j$ is attached we obtain
\begin{equation}\label{eq:betaalpha}
\beta_{V_{S}^j}^{(0)}=\alpha_{\Omega_{S}^{j-1}}(a_j)-\alpha_{\Omega_{S}^j}(a_j).  
\end{equation}
At the same time $ \eta_{\Omega_{S}^{j}}=\alpha_{\Omega_{S}^j}(a_{j+1})-\alpha_{\Omega_{S}^j}(a_{j})$. Summing over $j\in\Z_{2kp}$ gives \eqref{eq:muLavaurs}.	
\end{proof}

\begin{proposition}\label{prop:muasymptotic}
	The function $\mu_S(h)$ is asymptotic to the formal invariant $\hat\mu(h)=\sum_{l=0}^{+\infty} \mu_lh^l$ on the sector $S$, i.e. for all $N\in\Z_{\geq0}$
	\[\big|\mu_S(h)-\sum_{l=0}^N\mu_lh^l \big|=O(h^{N+1}),\quad h\in S. \]
\end{proposition}

\begin{proof}
	By \eqref{eq:muLavaurs} and \eqref{eq:gateintegral}
	$\mu_S(h)=
	\tfrac{1}{2\pi i}\sum_{j=0}^{2kp-1}\int_{\gamma_j}\bE.\alpha_{\Omega_{S}^j}\bE^{-1}.$
	By Proposition~\ref{prop:asymptotic} each Lavaurs  vector field $\bX_{\Omega_{S}^j}$ is asymptotic to the formal infinitesimal generator $\hat\bX=\frac{h^scP}{1+cP\hat R}\bE$ of $\phi^{\circ p}$, i.e. each $\bE.\alpha_{\Omega_{S}^j}(\xi)$ is asymptotic to $\hat R(\xi)=\sum r_{\bm m}\xi^{\bm m}$,
	i.e. for every $n\in\Z_{>0}$, $\bE.\alpha_{\Omega_{S}^j}-j^{(n)}\hat R(\xi)=O(|\xi|^{n+1})$ uniformly on each $\Omega$,
	where $j^{(n)}\hat R(\xi)$ denotes the $n$-jet of $\hat R(\xi)$ with respect to te variable $\xi$.
	By definition $\hat \mu(h)=\sum_{l=0}^{+\infty} r_{l,l}h^{l}$ (cf. \pageref{page:mu}), which by residue theorem means that
\[\tfrac{1}{2\pi i}\sum_{j=0}^{2kp-1}\int_{\gamma_j}j^{(n)}\hat R(\xi)\bE^{-1}=\sum_{l=0}^{\lfloor\frac{n}{2}\rfloor} r_{l,l}h^{l}=j^{(n)}\hat\mu(h),\]
as on each leaf $\{h=\const\neq 0\}$ the path $\sum_{j=0}^{2kp-1}\gamma_j$ is homotopic to a simple positive loop around $0$ (Proposition~\ref{prop:gate}).
 	Therefore 
	\[\mu_S(h)-j^{(n)}\hat{\mu}(h)=\tfrac{1}{2\pi i}\sum_{j=0}^{kp-1}\int_{\gamma_j}\big(\bE.\alpha_{\Omega_{S}^j}-j^{(n)}\hat R(\xi)\big)\bE^{-1}=O(|\xi|^{n+1}),\]
	and the statement follows.
\end{proof}

\subsubsection{Canonical normalizing cochains}
 
For each (cuspidal) sector $S$ the vector field
\begin{equation}\label{eq:bXS}
	\bXnfS:=h^s\frac{cP(u,h)}{1+\mu_S(h)cP(u,h)}\bE=h^s\frac{\bY}{1+\mu_S(h)\bY.\log\xi_1},\qquad h\in S,
\end{equation}
with $\mu_S(h)$ \eqref{eq:muLavaurs}, can be thought of as a sort of \emph{sectorial realization} on $B_S$ of the formal normal form $\hatbXnf=h^s\frac{cP(u,h)}{1+\hat\mu(h)cP(u,h)}\bE$ of Theorem~\ref{thm:1} (with $\hat\mu(h)=0$ in the formal case (c)), to which it is asymptotic by Proposition~\ref{prop:muasymptotic}. 
This means that we have a sectorial realization 
\[\phinfS:=\Lambda\exp(\tfrac1p\bXnfS)\] 
of the formal normal form $\hatphinf=\Lambda\exp(\tfrac1p\hatbXnf)$ of Theorem~\ref{thm:1}.
Therefore we can now also construct a cochain of normalizing transformations for $\phi$ towards this sectorial normal form.
The great advantage of such normalizing cochain over the one of Theorem~\ref{thm:cochain} is that it can be chosen in a completely \emph{canonical way} 
\Grn{as we shall see in Theorem~\ref{thm:reducedcochain} below. Moreover it turns out that this new canonical normalizing cochain is in fact asymptotic to
 the formal normalizing transformation of Theorem~\ref{thm:1} and therefore it can be viewed as its \emph{sectorial realization}.}

As $h^s\int\big(\bXnfS^{-1}-\bXmod^{-1}=\mu_S(h)\int\bE^{-1}$, where $\int\bE^{-1}=\tfrac12\log\big(\tfrac{\xi_1}{\xi_2}\big)$, 
\footnote{Under this choice of the primitive $\int\bE^{-1}=\tfrac12\log\big(\tfrac{\xi_1}{\xi_2}\big)$
we have	$\big(\int\bE^{-1}\big)\circ\sigma=-\int\bE^{-1}$.} 
which means that the map 
\begin{equation}\label{eq:FS}
F_{S}(\xi):=\exp(t\bY)(\xi)\big|_{t=\tfrac12\mu_S(h)\log\big(\tfrac{\xi_1}{\xi_2}\big)}
\end{equation}
is such that (cf. Lemma~\ref{lemma:X})
\[\bXnfS=F_{S}^*(\bXmod).\] 
Note that $\tfrac12\log\big(\tfrac{\xi_1}{\xi_2}\big)$ is multi-valued on each leaf $\{h=\const\}$, and so is $F_S$ \eqref{eq:FS},
therefore we shall denote $F_{\Omega_S}$ the restriction of its branch to each domain $\Omega_S$.

Given a normalizing cochain $\{\Omega_S\mapsto\Psi_{\Omega_S}\}$ as in Theorem~\ref{thm:cochain} over a sector $S$, and some functional cochain $\{\Omega_S\mapsto C_{\Omega_S}(h)\}$ as in Theorem~\ref{thm:cochain}, then the transformation cochain $\{\Omega_S\mapsto\tilde\Psi_{\Omega_S}\}$:
\begin{equation}\label{eq:normalizedcochain}
\tilde\Psi_{\Omega_S}:= F_{\Omega_S}^{\circ(-1)}\circ\exp(C_{\Omega_S}(h)\bY)\circ \Psi_{\Omega_S}
\end{equation} 
is normalizing for $\phi^{\circ p}$ with respect to the sectorial normal form $\phinfS^{\circ p}=\exp(\bXnfS)$,
\[ \tilde\Psi_{\Omega_S}\circ\phi^{\circ p} =\exp(\bXnfS)\circ\tilde\Psi_{\Omega_S}.\]
As in Section~\ref{sec:6.3.1} associated to this cochain there is a cocycle of transition maps
\begin{equation}\label{eq:reducedcocycle} 
\tilde\psi_{V_{S}^j}= \tilde\Psi_{\Omega_{S}^{j-1}}\circ\Big(\tilde\Psi_{\Omega_{S}^{j}}\Big)^{\circ(-1)}, 
\end{equation}
which commute with $\phinfS^{\circ p}$ on the intersection sectors.
Letting 
\[\tilde\bt_S:=\bt\circ F_S=\bt+\tfrac12\mu_S(h)\log\big(\tfrac{\xi_1}{\xi_2}\big),\]
then we have again a Fourier representation of each cocycle $\{\tilde\psi_{V_{S}^j}\}$:
\begin{equation}\label{eq:tildebeta}
\tilde\beta_{V_{S}^j}\circ\tilde\bt_S:=\tilde\bt_S\circ\tilde\psi_{V_{S}^j}-\tilde\bt_S,\qquad
\tilde\beta_{V_{S}^j}(t,h)=\sum_{n\in\Z}\tilde\beta_{V_{S}^j}^{(n)}(h)e^{\frac{2\pi i nt}{h^s}}.
\end{equation}
We shall show that we can choose the constants $C_{\Omega_S}$ in \eqref{eq:normalizedcochain} in such a way so that the constant terms 
$\tilde\beta_{V_S^j}^{0}(h)$ in the Fourier representation \eqref{eq:tildebeta} are null for all $j\in\Z_{2kp}$.
Such cochain $\{\Omega_S\mapsto\tilde\Psi_{\Omega_S}\}$ will be uniquely determined.

\begin{theorem}[Existence of a canonical normalizing cochain]\label{thm:reducedcochain}
There exists a unique cochain of ``constants'' $\{\Omega_S\mapsto C_{\Omega_S}(h)\}$ such that the outer normalizing cochain $\tilde\Psi_{\Omega_S^j}(\xi)=\xi+\hot$ \eqref{eq:normalizedcochain}
satisfies:
	\begin{enumerate}[label=(\arabic*)]
	
	\item  it conjugates $\phi$ to the sectorial normal form $\phinfS=\Lambda\exp(\tfrac1p \bXnfS)$ 
	\[\tilde\Psi_{\Lambda^{n}(\Omega_S)}\circ\phi^{\circ n}=\phinfS^{\circ n}\circ\tilde\Psi_{\Omega_S}, \quad n=1,\ldots, p, \] %
	\item  it is $\sigma$-equivariant, that is	
	\[\tilde\Psi_{\sigma(\Omega_S)}\circ\sigma=\sigma\tilde\Psi_{\Omega_S}, \]
	\item  \Grn{if $\Omega_S$ and $\Omega_S'$ share a gate intersection then}
	\[\tilde\Psi_{\Omega_S}=\tilde\Psi_{\Omega_S'}\qquad \text{on the gate intersection.},\] 
	\setcounter{enumi}{-1}
	\item has vanishing constant Fourier coefficients of the transition maps, $\tilde\beta_{V_S^j}^{(0)}(h)=0$ in \eqref{eq:tildebeta}.
\end{enumerate}
Such normaizing cochain $\{\Omega_S\mapsto \tilde\psi_{\Omega_S}\}$ is \textbf{unique}\footnote{Under the assumption of tangency to identity, otherwise it would be determined up to left composition with elements of the centralizer $\Cal Z^{\sigma,\Lambda}(\bXnfS)=\Cal Z^{\sigma,\Lambda}(\bXmod)$ (Definition~\ref{def:centralizer}).} 
and \textbf{asymptotic} to the unique $\sigma$-equivariant formal normalizing transformation $\hat\Psi(\xi)=\xi+\hot$ that conjugates
$\phi\circ\hat\Psi=\hat\Psi\circ\hatphinf$.	
\end{theorem}

\begin{corollary}
	If two $\sigma$-reversible analytic germs $\phi(\xi)$, $\phi'(\xi')$ are formally equivalent by a transformation $\xi'=\hat\Psi(\xi)$, then there 
	exists a uniquely determined $\sigma$-equivariant  cochain of transformations $\{\Omega_S\mapsto\Phi_{\Omega_{S}}\}$ asymptotic to the formal transformation $\hat\Phi$ that conjugates $\phi$ to $\phi'$.
\end{corollary}	

\begin{proof}
If $\tilde\Psi_{\Omega_S}= F_{\Omega_S}^{\circ(-1)}\circ\exp(C_{\Omega_S}\bY)\circ \Psi_{\Omega_S}$  and $\tilde\Psi_{\Omega_S}'= {F_{\Omega_S}'}^{\circ(-1)}\circ\exp(C'_{\Omega_S}\bY)\circ \Psi'_{\Omega_S}$ are the normalizing cochains \eqref{eq:normalizedcochain} of Theorem~\ref{thm:reducedcochain} for $\phi$ and $\phi'$, then 
\[\Phi_{\Omega_{S}}:={\Psi'_{\Omega_S}}^{\circ(-1)}\circ\exp\big((C_{\Omega_S}-C'_{\Omega_S})\bY\big)\circ \Psi_{\Omega_S}=
\big(\tilde\Psi_{\Omega_S}'\big)^{\circ(-1)}\circ {F_{\Omega_S}'}^{\circ(-1)}\circ F_{\Omega_S}\circ\tilde\Psi_{\Omega_S}\] 
is a conjugating transformation between $\phi$ and $\phi'$.
And since their invariants $\mu_S$, $\mu'_S$ are asymptotic to the same formal $\hat\mu$ the composition $F_S'\circ F_S^{\circ(-1)}$ is asymptotic to identity, therefore $\Phi_{\Omega_{S}}$ is asymptotic to the composition of the formal normalizing transformations $\hat\Phi=\big(\hat\Psi'\big)^{\circ(-1)}\circ \hat\Psi$ of Theorem~\ref{thm:1}.
\end{proof}

\begin{lemma}\label{lemma:beta}
	For each equilibrium $a(h)$ of $\bY_h$ the sum of $\beta_{V_S}^{(0)}(h)$ over all the intersections attached to $a(h)$ is null.
	Here the sum is taken over both outer and inner intersections, where for an inner intersection $V_S$, $\beta_{V_S}^{(0)}(h):=-\beta_{\sigma(V_S)}^{(0)}(h)$. 
\end{lemma}

\begin{proof}
	The identity \eqref{eq:betaalpha} expresses $\beta_{V_S}^{(0)}$ as the difference $\alpha_{\Omega_S}(a(h))-\alpha_{\tilde\Omega_S}(a(h))$,
	where $\Omega_{S}$ is the Lavaurs domain on the right of $V_S$ and $\tilde\Omega_S$ the one on the left.
	While the determination of $\bt$ determines also each of the $\alpha$ by an additive constant, the difference of them is independent of it as long as $\bt_{\Omega_S}=\bt_{\tilde\Omega_S}$ on the intersection. Hence it doesn't matter if it is an outer or an inner intersection. 
	\Grn{Moreover, if $\Omega_S$ and $\Omega_S'$ share a gate intersection, then $\alpha_{\Omega_S}=\alpha_{\Omega_S'}$.
	Hence when one expresses the sum of all the $\beta^{(0)}$'s in terms of $\alpha(a)$'s, they all cancel out.}
\end{proof}

\begin{proof}[Proof of Theorem~\ref{thm:reducedcochain}]
Let us show that the constants $C_{\Omega_S}(h)$ in \eqref{eq:normalizedcochain} can be chosen so that the cochain $\tilde\Psi_{\Omega_S}$ satisfies both \textit{(0)} and \textit{(1)-(3)}. 	
	
We may assume that the enumeration of the outer Lavaurs domains is such that $\Omega_S^0$ shares a gate with some inner domain $\sigma(\Omega_S^l)$ for some $l$
(such pair always exist by Proposition~\ref{prop:gate}).
Choosing some $C_{\Omega_S^0}$ on $\Omega_S^0$, then on the following outer domains $\Omega_S^1,\ldots,\Omega_S^{2kp-1}$ (in counterclockwise order) and $\Omega_S^{2kp}=e^{2\pi i J}(\Omega_S^0)$, where $J=\left(\begin{smallmatrix} 1&0\\[3pt] 0&-1\end{smallmatrix}\right)$,
we need to take $C_{\Omega_S^j}:=C_{\Omega_S^0}+\beta_{V_{S}^{1}}^{(0)}(h)+\ldots+\beta_{V_{S}^{j}}^{(0)}(h)$, $j=1,\ldots,2kp$, 
so that $C_{\Omega_S^{j}}-C_{\Omega_S^{j-1}}=\beta_{V_{S}^{j}}^{(0)}(h)$.
Choosing the determinations $F_{\Omega_S^j}$ of \eqref{eq:FS} such that they agree on the intersections $V_S^j\subseteq\Omega_{S}^j\cap\Omega_S^{j-1}$, $j=1,\ldots,2kp$, 
we then get ${\beta_{V_S^j}'}^{(0)}=\beta_{V_S^j}^{(0)}+C_{\Omega_S^{j-1}}-C_{\Omega_S^j}=0$, so \textit{(0)} is satisfied.
By Proposition~\ref{prop:muLavaurs}, $C_{\Omega_S^{2kp}}=C_{\Omega_S^0}+2\pi i\mu_S$,
while at the same time  
$F_{\Omega_S^{2kp}}=\exp(2\pi i\mu_S\bY)\circ F_{\Omega_S^{0}}$,
which means that $\tilde\Psi_{\Omega_S^{2kp}}=\tilde\Psi_{\Omega_S^0}$ is well defined.

On the inner domains $\sigma(\Omega_S^0),\ldots,\sigma(\Omega_S^{2kp-1}),\sigma(\Omega_S^{2kp})$ (in clockwise order)
we need to take $\tilde\Psi_{\sigma(\Omega_S^{j})}=\sigma\tilde\Psi_{\Omega_S^{j}}\circ\sigma$ in order to satisfy \textit{(2)}.
We now choose $C_{\Omega_S^0}$ so that on the gate intersection between the outer domain $\Omega_S^0$ and its inner neighbor $\sigma(\Omega_S^l)$
the two transformation agree $\tilde\Psi_{\Omega_S^{0}}=\tilde\Psi_{\sigma(\Omega_S^{l})}$. 
If one selects the two branches of \eqref{eq:FS} such that $F_{\Omega_S^{0}}=F_{\sigma(\Omega_S^{l})}$, then this translates to asking that
$C_{\Omega_S^0}=-C_{\Omega_S^l}$.

We now need to verify that the properties \textit{(1)} and \textit{(3)} hold, while the properties \textit{(2)} and \textit{(0)} are already satisfied by our construction.

\Grn{
Let us verify \textit{(1)}: 
Using \eqref{eq:normalizedcochain} and the condition \textit{(1)} of Theorem~\ref{thm:cochain}, this is equivalent to
\[F_{\Lambda(\Omega^j_S)}^{\circ(-1)}\circ\exp\big(C_{\Lambda(\Omega_S^{j})}\bY\big)=\Lambda F_{\Omega^j_S}^{\circ(-1)}\circ\Lambda^{-1}\circ\exp\big(C_{\Omega_S^{j}}\bY\big).\]
If $\Lambda=\left(\begin{smallmatrix}\lambda&0\\0&\lambda^{-1}\end{smallmatrix}\right)$, then on one hand, by Proposition~\ref{prop:muLavaurs}, $C_{\Lambda(\Omega_S^{j})}=C_{\Omega_S^j}+\mu_S\log\lambda$, on the other hand we have $\Lambda^{-1}F_{\Lambda(\Omega_S^j)}\circ\Lambda=\exp\big(\mu_S\log\lambda\,\bY\big)\circ F_{\Omega_S^j}$, 
and the condition is satisfied.
}

Let us verify \textit{(3)}:
Removing from the gate graph the edge corresponding to the gate intersection between $\Omega_S^0$ and its inner neighbor we obtain a tree by Proposition~\ref{prop:gate}.
For any gate (edge) of this tree, any branch of the tree starting at this gate has the sum of all the $\beta$'s at its vertices null by Lemma~\ref{lemma:beta}, which means that the constants $C_{\Omega_{S}}$ at the two sides of this gate are the same. And the determinations of $F_{\Omega_{S}}$ are chosen so that they agree along gates of the branch as well as along the intersection sectors attached to its vertices.

Let us now prove the asymptoticity of the normalizing cochain $\{\Omega_S\mapsto\tilde\Psi_{\Omega_S}\}$ \eqref{eq:normalizedcochain}. 
Let $\alpha_{\Omega_S}=\bt\circ\Psi_{\Omega_S}-\bt$ \eqref{eq:alpha}, and let
\[a_{\Omega_S}:=\tilde\bt_{S}\circ\tilde\Psi_{\Omega_S}-\tilde\bt_{S}=\alpha_{\Omega_S}(\xi)+C_{\Omega_S}(h)-\tfrac12\mu_S(h)\log\big(\tfrac{\xi_1}{\xi_2}\big).\]
This means that $\tilde\Psi_{\Omega_S}=\exp(th^{-s}\bXnfS)\big|_{t=a_{\Omega_S}}$, and (see \eqref{eq:LavaursX})
\[\bX_{\Omega_S}=\frac{\bXmod}{1+h^{-s}\bXmod.\alpha_{\Omega_S}}=\frac{\bXnfS}{1+h^{-s}\bXnfS.a_{\Omega_S}}.\] 
Correspondingly, let $\hat\bX$ be the formal infinitesimal generator of $\phi^{\circ p}=\exp(\hat{\bX})$, let $\hatbXnf=\frac{\bXmod}{1+h^{-s}\bXmod.\big(\tfrac12\hat\mu\log\big(\tfrac{\xi_1}{\xi_2}\big)\big)}$ be its formal normal form,
and let $\hat a(\xi)$ be the formal power series such that
$\hat a\circ\sigma=-\hat a$ and 
\[\hat\bX=\frac{\bXmod}{1+h^{-s}\bXmod.\big(\tfrac12\hat\mu\log\big(\tfrac{\xi_1}{\xi_2}\big)+\hat a\big)}=\frac{\hatbXnf}{1+h^{-s}\hatbXnf.\hat a}.\]
Then the $\sigma$-equivariant formal normalizing transformation $\hat\Psi$, such that $\hat\bX=\hat\Psi^*\hatbXnf$, can be also expressed as  $\hat\Psi=\exp(th^{-s}\hatbXnf)\big|_{t=\hat a}$. 
We know that $\mu_S$ is asymptotic to $\hat\mu$, hence $\bXnfS$ is asymptotic to $\hatbXnf$, so all we need to prove is that $a_{\Omega_S}$ is asymptotic to $\hat a$.

We also know that $\bX_{\Omega_S}$ is asymptotic to $\hat\bX$, which means that $\bE.a_{\Omega_S}$ is asymptotic to $\bE.\hat a$.
In the notation of the proof of Proposition~\ref{prop:asymptotic} this means that for any $n\in\Z_{>0}$
\[\bE.(a_{\Omega_S}-j^{(n)}\hat a)=0\mod \Cal J_{\Omega_S}^n, \]
which implies that $a_{\Omega_S}-j^{(n)}\hat a=c_{n,\Omega_S}(h)\mod \Cal J_{\Omega_S}^n$ for some bounded function $c_{n,\Omega_S}(h)$ on $S$.

On the outer intersection $V_S^j$ we have
\begin{align*}\tilde\beta_{V_S^j}\circ\tilde\bt_S=\big(\tilde\bt_{S}\circ\tilde\Psi_{\Omega_S^{j-1}}-\tilde\bt_{S}\circ\tilde\Psi_{\Omega_S^j}\big)\circ\tilde\Psi_{\Omega_S^j}^{\circ(-1)}
&=\big(a_{\Omega_S^{j-1}}-a_{\Omega_S^{j}}\big)\circ\tilde\Psi_{\Omega_{S}^{j}}^{\circ(-1)}\\
&=c_{n,\Omega_S^{j-1}}-c_{n,\Omega_S^{j}}\mod \Cal J_{V_S^j}^n, 
\end{align*}
with $\tilde\beta_{V_S^j}\circ\tilde\bt_S=\sum_{n\in\Z}\tilde\beta_{V_{S}^j}^{(n)}(h)e^{\frac{2\pi i n\tilde\bt_S}{h^s}}$ on the left side.
By the construction of the Lavaurs domains, 
\[\RE\big(-2\pi in\tfrac{\tilde\bt_S}{h^s}\big)\geq 2\pi n\sin\delta_3\tfrac{|\tilde\bt_S|}{|h|^s}\geq An\sin\delta_3|\xi|^{-2kps-2s},\]
for some $A>0$, where $s$ is as in Lemma~\ref{lemma:asymptotic-h}.
Since $\beta_{V_{S}^j}^{(0)}=0$, this means that the left side is exponentially flat in $|\xi|$, and therefore 
$c_{n,\Omega_S^{j-1}}-c_{n,\Omega_S^{j}}=0\mod \Cal J_{V_S^j}^n$.
We also have 
\[a_{\sigma(\Omega_S)}\circ\sigma=\big(\tilde\bt_S\circ\tilde\Psi_{\sigma(\Omega_S)}-\tilde\bt_S\big)\circ\sigma=-\tilde\bt\circ\tilde\Psi_{\Omega_S}+\tilde\bt_S=-a_{\Omega_S}, \]
which means that $c_{n,\sigma(\Omega_S)}(h)=-c_{n,\Omega_S}(h)\mod \Cal J_{\Omega_S}^n$.
We conclude that $c_{n,\Omega_S}=0\mod \Cal J_{\Omega_S}^n$ for all Lavaurs domains $\Omega_S$, and for any $n\in\Z_{>0}$. Hence $a_{\Omega_S^j}(\xi)$ is asymptotic to $\hat a(\xi)$, which we wanted to prove. 
\end{proof}

\GRN

\subsubsection{Proof of Theorem~\ref{thm:analytic}.}

\begin{proof}[Proof of Theorem~\ref{thm:analytic}]
The proof of Theorem~\ref{thm:analyticclassification} establishes that for any sector $S$, the two pairs $(\phi,\sigma)$ and $(\phi',\sigma)$ are conjugated by $G_S\in\Diff_{\id}^{h}(B_S,0)$ if and only if their outer cocycles over $S$ agree.
Let us show that this in fact implies that the analytic conjugacy extends from $B_S$ to a full neighborhood of $0\in\C^2$.

Since the cocycles over $S$ agree, then also the sectoral invariants $\mu_S=\mu_S'$ agree. So let $\{\Omega_{S}\mapsto\tilde\Psi_{\Omega_{S}}\}$, $\{\Omega_{S}\mapsto\tilde\Psi'_{\Omega_{S}}\}$ be the canonical normalizing cochains of Theorem~\ref{thm:reducedcochain} for $\phi,\phi'$, and let
$\{V_{S}\mapsto\tilde\psi_{V_{S}}\}$, $\{V_{S}\mapsto\tilde\psi'_{V_{S}}\}$ be the associated transition maps \eqref{eq:reducedcocycle}.
By the unicity, $\tilde\Psi'_{\Omega_{S}}=\tilde\Psi_{\Omega_{S}}\circ G_S$ for all $\Omega_{S}$, and therefore the transition maps agree $\tilde\psi'_{V_{S}}=\tilde\psi_{V_{S}}$.

Now if $\tilde S$ is another sector with non-trivial intersection $\tilde S\cap S\neq\{0\}$, then by the same argument, on this intersection one has $\mu_{\tilde S}=\mu'_{\tilde S}$, $\tilde\Psi'_{\Omega_{S}}=\tilde\Psi_{\Omega_{S}}\circ G_S$ and $\tilde\psi'_{V_{\tilde S}}=\tilde\psi_{V_{\tilde S}}$, and therefore it is true on the whole $\tilde S$. Hence the conjugacy $G_S$ extends  analytically also to $\tilde S$.

Repeating this argument we see that the conjugacy $G_S$ is in fact analytic on the union of all the domains $B_S$, that is, on a full neighborhood of $0\in\C^2$.
\end{proof}

\FGRN

\subsubsection{Compatibility of cocycles.}\label{sec:6.3compatibility}

Over each sector $S$ in the $h$-space the associated cocycle carries complete information about dynamics of $\phi$ on $B_S$. On the intersection of two different sectors $S$ and $S'$, the two cocycles have to describe the same dynamics on $B_{S\cap S'}$, therefore cannot be independent of each other. Our goal is to formulate a necessary compatibility condition between them.
We shall identify the cocycles with certain ``pseudo-representations'' of a fundamental groupoid $\Pi_1(B_S\sminus\{P=0\},\Ends_S)$ (Definition~\ref{def:fundamentalgroupoid}) in the pseudogroup of transformations commuting with $\phimod^{\circ p}$, and express the compatibility condition as a conjugacy of these ``pseudo-representations''. 
The basic idea is an analogy with the compatibility conditions encountered in problems of unfolding of moduli spaces in \cite{Hurtubise-Rousseau, Rousseau-Teyssier2}, where they are formulated in terms of conjugacy of monodromies of linear systems \cite{Hurtubise-Rousseau} or of holonomies of foliations \cite{Rousseau-Teyssier2}.
One difference is that we shall work with fundamental groupoids in place of fundamental groups, which is a natural generalization in situations involving a (non-linear) Stokes phenomenon. 
Another difference lies in the general impossibility of composition of transformations associated to different paths due to potential lack of their analytic extendability, thus the name ``pseudo-representation''.

\begin{definition}[Fundamental groupoid]\label{def:fundamentalgroupoid}
An \emph{end} of a Lavaurs domain $\Omega_{S,h}$ is the intersection of the closure of $\Omega_{S,h}$ with the boundary of $B_h$, and is identified with some marked point $e_{\Omega_{S,h}}\in\partial B_h\cap \ov\Omega_{S,h}$ (the choice is such that it depends continuously on $h\in S$ and respects the cyclic ordering of the inner/outer domains and the actions of $\sigma,\Lambda$). 
	
For a given sector $S$ in the $h$-space, and $h\in S$ we denote $\Ends_{S,h}$ the set of the  ends of the $4kp$ associated domains $\Omega_{S,h}$.
The fundamental groupoids $\Pi_1(B_h\sminus\{P_h=0\},\Ends_{S,h})$, $h\in S^*=S\sminus\{0\}$, consisting of relative classes of paths in $B_h\sminus\{P_h=0\}$ with fixed endpoints in $\Ends_{S,h}$, are identified with each other over the sector $S^*$ as a single groupoid
\[\Pi_1(B_{S^*}\sminus\{P=0\},\Ends_{S^*}),\qquad B_{S^*}=\coprod_{h\in S^*}B_h,\quad \Ends_{S^*}=\coprod_{h\in S^*}\Ends_{S,h}.\] 
\end{definition}

The Fatou coordinates $\bT_{\Omega_S}$ for $\phi^{\circ p}$ on the Lavaurs domains $\Omega_S$ are characterized by two conditions
\begin{enumerate}  \setlength\itemsep{0em}
	\item Fatou relation (Abel equation): $\bT_{\Omega_S}\circ\phi^{\circ p}=\bT_{\Omega_S}+h^s$,
	\item asymptotic condition: $\bT_{\Omega_S}-\bt$ has at most moderate growth on $\Omega_S$.
\end{enumerate}
The second condition assures its uniqueness up to addition of a constant. 
We will look at what information do the cochains of Fatou coordinates carry when the asymptotic condition is forgotten. 
In that case, the coordinate $\bT_{\Omega_S}$ is determined up to an addition of a term $m_{\Omega_S}\circ\bT_{\Omega_S}$ where $m_{\Omega_S}(t)=m_{\Omega_S}(t+h^s)$ can be any $h^s$-periodic function.
Equivalently, the normalizing cochain $\{\Omega_S\mapsto\Psi_{\Omega_S}\}$ is then defined only up to left composition with any cochain of transformations $\{\Omega_S\mapsto M_{\Omega_S}\}$ commuting with $\phimod^{\circ p}=\exp(\bXmod)$.
The point is that in the absence of the asymptotic condition the form of the domains $\Omega_S$ no longer plays a role,
so the cochains can be thought as defined on (neighborhoods of) the ends  $\{e_{\Omega_S}\mapsto\Psi_{\Omega_S}\}$, 
which then allows to compare the cocycles associated to normalizing cochains on different coverings over different sectors $S$.

\begin{definition}[Fatou pseudo-representations]
The set of Lavaurs domains $\{\Omega_S\}$ over a sector $S$ is in bijective correspondence with the set of ends $\Ends_S=\{e_{\Omega_S}\}$.
\emph{Elementary paths} in  $\Pi_1(B_{S^*}\sminus\{P=0\},\Ends_{S^*})$ are those that correspond to 
\begin{itemize}
	  \setlength\itemsep{0em}
	\item[-] either inner or outer intersections: they connect the ends of two neighboring inner or outer Lavaurs domains and lie inside their union, 
	\item[-] or gate intersections: they connect the ends of two  Lavaurs domains sharing a gate intersection and lie inside their union.
\end{itemize}
Clearly, the elementary paths generate the fundamental groupoid $\Pi_1(B_{S^*}\sminus\{P=0\},\Ends_{S^*})$.
	
A cocycle $\{V_S\mapsto\psi_{V_{S}}\}$ over $S$ gives rise to a \emph{Fatou pseudo-representation} $\mathfrak{r}=\mathfrak{r}_S$ of $\Pi_1(B_{S^*}\sminus\{P=0\},\Ends_{S^*})$ by
associating to each elementary path 
	$e_{\Omega_{S}}\xrightarrow{\gamma_{V_S}} e_{\Omega_{S}'}$ corresponding to an outer/inner/gate intersection $V_S\subseteq \Omega_{S}\cap \Omega_{S}'$ 
	a transformation	
	\[\mathfrak{r}(\gamma_{V_S}):=\psi_{V_S},\qquad 
	\bt\circ\mathfrak{r}(\gamma_{V_S})=\bt+\beta_{V_S}\circ\bt, \] 
	where $\beta_{V_S}(t)=\sum_{n\in\Z}\beta_{V_S}^{(n)}(h)\,e^{\frac{2n\pi i t}{h^s}}$
	is an analytic $h^s$-periodic function on $\bt(V_S)$.
Whenever some elementary paths $e_{\Omega_{0,S}}\xrightarrow{\gamma_{V_{1,S}}} e_{\Omega_{1,S}},\ldots,e_{\Omega_{n-1,S}}\xrightarrow{\gamma_{V_{n,S}}} e_{\Omega_{n,S}}$,
correspond to intersections $V_{1,S},\ldots,V_{n,S}$ attached to the same simple singularity, then to their composition $e_{\Omega_{0,S}}\xrightarrow{\gamma_{V_{1,S}}\cdots\gamma_{V_{n,S}}} e_{\Omega_{n,S}}$
is associated the corresponding composition of the transformations
\[\mathfrak{r}(\gamma_{V_{1,S}}\cdots\gamma_{V_{n,S}})=\mathfrak{r}(\gamma_{V_{1,S}})\circ\ldots\circ\mathfrak{r}(\gamma_{V_{n,S}})\]
which in this case is well-defined because each $\beta_{V_{j,S}}$ extends by the $h^s$-periodicity to a $\bt$-image of a full neighborhood of the singularity.

Two pseudo-representations $\mathfrak{r},\mathfrak{r}'$ are \emph{conjugated} if there exists a map $M$ that associates to each end $e=e_{\Omega_S}\in\Ends_{S^*}$ a transformation 
$M_{e_{\Omega_S}}:e_{\Omega_{S}}\to e_{\Omega_{S}}$ commuting with $\exp(h^s\bY)$
that conjugates 
\[\mathfrak{r}'(\gamma)=M_{e'}\circ\mathfrak{r}(\gamma)\circ M_{e}^{\circ(-1)}, \quad \text{for }\  e\xrightarrow{\gamma} e'.\]
\end{definition}

So far the Fatou pseudo-representations are not much more than just cocycles minus the asymptotic condition. 
The point is that when passing from one sector $S$ to another sector $S'$ with a non-trivial intersection $S\cap S'$, there is a natural identification of $\Ends_{S^*}$ and $\Ends_{{S'}^*}$:
in fact, the ends of the zones of $e^{i\theta}h^s\bY_h$ (p.~\pageref{page:ends}) depend continuously on $\theta\in\,\,]\delta_3,\pi-\delta_3[$ by rotating, even for unstable values of $\theta$, and this correspondence is carried also to the ends of the Lavaurs domains.
Therefore we can identify the fundamental groupoids 
\begin{equation}\label{eq:fundamentalgroupoid}
\begin{aligned}\Pi_1(B_{S^*}\sminus\{P=0\},\Ends_{S^*})&\simeq \Pi_1(B_{S^*\cap {S'}^*}\sminus\{P=0\},\Ends_{S^*\cap {S'}^*})\\
	&\simeq \Pi_1(B_{{S'}^*}\sminus\{P=0\},\Ends_{{S'}^*}).
\end{aligned}
\end{equation}

\begin{theorem}[Compatibility condition]
For two different sectors $S$, $S'$ with a non-trivial intersection	$S\cap S'$, the Fatou pseudo-representations $\mathfrak{r}_{S},\mathfrak{r}_{S'}$ of 
\eqref{eq:fundamentalgroupoid} 
generated by the classifying cocycles
$\{\mathfrak{r}_{S}(\gamma_{V_S})=\phi_{V_S}\}$ and $\{\mathfrak{r}_{S'}(\gamma_{V'_{S'}})=\phi_{V'_{S'}}\}$ are conjugated to each other.	
\end{theorem}
\begin{proof}
If the end $e=e_{\Omega_S}=e_{\Omega'_{S'}}$ of the Lavaurs domain $\Omega_{S}$ over $S$ is identified with the end of $\Omega'_{S'}$ over $S'$, then the map $e\mapsto M_e:=\Psi_{\Omega'_{S'}}\circ \Psi_{\Omega_{S}}^{\circ(-1)}$ conjugates $\mathfrak{r}_S$ and $\mathfrak{r}_{S'}$.
\end{proof}

\begin{figure}[t]
\centering
\begin{subfigure}{.45\textwidth} \includegraphics{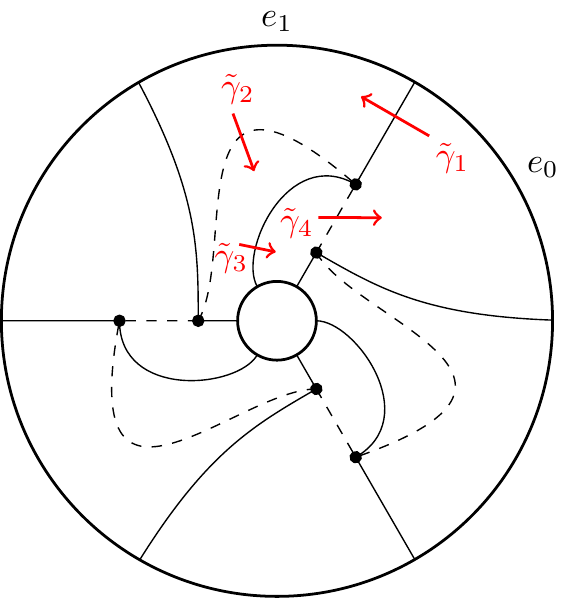}  \caption{} \label{figure:compatibility-b} \end{subfigure}
\qquad
\begin{subfigure}{.45\textwidth} \includegraphics{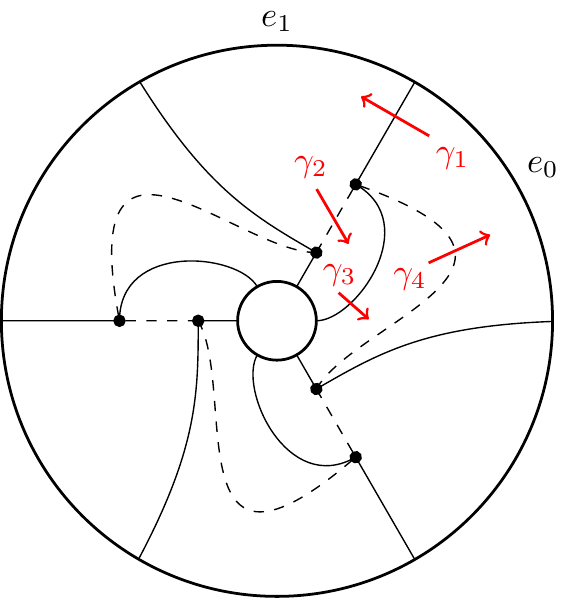} \caption{} \label{figure:compatibility-a} \end{subfigure}
	\caption{Schematic depiction of the topological organization of the Lavaurs domains corresponding to Figure~\ref{figure:vf3-a} and~\ref{figure:vf3-c} for $k=1$, $p=3$
		(see also Figure~\ref{figure:sectors}). 
	The dashed lines represent gate intersections, while the full lines represent outer and inner intersections. The arcs of the outer and inner boundary circle represent ends of the domains, and the red arrows represent selected paths between the ends of the corresponding domains.}
	\label{figure:compatibility}
\end{figure}

Let us illustrate this compatibility conditions on an example.

\begin{example}
Let $k=1$, and consider the case when all the equilibria are simple, and let us look at the two neighboring sectors $\tilde S$, resp. $S$, over which the topological organization of the Lavaurs domains is as depicted in Figure~\ref{figure:compatibility} (see also Figure~\ref{figure:sectors}).	
Under the identification of ends
\[e_0=e_{\Omega_{\tilde S}^0}=e_{\Omega_{S}^0},\qquad e_1=e_{\Omega_{\tilde S}^1}=e_{\Omega_{S}^1},\]
the four successive paths
$\ e_{\Omega_{\tilde S}^0} \xrightarrow{\tilde\gamma_1} e_{\Omega_{\tilde S}^1} \xrightarrow{\tilde\gamma_2} e_{\sigma\Lambda(\Omega_{\tilde S}^1)} \xrightarrow{\tilde\gamma_3} e_{\Lambda\sigma(\Omega_{\tilde S}^0)} \xrightarrow{\tilde\gamma_4} e_{\Omega_{\tilde S}^0},\ $
resp.
$\ e_{\Omega_{S}^0} \xrightarrow{\gamma_1} e_{\Omega_{S}^1} \xrightarrow{\gamma_2} e_{\Lambda\sigma(\Omega_{S}^1)} \xrightarrow{\gamma_3} e_{\sigma(\Omega_{ S}^0)} \xrightarrow{\gamma_4} e_{\Omega_{S}^0},\ $
which form a simple loop around the same simple equilibrium $a(h)$, have pseudo-representations
\begin{align*}
\mathfrak{r}_{\tilde S}(\tilde\gamma_1)&=\psi_{V_{\tilde S}^1} &
\mathfrak{r}_{\tilde S}(\tilde\gamma_2)&=\exp\big(-\nu_{0\infty}\bY\big)\\
\mathfrak{r}_{\tilde S}(\tilde\gamma_3)&=\Lambda\sigma\psi_{V_{\tilde S}^2}\circ(\Lambda\sigma) &
\mathfrak{r}_{\tilde S}(\tilde\gamma_4)&=\exp\big((\nu_a+\nu_{0\infty})\bY\big),
\end{align*}
resp.
\begin{align*}
\mathfrak{r}_{S}(\gamma_1)&=\psi_{V_{S}^1}&
\mathfrak{r}_{S}(\gamma_2)&=\exp\big((\nu_a-\nu_{0\infty})\bY\big)\\
\mathfrak{r}_{S}(\gamma_3)&=\sigma\Lambda\psi_{V_{S}^2}\circ(\sigma\Lambda)&
\mathfrak{r}_{S}(\gamma_4)&=\exp\big(\nu_{0\infty}\bY\big),
\end{align*}
where $\nu_a(h)$ is the dynamical residue of $\bY_h^{-1}$ at the equilibrium $a(h)$ \eqref{eq:nu_a}, \eqref{eq:nu12}, and $\nu_{0\infty}(h)$ is the period \eqref{eq:nu08} of $\bY_h^{-1}$ along the path from $0$ to $\infty$ corresponding to $\gamma_4$.
The whole loop also has a well defined pseudo-representation
$\mathfrak{r}_{\tilde S}(\tilde\gamma_1\tilde\gamma_2\tilde\gamma_3\tilde\gamma_4)=\mathfrak{r}_{\tilde S}(\tilde \gamma_1)\circ\mathfrak{r}_{\tilde S}(\tilde \gamma_2)\circ\mathfrak{r}_{\tilde S}(\tilde \gamma_3)\circ\mathfrak{r}_{\tilde S}(\tilde \gamma_4)$,
resp.
$\mathfrak{r}_{S}(\gamma_1\gamma_2\gamma_3\gamma_4)=\mathfrak{r}_{S}(\gamma_1)\circ\mathfrak{r}_{S}(\gamma_2)\circ\mathfrak{r}_{S}(\gamma_3)\circ\mathfrak{r}_{S}(\gamma_4)$.
Now the compatibility condition between the cocycles on the two sector demands that there exists a pair of transformations  $M_{e_0}$, $M_{e_1}$ 
(in our case given by $M_{e_j}=\Psi_{\Omega_{\tilde S}^{j}}\circ \Psi_{\Omega_{S}^{j}}^{\circ(-1)}$, $j=0,1$), 
commuting with $\exp(\bXmod)$, such that:
\begin{enumerate}
	\item $\psi_{V_{\tilde S}^1}\circ M_{e_1}=M_{e_0}\circ\psi_{V_{S}^1}$,
	\item $\psi_{V_{\tilde S}^2}\circ (\Lambda M_{e_0}\circ\Lambda^{-1})=M_{e_1}\circ\psi_{V_{S}^2}$,
	\item $\mathfrak{r}_{\tilde S}(\tilde\gamma_1\tilde\gamma_2\tilde\gamma_3\tilde\gamma_4)\circ M_{e_0}=M_{e_0}\circ \mathfrak{r}_{S}(\gamma_1\gamma_2\gamma_3\gamma_4)$.
\end{enumerate}
All other conjugacy relations follow from these three by $(\sigma,\Lambda)$-equivariance.
\end{example}

\goodbreak

\bibliographystyle{alpha}

\end{document}